\documentclass[11pt,twoside]{article}

\usepackage{amsmath}
\usepackage{amssymb}
\usepackage{mathrsfs}
\usepackage{amsthm}

\usepackage{latexsym}

\usepackage{indentfirst}
\usepackage{color}
\usepackage{txfonts}

\usepackage{anysize}

\allowdisplaybreaks

\usepackage[colorlinks=true,
  linkcolor=red,
  citecolor=blue,
  urlcolor=magenta]{hyperref}

\newtheorem{theorem}{Theorem}[section]
\newtheorem{lemma}[theorem]{Lemma}

\newtheorem{proposition}[theorem]{Proposition}

\theoremstyle{definition}
\newtheorem{remark}[theorem]{Remark}
\newtheorem{definition}[theorem]{Definition}

\numberwithin{equation}{section}

\begin{document}
\title{\bf\Large Matrix-Weighted Variable Besov Spaces
\footnotetext{\hspace{-0.35cm} 2020 {\it
Mathematics Subject Classification}. Primary 46E35; Secondary 47A56, 15A15,
46E40, 42B35. \endgraf
{\it Key words and phrases}. variable Besov space,
variable matrix weight, almost diagonal operator, atom,
molecule, wavelet, trace operator, Calder\'on--Zygmund operator.
\endgraf
This project is partially supported by
the National Natural Science Foundation of China
(Grant Nos. 12371093 and 12431006), the Beijing Natural
Science Foundation (Grant No. 1262011),
and the Fundamental Research Funds for the
Central Universities (Grant No. 2253200028).
}}
\date{}
\author{}
\author{Dachun Yang\footnote{Corresponding author,
E-mail: \texttt{dcyang@bnu.edu.cn}/{\color{red}\today}/Final version.},\ \
Wen Yuan and Zongze Zeng}
\maketitle

\vspace{-0.8cm}

\begin{center}
\begin{minipage}{13cm}
{\small {\bf Abstract:}\quad
In this article, using variable matrix ${\mathscr{A}}_{p(\cdot),\infty}$
weights, we introduce the matrix-weighted variable Besov
space $B^{s(\cdot)}_{p(\cdot),q(\cdot)}(W)$
and the corresponding
averaging variable Besov space $B^{s(\cdot)}_{p(\cdot),q(\cdot)}(\mathbb{A})$
and prove that  they are equivalent. Applying this,
we establish the $\varphi$-transform characterization of
$B^{s(\cdot)}_{p(\cdot),q(\cdot)}(W)$.
By this and via first establishing  the boundedness of
$\alpha$-convexification $\eta$-type operators on variable Lebesgue spaces,
we obtain the boundedness  of almost diagonal operators on the sequence space $b^{s(\cdot)}_{p(\cdot),q(\cdot)}(W)$ related to $B^{s(\cdot)}_{p(\cdot),q(\cdot)}(W)$,
which is further used to establish various decomposition characterizations of $B^{s(\cdot)}_{p(\cdot),q(\cdot)}(W)$, respectively,
in terms of molecules, wavelets, and atoms. Applying the wavelet decomposition of  $B^{s(\cdot)}_{p(\cdot),q(\cdot)}(W)$,
we   obtain the   trace theorem  and the extension properties of $B^{s(\cdot)}_{p(\cdot),q(\cdot)}(W)$,
and, applying  the molecular characterization,
we obtain  the boundedness of Calder\'on--Zygmund operators
on $B^{s(\cdot)}_{p(\cdot),q(\cdot)}(W)$.
}
\end{minipage}
\end{center}

\vspace{0.2cm}

\tableofcontents

\section{Introduction}
The theory of Besov spaces has found wide applications
in harmonic analysis and partial differential equations
(see, for example, \cite{b20 2,cgn17,cgn19,gn16,os24,oao24,whhy21}).
In recent decades, there also
exists an increasing interest in various Besov-type spaces,
such as variable Besov spaces (\cite{adh18,ah10,ah14,d12,d15,d15 1,yzy15}),
Besov--Morrey spaces (\cite{fx11,tx05,x05,xf12,ysy10}),
and Besov spaces associated with operators
(\cite{bbd20,b20,bd17,bd21,bd21 2,gkkp17,gkkp19,zy20}).

Besov spaces with variable smoothness $s(\cdot)$ and fixed $p = q$
was first studied by Leopold \cite{l89,l89 2,l91}
and Leopold and Schrohe \cite{ls96}
in order to characterize the boundedness of pseudo-differential operators,
which were further generalized to the case that $p\neq q$ by Besov \cite{b99,b03,b05}.
Besov spaces with variable integrability
$p(\cdot)$ and fixed $q$ and $s$ were later introduced by Xu \cite{x08,x08 2}
along a different line of study.
Through the variable mixed Lebesgue-sequence
space $l^{q(\cdot)}(L^{p(\cdot)})$,
Almeida and H\"ast\"o \cite{ah10} introduced variable
Besov spaces $B^{s(\cdot)}_{p(\cdot),q(\cdot)}$
with variable integrability indices $p(\cdot)$ and $q(\cdot)$
and variable smoothness index $s(\cdot)$.
Function spaces with variable exponents have been used to study the existence and the regularity
of solutions of  some
partial differential equations such as (fractional) Navier--Stokes equations
(see, for example, \cite{ac21,ac21 2,cf13,dhr17,os24}).
We also refer to \cite{adh18,ah14,dx14,d12,d15,dh17,yzy15,zd23}
for more studies about variable Besov  and   Besov-type spaces.

On the other hand, in the study of the boundedness of the Hardy--Littlewood maximal operator
on weighted variable Lebesgue spaces,
Cruz-Uribe et al. \cite{cdh11} introduced the concept of
variable weights and
proved in \cite{cfn12} the related weak boundedness
of the maximal operators.
Later on  Cruz-Uribe and Wang \cite{cw17} established the extrapolation
theorem of Lebesgue spaces with variable weights.
Recently, Cruz-Uribe and Penrod \cite{cp24} proved the reverse H\"older inequality
for variable weights in variable Lebesgue spaces.
We refer to \cite{cs23,wgx25} for more studies about variable  weights.
Recently, Wang and Xu \cite{wx22} studied the embedding and
the interpolation properties of weighted variable Besov spaces,
Guo et al. \cite{gwx23} obtained an equivalent characterization
of weighted variable Besov spaces,
and Wang et al. \cite{wgx24} further characterized
weighted variable Besov spaces by decompositions in terms of
 atoms, molecules, and wavelets.
We refer also to \cite{cms25} for
weighted variable Besov spaces associated with operators.

The study of matrix weights can be tracked back to
the work of Wiener and Masani \cite{wm58}
on the prediction theory for multivariate stochastic processes.
In 1990s, Nazarov and Treil \cite{nt96}, Treil and Volberg \cite{tv97},
and Volberg \cite{v97} generalized the scalar Muckenhoupt $A_p$ weights
to the matrix $A_p$ weights acting on vector-valued functions.
From then on, a lot of attention  has been paid to the theory of
matrix $A_p$ weights;
see, for example, \cite{bpw16,b01,bc22,c01,dhl20,g03,llor24,llor24 2,nptv17}
and see, for example, \cite{n13,nr18,ns21,ns25,ns26} for more studies about matrix weights
on more general bases.
The matrix $A_p$ weighted Besov spaces $B^s_{p,q}(W)$
were later introduced by Roudenko \cite{r03}
for any $p\in (1,\infty)$ and by Frazier and Roudenko
\cite{fr04} for any $p\in (0,1]$.
In these works, they proved the boundedness of almost
diagonal operators
and, using this boundedness, studied the boundedness
of Calder\'on--Zygmund operators
and established the wavelet characterization of matrix
$A_p$ weighted Besov spaces.
In   \cite{fr21}, Frazier and Roudenko
also introduced and studied the matrix $A_p$ weighted Triebel--Lizorkin spaces.
Very recently, Bu et al. \cite{bhyy23,bhyy23 2,bhyy23 3} systematically
developed the matrix $A_p$ weighted Besov-type and Triebel--Lizorkin-type spaces.
In particular, Bu et al.  introduced the upper and the lower  dimensions
for matrix weights, which were used to obtain the optimal ranges of  indexes
of the related boundedness of almost diagonal operators
and then improve  some corresponding results in \cite{r03,fr04,fr21}.
We refer the survey \cite{byyz26} for more details on the  history and
developments of matrix weights.

The matrix variants of Muckenhoupt $A_\infty$ weights were also first
 introduced by Volberg \cite{v97}. Differently from the scalar-valued case,
the class of matrix Muckenhoupt $A_\infty$ weights are split into different
 matrix $A_{p,\infty}$ weights with $p\in(0,\infty)$.
Recently, Bu et al.  \cite{bhyy23 1}
found various equivalent characterizations of $A_{p,\infty}$ weights
and gave several important properties of these weights.
Based on these, Yang et al. \cite{bhyy24,byyz25,yymz25}
further introduced and studied the matrix $A_{p,\infty}$ weighted
Besov-type spaces.
We also refer to \cite{bcyy24,cyy25} in which Yang et al. introduced the
matrix $A_{p,\infty}$ weighted Hardy spaces and characterized them
in terms of various maximal functions and, especially, atoms
in the matrix $A_p$ weight case and, as applications,
Yang et al. introduced matrix-weighted
Campanato spaces, proved that they are the dual spaces
of matrix $A_p$ weighted Hardy spaces,
and obtained the boundedness of Calder\'on--Zygmund operators
on both matrix-weighted Hardy spaces and Campanato spaces.

Combining the theories of variable weights and matrix weights,
Cruz-Uribe and Penrod \cite{cp23} introduced  variable matrix
$\mathscr{A}_{p(\cdot)}$ weights, established the related 1identity
approximation theorem and
introduced the  variable matrix $\mathscr{A}_{p(\cdot)}$ weighted Sobolev spaces.
They also established the reverse H\"older inequality
for matrix $\mathscr{A}_{p(\cdot)}$ weights on variable Lebesgue spaces in \cite{cp24}.
Nieraeth and Penrod \cite{np25} later proved the boundedness of Christ--Goldberg
maximal operators and Calder\'on--Zygmund operators on
matrix $\mathscr{A}_{p(\cdot)}$ weighted
variable Lebesgue spaces.
Inspired by these, we  \cite{yyz25} introduced
variable matrix $\mathscr{A}_{p(\cdot),\infty}$
weights, proved the existence of their related
reducing operators, and  established the related
reverse H\"older inequality.

Based on \cite{yyz25}, the main purpose of this  article is to
develop the   matrix weighted variable Besov spaces
with matrix $\mathscr{A}_{p(\cdot),\infty}$ weights.
We introduce the matrix $\mathscr{A}_{p(\cdot),\infty}$
weighted variable Besov space $B^{s(\cdot)}_{p(\cdot),q(\cdot)}(W)$
and the corresponding ``averaging'' variable Besov space $B^{s(\cdot)}_{p(\cdot),q(\cdot)}(\mathbb{A})$, as well as their corresponding sequences spaces $b^{s(\cdot)}_{p(\cdot),q(\cdot)}(W)$ and $b^{s(\cdot)}_{p(\cdot),q(\cdot)}(\mathbb{A})$,
and prove the  equivalences
$$B^{s(\cdot)}_{p(\cdot),q(\cdot)}(\mathbb{W})
=B^{s(\cdot)}_{p(\cdot),q(\cdot)}(\mathbb{A})\ \mathrm{and}\
b^ {s(\cdot)}_{p(\cdot),q(\cdot)}(\mathbb{W})
=b^{s(\cdot)}_{p(\cdot),q(\cdot)}(\mathbb{A}).$$
Using these,
we establish the $\varphi$-transform characterization of  $B^{s(\cdot)}_{p(\cdot),q(\cdot)}(W)$.
By means of the $\alpha$-convexification $\eta$-type operator
[see \eqref{def etaf}] and
the $\alpha$-convexification maximal operator
[see \eqref{def alpha max}] and by first
obtaining their boundedness on $\mathscr{A}_{p(\cdot),\infty}$ matrix
weighted variable Lebesgue spaces,
we establish a vector-valued inequality
related to $\alpha$-convexification $\eta$-type operators, and employ it to
obtain the  boundedness of almost diagonal operators on
$b^{s(\cdot)}_{p(\cdot),q(\cdot)}(W)$,
which are further used to establish
various decomposition characterizations of
$B^{s(\cdot)}_{p(\cdot),q(\cdot)}(W)$, respectively,
in terms of molecules, wavelets, and atoms. Applying the wavelet decomposition of  $B^{s(\cdot)}_{p(\cdot),q(\cdot)}(W)$,
we then obtain their   trace and   extension properties.
Finally, by  the aforementioned  molecular characterization,
we obtain the boundedness of Calder\'on--Zygmund operators
on $B^{s(\cdot)}_{p(\cdot),q(\cdot)}(W)$.

We point out that, since the almost diagonal operators are the discrete version
of the convolution of $\eta$ functions [see \eqref{def eta}] with vectors,
to obtain their boundedness on $b^{s(\cdot)}_{p(\cdot),q(\cdot)}(W)$,
we need to establish a vector-valued inequality (see Proposition
\ref{eta bound seq}) of some  $\eta$-type operators, namely the $\alpha$-convexification
$\eta$-type operators related to
matrix weights [see \eqref{def etaf}], which strongly relies on
their boundedness on variable Lebesgue spaces.
In the scalar-valued case,
since weights and functions can be separated from each other,
the boundedness of $\alpha$-convexification $\eta$-type operators
can be directly obtained by convexifying the weights.
However, in the matrix case, weights and functions cannot be
separated, and so the  aforementioned approach fails.
To overcome this obstacle, we introduce the concept of $\alpha$-convexification
maximal operators related to matrix weights.
Then, by establishing their sparse domination and using the reverse H\"older inequality
for matrix $\mathscr{A}_{p(\cdot),\infty}$ weights in variable Lebesgue spaces
obtained in our previous article \cite{yyz25},
we establish the boundedness of $\alpha$-convexification maximal operators
and hence the boundedness of $\alpha$-convexification $\eta$-type operators.
These boundedness results, to the best of our knowledge,
are new even when $p(\cdot)\equiv p$ is a constant exponent.

The organization of the reminder of this article is as follows.

In Section \ref{sec Apinfty}, we recall some basic concepts
and properties of matrix $\mathscr{A}_{p(\cdot),\infty}$ weights obtained in \cite{yyz25},
including the related  reducing operators
(see Definition \ref{def reducing operator})
and their upper and their lower dimensions
(see Theorem \ref{dim ext}).
Then we introduce the $\alpha$-convexification $\eta$-type operator
and the $\alpha$-convexification maximal operator
for $\mathscr{A}_{p(\cdot),\infty}$ matrix weights and,
through establishing some sparse domination,
we prove their boundedness on variable Lebesgue spaces (see Theorem \ref{bound max}),
which are further used to show the boundedness of  almost diagonal operators
in Section \ref{sec abo}.

In Section \ref{sec Besov}, we introduce the matrix $\mathscr{A}_{p(\cdot),\infty}$
weighted variable Besov space and the corresponding variable Besov sequence space.
More precisely, in Subsection \ref{sec Besov 1},
we introduce the matrix $\mathscr{A}_{p(\cdot),\infty}$
weighted variable Besov space $B^{s(\cdot)}_{p(\cdot),q(\cdot)}(W)$
and the corresponding averaging weighted variable Besov space
$B^{s(\cdot)}_{p(\cdot),q(\cdot)}(\mathbb{A})$, and
prove their equivalence (see Theorem \ref{W aa supp}).
Then, in Subsection \ref{sec Besov 2},
we introduce the matrix $\mathscr{A}_{p(\cdot),\infty}$ weighted
variable Besov sequence space $b^{s(\cdot)}_{p(\cdot),q(\cdot)}(W)$
and the corresponding averaging weighted variable Besov sequence space
$b^{s(\cdot)}_{p(\cdot),q(\cdot)}(\mathbb{A})$
and prove their equivalence (see Theorem \ref{W aa 3}).
Finally, in Subsection \ref{sec varphi},
by using these equivalences,
we establish the $\varphi$-transform characterization
of $B^{s(\cdot)}_{p(\cdot),q(\cdot)}(W)$ (see Theorem \ref{phi bound}),
which, as an application, implies
that $B^{s(\cdot)}_{p(\cdot),q(\cdot)}(W)$ is independent of the choice of $\varphi$
(see Proposition \ref{ind choice}).

In Section \ref{sec abo}, by using the above
obtained boundedness of $\alpha$-convexification $\eta$-type operators
for $\mathscr{A}_{p(\cdot),\infty}$ matrix weights,
we establish a vector-valued inequality related to $\alpha$-convexification $\eta$-type operators
and then, applying this, we obtain the boundedness of almost diagonal operators
on $b^{s(\cdot)}_{p(\cdot),q(\cdot)}(W)$ (see Theorem \ref{abo bound 1}).

In Section \ref{sec mole wave}, we apply
the boundedness of almost diagonal operators
to obtain several characterizations of $B^{s(\cdot)}_{p(\cdot),q(\cdot)}(W)$.
Precisely, in Subsection \ref{sec mole},
we establish the molecular characterization of $B^{s(\cdot)}_{p(\cdot),q(\cdot)}(W)$
(see Theorem \ref{mole com}) by combining the $\varphi$-transform
characterization (Theorem \ref{phi bound}) with the boundedness of
almost diagonal operators (Theorem \ref{abo bound 1}).
Then, in Subsection \ref{sec wave}, as an application of
the molecular characterization (Theorem \ref{mole com}),
we obtain the wavelet characterization of $B^{s(\cdot)}_{p(\cdot),q(\cdot)}(W)$
(see Theorem \ref{dep wave 1})
and then, by using this,
we establish the atomic characterization of $B^{s(\cdot)}_{p(\cdot),q(\cdot)}(W)$
(see Theorem \ref{dep atom}).

Finally, in Section \ref{sec app}, applying the wavelet
and the molecular characterizations obtained in Section \ref{sec mole wave},
we further establish the boundedness of trace operators
and Calder\'on--Zygmund operators on $B^{s(\cdot)}_{p(\cdot),q(\cdot)}(W)$.
Indeed, in Subsection \ref{sec trace},
we introduce the trace and the extension operators
on $B^{s(\cdot)}_{p(\cdot),q(\cdot)}(W)$ by using its wavelet characterization
(Theorem \ref{dep wave 1}) and then, together this with the above
obtained molecular characterization (Theorem \ref{mole com}),
we establish the boundedness of these trace and extension
operators (see Theorems \ref{trace 1} and \ref{WV ext}).
In Subsection \ref{sec CZ}, we further obtain the boundedness of
Calder\'on--Zygmund operators on $B^{s(\cdot)}_{p(\cdot),q(\cdot)}(W)$
by using its molecular characterization (see Theorem \ref{CZ theo}).

We end this introduction by making some conventions on symbols.
Throughout this article, we work in $\mathbb{R}^n$
and, unless otherwise specified, we always take $\mathbb{R}^n$
as the default underlying space.
Let $\mathbb{Z}$ be the collection of all integers,
$\mathbb{Z}_+:=\{0,1,\dots\}$,
and $\mathbb{N} := \{1,2,\dots\}$.
For any $\gamma := (\gamma_1,\dots,\gamma_n)\in \mathbb{Z}^n_+$,
let $|\gamma| := \gamma_1 + \cdots + \gamma_n$ and, for any
$x := (x_1,\dots,x_n) \in \mathbb{R}^n$, let $x^\gamma := x_1^{\gamma_1}\cdots x_n^{\gamma_n}$
and $D^\gamma := (\frac{\partial}{\partial x_1})^{\gamma_1}\cdots(\frac{\partial}{\partial x_n})^{\gamma_n}$.
For any measurable set $E$ in $\mathbb{R}^n$,
denote by the \emph{symbol $\mathscr{M}(E)$} the set
of all measurable functions on $E$
and, when $E = \mathbb{R}^n$, simply write $\mathscr{M}(\mathbb{R}^n)$ as $\mathscr{M}$.
In addition, we use the symbol $L^p_{\rm loc}$ with $p\in (0,\infty)$ to denote
the set of all locally $p$-integrable functions on $\mathbb{R}^n$
and use the symbol $C^{\infty}$ to denote the set of all infinitely differentiable functions on $\mathbb{R}^n$.
For any $t\in (0,\infty)$, let $\log_+ t := \max\{0,\log t\}$.
For any $x\in\mathbb{R}^n$ and $r\in (0,\infty)$,
the \emph{open ball $B(x,r)$} is defined to be the set $\{y\in\mathbb{R}^n:\ |x-y|< r\}$ and let
$\mathbb{B} := \{B(x,r):\ x\in\mathbb{R}^n\ \text{and}\ r\in(0,\infty)\}.$
A \emph{cube} $Q$ in $\mathbb{R}^n$ always has finite edge length
and edges of cubes are always assumed to be parallel to the coordinate axes,
but $Q$ is not necessary to be open or closed.
For any cube $Q$ in $\mathbb{R}^n$,
we always use $l(Q)$ to denote its edge length and $x_Q$ to denote its center.
For any $k\in\mathbb{Z}^n$ and $j\in\mathbb{Z}$,
let
$\mathcal{Q}:=\{ Q_{k,j}:= 2^{-j}([0,1)^n + k):\
k\in\mathbb{Z}^n\ \text{and}\ j\in\mathbb{Z}\}$
and, for any $j\in\mathbb{Z}$,
let
$$\mathcal{Q}_j:=\left\{ Q_{k,j}:= 2^{-j}([0,1)^n + k):\ k\in\mathbb{Z}^n\right\}$$
and
$$\mathcal{Q}_+:= \left\{ Q_{k,j}:= 2^{-j}([0,1)^n + k):\ k\in\mathbb{Z}^n\ \text{and}\ j\in\mathbb{Z}_+\right\}.$$
If $E$ is a measurable set in $\mathbb{R}^n$,
then we denote by $\mathbf{1}_E$ its \emph{characteristic function}
and, for any bounded measurable set $E\subset \mathbb{R}^n$ with $|E| \neq 0$
and for any $f\in L^1_{\rm loc}$,
let $\fint_E f(x)\,dx := \frac{1}{|E|} \int_E f(x)\,dx. $
For any $p\in [1,\infty]$, let $p'$ be its conjugate number,
that is, $\frac1p+\frac{1}{p'} = 1$.
We always use $C$ to denote a positive constant
independent of the main parameters involved,
but it may vary from line to line.
The symbol $f\lesssim g$ means $f\leq  Cg$
and, if $f\lesssim g\lesssim f$, we then write $f\sim g$.
Finally, in all proofs we
consistently retain the symbols introduced
in the original theorem (or related statement).

\section{Matrix $\mathscr{A}_{p(\cdot),\infty}$ Weights\label{sec Apinfty}}

In this section, we first recall some basic properties of matrix
$\mathscr{A}_{p(\cdot),\infty}$ weights
obtained in \cite{yyz25}.
Then we introduce the  concepts of $\alpha$-convexification maximal operators
and $\alpha$-convexification $\eta$-type operators  related to $\mathscr{A}_{p(\cdot),\infty}$ weights,
and prove their boundedness on variable
Lebesgue spaces. These boundedness results are further used to establish a
matrix weighted version of the vector-valued inequality
involving $\eta$ functions in Section \ref{sec abo}, which is a key tool for establishing the boundedness of  almost diagonal operators.

We begin with the variable Lebesgue spaces.
A measurable function $p:\ \mathbb{R}^n\to(0,\infty]$ is called an \emph{exponent function}.
We use the \emph{symbol $\mathcal{P}$} to denote the set of all exponent functions $p:\ \mathbb{R}^n \rightarrow [1,\infty]$,
and the \emph{symbol $\mathcal{P}_0$} to denote the set of all exponent functions $p:\ \mathbb{R}^n \rightarrow (0,\infty]$
satisfying $\mathop{\rm{ess}\inf}_{x\in \mathbb{R}^n} p(x) > 0$.
For any $p(\cdot)\in \mathcal{P}_0$ and any set $E$ in $\mathbb{R}^n$,
let
$$ p_+(E) := \mathop{\rm{ess}\sup}_{x\in E} p(x)\quad \text{and}\quad p_-(E) := \mathop{\rm{ess}\inf}_{x\in E} p(x); $$
moreover, write $p_+ := p_+(\mathbb{R}^n)$ and $p_- := p_-(\mathbb{R}^n)$.

Then we recall the definition of variable Lebesgue spaces
(see, for instance, \cite[Definition 2.16]{cf13}).
\begin{definition}\label{def Leb}
The \emph{variable Lebesgue space $L^{p(\cdot)}$} associated
with $p(\cdot)\in \mathcal{P}_0$
is defined to be the set of all $f\in\mathscr{M}$ such that
$$ \|f\|_{L^{p(\cdot)}} := \inf\left\{ \lambda\in (0,\infty):\ \rho_{L^{p(\cdot)}}\left(\frac{f}{\lambda}\right) \leq 1 \right\} < \infty, $$
where $\rho_{L^{p(\cdot)}}$ is the \emph{variable exponent modular}
defined by setting
$$\rho_{L^{p(\cdot)}} (f) :=  \int_{\mathbb{R}^n\setminus \Omega_\infty} \left|f(x)\right|^{p(x)}\,dx
+ \mathop{\rm{ess}\sup}_{x\in \Omega_\infty} |f(x)| $$
with $\Omega_\infty := \{x\in\mathbb{R}^n:\ p(x) = \infty\}$.
\end{definition}

The following log-H\"older continuous condition of variable exponents
(see, for instance, \cite[Definition 2.2]{cf13}) is frequently used in the theory of variable function spaces.

\begin{definition}
A measurable real-valued function $r$ on $\mathbb{R}^n$ is said to be
\emph{locally log-H\"older continuous},
denoted by $r(\cdot) \in LH_0$,
if there exists a positive constant $C_0$
such that, for any $x,y\in\mathbb{R}^n$ with $|x-y| < \frac12$,
\begin{align}\label{clogp 1}
|r(x)-r(y)| \leq -\frac{C_0}{\log(|x-y|)}.
\end{align}
A measurable real-valued function $r$ on $\mathbb{R}^n$
is \emph{log-H\"older continuous at infinity},
denoted by $r(\cdot) \in LH_\infty$,
if there exist positive constants $r_\infty$ and $C_{\infty}$
such that, for any $x\in \mathbb{R}^n$,
\begin{align*}
|r(x)-r_\infty| \leq \frac{C_\infty}{\log(e+|x|)}.
\end{align*}
Furthermore, a measurable real-valued function $r$  on $\mathbb{R}^n$ is said to be
\emph{globally log-H\"older continuous},
denoted by $r(\cdot) \in LH$,
if $r(\cdot)$ is both locally log-H\"older continuous
and log-H\"older continuous at infinity.
\end{definition}
\begin{remark}
\begin{itemize}
\item[{\rm (i)}] If $r(\cdot) \in LH$,
then \eqref{clogp 1} can be replaced by the following condition:
\begin{align}\label{clogp}
|r(x)-r(y)| \leq \frac{C_{\rm log}}{\log(e + \frac{1}{|x-y|})}
\ \ \text{for any}\ \  x,y\in\mathbb{R}^n.
\end{align}
\item[{\rm (ii)}] From \cite[Proposition 2.3]{cf13},
we infer that, if $r(\cdot) \in LH$,
then $\frac{1}{r(\cdot)} \in LH$.
\end{itemize}
\end{remark}

We now recall some basic properties of $L^{p(\cdot)}$ which are used below.
In what follows, for any $p(\cdot) \in \mathcal{P}_0$ and any cube $Q$ in $\mathbb{R}^n$,
let $p_Q:= [\fint_{Q} \frac{1}{p(x)}\,dx]^{-1}$.
Also, if a positive constant $C$ depends on some positive constants associated with $p(\cdot)$
or, more precisely, $C$ depends on some of $\{p_-,p_+,p_\infty,C_0,C_\infty\}$,
then we simply say that \emph{$C$ depends on $p(\cdot)$}.
The following lemma is precisely \cite[Theorem 4.5.7]{dhr17}.

\begin{lemma}\label{est Q}
Let $p(\cdot) \in \mathcal{P}\cap LH$.
Then, for any cube $Q$ in $\mathbb{R}^n$,
$$ \left\| \mathbf{1}_Q \right\|_{L^{p(\cdot)}}\sim\left| Q \right|^{\frac{1}{p_Q}},
\ \
\left\| \mathbf{1}_Q \right\|_{L^{p'(\cdot)}}\sim \left| Q \right|^{\frac{1}{p'_Q}},
\ \ \text{and}\ \
\left\| \mathbf{1}_Q \right\|_{L^{p(\cdot)}}\left\| \mathbf{1}_Q \right\|_{L^{p'(\cdot)}} \sim |Q|$$
where the positive equivalence constants depend only on $p(\cdot)$ and $n$.
\end{lemma}

The following H\"older's inequality in variable Lebesgue spaces is exactly \cite[Theorem 2.26]{cf13}.

\begin{lemma}\label{Holder}
Let $p(\cdot) \in \mathcal{P}$.
If $f\in L^{p(\cdot)}$ and $g\in L^{p'(\cdot)}$,
then $fg\in L^1$ and, moreover,
$$ \int_{\mathbb{R}^n} \left| f(x)g(x) \right|\,dx \lesssim \|f\|_{L^{p(\cdot)}} \|g\|_{L^{p'(\cdot)}}, $$
where the implicit positive constant depends only on $p(\cdot)$.
\end{lemma}

As a consequence of Lemmas \ref{est Q} and \ref{Holder}, we have the following conclusion (see, for instance, \cite[Lemma 2.8]{yyz25}).

\begin{lemma}\label{est fQ}
Let $p(\cdot) \in \mathcal{P}\cap LH$.
Then, for any $f\in \mathscr{M}$ and any cube $Q$ in $\mathbb{R}^n$,
\begin{align*}
\fint_Q \left|f(x)\right|\,dx \lesssim \frac{1}{\|\mathbf{1}_Q\|_{L^{p(\cdot)}}} \left\| f \mathbf{1}_Q \right\|_{L^{p(\cdot)}},
\end{align*}
where the implicit positive constant depends only on $p(\cdot)$ and $n$.
\end{lemma}

The following one is precisely \cite[Theorem 2.34]{cf13}.

\begin{lemma}\label{fg Lp}
Let $p(\cdot) \in \mathcal{P}$.
Then, for any $f\in \mathscr{M}$,
$f\in L^{p(\cdot)}$ if and only if
$$ \left\|f\right\|'_{L^{p(\cdot)}} := \sup_{\|g\|_{L^{p'(\cdot)}} \leq 1} \left|\int_{\mathbb{R}^n}  f(x)g(x) \,dx\right| < \infty $$
and, moreover, $ \|f\|_{L^{p(\cdot)}} \sim \|f\|'_{L^{p(\cdot)}}$,
where the positive equivalence constants depend only on $p(\cdot)$.
\end{lemma}

The following is the convexification of variable Lebesgue spaces
(see, for instance, \cite[Proposition 2.18]{cf13} and \cite[Lemma 3.2.6]{dhr17}).

\begin{lemma}\label{con f}
Let $p(\cdot) \in\mathcal{P}_0$ with $p_+ < \infty$.
Then, for any $r \in (0,\infty)$ and $f\in\mathscr{M}$,
$ \|f\|_{L^{rp(\cdot)}} = \| |f|^r \|^\frac1r_{L^{p(\cdot)}} $.
\end{lemma}

Next, we recall some basic concepts of matrices and matrix weights.
For any $m,n\in\mathbb{N}$, the set of all $m\times n$ complex-valued matrices
is denoted by the \emph{symbol $M_{m, n}$},
and $M_{m, m}$ is simply denoted by $M_m$.
For any $A\in M_m$, let
$\|A\| := \sup_{\vec{z}\in\mathbb{C}^m, |\vec{z}| = 1}| A\vec{z}|.$
Then $(M_m, \|\cdot\|)$ is a Banach space.
Moreover, we have the following well-known result
(see, for instance, \cite[Lemma 2.3]{bhyy23}).

\begin{lemma}\label{pA norm}
Let $A,B\in M_m$ be two nonnegative definite matrices.
Then $\|AB\| = \|BA\|$.
\end{lemma}

Now, we recall the concept of matrix weights
(see, for instance, \cite[Definition 2.7]{bhyy23}).
\begin{definition}
A matrix-valued function $W:\ \mathbb{R}^n \rightarrow M_m$ is called a \emph{matrix weight}
if $W$ satisfies that
\begin{itemize}
\item[{\rm (i)}] for almost every $x\in\mathbb{R}^n$, $W(x)$ is nonnegative definite,
\item[{\rm (ii)}] for almost every $x\in\mathbb{R}^n$, $W(x)$ is invertible,
\item[{\rm (iii)}] the entries of $W$ are all locally integrable.
\end{itemize}
\end{definition}

Now we recall the  matrix $\mathscr{A}_{p(\cdot),\infty}$ weights
introduced in   \cite[Definition 1.1(ii)]{yyz25}.

\begin{definition}
Let $p(\cdot)\in \mathcal{P}_0$.
A matrix weight $W$ on $\mathbb{R}^n$ is called a \emph{matrix $\mathscr{A}_{p(\cdot),\infty}$ weight}
if
$$ \left[W\right]_{\mathscr{A}_{p(\cdot),\infty}} := \sup_{Q} \exp\left( \fint_{Q} \log\left( \frac{1}{\|\mathbf{1}_Q\|_{L^{p(\cdot)}}} \left\|\, \left\| W(\cdot)W^{-1}(y) \right\| \mathbf{1}_Q \right\|_{L^{p(\cdot)}}
 \right)\,dy \right) <\infty , $$
where the supremum is taken over all cubes $Q$ in $\mathbb{R}^n$.
\end{definition}
\begin{remark}\label{rem Apinfty}
\begin{itemize}
\item[{\rm (i)}] If $p(\cdot) \equiv p$ is a constant exponent,
then, for any $W\in \mathscr{A}_{p,\infty}$,
the $p$-th power of $W$ is a matrix $A_{p,\infty}$ weight
(see, for instance, \cite[(2.2)]{v97} or \cite[Definition 3.1]{bhyy23 1} for the definition of $A_{p,\infty}$ weights).
\item [{\rm (ii)}] From \cite[Theorem 3.1]{yyz25},
it follows that, for any scalar-valued weight $w$,
if $p(\cdot)\in \mathcal{P}_0$ with $p(\cdot) \in LH$, then
$w \in \mathscr{A}_{p(\cdot),\infty}$ if and only if $w^{p(\cdot)} \in A_\infty$.
\end{itemize}
\end{remark}

Next, we recall the concept of  reducing operators of order $p(\cdot)$
for matrix weights,
which is exactly \cite[Definition 3.8]{yyz25}.

\begin{definition}\label{def reducing operator}
Let $p(\cdot) \in \mathcal{P}_0$ and $W$ be a matrix weight
and let $Q$ be any cube in $\mathbb{R}^n$.
The matrix $A_Q\in M_m$ is called a
\emph{reducing operator of order $p(\cdot)$ for $W$}
if $A_Q$ is positive definite and self-adjoint such that,
for any $\vec{z} \in \mathbb{C}^m$,
\begin{align}\label{eq redu}
\left| A_Q \vec{z} \right|
\sim \frac{1}{\|\mathbf{1}_Q\|_{L^{p(\cdot)}}} \left\|\, \left| W(\cdot) \vec{z} \right|  \mathbf{1}_{Q} \right\|_{L^{p(\cdot)}},
\end{align}
where the positive equivalence constants depend only on $m$ and $p(\cdot)$.
\end{definition}

The following lemma guarantees the existence of reducing operators of order $p(\cdot)$
for matrix weights, which is precisely \cite[Proposition 3.9]{yyz25}.

\begin{lemma}
Let $p(\cdot) \in \mathcal{P}_0$.
Then, for any matrix weight $W$ and any cube $Q$ in $\mathbb{R}^n$,
the reducing operator $A_Q$ of order $p(\cdot)$ for $W$ exists.
\end{lemma}

The next lemma  extends \eqref{eq redu} from any vector $\vec{z}$ to any matrix $M\in M_m$,
which is exactly \cite[Lemma 3.10]{yyz25}.

\begin{lemma}\label{eq reduc M}
Let $p(\cdot) \in \mathcal{P}_0$ and $W$ be a matrix weight and
let $Q$ be any cube in $\mathbb{R}^n$.
If $A_Q$ is a reducing operator of order $p(\cdot)$ for $W$,
then, for any matrix $M\in M_m$,
$$ \left\| A_Q M \right\| \sim \frac{1}{\|\mathbf{1}_Q\|_{L^{p(\cdot)}}} \left\| \, \left\| W(\cdot) M  \right\|  \mathbf{1}_{Q} \right\|_{L^{p(\cdot)}}, $$
where the positive equivalence constants depend only on $m$ and $p(\cdot)$.
\end{lemma}

We also recall the following concepts of  the lower and the upper $\mathscr{A}_{p(\cdot),\infty}$
weight dimensions introduced in \cite[Definition 3.21]{yyz25}.

\begin{definition}
Let $p(\cdot) \in \mathcal{P}_0$ and $d\in\mathbb{R}$.
A matrix weight $W$ is said to have \emph{$\mathscr{A}_{p(\cdot),\infty}$-lower dimension} $d$,
denoted by $W \in \mathbb{D}^{\rm lower}_{p(\cdot),\infty,d}$,
if there exists a positive constant $C$ such that,
for any $\lambda \in [1,\infty)$ and any cube $Q$ in $\mathbb{R}^n$,
\begin{align*}
\exp\left( \fint_{\lambda Q} \log\left( \frac{1}{\|\mathbf{1}_Q\|_{L^{p(\cdot)}}} \left\|\, \left\| W(\cdot)W^{-1}(y) \right\| \mathbf{1}_Q \right\|_{L^{p(\cdot)}}
 \right)\,dy \right) \leq C \lambda^d.
\end{align*}
A matrix weight $W$ is said to have \emph{$\mathscr{A}_{p(\cdot),\infty}$-upper dimension} $d$,
denoted by $W \in \mathbb{D}^{\rm upper}_{p(\cdot),\infty,d}$,
if there exists a positive constant $C$ such that,
for any $\lambda \in [1,\infty)$ and any cube $Q$ in $\mathbb{R}^n$,
\begin{align*}
\exp\left( \fint_{Q} \log\left( \frac{1}{\|\mathbf{1}_{\lambda Q}\|_{L^{p(\cdot)}}} \left\|\, \left\| W(\cdot)W^{-1}(y) \right\| \mathbf{1}_{\lambda Q} \right\|_{L^{p(\cdot)}}
 \right)\,dy \right) \leq C \lambda^d.
\end{align*}
\end{definition}

We have the following basic properties,
which is precisely \cite[Proposition 3.22]{yyz25}.

\begin{proposition}\label{dim ext}
Let $p(\cdot) \in \mathcal{P}_0\cap LH$.
Then the following statements hold:
\begin{itemize}
\item[{\rm (i)}] For any $d \in (-\infty,0)$,
$\mathbb{D}^{\rm lower}_{p(\cdot),\infty,d} = \emptyset$
and $\mathbb{D}^{\rm upper}_{p(\cdot),\infty,d} = \emptyset$.
\item[{\rm (ii)}] For any $W \in \mathscr{A}_{p(\cdot),\infty}$,
there exists $d_1 \in [0,\frac{n}{p_-})$
such that $W \in \mathbb{D}^{\rm lower}_{p(\cdot),\infty,d_1}$.
\item[{\rm (iii)}] For any $W \in \mathscr{A}_{p(\cdot),\infty}$,
there exists $d_2 \in [0,\infty)$
such that $W \in \mathbb{D}^{\rm upper}_{p(\cdot),\infty,d_2}$.
\end{itemize}
\end{proposition}

Let $p(\cdot)\in \mathcal{P}_0\cap LH$.
Then, for any matrix weight $W \in \mathscr{A}_{p(\cdot),\infty}$,
let
\begin{align*}
d^{\rm lower}_{p(\cdot),\infty}(W) := \inf\left\{ d\in \left(0,\frac{n}{p_-}\right):\ W\ \text{has}\ \mathscr{A}_{p(\cdot),\infty}\text{-lower dimension}\ d \right\}
\end{align*}
and
\begin{align*}
d^{\rm upper}_{p(\cdot),\infty}(W) := \inf\left\{ d\in (0,\infty):\ W\ \text{has}\ \mathscr{A}_{p(\cdot),\infty}\text{-upper dimension}\ d \right\}.
\end{align*}
Let
\begin{align*}
\left[\!\left[ d^{\rm lower}_{p(\cdot),\infty}(W), \infty\right)\right. :=
\begin{cases}
\displaystyle \left[d^{\rm lower}_{p(\cdot),\infty}(W), \frac{n}{p_-}\right) &\text{if } d^{\rm lower}_{p(\cdot),\infty}(W)\ \text{is an} \ \mathscr{A}_{p(\cdot),\infty}\text{-lower dimension of }W \\
\displaystyle \left(d^{\rm lower}_{p(\cdot),\infty}(W), \frac{n}{p_-}\right) &{\rm otherwise}
\end{cases}
\end{align*}
and
\begin{align*}
\left[\!\left[ d^{\rm upper}_{p(\cdot),\infty}(W), \infty\right)\right. :=
\begin{cases}
\displaystyle \left[d^{\rm upper}_{p(\cdot),\infty}(W), \infty\right) &\text{if } d^{\rm upper}_{p(\cdot),\infty}(W)\ \text{is an} \ \mathscr{A}_{p(\cdot),\infty}\text{-upper dimension of }W \\
\displaystyle \left(d^{\rm upper}_{p(\cdot),\infty}(W), \infty\right) &{\rm otherwise.}
\end{cases}
\end{align*}

\begin{remark}
If $p(\cdot) \equiv p$ is a constant exponent,
then Proposition \ref{dim ext}{\rm (ii)}
shows that, for any $W\in\mathscr{A}_{p,\infty}$,
$ d^{\rm lower}_{p(\cdot),\infty}(W) \in [0,\frac{n}{p}) $.
From Remark \ref{rem Apinfty}{\rm (i)},
it follows that $W \in \mathscr{A}_{p,\infty}$ if and only if $\widetilde{W} := W^p \in A_{p,\infty}$.
Hence, by this and Proposition \ref{dim ext}{\rm (ii)},
we find that, if $\widetilde{W} \in A_{p,\infty}$ and $d \in (-\infty,\infty)$ satisfy
\begin{align*}
\exp\left( \fint_{\lambda Q} \log\left( \fint_Q \left\| \widetilde{W}^{\frac1p}(x)\widetilde{W}^{-\frac1p}(y) \right\|^p \,dx \right)\,dy \right) \lesssim \lambda^d
\end{align*}
for any cube $Q$ in $\mathbb{R}^n$ and any $\lambda\in (1,\infty)$,
then $d \in (0,n)$,
which is exactly \cite[Proposition 6.3{\rm (ii)}]{bhyy23}.
\end{remark}

The upper and the lower $\mathscr{A}_{p(\cdot),\infty}$ weight dimensions play an important role in the
following estimate,
which is precisely \cite[Lemma 3.27]{yyz25}
(see, for instance, \cite[Proposition 6.5]{bhyy23 1} for a similar result for matrix $A_{p,\infty}$ weights) and is frequently used below.

\begin{lemma}\label{QP5}
Let $p(\cdot) \in \mathcal{P}_0\cap LH$, $W \in \mathscr{A}_{p(\cdot),\infty}$,
$d_1 \in [\![ d^{\rm lower}_{p(\cdot),\infty}(W) ,\frac{n}{p_-})$,
$d_2 \in [\![ d^{\rm upper}_{p(\cdot),\infty}(W) ,\infty)$,
and $\{A_Q\}_{{\rm cube}\ Q}$ be a family of reducing operators of order $p(\cdot)$ for $W$.
Then, for any cubes $Q$ and $R$ in $\mathbb{R}^n$, any $x\in Q$, and $y\in R$,
\begin{align}\label{eq strong doubling}
\left\| A_Q A_R^{-1} \right\| \lesssim \max\left\{ \left[ \frac{l(R)}{l(Q)} \right]^{d_1},
\left[ \frac{l(Q)}{l(R)} \right]^{d_2} \right\}\left[ 1+ \frac{|x-y|}{l(Q)\vee l(R)} \right]^{\Delta},
\end{align}
where $\Delta := d_1 + d_2$
and the implicit positive constant is independent of $Q$ and $R$.
\end{lemma}

\begin{definition}
Let $\mathbb{A} := \{A_Q\}_{Q\in\mathcal{Q}_+}$ be a sequence of positive definite and
self-adjoint matrices.
Then $\mathbb{A}$ is said to be \emph{strong $(d_1, d_2)$-doubling}
if \eqref{eq strong doubling} holds for any $Q,R \in \mathcal{Q}_+$.
\end{definition}

For any $j\in \mathbb{Z}_+$, $m\in (0,\infty)$,
and $x\in\mathbb{R}^n$, let
\begin{align}\label{def eta}
\eta_{j,m}(x) := \frac{2^{jn}}{(1+2^j|x|)^m}.
\end{align}
Let $p(\cdot) \in \mathcal{P}_0$ and $W\in \mathscr{A}_{p(\cdot),\infty}$.
For any $j\in \mathbb{N}$, $m\in (0,\infty)$, and $\alpha\in (0,\infty)$,
the \emph{$\alpha$-convexification $\eta$-type operator $\eta_{j,m,W}^{(\alpha)}$}
is defined by setting, for any $\vec{f} \in (L^1_{\rm loc})^m$ and $x\in\mathbb{R}^n$,
\begin{align}\label{def etaf}
\eta_{j,m,W}^{(\alpha)}\left(\vec{f}\right)(x) :=
\left[\int_{\mathbb{R}^n} \frac{2^{jn}|W(x)W^{-1}(y)\vec{f}(y)|^{\alpha}}{(1 + 2^j|x-y|)^{\alpha m}}\,dy\right]^\frac{1}{\alpha}.
\end{align}
Then we have the following   boundedness of $\eta_{j,m,W}^{(\alpha)}$ on variable Lebesgue spaces.
In what follows,  for any vector-valued function $\vec f \in (L^1_{\rm loc})^m$, we write
$\|\vec f\|_{L^{p(\cdot)}}:=\|\, |\vec f|\, \|_{L^{p(\cdot)}}.$

\begin{theorem}\label{bound eta}
Let $p(\cdot) \in \mathcal{P}_0\cap LH$
and $W \in \mathscr{A}_{p(\cdot),\infty}$.
Then there exists $\alpha \in (0,1)$,
depending on $[W]_{\mathscr{A}_{p(\cdot),\infty}}$, such that, for any given
$m\in (\frac{n}{\alpha},\infty)$ and any $\vec{f} \in (L^{p(\cdot)})^m$,
\begin{align}\label{bound eta 1}
\left\|\eta_{j,m,W}^{(\alpha)}\left(\vec{f}\right)\right\|_{L^{p(\cdot)}} \lesssim \left\|\vec{f}\right\|_{L^{p(\cdot)}},
\end{align}
where the implicit positive constant is independent of $j$ and $\vec{f}$.
\end{theorem}

To prove Theorem \ref{bound eta},
we need to introduce the $\alpha$-convexification maximal operators related to matrix weights.
Let $p(\cdot) \in \mathcal{P}_0$, $\alpha \in (0,\infty)$,
and $W \in \mathscr{A}_{p(\cdot),\infty}$.
The \emph{$\alpha$-convexification maximal operator $\mathcal{M}_{W}^{(\alpha)}$} is defined by setting,
for any $\vec{f} \in (L^1_{\rm loc})^m$ and $x\in \mathbb{R}^n$,
\begin{align}\label{def alpha max}
\mathcal{M}_{W}^{(\alpha)} \left( \vec{f} \right)(x) :=
\sup_{Q} \left[ \fint_Q \left|W(x) W^{-1}(y) \vec{f}(y) \right|^\alpha\,dy \right]^{\frac{1}{\alpha}},
\end{align}
where the supremum is taken over all cubes $Q$ in $\mathbb{R}^n$.
When $\alpha = 1$, $\mathcal{M}^{(1)}_W$ is exactly the
\emph{Christ--Goldberg maximal operator $\mathcal{M}_W$} (see \cite{g03,np25}).
Then we have the following boundedness result.

\begin{theorem}\label{bound max}
Let $p(\cdot) \in \mathcal{P}_0\cap LH$.
Then, for any $W \in \mathscr{A}_{p(\cdot),\infty}$,
there exists $\alpha \in (0,1)$, depending on $[W]_{\mathscr{A}_{p(\cdot),\infty}}$, such that,
for any $\vec{f} \in (L^{p(\cdot)})^m$,
\begin{align}\label{bdal}
\left\| \mathcal{M}_{W}^{(\alpha)} \left( \vec{f} \right) \right\|_{L^{p(\cdot)}}
\lesssim \left\|\vec{f}\right\|_{L^{p(\cdot)}},
\end{align}
where the implicit positive constant is independent of $\vec{f}$.
\end{theorem}
\begin{remark}\label{rem bound}
Let $p(\cdot) \in \mathcal{P}\cap LH$ with $p_-\in (1,\infty)$.
Then it follows from \cite[Theorem A]{np25} that
$\mathcal{M}_{W}$ is bounded on $L^{p(\cdot)}$
if and only if $W \in \mathscr{A}_{p(\cdot)}$
(see \cite[(1.2)]{cp23} for its definition).
Thus, using this, we conclude that \eqref{bdal}   holds for $\alpha = 1$
if and only if $W \in \mathscr{A}_{p(\cdot)}$.
\end{remark}

To prove Theorem \ref{bound max}, we need more properties of
$\mathscr{A}_{p(\cdot),\infty}$ weights.
The first one is the following equivalent characterization of $\mathscr{A}_{p(\cdot),\infty}$ weights.

\begin{lemma}\label{Apinfty plus}
Let $p(\cdot) \in \mathcal{P}_0\cap LH$.
Then, for any matrix weight $W$, $W \in \mathscr{A}_{p(\cdot),\infty}$
if and only if
$$ [W]_{\widetilde{\mathscr{A}}_{p(\cdot),\infty}} := \sup_{Q} \exp\left( \fint_Q \log_+ \left\| W^{-1}(y) A_Q \right\|\,dy \right) < \infty, $$
where the supremum is taken over all cubes $Q$ in $\mathbb{R}^n$,
and, moreover, $ [W]_{\widetilde{\mathscr{A}}_{p(\cdot),\infty}} \sim [W]_{\mathscr{A}_{p(\cdot),\infty}} $,
where the positive equivalence constants depend only on $p(\cdot)$, $n$, and $m$.
\end{lemma}
\begin{proof}
We first show the sufficiency.
From Lemma \ref{eq reduc M} with $M := W^{-1}(y)$ and from the fact that $\log t\leq \log_+ t$
for any $t \in (0,\infty)$,
we infer that, for any cube $Q$ in $\mathbb{R}^n$,
\begin{align*}
\exp\left( \fint_Q \log\left( \frac{1}{\|\mathbf{1}_{Q}\|} \left\|\,\left\|W(\cdot)W^{-1}(y)\right\|\mathbf{1}_Q\right\|_{L^{p(\cdot)}} \right)\,dy\right)
\lesssim \exp\left( \fint_Q \log_+ \left\| W^{-1}(y) A_Q \right\|\,dy \right),
\end{align*}
which, combined with the definition of $\mathscr{A}_{p(\cdot),\infty}$,
further implies that $W\in \mathscr{A}_{p(\cdot),\infty}$.
This finishes the proof of the sufficiency.

Next, we give the proof of the necessity.
By the definition of $\log_+$,
we find that, for any $y\in (0,\infty)$, $\log_+ y \leq \log y + \log_+ y^{-1}$
and hence, for any cube $Q$ in $\mathbb{R}^n$,
\begin{align*}
\exp\left(\fint_Q \log_+ \left\| W^{-1}(y) A_Q \right\|\,dy\right)
& \leq \exp\left(\fint_Q \log \left\| W^{-1}(y) A_Q \right\|\,dy\right)
\exp\left(\fint_Q \log_+ \left\| W^{-1}(y) A_Q \right\|^{-1}\,dy\right)\\
& =: {\rm I} \times {\rm II}.
\end{align*}
Then it follows from Lemmas \ref{pA norm} and \ref{eq reduc M} with $M:= W^{-1}(y)$ that,
for any $y\in \mathbb{R}^n$,
$$ \left\| W^{-1}(y) A_Q \right\|
= \left\|A_Q W^{-1}(y)\right\| \sim \frac{1}{\|\mathbf{1}_{Q}\|} \left\|\,\left\|W(\cdot)W^{-1}(y)\right\|\mathbf{1}_Q\right\|_{L^{p(\cdot)}}.  $$
By this and the definition of $\mathscr{A}_{p(\cdot),\infty}$,
we conclude that
\begin{align}\label{Apinfty plus eq 1}
{\rm I} \lesssim \exp\left( \fint_Q \log\left( \frac{1}{\|\mathbf{1}_{Q}\|} \left\|\,\left\|W(\cdot)W^{-1}(y)\right\|\mathbf{1}_Q\right\|_{L^{p(\cdot)}} \right)\,dy\right) \leq [W]_{\mathscr{A}_{p(\cdot),\infty}},
\end{align}
where the implicit positive constant depends only on $p(\cdot)$, $n$, and $m$.

Let $I_m$ be the identity matrix of $M_m$ and $r \in (0, p_-)$.
Then, using the fact that, for any matrices $A, B\in M_m$,
$\|AB\| \leq \|A\|\|B\|$ and Lemma \ref{pA norm},
we obtain, for any cube $Q$ in $\mathbb{R}^n$ and any $y\in \mathbb{R}^n$,
$$1 = \left\| I_m \right\| = \left\|W^{-1}(y) A_Q A_Q^{-1} W(y)\right\|
\leq \left\|W^{-1}(y) A_Q\right\| \left\|A_Q^{-1}W(y)\right\| = \left\|W^{-1}(y) A_Q\right\| \left\|W(y)A_Q^{-1}\right\| $$
and hence $ \|W^{-1}(y) A_Q\|^{-1}\leq \|W(y)A_Q^{-1}\| $.
This, combined with the well-known inequality that $\log_+ t \leq t$ for any $t \in(0,\infty)$
and Lemma \ref{est fQ} with $p(\cdot) := \frac{p(\cdot)}{r}$,
further implies that
\begin{align*}
{\rm II}
&\leq \exp\left( \fint_Q \frac1r \log_+ \left(\left\| W(y) A_Q^{-1} \right\|^r\right)\,dy\right)
\leq \exp\left( \frac1r \fint_Q \left\| W(y) A_Q^{-1} \right\|^r\,dy\right)\\
&\lesssim \left[\exp\left( \frac{1}{\|\mathbf{1}_Q\|_{L^{\frac{p(\cdot)}{r}}}} \left\|\,\left\| W(y) A_Q^{-1} \right\|^r\mathbf{1}_Q\right\|_{L^{\frac{p(\cdot)}{r}}}\right)\right]^\frac1r.
\end{align*}
Applying this with Lemma \ref{con f} and Lemma \ref{eq reduc M} with $M := A_Q^{-1}$,
we conclude that there exists a positive constant $C$,
depending only on $p(\cdot)$, $n$, and $m$, such that
\begin{align*}
{\rm II} \lesssim \exp\left( \frac{1}{\|\mathbf{1}_Q\|_{L^{p(\cdot)}}} \left\|\,\left\| W(y) A_Q^{-1} \right\|\mathbf{1}_Q\right\|_{L^{p(\cdot)}}\right)
\lesssim \exp\left(C \left\| A_Q A_Q^{-1} \right\|\right)
\lesssim 1,
\end{align*}
where the implicit positive constant depends only on $p(\cdot)$ and $n$.
This, together with \eqref{Apinfty plus eq 1}, further implies that
$[W]_{\widetilde{\mathscr{A}}_{p(\cdot),\infty}} \lesssim [W]_{\mathscr{A}_{p(\cdot),\infty}}$
and hence finishes the proof of Lemma \ref{Apinfty plus}.
\end{proof}

For any given cube $Q_0$ in $\mathbb{R}^n$,
let $\mathcal{Q}(Q_0)$ be the set of all dyadic cubes in $Q_0$.
\begin{proposition}\label{Apinfty u}
Let $p(\cdot) \in \mathcal{P}_0\cap LH$.
Then there exist positive constants $C_1$ and $C_2$,
depending only on $p(\cdot)$ and $n$, such that,
for any $W\in \mathscr{A}_{p(\cdot),\infty}$ and $u\in (0,\frac{\log 2}{C_1 + C_2 \log([W]_{\mathscr{A}_{p(\cdot),\infty}})})$,
$$
\sup_Q \fint_Q \left\|W^{-1}(x) A_Q\right\|^u\,dx < \infty,
$$
where the supremum is taken over all cubes $Q$ in $\mathbb{R}^n$.
\end{proposition}

\begin{proof}
By Lemma \ref{Apinfty plus}, we find that
$$M := \log\left( [W]_{\widetilde{\mathscr{A}}_{p(\cdot),\infty}} \right)
= \sup_Q \fint_Q \log_+ \left\| W^{-1}(y) A_Q \right\| \,dy < \infty. $$
Let cube $Q \subset \mathbb{R}^n$ be fixed.
Then, since $\fint_Q \log_+ \|W^{-1}(y)A_Q\|\,dy \leq M$,
it follows that there exists a collection of pairwise disjoint maximal
dyadic subcubes $\{Q_j\}_{j\in\mathbb{Z}_+}\subset \mathcal{Q}(Q)$ such that, for any $j\in\mathbb{Z}_+$,
\begin{align}\label{Apinfty u 1}
\fint_{Q_j} \log_+ \left\|W^{-1}(y) A_Q\right\|\,dy > 2M.
\end{align}
Let $\Omega_1 := \bigcup_{j\in\mathbb{Z}_+} Q_j$.
Then, using the disjointness of cubes in $\{Q_j\}_{j\in\mathbb{Z}_+}$,
\eqref{Apinfty u 1}, and the definition of $M$, we find that
\begin{align*}
|\Omega_1| = \sum_{j=0}^\infty |Q_j|
&< \frac{1}{2M} \sum_{j=0}^\infty \int_{Q_j} \log_+ \left\|W^{-1}(y) A_Q\right\|\,dy\nonumber \\
&\leq \frac{1}{2M} \int_Q \log_+ \left\|W^{-1}(y) A_Q\right\|\,dy
\leq \frac{1}{2M} M|Q| = \frac12 |Q|.
\end{align*}
From the definition of $M$, we deduce that
$\fint_{Q} \log_+ \|W^{-1}(y) A_Q\|\,dy \leq M$
and hence, for any $j\in\mathbb{Z}_+$, $Q_j\subsetneqq Q$.
Now, let $\widehat{Q}_i := \min\{R\in \mathcal{Q}(Q):\ Q_i\subsetneqq R\}$.
Then, by the maximality of $Q_j$,
we find that, for any $j\in\mathbb{Z}_+$,
\begin{align}\label{Apinfty u 7}
\fint_{Q_j} \log_+ \|W^{-1}(y) A_Q\|\,dy
\leq 2^n \fint_{\widehat{Q_j}} \log_+ \|W^{-1}(y) A_Q\|\,dy
\leq 2^{n+1} M.
\end{align}
Observe that, by the fact that, for any matrices $A,B\in M_m$, $\|AB\| \leq \|A\|\|B\|$
and Lemmas \ref{est fQ} and \ref{pA norm},
we obtain, for any $j\in\mathbb{Z}_+$,
\begin{align*}
\log_+ \left\|A_{Q_j}^{-1}A_Q\right\| &= \fint_{Q_j} \log_+ \left\|A_{Q_j}^{-1}A_Q\right\|\,dy
\leq \fint_{Q_j} \log_+ \left( \left\|A_{Q_j}^{-1}W(y)\right\|\left\|W^{-1}(y)A_Q\right\|\right)\,dy\\
& = \fint_{Q_j} \log_+ \left( \left\|W(y)A_{Q_j}^{-1}\right\|\left\|W^{-1}(y)A_Q\right\|\right)\,dy\nonumber\\
& = \fint_{Q_j} \log_+ \left\|W(y)A_{Q_j}^{-1}\right\|\,dy + \fint_{Q_j}\log_+\left\|W^{-1}(y)A_Q\right\|\,dy
\end{align*}
Using this, \eqref{Apinfty u 7},
and Lemma \ref{eq reduc M} with $M := A_{Q_j}^{-1}$
and using the fact that $\log_+ t \leq t$ for any $t\in(0,\infty)$,
we find that
\begin{align}\label{Apinfty u 5}
\log_+ \left\|A_{Q_j}^{-1}A_Q\right\|
& \leq \fint_{Q_j} \log_+ \left\|W(y)A_{Q_j}^{-1}\right\|\,dy + 2^{n+1}M\nonumber\\
&\lesssim \frac{1}{\|\mathbf{1}_{Q_j}\|_{L^{p(\cdot)}}} \left\|\,\left\| W(y) A_{Q_j}^{-1} \right\|\mathbf{1}_{Q_j}\right\|_{L^{p(\cdot)}} + 2^{n+1}M\nonumber\\
&\sim \left\| A_{Q_j} A_{Q_j}^{-1} \right\| + 1
\lesssim 1 + 2^{n+1} M.
\end{align}
Notice that,
for any $y\in Q_j$ with $j\in\mathbb{Z}_+$,
\begin{align}\label{Apinfty u 4}
\log_+ \left\|W^{-1}(y)A_Q\right\| \leq \log_+ \left(\left\|W^{-1}(y)A_Q\right\| \left\|A_{Q_j}^{-1}A_Q\right\| \right)
\leq \log_+ \left\|W^{-1}(y)A_{Q_j}\right\| + \log_+ \left\|A_{Q_j}^{-1}A_Q\right\|.
\end{align}
Moreover, from the definition of $\{Q_j\}_{j\in\mathbb{Z}_+}$
and the Lebesgue differential theorem,
we infer that, for any $x\in Q\setminus (\bigcup_{j\in\mathbb{Z}_+} Q_j)$,
$\log_+ \|W^{-1}(x)A_Q\| \leq 2M. $
Thus, combining this with \eqref{Apinfty u 4} and \eqref{Apinfty u 5},
we conclude that there exists a positive constant $C := \widetilde{C}(1 + 2^{n+1}M)$,
where $\widetilde{C}$ is the positive constant appearing in \eqref{Apinfty u 5},
independent of $Q$, such that, for any $x\in Q$,
\begin{align}\label{Apinfty u 6}
\log_+ \left\|W^{-1}(x)A_Q\right\| \mathbf{1}_Q(x)
\leq C\mathbf{1}_Q(x) + \sum_{j=0}^\infty \log_+ \left\|W^{-1}(x)A_{Q_j}\right\| \mathbf{1}_{Q_j}(x).
\end{align}

Observing that each term on the right-hand side of \eqref{Apinfty u 6}
has the same expression as the term on the left-hand side with $Q$ replaced by $Q_j$,
we can iterate these estimations.
Thus, after infinitely many iterations,
we conclude that, for any $y \in \mathbb{R}^n$,
\begin{align*}
\log_+ \left\|W(y)^{-1}A_Q\right\| \leq C \sum_{k=0}^\infty \mathbf{1}_{\Omega_k}(y)
 = C \sum_{k=0}^\infty (k + 1) \mathbf{1}_{\Omega_k\setminus \Omega_{k+1}}(y),
\end{align*}
where $\Omega_0 := Q$ and, in general, for any $k\in\mathbb{N}$,
$\Omega_k \subset \Omega_{k-1}$ is the union of cubes with
$|\Omega_k| \leq |\Omega_{k-1}| \leq \cdots \leq 2^{-k}|Q|$.
Hence, applying these with letting $u \in (0, \frac{\log2}{C})$,
we obtain
\begin{align*}
\fint_Q \left\| W^{-1}(y) A_Q \right\|^u\,dy
& = \fint_Q \exp\left(u\log_+ \left\| W^{-1}(y) A_Q \right\|\right)\,dy
 \leq \fint_{Q} \sum_{k=0}^\infty e^{u(k + 1)C} \mathbf{1}_{\Omega_k\setminus \Omega_{k+1}}(y) \,dy\\
&\leq \sum_{k=0}^\infty e^{u(k + 1)C} 2^{-k}
 = e^{uC} \sum_{k=0}^\infty e^{k(Cu - \log 2)}< \infty,
\end{align*}
which completes the proof of Proposition \ref{Apinfty u}.
\end{proof}
Now, let $t \in \{0,\frac13\}^n$.
Then the dyadic grid $\mathcal{Q}^t$ is defined by setting
$$ \mathcal{Q}^t := \left\{ 2^k\left([0,1)^n + m + (-1)^kt\right):\ k\in\mathbb{Z},
\ m\in\mathbb{Z}^n \right\}. $$
The following result is known as the ``$\frac13$''-trick
(see, for instance, \cite[Lemma 4.3]{np25}).

\begin{lemma}\label{Q Qt}
For any cube $Q$ in $\mathbb{R}^n$,
there exist $t \in \{0,\frac13\}^n$ and $Q_t \in \mathcal{Q}^t$
such that $Q\subset Q_t$ and $l(Q_t) \leq 6l(Q)$.
\end{lemma}

For any dyadic grid $\mathcal{Q}^t$, we define the
corresponding \emph{$\alpha$-convexification dyadic maximal operator
$\mathcal{M}_{W, \mathcal{Q}^t}^{(\alpha)}$} by setting,
for any $f\in L^{1}_{\rm loc}$ and $x \in \mathbb{R}^n$,
$$\mathcal{M}_{W, \mathcal{Q}^t}^{(\alpha)} (f)(x) :=
\sup_{x\in Q\in \mathcal{Q}^t} \left[ \fint_Q \left|W(x) W^{-1}(y) \vec{f}(y) \right|^\alpha\,dy \right]^{\frac{1}{\alpha}}. $$
Then, by  Lemma \ref{Q Qt},
we obtain the following result;
we omit the details here.

\begin{lemma}\label{mm2}
Let $p(\cdot) \in \mathcal{P}_0$, $W \in \mathscr{A}_{p(\cdot),\infty}$, and $\alpha \in (0,\infty)$.
Then, for any $f \in L^{p(\cdot)}$ and $x \in \mathbb{R}^n$,
\begin{align*}
\mathcal{M}_{W}^{\alpha} (f)(x) \leq 6^n \sum_{t\in \{0,\frac13\}^n} \mathcal{M}_{W, \mathcal{Q}^t}^{(\alpha)} (f)(x).
\end{align*}
\end{lemma}

Next, for any given cube $Q_0$ in $\mathbb{R}^n$ and any $\alpha \in (0,\infty)$,
the \emph{local $\alpha$-convexification dyadic maximal operator $\mathcal{M}_{W,Q_0}^{d,(\alpha)}$
on $Q_0$} is defined by setting,
for any $\vec{f} \in \mathscr{M}^m$ and $x \in Q_0$,
\begin{align*}
\mathcal{M}_{W,Q_0}^{d,(\alpha)} \left( \vec{f} \right)(x) :=
\sup_{x\in Q\in \mathcal{Q}(Q_0)} \left[ \fint_Q \left|W(x) W^{-1}(y) \vec{f}(y) \right|^\alpha\,dy \right]^{\frac{1}{\alpha}}
\end{align*}
and, for any $x\in \mathbb{R}^n\setminus Q_0$,
$\mathcal{M}_{W,Q_0}^{d,(\alpha)} ( \vec{f} )(x) := 0$.
Moreover, let $\mathcal{M}^{d,(\alpha)}_{Q_0}$ denote the \emph{local $\alpha$-convexification dyadic maximal operator on $Q_0$}
in the scalar setting, that is, for any $f \in L^1_{\rm loc}$ and $x \in Q_0$,
\begin{align*}
\mathcal{M}^{d,(\alpha)}_{Q_0}(f)(x) := \sup_{x\in Q\in \mathcal{Q}(Q_0)} \left[\fint_Q |f(y)|^\alpha\,dy\right]^{\frac{1}{\alpha}}
\end{align*}
and, for any $x \in \mathbb{R}^n\setminus Q_0$,
$\mathcal{M}^{d,(\alpha)}_{Q_0}(f)(x) := 0$.

We now recall the concept of the sparse family (see, for instance, \cite[Definition 2.2]{nptv17}).

\begin{definition}
Let $\eta\in (0,1)$.
A collection of cubes $\mathscr{S}$ is called the \emph{$\eta$-sparse family}
if there exists a disjoint collection of measurable sets $E_Q \subset Q$ with $Q \in \mathscr{S}$
such that $|E_Q| \geq \eta |Q|$ for any $Q \in \mathscr{S}$.
\end{definition}

\begin{lemma}\label{Apinfty spa}
Let $p(\cdot) \in \mathcal{P}_0\cap LH$,
$W\in \mathscr{A}_{p(\cdot),\infty}$, and $\alpha \in (0,1)$.
Then, for any given cube $Q_0$ in $\mathbb{R}^n$
and any given $\vec{f} \in (L^{p(\cdot)})^m$,
there exists one $\frac12$-sparse family $\mathscr{S} \subset \mathcal{Q}(Q_0)$
such that, for any $r \in (0,1]$ and almost every $x \in Q_0$,
\begin{align*}
\mathcal{M}_{W,Q_0}^{d,(\alpha)} \left( \vec{f} \right)(x)
\lesssim \left\{ \sum_{R\in \mathscr{S}} \left\|W(x) A_R^{-1}\right\|^r
\left[\fint_{R} \left|A_R W^{-1}(y)\vec{f}(y)\right|^\alpha\,dy \right]^\frac{r}{\alpha}\right\}^\frac1r,
\end{align*}
where the implicit positive constant is independent of $Q_0$, $\mathscr{S}$, and $\vec{f}$.
\end{lemma}

\begin{proof}
Let
$$\Omega := \left\{x\in Q_0:\ \mathcal{M}^{d,(\alpha)}_{Q_0}\left(\left| A_{Q_0} W^{-1}\vec{f} \right|\right)(x) >
\left[2\fint_{Q_0} \left| A_{Q_0} W^{-1}(y) \vec{f}(y) \right|^\alpha\,dy\right]^{\frac{1}{\alpha}}\right\}.$$
Then, by the definition of $\mathcal{M}^{d,(\alpha)}_{Q_0}$,
for any $x \in \Omega$,
there exists a dyadic cube $Q_x \in \mathcal{Q}(Q_0)$
such that
$$\left[\fint_{Q_x} \left| A_{Q_0} W^{-1}(y) \vec{f}(y) \right|^\alpha\,dy\right]^{\frac{1}{\alpha}}
 > \left[2\fint_{Q_0} \left| A_{Q_0} W^{-1}(y) \vec{f}(y) \right|^\alpha\,dy\right]^{\frac{1}{\alpha}}. $$
Thus, using the definition of dyadic cubes,
we find that there exists a sequence of pairwise disjoint maximal dyadic cubes
$\{R_j\}_{j\in\mathbb{Z}_+} \subset \mathcal{Q}(Q_0)$
such that $\Omega = \bigcup_{j\in\mathbb{Z}_+} R_j$
and, for any $j\in\mathbb{Z}_+$,
\begin{align*}
\left[\fint_{R_j} \left| A_{Q_0} W^{-1}(y) \vec{f}(y) \right|^\alpha\,dy\right]^{\frac{1}{\alpha}}
 > \left[2\fint_{Q_0} \left| A_{Q_0} W^{-1}(y) \vec{f}(y) \right|^\alpha\,dy\right]^{\frac{1}{\alpha}}.
\end{align*}
Hence, from this and the disjointness of $\{R_j\}_{j\in\mathbb{Z}_+}$, we deduce that
\begin{align*}
|\Omega| &= \sum_{j\in\mathbb{Z}_+} |R_j|
< \left[2\fint_{Q_0} \left| A_{Q_0} W^{-1}(y) \vec{f}(y) \right|^\alpha\,dy\right]^{-1}
\sum_{j\in\mathbb{Z}_+} \int_{R_j} \left| A_{Q_0} W^{-1}(y) \vec{f}(y) \right|^\alpha\,dy\nonumber\\
& = \left[2\fint_{Q_0} \left| A_{Q_0} W^{-1}(y) \vec{f}(y) \right|^\alpha\,dy\right]^{-1}
\int_{\bigcup_{j\in\mathbb{Z}_+} R_j} \left| A_{Q_0} W^{-1}(y) \vec{f}(y) \right|^\alpha\,dy\\
&\leq \left[2\fint_{Q_0} \left| A_{Q_0} W^{-1}(y) \vec{f}(y) \right|^\alpha\,dy\right]^{-1}
\int_{Q_0} \left| A_{Q_0} W^{-1}(y) \vec{f}(y) \right|^\alpha\,dy
\leq \frac12 |Q_0|.
\end{align*}
Notice that, by the definition of $\Omega$,
we obtain, for any $x \in Q_0\setminus \Omega$,
\begin{align*}
\mathcal{M}^{d,(\alpha)}_{Q_0}\left(\left| A_{Q_0} W^{-1}\vec{f} \right|\right)(x) \leq
\left[2 \fint_{Q_0} \left| A_{Q_0} W^{-1}(y) \vec{f}(y) \right|^\alpha\,dy\right]^{\frac{1}{\alpha}}.
\end{align*}
Applying this with the definition of $\mathcal{M}_{W,Q_0}^{d,(\alpha)}$,
we conclude that, for any $x \in Q_0\setminus \Omega$,
\begin{align}\label{Apinfty spa 2}
\mathcal{M}_{W,Q_0}^{d,(\alpha)}\left(\vec{f}\right) (x)
&= \sup_{x \in Q\in \mathcal{Q}(Q_0)} \left[ \fint_Q \left|W(x) W^{-1}(y) \vec{f}(y) \right|^\alpha\,dy \right]^{\frac{1}{\alpha}}\nonumber\\
& \leq \sup_{x \in Q\in \mathcal{Q}(Q_0)} \left\|W(x) A_{Q_0}^{-1}\right\| \left[ \fint_Q \left|A_{Q_0} W^{-1}(y) \vec{f}(y) \right|^\alpha\,dy \right]^{\frac{1}{\alpha}}\nonumber\\
& \leq \left\|W(x) A_{Q_0}^{-1}\right\| \left[2 \fint_{Q_0} \left|A_{Q_0} W^{-1}(y) \vec{f}(y) \right|^\alpha\,dy \right]^{\frac{1}{\alpha}}.
\end{align}
Now, for any $j \in \mathbb{Z}_+$,
let
$F_j := \{x \in R_j :\ \mathcal{M}_{W,Q_0}^{d,(\alpha)}(\vec{f}) (x) \neq \mathcal{M}_{W,R_j}^{d,(\alpha)}(\vec{f}) (x)\}. $
By the definition of $\mathcal{M}_{W,Q_0}^{d,(\alpha)}$ and the fact that $R_j \subset Q_0$,
we find that, for any $j\in\mathbb{Z}_+$ and $x \in F_j$,
$\mathcal{M}_{W,Q_0}^{d,(\alpha)}(\vec{f}) (x) > \mathcal{M}_{W,R_j}^{d,(\alpha)}(\vec{f}) (x), $
which further implies that, in this case,
$$\mathcal{M}_{W,Q_0}^{d,(\alpha)}\left(\vec{f}\right) (x)
 = \sup_{R_j \subset Q \in \mathcal{Q}(Q_0)} \left[ \fint_Q \left|W(x) W^{-1}(y) \vec{f}(y) \right|^\alpha\,dy \right]^{\frac{1}{\alpha}}. $$
From this and the maximality of $R_j$,
we infer that, for any $j\in\mathbb{Z}_+$ and $x \in R_j$,
\begin{align*}
\mathcal{M}_{W,Q_0}^{d,(\alpha)}\left(\vec{f}\right) (x)
& = \sup_{R_j \subset Q \in \mathcal{Q}(Q_0)} \left[ \fint_Q \left|W(x) W^{-1}(y) \vec{f}(y) \right|^\alpha\,dy \right]^{\frac{1}{\alpha}}\\
&\leq \sup_{R_j \subset Q \in \mathcal{Q}(Q_0)} \left\|W(x)A_{Q_0}^{-1}\right\| \left[ \fint_Q \left|A_{Q_0} W^{-1}(y) \vec{f}(y) \right|^\alpha\,dy \right]^{\frac{1}{\alpha}}\\
&\leq \left\|W(x) A_{Q_0}^{-1}\right\| \left[2 \fint_{Q_0} \left|A_{Q_0} W^{-1}(y) \vec{f}(y) \right|^\alpha\,dy \right]^{\frac{1}{\alpha}}.
\end{align*}
Using this, the disjointness of $\{R_j\}_{j\in \mathbb{Z}_+}$,
and \eqref{Apinfty spa 2},
we conclude that, for any $x \in Q_0$ and $r\in (0,1]$,
\begin{align}\label{Apinfty spa 3}
\left[\mathcal{M}_{W,Q_0}^{d,(\alpha)}\left(\vec{f}\right)(x)\right]^r \mathbf{1}_{Q_0}(x)
&= \left[\mathcal{M}_{W,Q_0}^{d,(\alpha)}\left(\vec{f}\right) (x)\right]^r \mathbf{1}_{Q_0 \setminus \Omega}(x)
+ \left[\mathcal{M}_{W,Q_0}^{d,(\alpha)}\left(\vec{f}\right) (x)\right]^r \mathbf{1}_{\bigcup_{j\in\mathbb{Z}_+} F_j}(x)\nonumber\\
&\quad + \sum_{j \in \mathbb{Z}_+} \left[\mathcal{M}_{W,Q_0}^{d,(\alpha)}\left(\vec{f}\right) (x)\right]^r \mathbf{1}_{R_j}(x)\nonumber\\
& \leq 2^{1 + \frac{r}{\alpha}} \left\|W(x) A_{Q_0}^{-1}\right\|^r \left[ \fint_{Q_0} \left|A_{Q_0} W^{-1}(y) \vec{f}(y) \right|^\alpha\,dy \right]^{\frac{r}{\alpha}}\mathbf{1}_{Q_0}(x)\nonumber\\
& \quad + \sum_{j \in \mathbb{Z}_+} \left[\mathcal{M}_{W,R_j}^{d,(\alpha)}\left(\vec{f}\right) (x)\right]^r \mathbf{1}_{R_j}(x).
\end{align}
Notice that each term of the second factor on the right-hand side of \eqref{Apinfty spa 3}
has the same expression as the term on the left-hand side with $Q_0$ replaced by $R_j$.
Iterating all the above estimations, we obtain a series of collections of dyadic cubes $\{\{R^{(k)}_j\}_{j\in\mathbb{Z}_+}\}_{k\in\mathbb{Z}_+}$
with $R^{(0)}_j := Q_0$
such that, for any $k,j\in\mathbb{Z}_+$ with letting $\mathcal{J}^{k}_j := \{i\in\mathbb{Z}_+: R^{(k+1)}_i \subset R^{(k)}_j\}$,
$|\bigcup_{i\in \mathcal{J}^{k}_j}  R^{(k+1)}_i| < \frac12 |R^{(k)}_j| $ and, for any $x \in R^{(k)}_j$,
\begin{align*}
\left[\mathcal{M}_{W,R^{(k)}_j}^{d,(\alpha)}\left(\vec{f}\right)(x)\right]^r\mathbf{1}_{R^{(k)}_j}(x)
&\leq 2^{1 + \frac{r}{\alpha}} \left\|W(x) A_{R^{(k)}_j}^{-1}\right\|^r \left[ \fint_{Q_0} \left|A_{R^{(k)}_j} W^{-1}(y) \vec{f}(y) \right|^\alpha\,dy \right]^{\frac{r}{\alpha}}\mathbf{1}_{R^{(k)}_j}(x)\\
&\quad + \sum_{i \in \mathcal{J}^{k}_j} \left[\mathcal{M}_{W,R^{(k+1)}_i}^{d,(\alpha)}\left(\vec{f}\right) (x)\right]^r \mathbf{1}_{R^{(k+1)}_i}(x).
\end{align*}
Thus, using this and \eqref{Apinfty spa 3},
we conclude that, for almost every $x\in Q_0$,
\begin{align*}
\mathcal{M}_{W,Q_0}^{d,(\alpha)} \left( \vec{f} \right)(x)
\lesssim \left\{\sum_{R\in \mathscr{S}} \left\|W(x) A_R^{-1}\right\|^r
\left[\fint_{R} \left|A_R W^{-1}(y)\vec{f}(y)\right|^\alpha\,dy \right]^\frac{r}{\alpha} \mathbf{1}_{R}(x) \right\}^\frac1r,
\end{align*}
which completes the proof of Lemma \ref{Apinfty spa}.
\end{proof}

For any variable exponent $q(\cdot)$,
the \emph{variable maximal operator $ \mathcal{M}_{q(\cdot)} $}
is defined by setting, for any $f \in L^{1}_{\rm loc}$ and $x \in \mathbb{R}^n$,
\begin{align*}
\mathcal{M}_{q(\cdot)}(f)(x) := \sup_{x\in Q} \frac{1}{\|\mathbf{1}_Q\|_{L^{q(\cdot)}}} \left\|f\mathbf{1}_Q\right\|_{L^{q(\cdot)}},
\end{align*}
where the supremum is taken over all cubes $Q$ containing $x$.
The following is precisely \cite[Theorem 7.3.27]{dhr17}.
\begin{lemma}\label{Mq bound}
Let $p(\cdot),q(\cdot),r(\cdot)\in \mathcal{P}\cap LH$
such that $p(\cdot) = r(\cdot)q(\cdot)$ and $r_- \in (1,\infty)$.
Then, for any $f\in L^{p(\cdot)}$,
\begin{align*}
\left\|\mathcal{M}_{q(\cdot)}(f)\right\|_{L^{p(\cdot)}} \lesssim \left\|f\right\|_{L^{p(\cdot)}},
\end{align*}
where the implicit positive constant depends only on $p(\cdot),q(\cdot)$, and $r(\cdot)$.
\end{lemma}

Let $p(\cdot) \in \mathcal{P}_0$ and $r\in (0,1]$ and
let $\mathscr{S}$ be an $\eta$-sparse family in $\mathbb{R}^n$
and, for any cube $Q$ in $\mathbb{R}^n$,
$\lambda_Q$ be a non-negative function.
Define  $T^{(r)}_{p(\cdot), \lambda,\mathscr{S}}$  by setting,
for any $f \in L^{1}_{\rm loc}$ and $x \in \mathbb{R}^n$,
$$ T_{p(\cdot),\lambda,\mathscr{S}}^{(r)}\left(f\right)(x)
:= \left[\sum_{Q\in \mathscr{S}} \lambda^r_Q(x) T_{p(\cdot),Q}^{r} \left(f\right) \mathbf{1}_Q(x)\right]^\frac1r, $$
where $T_{p(\cdot),Q}(f) := \frac{1}{\|\mathbf{1}_Q\|_{L^{p(\cdot)}}} \|{f}\mathbf{1}_Q\|_{L^{p(\cdot)}}$.
We have the following conclusion.

\begin{lemma}\label{Tf bound}
Let $p(\cdot) \in \mathcal{P}_0\cap LH$, $r\in (0,p_-)$,
$\alpha \in (0,1)$, and $\eta \in (0,1)$ and, for any cube $Q$ in $\mathbb{R}^n$,
let $\lambda_Q$ be a non-negative function.
If there exists $\theta \in (1,\infty)$ such that
\begin{align}\label{Tf bound 1}
\sup_{Q} \frac{1}{\|\mathbf{1}_Q\|_{L^{\theta p(\cdot)}}} \left\|\lambda_Q \mathbf{1}_Q\right\|_{L^{\theta p(\cdot)}} < \infty,
\end{align}
where the supremum is taken over all cubes $Q$ in $\mathbb{R}^n$,
then, for any $\eta$-sparse family $\mathscr{S}$ of $\mathbb{R}^n$
and for any $f \in L^{p(\cdot)}$,
\begin{align*}
\left\|T_{\alpha p(\cdot),\lambda,\mathscr{S}}^{(r)}\left(f\right)\right\|_{L^{p(\cdot)}}
\lesssim \left\|f\right\|_{L^{p(\cdot)}},
\end{align*}
where the implicit positive constant is independent of $\mathscr{S}$ and $f$.
\end{lemma}

\begin{proof}
Observe that $\frac{p(\cdot)}{r} \in \mathcal{P} \cap LH$.
By this and Lemmas \ref{con f} and \ref{fg Lp} with $p(\cdot) := \frac{p(\cdot)}{r}$, we find that
\begin{align}\label{Tf bound 2}
\left\|T_{\alpha p(\cdot),\lambda,\mathscr{S}}^{(r)}\left(f\right)\right\|^r_{L^{p(\cdot)}}
\sim \sup_{\|g\|_{L^{(\frac{p(\cdot)}{r})'}}\leq 1} \int_{\mathbb{R}^n} \left[T_{\alpha p(\cdot),\lambda,\mathscr{S}}^{(r)}\left(f\right)(x)\right]^r |g(x)|\,dx.
\end{align}
Now, fix $g \in L^{p'(\cdot)}$ with $\|g\|_{L^{(\frac{p(\cdot)}{r})'}}\leq 1$.
By the definition of $T_{p(\cdot),Q}(f)$ and Lemma \ref{con f},
we find that, for any $r\in (0,p_-)$,
$$\left[T_{p(\cdot),Q}(f)\right]^r
= \left[\frac{1}{\|\mathbf{1}_Q\|_{L^{p(\cdot)}}} \|{f}\mathbf{1}_Q\|_{L^{p(\cdot)}} \right]^r
= \frac{1}{\|\mathbf{1}_Q\|_{L^{\frac{p(\cdot)}{r}}}} \||f|^r\mathbf{1}_Q\|_{L^{\frac{p(\cdot)}{r}}}
= T_{\frac{p(\cdot)}{r},Q}\left(|f|^r\right). $$
Then, from this, the definition of $T_{p(\cdot),\alpha,\lambda,\mathscr{S}}$,
and Lemma \ref{Holder} with $p(\cdot) := \frac{\theta p(\cdot)}{r}$,
we deduce that
\begin{align*}
&\int_{\mathbb{R}^n} \left[T_{\alpha p(\cdot),\lambda,\mathscr{S}}^{(r)}\left(f\right)(x)\right]^r |g(x)|\,dx\\
&\quad = \int_{\mathbb{R}^n} \sum_{Q\in \mathscr{S}} \left[\lambda_Q(x)\right]^r \left[T_{\alpha p(\cdot),Q} \left({f}\right)\right]^r \mathbf{1}_Q(x) |g(x)|\,dx\\
&\quad = \sum_{Q\in \mathscr{S}} \left[T_{\alpha p(\cdot),Q} \left(f\right)\right]^r \int_{Q} \left[\lambda_Q(x)\right]^r |g(x)|\,dx\\
&\quad \lesssim \sum_{Q\in \mathscr{S}} T_{\frac{\alpha p(\cdot)}{r},Q} \left(|f|^r\right) \left\|\lambda^r_Q \mathbf{1}_Q\right\|_{L^{\frac{\theta p(\cdot)}{r}}} \left\|g \mathbf{1}_Q\right\|_{L^{(\frac{\theta p(\cdot)}{r})'}}.
\end{align*}
By this and Lemma \ref{est Q} with $p(\cdot) := \frac{\theta p(\cdot)}{r}$
and by Lemmas \ref{con f} and \eqref{Tf bound 1},
we deduce that
\begin{align*}
&\int_{\mathbb{R}^n} \left[T_{\alpha p(\cdot),\lambda,\mathscr{S}}^{(r)}\left(f\right)(x)\right]^r |g(x)|\,dx\\
&\quad \sim \sum_{Q\in \mathscr{S}} T_{\frac{\alpha p(\cdot)}{r},Q} \left(|f|^r\right) \frac{1}{\|\mathbf{1}_Q\|_{L^{\frac{\theta p(\cdot)}{r}}}} \left\|\lambda^r_Q \mathbf{1}_Q\right\|_{L^{\frac{\theta p(\cdot)}{r}}}
\frac{1}{\|\mathbf{1}_Q\|_{L^{(\frac{\theta p(\cdot)}{r})'}}} \left\|g \mathbf{1}_Q\right\|_{L^{(\frac{\theta p(\cdot)}{r})'}} |Q| \\
&\quad \sim \sum_{Q\in \mathscr{S}} T_{\frac{\alpha p(\cdot)}{r},Q} \left(|f|^r\right) \left[\frac{1}{\|\mathbf{1}_Q\|_{L^{\theta p(\cdot)}}} \left\|\lambda_Q \mathbf{1}_Q\right\|_{L^{\theta p(\cdot)}}\right]^r
\frac{1}{\|\mathbf{1}_Q\|_{L^{(\frac{\theta p(\cdot)}{r})'}}} \left\|g \mathbf{1}_Q\right\|_{L^{(\frac{\theta p(\cdot)}{r})'}} |Q| \\
&\quad \lesssim \sum_{Q\in \mathscr{S}} T_{\frac{\alpha p(\cdot)}{r},Q} \left(|f|^r\right) T_{(\frac{\theta p(\cdot)}{r})',Q} \left(g\right) |Q|.
\end{align*}
Using this, the assumption that $\mathscr{S}$ is $\eta$-sparse,
and Lemma \ref{Holder} with $p(\cdot) := \frac{p(\cdot)}{r}$ and
using Lemma \ref{Mq bound},
we conclude that
\begin{align*}
\int_{\mathbb{R}^n} \left[T_{\alpha p(\cdot),\lambda,\mathscr{S}}^{(r)}\left(f\right)(x)\right]^r |g(x)|\,dx
& \lesssim \sum_{Q\in \mathscr{S}} |E_Q| T_{\frac{\alpha p(\cdot)}{r},Q} \left(|f|^r\right) T_{(\frac{\theta p(\cdot)}{r})',Q} \left(g\right) \\
& \leq \int_{\mathbb{R}^n} \mathcal{M}_{\frac{\alpha p(\cdot)}{r}}(|f|^r)(x) \mathcal{M}_{(\frac{\theta p(\cdot)}{r})'}(g)(x)\,dx\\
& \lesssim \left\|\mathcal{M}_{\frac{\alpha p(\cdot)}{r}}(|f|^r)\right\|_{L^{\frac{p(\cdot)}{r}}} \left\|\mathcal{M}_{(\frac{\theta p(\cdot)}{r})'}(g)\right\|_{L^{(\frac{p(\cdot)}{r})'}}
\lesssim \left\|f\right\|^r_{L^{p(\cdot)}} \left\|g\right\|_{L^{(\frac{p(\cdot)}{r})'}},
\end{align*}
which, together with \eqref{Tf bound 2},
further gives that
$\|T_{\alpha p(\cdot),\lambda,\mathscr{S}}^{(r)}(f)\|_{L^{p(\cdot)}}
\lesssim \|f\|_{L^{p(\cdot)}}. $
This finishes the proof of Lemma \ref{Tf bound}.
\end{proof}
The following is the reverse H\"older inequality for $\mathscr{A}_{p(\cdot),\infty}$ weights
in variable Lebesgue spaces,
which is exactly \cite[Theorem 3.15]{yyz25}.
\begin{lemma}\label{reverse Holder}
Let $p(\cdot)\in\mathcal{P}_0\cap LH$.
Then, for any $W \in \mathscr{A}_{p(\cdot),\infty}$,
there exist positive constants $C$, $C_1$, $C_2$,
$A$, and $A_1$,
depending only on $p(\cdot)$ and $n$,
such that, for any $r \in (1,r_w]$ with
$$ r_w := 1+ \frac{1}{C_1 [W]_{\mathscr{A}_{p(\cdot),\infty}}^{A_1} 2^{C_2[W]_{\mathscr{A}_{p(\cdot),\infty}}}}, $$
any cube $Q$ in $\mathbb{R}^n$, and any matrix $M \in M_{m}$,
\begin{align*}
\frac{1}{\|\mathbf{1}_{Q}\|_{L^{rp(\cdot)}}} \left\|\left\| W(\cdot)M \right\|\mathbf{1}_Q\right\|_{L^{rp(\cdot)}}
\leq C [W]_{\mathscr{A}_{p(\cdot),\infty}}^{A}
\frac{1}{\|\mathbf{1}_{Q}\|_{L^{p(\cdot)}}} \left\| \left\| W(\cdot)M \right\|\mathbf{1}_Q\right\|_{L^{p(\cdot)}}.
\end{align*}
\end{lemma}

Based on the previous results, we now give the proof of Theorem \ref{bound max}.

\begin{proof}[Proof of Theorem \ref{bound max}]
Let $r \in (0,p_-) \cap (0,1]$ and
$\alpha := \frac{u}{(\frac{p_-}{r})'}$,
where $u$ is the same as in Proposition \ref{Apinfty u}.
Without loss of generality, we may assume that $\alpha \in (0,r)$.
Indeed, by H\"older's inequality,
for any $\widetilde{u}\in (0,u)$ and any cube $Q$ in $\mathbb{R}^n$,
\begin{align*}
\fint_Q \left\|W^{-1}(y) A_Q\right\|^{\widetilde{u}}\,dy
\leq \left[\fint_Q \left\|W^{-1}(y) A_Q\right\|^{u}\,dy \right]^{\frac{\widetilde{u}}{u}}
\end{align*}
and hence $\widetilde{u}$ also satisfies Proposition \ref{Apinfty u}.
Thus, we can choose $u$ small enough such that $\alpha\in (0,r)$.

We first fix a cube $Q_0$ in $\mathbb{R}^n$.
From Lemma \ref{Apinfty spa}, it follows that
there exists a $\frac12$-sparse family $\mathscr{S}$ such that,
for almost every $x\in\mathbb{R}^n$,
\begin{align}\label{bound max 1}
\mathcal{M}_{W,Q_0}^{d,(r \alpha)} \left( \vec{f} \right)(x)
&\lesssim \left\{\sum_{R\in \mathscr{S}} \left\|W(x) A_R^{-1}\right\|^r
\left[\fint_{R} \left|A_R W^{-1}(y)\vec{f}(y)\right|^{\alpha}\,dy \right]^\frac{r}{\alpha} \mathbf{1}_R(x) \right\}^\frac1r \nonumber \\
&\leq \left\{\sum_{R\in \mathscr{S}} \left\|W(x) A_R^{-1}\right\|^r
\left[\fint_{R} \left\|A_R W^{-1}(y)\right\|^{\alpha} \left|\vec{f}(y)\right|^{\alpha}\,dy \right]^\frac{r}{\alpha} \mathbf{1}_R(x) \right\}^\frac1r.
\end{align}
Notice that $\frac{p_-}{r}\in (1,\infty)$.
Using this and H\"older's inequality with $p := \frac{p_-}{r}$
and using Lemma \ref{con f} and Proposition \ref{Apinfty u},
we find that, for any $R\in \mathscr{S}$,
\begin{align*}
&\left[\fint_{R} \left\|A_R W^{-1}(y)\right\|^{ \alpha} \left|\vec{f}(y)\right|^{ \alpha}\,dy \right]^\frac r\alpha\\
&\quad  \leq \left[\fint_{R} \left\|A_R W^{-1}(y)\right\|^{\alpha (\frac{p_-}{r})'}\,dy\right]^{\frac{r}{\alpha (\frac{p_-}{r})'}}
\left[ \fint_R \left|\vec{f}(y)\right|^{\alpha p_-}\,dy \right]^{\frac{r^2}{\alpha p_-}}\\
&\quad  \lesssim \left[\fint_{R} \left\|W^{-1}(y) A_R \right\|^{u}\,dy\right]^{\frac{r}{\alpha (\frac{p_-}{r})'}}
\left[ \fint_R \left|\vec{f}(y)\right|^{\alpha \frac{p_-}{r}}\,dy \right]^{\frac{r^2}{\alpha p_-}}
\lesssim \left[ \fint_R \left|\vec{f}(y)\right|^{\alpha \frac{p_-}{r}}\,dy \right]^{\frac{r^2}{\alpha p_-}},
\end{align*}
which, combined with Lemma \ref{est fQ} with $p(\cdot) := \frac{p(\cdot)}{r}$ and combined with Lemma \ref{con f},
further implies that
\begin{align*}
\left[\fint_{R} \left\|A_R W^{-1}(y)\right\|^{\alpha} \left|\vec{f}(y)\right|^{\alpha}\,dy \right]^\frac r\alpha
& \lesssim \left[\frac{1}{\|\mathbf{1}_R\|_{L^{\frac{p(\cdot)}{r}}}} \left\|\left|\vec{f}\right|^\alpha\mathbf{1}_R\right\|_{L^{\frac{p(\cdot)}{r}}}\right]^\frac{r}{\alpha}
 = \left[\frac{1}{\|\mathbf{1}_R\|_{L^{\frac{\alpha p(\cdot)}{r}}}} \left\|\vec{f}\mathbf{1}_R\right\|_{L^{\frac{\alpha p(\cdot)}{r}}}\right]^{r}.
\end{align*}
Applying this with \eqref{bound max 1},
we conclude that, for almost every $x \in Q_0$,
\begin{align*}
\mathcal{M}_{W,Q_0}^{d,(\alpha)} \left( \vec{f} \right)(x)
\lesssim \sum_{R\in \mathscr{S}} \left\|W(x) A_R^{-1}\right\|^r \left[T_{\frac{\alpha p(\cdot)}{r},R}(|\vec{f}|)\right]^r.
\end{align*}

Now, for any cube $Q$ in $\mathbb{R}^n$,
let $\lambda_Q := \|WA_Q^{-1}\|$
and, for any $x\in \mathbb{R}^n$, let
\begin{align*}
T_{\frac{\alpha p(\cdot)}{r},\lambda,\mathscr{S}}^{(r)}( |\vec{f}| )(x) := \left\{\sum_{R\in \mathscr{S}} \left\|W(x) A_R^{-1}\right\|^r \left[T_{\frac{\alpha p(\cdot)}{r},R}( |\vec{f}| )\right]^r \mathbf{1}_R(x)\right\}^\frac1r.
\end{align*}
Observe that, by Lemma \ref{reverse Holder},
for any $W \in \mathscr{A}_{p(\cdot),\infty}$,
there exists $\theta \in (1,\infty)$
such that, for any cube $Q$ in $\mathbb{R}^n$,
$$ \frac{1}{\|\mathbf{1}_{Q}\|_{L^{\theta p(\cdot)}}} \left\|\left\| W(\cdot)A_Q^{-1} \right\|\mathbf{1}_Q\right\|_{L^{\theta p(\cdot)}}
\lesssim \frac{1}{\|\mathbf{1}_{Q}\|_{L^{p(\cdot)}}} \left\| \left\| W(\cdot)A_Q^{-1} \right\|\mathbf{1}_Q\right\|_{L^{p(\cdot)}},
 $$
which, together with Lemma \ref{eq reduc M} with $M := A_Q^{-1}$, further implies that
$$ \frac{1}{\|\mathbf{1}_{Q}\|_{L^{\theta p(\cdot)}}} \left\|\left\| W(\cdot)A_Q^{-1} \right\|\mathbf{1}_Q\right\|_{L^{\theta p(\cdot)}}
\lesssim \frac{1}{\|\mathbf{1}_{Q}\|_{L^{p(\cdot)}}} \left\| \left\| W(\cdot)A_Q^{-1} \right\|\mathbf{1}_Q\right\|_{L^{p(\cdot)}} \lesssim 1. $$
Combining this with Lemma \ref{Tf bound} with $\alpha := \frac{\alpha}{r}$,
we conclude that
\begin{align*}
\left\|T_{\alpha p(\cdot),\lambda,\mathscr{S}}^{(r)}( |\vec{f}| )\right\|_{L^{p(\cdot)}}
\lesssim \left\|\vec{f}\right\|_{L^{p(\cdot)}}
\end{align*}
and hence
\begin{align}\label{keyest}
\left\|\mathcal{M}_{W,Q_0}^{d,(\alpha)} \left( \vec{f} \right)\right\|_{L^{p(\cdot)}}
\lesssim \left\|\vec{f}\right\|_{L^{p(\cdot)}},
\end{align}
where the implicit positive constant is independent of $Q_0$ and $\vec{f}$.

Finally, by the definitions of dyadic grids and dyadic maximal operators,
we find that, for any $t\in \{0,\frac13\}^n$,
there exists a sequence of cubes $\{Q^{(t)}_j\}_{j\in\mathbb{Z}_+}$
with $Q^{(t)}_j \subset Q^{(t)}_{j+1}$ for any $j\in\mathbb{Z}_+$
and $\lim_{j\rightarrow \infty} Q^{(t)}_j = \mathbb{R}^n$
such that, for any $j\in\mathbb{Z}_+$ and $x\in \mathbb{R}^n$,
$\mathcal{M}_{W,Q^{(t)}_j}^{d,(\alpha)} ( \vec{f} )(x) \leq \mathcal{M}_{W,Q^{(t)}_{j+1}}^{d,(\alpha)} ( \vec{f} )(x)$
and
\begin{align*}
\mathcal{M}_{W,\mathcal{Q}^t}^{(\alpha)}\left( \vec{f} \right)(x) = \lim\limits_{j\to \infty} \mathcal{M}_{W,Q^{(t)}_j}^{d,(\alpha)} \left( \vec{f} \right)(x).
\end{align*}
Thus, using this, the variable version of Fatou's lemma
(see, for instance, \cite[Theorem 2.61]{cf13}), and the proved estimate \eqref{keyest},
we conclude that, for any $t\in \{0,\frac13\}^n$,
\begin{align*}
\left\|\mathcal{M}_{W,\mathcal{Q}^t}^{(\alpha)} \left( \vec{f} \right)\right\|_{L^{p(\cdot)}}
\leq \liminf\limits_{j\to \infty} \left\|\mathcal{M}_{W,Q^{(t)}_j}^{d,(\alpha)} \left( \vec{f} \right)\right\|_{L^{p(\cdot)}}
\lesssim \left\|\vec{f}\right\|_{L^{p(\cdot)}},
\end{align*}
which, together with Lemma \ref{mm2},
further implies that
\begin{align*}
\left\|\mathcal{M}_{W}^{(\alpha)} \left( \vec{f} \right)\right\|_{L^{p(\cdot)}}
\lesssim \left\|\vec{f}\right\|_{L^{p(\cdot)}}.
\end{align*}
This finishes the proof of Theorem \ref{bound max}.
\end{proof}
Finally, we give the proof of Theorem \ref{bound eta}.
\begin{proof}[Proof of Theorem \ref{bound eta}]
Using the same argument as in the proof of \cite[Lemma 5.2]{dhr09}
with $\eta_{v,m}\ast g$ replaced by $\eta_{j,m,W}^{(\alpha)}(\vec{f})$,
we obtain, for any $x\in\mathbb{R}^n$,
\begin{align*}
\eta_{j,m,W}^{(\alpha)}(\vec{f})(x) \lesssim \left[ \sum_{k = 0}^\infty 2^{-k(\alpha m-n)} \sum_{Q\in \mathcal{Q}_{j-k}}
\fint_{Q} \left| W(x)W^{-1}(y) \vec{f}(y) \right|^\alpha\,dy\mathbf{1}_{3Q}(x) \right]^{\frac{1}{\alpha}}
\lesssim \mathcal{M}^{(\alpha)}_W(\vec{f})(x),
\end{align*}
which, combined with Theorem \ref{bound max}, further implies \eqref{bound eta 1}.
This finishes the proof of Theorem \ref{bound eta}.
\end{proof}
\begin{remark}\label{rem eta}
By the proof of Theorem \ref{bound eta},
we find that, for any $p(\cdot) \in \mathcal{P}_0 \cap LH$
and $W\in \mathscr{A}_{p(\cdot),\infty}$,
if $\alpha \in (0,1)$ ensures that $\mathcal{M}^{(\alpha)}_{W}$
is bounded on $L^{p(\cdot)}$,
then $\eta^{(\alpha)}_{j,m,W}$ with $m\in (n,\infty)$ is bounded
on $L^{p(\cdot)}$.
By this and Remark \ref{rem bound},
we conclude that, if $W\in \mathscr{A}_{p(\cdot)}$,
then $\eta^{(\alpha)}_{j,m,W}$ with $\alpha = 1$
is bounded on $L^{p(\cdot)}$.
\end{remark}

\section{Matrix-Weighted Variable Besov Spaces}\label{sec Besov}
In this section, we introduce the matrix-weighted variable Besov spaces
and the related sequence spaces, including:
\begin{itemize}
\item[$\bullet$] the (pointwise) matrix-weighted variable Besov space $B^{s(\cdot)}_{p(\cdot),q(\cdot)}(W)$
and the related sequence space $b^{s(\cdot)}_{p(\cdot),q(\cdot)}(W)$,
where $W:\ \mathbb{R}^n \to M_m$ is a matrix weight,
\item[$\bullet$] the averaging matrix-weighted space $B^{s(\cdot)}_{p(\cdot),q(\cdot)}(\mathbb{A})$
and the related averaging sequence space $b^{s(\cdot)}_{p(\cdot),q(\cdot)}(\mathbb{A})$,
where $\mathbb{A}:= \{A_Q\}_{Q\in \mathcal{Q}_+}$ are reducing operators of order $p(\cdot)$ for $W$.
\end{itemize}
We prove the equivalence between $B^{s(\cdot)}_{p(\cdot),q(\cdot)}(W)$ and $B^{s(\cdot)}_{p(\cdot),q(\cdot)}(\mathbb{A})$
in Subsection \ref{sec Besov 1}
and the equivalence between $b^{s(\cdot)}_{p(\cdot),q(\cdot)}(W)$
and $b^{s(\cdot)}_{p(\cdot),q(\cdot)}(\mathbb{A})$ in Subsection \ref{sec Besov 2}.
Finally, in Subsection \ref{sec varphi},
we establish the $\varphi$-transform characterization for matrix-weighted variable Besov spaces
and conclude that the matrix-weighted variable Besov space is independent of
the choice of $\{\varphi_j\}_{j\in\mathbb{Z}_+}$.

Now, we recall the following space  introduced by Almeida and H\"ast\"o in \cite{ah10}.
\begin{definition}\label{def seq}
Let $p(\cdot), q(\cdot) \in \mathcal{P}_0$.
The \emph{variable mixed Lebesgue-sequence space $l^{q(\cdot)}(L^{p(\cdot)})$} is defined
to be the set of all measurable function sequences $\{f_j\}_{j\in\mathbb{Z}_+}\subset \mathscr{M}$ such that
$$ \left\| \left\{ f_j \right\}_{j\in\mathbb{Z}_+} \right\|_{l^{q(\cdot)}(L^{p(\cdot)})} :=
\inf\left\{ \mu\in (0,\infty) :\ \rho_{l^{q(\cdot)}(L^{p(\cdot)})} \left( \left\{\frac{f_j}{\mu}\right\}_{j\in\mathbb{Z}_+} \right)\leq 1 \right\} < \infty, $$
where the \emph{variable mixed Lebesgue-sequence modular $\rho_{l^{q(\cdot)}(L^{p(\cdot)})}$} is defined by setting
$$ \rho_{l^{q(\cdot)}(L^{p(\cdot)})} \left( \left\{f_j\right\}_{j\in\mathbb{Z}_+} \right) :=
\sum_{j = 0}^\infty \inf\left\{ \lambda_j \in (0,\infty):\ \rho_{p(\cdot)}\left( \lambda_j^{-\frac{1}{q(\cdot)}}f_j \right) \leq 1 \right\}. $$
\end{definition}
\begin{remark}\label{rem def seq}
\begin{itemize}
\item[{\rm (i)}] From Definitions \ref{def seq} and \ref{def Leb},
we infer that, if $q_+ < \infty$, then
$$ \rho_{l^{q(\cdot)}(L^{p(\cdot)})} \left( \left\{f_j\right\}_{j\in\mathbb{Z}_+} \right)
= \sum_{j = 0}^\infty \left\| \left|f_j\right|^{q(\cdot)} \right\|_{L^{\frac{p(\cdot)}{q(\cdot)}}}. $$
\item[{\rm (ii)}] By \cite[Proposition 3.3]{ah10},
we know that, if $p(\cdot),q(\cdot)$ are both constant exponents,
then the quasi-norm $\|\cdot\|_{l^{q}(L^{p})}$ defined in Definition \ref{def seq}
is precisely the mixed Lebesgue-sequence quasi-norm.
\item[{\rm (iii)}] From \cite[Proposition 3.5]{ah10},
it follows that $\rho_{l^{q(\cdot)}(L^{p(\cdot)})}$ in Definition \ref{def seq}
is a semimodular and, if $p_+,q_+<\infty$,
then $\rho_{l^{q(\cdot)}(L^{p(\cdot)})}$ is continuous
(see \cite[Definition 2.1]{ah10} or \cite[Definition 2.1.1]{dhr17} for more details).
\item[{\rm (iv)}] Let $p(\cdot),q(\cdot)\in\mathcal{P}_0$ with $p_+,q_+<\infty$
and let $r \in (0,\infty)$.
Then, by the definition of $\left\| \cdot \right\|_{l^{q(\cdot)}(L^{p(\cdot)})}$,
it is easy to find that,
for any sequence of measurable functions $\{f_j\}_{j\in\mathbb{Z}_+}$,
$$ \left\| \left\{ f_j \right\}_{j\in\mathbb{Z}_+} \right\|_{l^{q(\cdot)}(L^{p(\cdot)})}
 = \left\| \left\{ \left| f_j \right|^r \right\}_{j\in\mathbb{Z}_+} \right\|^{\frac1r}_{l^{\frac{q(\cdot)}{r}}(L^{\frac{p(\cdot)}{r}})}. $$
\end{itemize}
\end{remark}
Next, we recall the concept of admissible pairs
(see, for instance, \cite[Definition 5.1]{ah10}).
In what follows, let $\mathcal{S}$ be the space  of all Schwartz functions on $\mathbb{R}^n$,
equipped with the well-known topology determined by a countable family of norms,
and let $\mathcal{S}'$ be the set of all linear functionals on $\mathcal{S}$,
equipped with the weak-$\ast$ topology
(see, for instance, \cite[Chapters 2.2 and 2.3]{g14} for more details).
Moreover, for any $f\in \mathcal{S}$, its \emph{Fourier transform}
$\widehat{f}$ is defined by setting, for any $\xi\in \mathbb{R}^n$,
$\widehat{f}(\xi) := \int_{\mathbb{R}^n} f(x) e^{-i x\cdot \xi}\,dx$
(see, for instance, \cite[Chapter 2.3]{g14} for more details).
\begin{definition}\label{phi pair}
A pair of measurable functions $(\varphi,\varPhi)$ is said to be \emph{admissible}
if $\varphi,\varPhi\in \mathcal{S}$ satisfy
$$
{\mathop\mathrm{\,supp\,}} \widehat{\varphi} \subset
\left\{ \xi\in\mathbb{R}^n:\ \frac12 \leq |\xi| \leq 2 \right\}
\quad\text{and}\quad\left| \widehat{\varphi}(\xi) \right| \geq c >0\
\text{when}\ \frac35\leq |\xi|\leq \frac53
$$
and
$$
{\mathop\mathrm{\,supp\,}} \widehat{\varPhi}
\subset \left\{ \xi\in\mathbb{R}^n:\ |\xi| \leq 2 \right\}
\quad\text{and}\quad\left| \widehat{\varPhi}(\xi)
\right| \geq c >0\ \text{when}\ |\xi|\leq \frac53,
$$
where $c$ is a positive constant independent of $\xi \in \mathbb{R}^n$.
\end{definition}
In what follows, we always let $\varphi_0 := \varPhi$ and
$\varphi_j := 2^{jn}\varphi(2^j\cdot)$ for any $j\in\mathbb{N}$.

\subsection{Matrix-Weighted Variable Besov Spaces}\label{sec Besov 1}

In this subsection, we first introduce the
(pointwise) matrix-weighted variable Besov space
(see \cite[Definition 3.22]{bhyy24} for the definition of
matrix $A_{p,\infty}$ weighted Besov spaces).
\begin{definition}
Let $p(\cdot), q(\cdot) \in \mathcal{P}_0$, $s(\cdot) \in L^\infty$,
and $\{\varphi_j\}_{j\in\mathbb{Z}_+}$ be as in Definition \ref{phi pair}
and let $W \in \mathscr{A}_{p(\cdot),\infty}$.
The \emph{ (pointwise) matrix-weighted variable Besov space $B^{s(\cdot)}_{p(\cdot),q(\cdot)}(W,\varphi)$}
is defined to be the set of all $\vec{f} \in (\mathcal{S}')^m$ such that
\begin{align*}
\left\|\vec{f}\right\|_{B^{s(\cdot)}_{p(\cdot),q(\cdot)}(W,\varphi)}
:= \left\| \left\{ 2^{js(\cdot)} \left|W(\cdot) \left(\varphi_j\ast \vec{f}\right)(\cdot)\right| \right\}_{j\in \mathbb{Z}_+} \right\|_{l^{q(\cdot)}(L^{p(\cdot)})} <\infty.
\end{align*}
\end{definition}
Next we introduce the  averaging matrix-weighted variable Besov space
(see \cite[Definition 3.11]{bhyy24} for the definition of averaging matrix $A_{p,\infty}$
weighted Besov spaces).
\begin{definition}
Let $p(\cdot), q(\cdot) \in \mathcal{P}_0$, $s(\cdot) \in L^\infty$,
and $\{\varphi_j\}_{j\in\mathbb{Z}_+}$ be as in Definition \ref{phi pair}
and let $W\in \mathscr{A}_{p(\cdot),\infty}$ and $\mathbb{A} := \{A_Q\}_{Q\in \mathcal{Q}_+}$
be reducing operators of order $p(\cdot)$ for $W$.
The \emph{averaging matrix-weighted variable Besov space
$B^{s(\cdot)}_{p(\cdot),q(\cdot)}(\mathbb{A},\varphi)$}
is defined to be the set of all $\vec{f} \in (\mathcal{S}')^m$ such that
$$ \left\|\vec{f}\right\|_{B^{s(\cdot)}_{p(\cdot),q(\cdot)}(\mathbb{A},\varphi)}
:= \left\| \left\{ 2^{js(\cdot)} \left|A_j \left(\varphi_j\ast \vec{f}\right)(\cdot)\right| \right\}_{j\in \mathbb{Z}_+} \right\|_{l^{q(\cdot)}(L^{p(\cdot)})} <\infty, $$
where, for any $j\in \mathbb{Z}_+$,
\begin{align}\label{def Aj}
A_j := \sum_{Q\in \mathcal{Q}_j} A_Q \mathbf{1}_Q .
\end{align}
\end{definition}
To show the equivalence of $ B^{s(\cdot)}_{p(\cdot),q(\cdot)}(W,\varphi) $
and $ B^{s(\cdot)}_{p(\cdot),q(\cdot)}(\mathbb{A},\varphi) $,
we recall the concept of the variable Besov sequence space
(see \cite[Definition 3]{d12}).
\begin{definition}\label{def b}
Let $p(\cdot), q(\cdot) \in \mathcal{P}_0$ and $s(\cdot) \in L^\infty$.
The \emph{variable Besov sequence space $ b^{s(\cdot)}_{p(\cdot),q(\cdot)} $}
is defined to be the set of all sequences
$t := \{t_Q\}_{Q\in \mathcal{Q}_+} \subset \mathbb{C}$
such that
$$ \left\| t \right\|_{b^{s(\cdot)}_{p(\cdot),q(\cdot)}}
:= \left\| \left\{ 2^{js(\cdot)} t_j \right\}_{j\in \mathbb{Z}_+} \right\|_{l^{q(\cdot)}(L^{p(\cdot)})} <\infty, $$
where, for any $j\in \mathbb{Z}_+$,
$t_j := \sum_{Q\in \mathcal{Q}_j} t_Q \widetilde{\mathbf{1}}_Q\ \ \text{and}\ \ \widetilde{\mathbf{1}}_Q := |Q|^{-\frac12}\mathbf{1}_Q.$
\end{definition}
\begin{remark}
If $p(\cdot),q(\cdot)$, and $s(\cdot)$ are all constant exponents,
then, from Remark \ref{rem def seq}{\rm (ii)},
we infer that $b^{s(\cdot)}_{p(\cdot),q(\cdot)}$ defined in Definition \ref{def b}
reduces to the classical Besov sequence space.
\end{remark}
For any reducing operators $\mathbb{A} := \{A_Q\}_{Q\in\mathcal{Q}_+}$ of order $p(\cdot)$ for $W$,
any $\varphi \in \mathcal{S}$, and any $\vec{f} \in (\mathcal{S}')^m$,
we define
\begin{align}\label{def supp f}
\sup_{\mathbb{A},\varphi}\left( \vec{f} \right) := \left\{ \sup_{\mathbb{A},\varphi,Q}\left( \vec{f} \right) \right\}_{Q\in\mathcal{Q}_+},
\end{align}
where, for any $Q\in\mathcal{Q}_+$,
\begin{align}\label{def supp f 1}
\sup_{\mathbb{A},\varphi,Q}\left( \vec{f} \right) := |Q|^{\frac12} \sup_{y\in Q} \left| A_Q\left( \varphi_{j_Q}\ast \vec{f} \right)(y) \right|.
\end{align}
The following equivalence is the main result of this subsection
(see \cite[Theorem 3.24]{bhyy24} for the matrix-weighted Besov space case).
\begin{theorem}\label{W aa supp}
Let $p(\cdot), q(\cdot)\in \mathcal{P}_0\cap LH$, $s(\cdot)\in LH$,
and $\{\varphi_j\}_{j\in\mathbb{Z}_+}$ be the same as in Definition \ref{phi pair}
and let $W \in \mathscr{A}_{p(\cdot),\infty}$ and
$\mathbb{A} := \{A_Q\}_{Q\in\mathcal{Q}}$ be reducing operators of order $p(\cdot)$ for $W$.
Then $\vec{f} \in B^{s(\cdot)}_{p(\cdot),q(\cdot)}(W,\varphi)$
if and only if $\vec{f} \in B^{s(\cdot)}_{p(\cdot),q(\cdot)}(\mathbb{A}, \varphi)$
and, moreover,
$$ \left\| \vec{f} \right\|_{B^{s(\cdot)}_{p(\cdot),q(\cdot)}(W,\varphi)}
\sim \left\| \sup_{\mathbb{A},\varphi}\left( \vec{f} \right) \right\|_{b^{s(\cdot)}_{p(\cdot),q(\cdot)}}
\sim \left\| \vec{f} \right\|_{B^{s(\cdot)}_{p(\cdot),q(\cdot)}(\mathbb{A},\varphi)}, $$
where the positive equivalence constants are independent of $\vec{f}$.
\end{theorem}
For the sake of clarity,
we break the proof of Theorem \ref{W aa supp} into the following two parts:
the first equivalence in Lemmas \ref{W supp} and \ref{W aa}
and the second equivalence in Lemma \ref{aa supaa}.
Here, we first show the latter equivalence of Theorem \ref{W aa supp},
which is exactly the following result.
\begin{lemma}\label{aa supaa}
Let $p(\cdot)$, $q(\cdot)$, $s(\cdot)$, $\{\varphi_j\}_{j\in\mathbb{Z}_+}$,
$W$, and $\mathbb{A}$ be the same as in Theorem \ref{W aa supp}.
Then $\vec{f} \in B^{s(\cdot)}_{p(\cdot),q(\cdot)}(\mathbb{A},\varphi)$
if and only if $\sup_{\mathbb{A}, \varphi}(\vec{f}) \in b^{s(\cdot)}_{p(\cdot),q(\cdot)}$
and, moreover,
\begin{align}\label{aa supaa eq 1}
\left\| \vec{f} \right\|_{B^{s(\cdot)}_{p(\cdot),q(\cdot)}(\mathbb{A},\varphi)}
\sim \left\| \sup_{\mathbb{A},\varphi}\left( \vec{f} \right) \right\|_{b^{s(\cdot)}_{p(\cdot),q(\cdot)}},
\end{align}
where the positive equivalence constants are independent of $\vec{f}$.
\end{lemma}
To prove Lemma \ref{aa supaa},
we need some basic tools.
The following lemma can be found in the proof of \cite[Theorem 2.4]{fr21}
(see also \cite[Lemma 3.15]{fr21}).
\begin{lemma}\label{varf}
Let $\gamma \in \mathcal{S}$ satisfy $\widehat{\gamma}(\xi) = 1$
for any $\xi \in\mathbb{R}^n$ with $|\xi| \leq 2$
and
$${\mathop\mathrm{\,supp\,}} \widehat{\gamma} \subset
\left\{ \xi\in\mathbb{R}^n:\ |\xi|\leq \pi \right\}.$$
Then, for any $j\in\mathbb{Z}_+$ and $f\in \mathcal{S}'$
with ${\mathop\mathrm{\,supp\,}} \widehat{f} \subset
\{\xi\in\mathbb{R}^n:\ |\xi| \leq 2^{j+1}\}$,
one has $ f\in C^{\infty}$ and,
for any $x,y\in\mathbb{R}^n$,
$$ f(x) = \sum_{R\in \mathcal{Q}_j} 2^{-jn} f\left(x_R+y\right) \gamma_j\left(x-x_R-y\right)\ \ \text{pointwise}. $$
\end{lemma}
The following lemma is precisely \cite[Lemma 19]{kv12}.
\begin{lemma}\label{s eta}
Let $s(\cdot) \in LH$ and $j,m\in\mathbb{N}$.
If $R \in (C_{log}(s),\infty)$, where $C_{log}(s)$ is the same as in \eqref{clogp},
then, for any $x,y \in\mathbb{R}^n$,
\begin{align*}
2^{js(x)} \eta_{j,m+R}(x-y) \lesssim 2^{js(y)} \eta_{j,m} (x-y),
\end{align*}
where the implicit positive constant is independent of $x$ and $j$.
Moreover, for any $f\in L^1_{\rm loc}$,
\begin{align*}
2^{js(x)} \left(\eta_{j,m+R}\ast f\right)(x) \lesssim \eta_{j,m} \ast \left[ 2^{js(\cdot)} f\right] (x),
\end{align*}
where the implicit positive constant is independent of $x$, $j$, and $f$.
\end{lemma}

The following vector-valued inequality is exactly \cite[Lemma 4.7]{ah10}.
\begin{lemma}\label{eta bound seq 2}
Let $p(\cdot), q(\cdot)\in \mathcal{P}\cap LH$.
For any $m\in (n,\infty)$ and any sequence of measurable functions $\{f_j\}_{j\in\mathbb{Z}_+}$,
$$ \left\| \left\{ \eta_{j,m}\ast f_j \right\}_{j\in\mathbb{Z}_+} \right\|_{l^{q(\cdot)}(L^{p(\cdot)})}
\lesssim \left\| \left\{ f_j \right\}_{j\in\mathbb{Z}_+} \right\|_{l^{q(\cdot)}(L^{p(\cdot)})}, $$
where the implicit positive constant is independent of $\{f_j\}_{j\in\mathbb{Z}_+}$.
\end{lemma}

The following lemma is precisely \cite[Lemma 2.31]{bhyy23}.
\begin{lemma}\label{QP4}
For any cubes $Q,R\subset \mathbb{R}^n$,
any $x,x' \in Q$, and any $y,y' \in R$,
$$ 1+\frac{|x-y|}{l(Q)\vee l(R)} \sim 1+\frac{|x'-y'|}{l(Q)\vee l(R)}, $$
where the positive equivalence constants depend only on $n$.
\end{lemma}
Now, we give the proof of Lemma \ref{aa supaa}.
\begin{proof}[Proof of Lemma \ref{aa supaa}]
From \eqref{def supp f 1} and the definition of $\widetilde{\mathbf{1}}_Q$,
it follows that, for any $j\in\mathbb{Z}_+$, any cube $Q\in \mathcal{Q}_j$, and any $x\in Q$,
$$ 2^{js(x)}\left| A_Q \left( \varphi_{j}\ast \vec{f}\right)(x) \right|
\leq 2^{js(x)} |Q|^{-\frac12} \sup_{\mathbb{A}, \varphi, Q}\left( \vec{f} \right) \mathbf{1}_Q(x)
= 2^{js(x)}\sup_{\mathbb{A}, \varphi, Q}\left( \vec{f} \right) \widetilde{\mathbf{1}}_Q(x), $$
which, together with the definition of $A_j$,
further implies that, for any $j\in\mathbb{Z}_+$ and $x \in \mathbb{R}^n$,
$$ 2^{js(x)}\left| A_j \left( \varphi_{j}\ast \vec{f}\right)(x) \right| \leq 2^{js(x)} \sum_{Q\in\mathcal{Q}_j} \sup_{\mathbb{A}, \varphi, Q}\left( \vec{f} \right) \widetilde{\mathbf{1}}_Q(x). $$
Using this and the definition of $\| \cdot \|_{B^{s(\cdot)}_{p(\cdot),q(\cdot)}(\mathbb{A},\varphi)}$, we conclude immediately that
\begin{align*}
\left\| \vec{f} \right\|_{B^{s(\cdot)}_{p(\cdot),q(\cdot)}(\mathbb{A},\varphi)}
&\leq \left\| \left\{ 2^{js(\cdot)} \sum_{Q\in\mathcal{Q}_j} \sup_{\mathbb{A}, \varphi, Q}\left( \vec{f} \right) \widetilde{\mathbf{1}}_Q  \right\}_{j\in\mathbb{Z}_+} \right\|_{l^{q(\cdot)}(L^{p(\cdot)})}
= \left\| \sup_{\mathbb{A},\varphi}\left( \vec{f} \right) \right\|_{b^{s(\cdot)}_{p(\cdot),q(\cdot)}},
\end{align*}
which gives the proof that the left-hand side of \eqref{aa supaa eq 1}
is not more than the right-hand side.

Next, we prove the converse inequality.
Let $\vec{f} := (f_1,\dots,f_m) \in (\mathcal{S}')^m$.
Then, for any $j\in\mathbb{Z}_+$ and $k\in \{1,\dots,m\}$,
since ${\mathop\mathrm{\,supp\,}} \widehat{\varphi_j} \subset \{\xi\in\mathbb{R}^n:\ |\xi| \leq 2^{j+1}\}$,
we infer that
${\mathop\mathrm{\,supp\,}} \widehat{\varphi_j\ast f_k} \subset \{\xi\in\mathbb{R}^n:\ |\xi| \leq 2^{j+1}\}. $
By this and Lemma \ref{varf} with $f := \varphi_j\ast f_k$,
we find that, for any $j\in \mathbb{Z}_+$, $k\in \{1,\dots,m\}$, and $x,y\in\mathbb{R}^n$,
\begin{align*}
\left( \varphi_j \ast f_k \right)(x) = \sum_{R\in\mathcal{Q}_j} 2^{-jn}\left( \varphi_j \ast \vec{f} \right)(x_R+y)
\gamma_j(x-x_R-y)
\end{align*}
and hence
\begin{align}\label{Qsupf 2}
\left( \varphi_j \ast \vec{f} \right)(x) = \sum_{R\in\mathcal{Q}_j} 2^{-jn}\left( \varphi_j \ast \vec{f} \right)(x_R+y)
\gamma_j(x-x_R-y),
\end{align}
where $\gamma_j \in \mathcal{S}$ is the same as in Lemma \ref{varf}.
Fix constants $r \in (0, \min\{p_-,q_-, 1\})$ and $M\in (\frac{n}{r} + C_{\rm log}(s) + \Delta, \infty)$,
where $C_{\rm log}(s)$ is the same as in \eqref{clogp}
and $\Delta$ the same as in Lemma \ref{QP5}.
Using the fact that $\gamma \in \mathcal{S}$ and the definition of $\gamma_j$,
we obtain, for any $j\in \mathbb{Z}_+$, $N\in (0,\infty)$, and $x\in \mathbb{R}^n$,
$|\gamma_j (x)| \lesssim \frac{2^{jn}}{(1 + 2^j|x|)^N}$.
From this, \eqref{Qsupf 2}, and Lemma \ref{QP4},
we deduce that, for any $j\in\mathbb{Z}_+$, $Q\in\mathcal{Q}_j$,
$x\in Q$, $y\in\mathbb{R}^n$, and $x'\in Q$,
\begin{align*}
\left| A_Q\left( \varphi_j \ast \vec{f} \right)(x') \right|^r
&\leq \sum_{R\in \mathcal{Q}_j} \left| 2^{-jn} \gamma_j\left( x'-x_R-y \right) \right|^r \left| A_Q\left( \varphi_j \ast\vec{f} \right)(x_R +y) \right|^r\\
&\lesssim \sum_{R\in \mathcal{Q}_j} \frac{1}{(1+2^j|x'-x_R-y|)^{Mr}} \left| A_Q\left( \varphi_j \ast\vec{f} \right)(x_R +y) \right|^r\\
&\lesssim \sum_{R\in \mathcal{Q}_j} \frac{1}{(1+2^j|x-x_R-y|)^{Mr}} \left| A_Q\left( \varphi_j \ast\vec{f} \right)(x_R +y) \right|^r,
\end{align*}
which, combined with \eqref{def supp f 1}, further implies that,
for any $x \in Q$ and $y\in \mathbb{R}^n$,
\begin{align}\label{Qsupf}
\left[ |Q|^{-\frac12} \sup_{\mathbb{A},\varphi,Q}\left( \vec{f} \right) \right]^r
&= \sup_{x'\in Q} \left| A_Q\left( \varphi_j \ast \vec{f} \right)(x') \right|^r\nonumber \\
&\lesssim \sum_{R\in \mathcal{Q}_j} \frac{1}{(1+2^j|x-x_R-y|)^{Mr}} \left| A_Q\left( \varphi_j \ast\vec{f} \right)(x_R +y) \right|^r.
\end{align}
Observe that this holds for any $y \in \mathbb{R}^n$.
Applying this with integrating \eqref{Qsupf} over all $y$ in the cube $(0,2^{-j}]^n$, we obtain
\begin{align}\label{Qsupf 3}
\left[ |Q|^{-\frac12} \sup_{\mathbb{A},\varphi,Q}\left( \vec{f} \right) \right]^r
&\lesssim 2^{jn}\sum_{R\in \mathcal{Q}_j} \int_{(0,2^{-j}]^n} \frac{1}{(1+2^j|x-x_R-z|)^{Mr}} \left| A_Q\left( \varphi_j \ast\vec{f} \right)(x_R +z) \right|^r\,dz\nonumber \\
& = 2^{jn}\sum_{R\in \mathcal{Q}_j} \int_R \frac{1}{(1+2^j|x-z|)^{Mr}} \left| A_Q\left( \varphi_j \ast\vec{f} \right)(z) \right|^r\,dz.
\end{align}
Observe that, by Lemmas \ref{QP5} and \ref{QP4},
for any $Q,R\in \mathcal{Q}_j$, $x\in Q$, and $z\in R$,
we obtain $\|A_Q A_R^{-1}\|\sim (1 + 2^{j} |x - z|)^\Delta$.
Applying this with \eqref{Qsupf 3}, Tonelli's theorem,
and the definition of $A_j$,
we conclude that, for any $j\in\mathbb{Z}_+$, $Q\in\mathcal{Q}_j$, and $x\in\mathbb{R}^n$,
\begin{align*}
\left[ \sup_{\mathbb{A},\varphi,Q}\left( \vec{f} \right) \widetilde{\mathbf{1}}_Q(x) \right]^r
& = \left[ |Q|^{-\frac12} \sup_{\mathbb{A},\varphi,Q}\left( \vec{f} \right) \right]^r \mathbf{1}_Q(x)\\
&\lesssim 2^{jn} \sum_{R\in \mathcal{Q}_j} \int_R \frac{1}{(1+2^j|x-z|)^{Mr}} \left| A_Q\left( \varphi_j \ast\vec{f} \right)(z) \right|^r\,dz \mathbf{1}_Q(x)\nonumber\\
&\leq 2^{jn}\sum_{R\in \mathcal{Q}_j} \int_R \frac{\|A_QA_R^{-1}\|^r}{(1+2^j|x-z|)^{Mr}} \left| A_R\left( \varphi_j \ast\vec{f} \right)(z) \right|^r\,dz \mathbf{1}_Q(x)\nonumber\\
&\lesssim 2^{jn}\sum_{R\in \mathcal{Q}_j} \int_{\mathbb{R}^n} \frac{1}{(1+2^j|x-z|)^{\widetilde{M}r}} \left| A_R\left( \varphi_j \ast\vec{f} \right)(z) \right|^r \mathbf{1}_{R}(z) \,dz \mathbf{1}_Q(x)\nonumber\\
& = 2^{jn} \int_{\mathbb{R}^n} \frac{1}{(1+2^j|x-z|)^{\widetilde{M}r}} \left| A_j\left( \varphi_j \ast\vec{f} \right)(z) \right|^r\,dz \mathbf{1}_Q(x),
\end{align*}
where $\widetilde{M} := M - \Delta$.
For any $j\in\mathbb{Z}_+$, using this and the disjointness of the dyadic cubes in $\mathcal{Q}_j$
and letting
$g_j := \sum_{Q\in \mathcal{Q}_j} \sup_{\mathbb{A},\varphi,Q} ( \vec{f} ) \widetilde{\mathbf{1}}_Q $
and $ h_j := | A_j ( \varphi_j\ast \vec{f} ) | $,
we find that, for any $x\in\mathbb{R}^n$,
\begin{align}\label{Qsupf 1}
\left|g_j(x)\right|^r
 &= \sum_{Q\in \mathcal{Q}_j} \left[ \sup_{\mathbb{A},\varphi,Q}\left( \vec{f} \right) \widetilde{\mathbf{1}}_Q(x) \right]^r
\lesssim \sum_{Q\in \mathcal{Q}_j} \int_{\mathbb{R}^n} \frac{2^{jn}}{(1+2^j|x-z|)^{\widetilde{M}r}} \left| h_j(z) \right|^r\,dz \mathbf{1}_{Q}(x) \nonumber\\
& = \int_{\mathbb{R}^n} \frac{2^{jn}}{(1+2^j|x-z|)^{\widetilde{M}r}} \left| h_j(z) \right|^r\,dz
= \left(\eta_{j,\widetilde{M}r}\ast \left|h_j\right|^r\right)(x).
\end{align}
Now, let $R' \in (rC_{\rm log}(s), \infty)$ satisfy $\widetilde{M}r-R' > n$.
Using this, \eqref{Qsupf 1}, and Lemma \ref{s eta} with $f := h_j$,
we find that, for any $j\in\mathbb{Z}_+$ and $x\in\mathbb{R}^n$,
$$ 2^{js(x)}\left|g_j (x)\right|
\lesssim 2^{js(x)} \left[ \eta_{j,\widetilde{M}r}\ast \left|h_j\right|^r(x) \right]^\frac1r
\lesssim \left[\eta_{j,\widetilde{M}r-R'}\ast \left(2^{jrs(\cdot)}\left|h_j\right|^r\right)(x)\right]^\frac1r. $$
From this, Remark \ref{rem def seq}(iv), and Lemma \ref{eta bound seq 2}
with $p(\cdot) := \frac{p(\cdot)}{r}$ and $q(\cdot) := \frac{q(\cdot)}{r}$,
we infer that
\begin{align*}
\left\| \left\{2^{js(\cdot)} \left|g_j\right|\right\}_{j\in\mathbb{Z}_+} \right\|_{l^{q(\cdot)}(L^{p(\cdot)})}
&\lesssim  \left\| \left\{ \left[\eta_{j,\widetilde{M}r-R'}\ast \left(2^{jrs(\cdot)}\left|h_j\right|^r\right)\right]^\frac1r\right\}_{j\in\mathbb{Z}_+} \right\|_{l^{q(\cdot)}(L^{p(\cdot)})}\\
& =  \left\| \left\{ \eta_{j,\widetilde{M}r-R}\ast \left(2^{jrs(\cdot)}\left|h_j\right|^r\right) \right\}_{j\in\mathbb{Z}_+} \right\|^\frac1r_{l^{\frac{q(\cdot)}{r}}(L^{\frac{p(\cdot)}{r}})}\\
&\lesssim \left\| \left\{ 2^{jrs(\cdot)}\left|h_j\right|^r \right\}_{j\in\mathbb{Z}_+} \right\|^{\frac1r}_{l^{\frac{q(\cdot)}{r}}(L^{\frac{p(\cdot)}{r}})}
 = \left\| \left\{ 2^{js(\cdot)}\left|h_j\right| \right\}_{j\in\mathbb{Z}_+} \right\|_{l^{q(\cdot)}(L^{p(\cdot)})},
\end{align*}
which, combined with the definitions of $\| \cdot \|_{b^{s(\cdot)}_{p(\cdot),q(\cdot)}} $,
$g_j$, $h_j$, and $\|\cdot\|_{B^{s(\cdot)}_{p(\cdot),q(\cdot)}(\mathbb{A},\varphi)}$,
further implies that
\begin{align*}
\left\| \sup_{\mathbb{A},\varphi}\left( \vec{f} \right) \right\|_{b^{s(\cdot)}_{p(\cdot),q(\cdot)}}
= \left\| \left\{2^{js(\cdot)} \left|g_j\right|\right\}_{j\in\mathbb{Z}_+} \right\|_{l^{q(\cdot)}(L^{p(\cdot)})}
\lesssim \left\| \left\{ 2^{js(\cdot)}\left|h_j\right| \right\}_{j\in\mathbb{Z}_+} \right\|_{l^{q(\cdot)}(L^{p(\cdot)})}
= \|\vec{f}\|_{B^{s(\cdot)}_{p(\cdot),q(\cdot)}(\mathbb{A}, \varphi)}.
\end{align*}
This finishes the proof of Lemma \ref{aa supaa}.
\end{proof}
Next, we show the first equivalence of Theorem \ref{W aa supp}.
To this end, we first prove the following lemma.
\begin{lemma}\label{W supp}
Let $p(\cdot)$, $q(\cdot)$, $s(\cdot)$, $\{\varphi_j\}_{j\in\mathbb{Z}_+}$,
$W$, and $\mathbb{A}$ be the same as in Theorem \ref{W aa supp}.
Then, for any $\vec{f} \in (\mathcal{S}')^m$,
\begin{align}\label{Ajtj 10}
\left\| \vec{f} \right\|_{B^{s(\cdot)}_{p(\cdot),q(\cdot)}(W,\varphi)}
\lesssim \left\| \sup_{\mathbb{A},\varphi}\left( \vec{f} \right) \right\|_{b^{s(\cdot)}_{p(\cdot),q(\cdot)}},
\end{align}
where the implicit positive constant is independent of $\vec{f}$.
\end{lemma}
Before giving the proof of Lemma \ref{W supp}, we recall some basic tools.
The following lemma is precisely \cite[Lemma 2.4]{cp23}.
In what follows, for any $p(\cdot) \in \mathcal{P}$,
we use $p'(\cdot)$ to denote its conjecture,
that is, $p'(\cdot)$ satisfies $ \frac{1}{p(x)} + \frac{1}{p'(x)} = 1 $
for almost every $x\in\mathbb{R}^n$.
\begin{lemma}\label{g pro}
Let $p(\cdot) \in \mathcal{P}\cap LH$.
Then, for any $f\in L^{p(\cdot)}$ and $g\in L^{p'(\cdot)}$
and for any pairwise disjoint collection $\mathcal{K}$ of cubes,
$$ \sum_{Q\in \mathcal{K}} \left\| f\mathbf{1}_Q \right\|_{L^{p(\cdot)}} \left\| g\mathbf{1}_Q \right\|_{L^{p'(\cdot)}}
\lesssim \left\| f \right\|_{L^{p(\cdot)}} \left\| g \right\|_{L^{p'(\cdot)}},$$
where the implicit positive constant depends only on $n$ and $p(\cdot)$.
\end{lemma}
The following lemma indicates the relationship between the modular and the norm of variable Lebesgue spaces,
which is a special case of \cite[Lemma 2.1.14]{dhr17} with the modular $\rho := \rho_{L^{p(\cdot)}}$.
\begin{lemma}\label{f rho}
Let $p(\cdot) \in \mathcal{P}_0$ with $p_+<\infty$.
Then, for any $f\in \mathscr{M}$,
$ \|f\|_{L^{p(\cdot)}} \leq 1 $
if and only if
$ \rho_{L^{p(\cdot)}}(f) \leq 1 $
and, moreover, $ \|f\|_{L^{p(\cdot)}} = 1 $
if and only if $\rho_{L^{p(\cdot)}}(f) = 1$.
\end{lemma}

The following lemma is a combination of the convexification for $L^{p(\cdot)}$
and Lemma \ref{f rho} and it has already been used in \cite{ah10}.
We omit the details here.
\begin{lemma}\label{f pq}
Let $p(\cdot),q(\cdot)\in \mathcal{P}_0\cap LH$.
Then, for any $f\in\mathscr{M}$,
$\||f|^{q(\cdot)}\|_{L^\frac{p(\cdot)}{q(\cdot)}} \leq 1$
if and only if
$ \|f\|_{L^{p(\cdot)}} \leq 1. $
\end{lemma}

The following lemma is a direct application of \cite[Lemma 2.1.14]{dhr17}
with the fact that $ \rho_{l^{q(\cdot)}(L^{p(\cdot)})} $ is a semimodular.
We omit the details here.
\begin{lemma}\label{f rho seq}
Let $p(\cdot),q(\cdot) \in \mathcal{P}_0\cap LH$.
Then, for any sequence of measurable functions $\{f_j\}_{j\in\mathbb{Z}_+}$,
$ \|\{f_j\}_{j\in\mathbb{Z}_+}\|_{l^{q(\cdot)}(L^{p(\cdot)})} \leq 1 $
if and only if
$ \rho_{l^{q(\cdot)}(L^{p(\cdot)})}(\{f_j\}_{j\in\mathbb{Z}_+}) \leq 1 $
and, moreover, $ \|\{f_j\}_{j\in\mathbb{Z}_+}\|_{l^{q(\cdot)}(L^{p(\cdot)})} = 1 $
if and only if $\rho_{l^{q(\cdot)}(L^{p(\cdot)})}(\{f_j\}_{j\in\mathbb{Z}_+}) = 1$.
\end{lemma}
The following result can be obtained directly by Definition \ref{def seq};
we omit the details here.
\begin{lemma}\label{f rho seq 1}
Let $p(\cdot),q(\cdot) \in \mathcal{P}_0\cap LH$.
For any sequence of measurable functions $\{f_j\}_{j\in\mathbb{Z}_+}$,
if there exists a positive constant $C_0$ such that
$\rho_{l^{q(\cdot)}(L^{p(\cdot)})}(\{f_j\}_{j\in\mathbb{Z}_+}) \leq C_0$,
then
$$ \|\{f_j\}_{j\in\mathbb{Z}_+}\|_{l^{q(\cdot)}(L^{p(\cdot)})} \leq \max\left\{C_0^{\frac{1}{q_+}}, C_0^{\frac{1}{q_-}}\right\}. $$
\end{lemma}
The following result can be found in the proof of \cite[Theorem 1]{d12}.
\begin{lemma}\label{2js eq}
Let $p(\cdot) \in LH$.
Then, for any $j\in\mathbb{Z}_+$, any cube $Q\in \mathcal{Q}_j$, and any $x,y\in Q$,
$2^{jp(x)} \sim 2^{jp(y)},$
where the positive equivalence constants depend only on $p(\cdot)$ and $n$.
Moreover, for any $j\in\mathbb{Z}_+$, any $\delta \in [1+2^{-j}, 1+2^{-j+1}]$,
any cube $Q\in \mathcal{Q}_j$, and any $x,y\in Q$,
$\delta^{jp(x)} \sim \delta^{jp(y)},$
where the positive equivalence constants depend only on $p(\cdot)$ and $n$.
\end{lemma}
Now, we give the proof of Lemma \ref{W supp}.
\begin{proof}[Proof of Lemma \ref{W supp}]
We first consider the case $ \|\sup_{\mathbb{A},\varphi}(\vec{f})\|_{b^{s(\cdot)}_{p(\cdot),q(\cdot)}} = 0$.
In this case, by the fact that $ \|\cdot\|_{b^{s(\cdot)}_{p(\cdot),q(\cdot)}} $ is a quasi-norm,
we obtain $\sup_{\mathbb{A},\varphi}(\vec{f}) = 0$ and hence,
using the definition of $\sup_{\mathbb{A},\varphi}(\vec{f})$, we find that,
for any $j\in \mathbb{Z}_+$, $\varphi_j \ast \vec{f} = 0$,
which, combined with the definition of $\|\cdot\|_{B^{s(\cdot)}_{p(\cdot),q(\cdot)}(W,\varphi)}$,
further implies that $\|\vec{f}\|_{B^{s(\cdot)}_{p(\cdot),q(\cdot)}(W,\varphi)} = 0$.
Thus, \eqref{Ajtj 10} holds in this case.

Next, we assume $ \|\sup_{\mathbb{A},\varphi}(\vec{f})\|_{b^{s(\cdot)}_{p(\cdot),q(\cdot)}} \neq 0$.
From the fact that $\|\cdot\|_{b^{s(\cdot)}_{p(\cdot),q(\cdot)}}$
and $\|\cdot\|_{B^{s(\cdot)}_{p(\cdot),q(\cdot)}(W,\varphi)}$ are both quasi-norms,
it follows that, to prove the present lemma in this case,
it is sufficient to show that, for any measurable function $\vec{f}$ satisfying
\begin{align}\label{Ajtj 1}
\left\| \sup_{\mathbb{A},\varphi}\left( \vec{f} \right) \right\|_{b^{s(\cdot)}_{p(\cdot),q(\cdot)}} =
\left\| \left\{ 2^{js(\cdot)}\sum_{Q\in\mathcal{Q}_j} \sup_{\mathbb{A},\varphi}\left(\vec{f}\right) \widetilde{\mathbf{1}_Q} \right\}_{j\in\mathbb{Z}_+} \right\|_{l^{q(\cdot)}(L^{p(\cdot)})} = 1,
\end{align}
we have
\begin{align}\label{Ajtj 12}
\left\| \vec{f} \right\|_{B^{s(\cdot)}_{p(\cdot),q(\cdot)}(W,\varphi)}
= \left\| \left\{ 2^{js(\cdot)} \left| W(\cdot) \left(\varphi_j\ast \vec{f}\right)(\cdot)\right| \right\}_{j\in\mathbb{Z}_+} \right\|_{l^{q(\cdot)}(L^{p(\cdot)})} \lesssim 1
\end{align}
with the implicit positive constant independent of $\vec{f}$.

For any $j\in \mathbb{Z}_+$,
let $t_j := \sum_{Q\in\mathcal{Q}_j} \sup_{\mathbb{A},\varphi,Q}( \vec{f} ) \widetilde{\mathbf{1}}_Q$.
We claim that, to prove \eqref{Ajtj 12},
we only need to show that
there exists a positive constant $C$ such that,
for any $j \in \mathbb{Z}_+$,
\begin{align}\label{Ajtj 2}
\left\| \delta_j^{-\frac{1}{q(\cdot)}}2^{js(\cdot)} \left| W(\cdot)\left(\varphi_j\ast \vec{f}\right)(\cdot) \right| \right\|_{L^{p(\cdot)}} \leq C,
\end{align}
where, for any $j\in\mathbb{Z}_+$,
\begin{align}\label{def deltaj}
\delta_j := \left\| 2^{js(\cdot)q(\cdot)}  t_j^{q(\cdot)} \right\|_{L^{\frac{p(\cdot)}{q(\cdot)}}} + 2^{-j}
\end{align}
and $ \delta_j \in [2^{-j}, 1+ 2^{-j}]$.
Indeed, if \eqref{Ajtj 2} holds,
then, for any $j\in\mathbb{Z}_+$,
$$\left\| C^{-1} \delta_j^{-\frac{1}{q(\cdot)}}2^{js(\cdot)} \left| W(\cdot)\left(\varphi_j\ast \vec{f}\right)(\cdot) \right| \right\|_{L^{p(\cdot)}} \leq 1.$$
Applying this with Lemma \ref{f pq},
we find that
$$\left\| C^{-q(\cdot)} \delta_j^{-1} 2^{js(\cdot)q(\cdot)} \left| W(\cdot)\left(\varphi_j\ast \vec{f}\right)(\cdot) \right|^{q(\cdot)} \right\|_{L^{\frac{p(\cdot)}{q(\cdot)}}} \leq 1,$$
which further implies that
\begin{align*}
\left\| C^{-q(\cdot)}  2^{js(\cdot)q(\cdot)} \left| W(\cdot)\left(\varphi_j\ast \vec{f}\right)(\cdot) \right|^{q(\cdot)} \right\|_{L^{\frac{p(\cdot)}{q(\cdot)}}} \leq \delta_j.
\end{align*}
Observe that, by Remark \ref{rem def seq}{\rm (i)} and the definition of $\delta_j$,
\begin{align*}
\rho_{l^{q(\cdot)}(L^{p(\cdot)})} \left( \left\{C^{-1} 2^{js(\cdot)} \left| W(\cdot)\left(\varphi_j\ast \vec{f}\right)(\cdot) \right|\right\}_{j\in\mathbb{Z}_+} \right)
&= \sum_{j=0}^\infty \left\| C^{-q(\cdot)}  2^{js(\cdot)q(\cdot)} \left| W(\cdot)\left(\varphi_j\ast \vec{f}\right)(\cdot) \right|^{q(\cdot)} \right\|_{L^{\frac{p(\cdot)}{q(\cdot)}}}\\
&\leq \sum_{j=0}^\infty \delta_j
 = \sum_{j=0}^\infty \left\| 2^{js(\cdot)q(\cdot)}  t_j^{q(\cdot)} \right\|_{L^{\frac{p(\cdot)}{q(\cdot)}}} + \sum_{j=0}^\infty 2^{-j}.
\end{align*}
This, together with Remark \ref{rem def seq}{\rm (i)},
Lemma \ref{f rho seq}, and \eqref{Ajtj 1},
further implies that
\begin{align*}
\rho_{l^{q(\cdot)}(L^{p(\cdot)})} \left( \left\{C^{-1} 2^{js(\cdot)} \left| W(\cdot)\left(\varphi_j\ast \vec{f}\right)(\cdot) \right|\right\}_{j\in\mathbb{Z}_+} \right)
&= \rho_{l^{q(\cdot)}(L^{p(\cdot)})}\left( \left\{ 2^{js(\cdot)} t_j \right\}_{j\in\mathbb{Z}_+} \right) +2\\
& = \left\| \sup_{\mathbb{A},\varphi}\left( \vec{f} \right) \right\|_{b^{s(\cdot)}_{p(\cdot),q(\cdot)}} + 2= 3. \nonumber
\end{align*}
From this and Lemma \ref{f rho seq 1},
we deduce that
$\| \{C^{-1} 2^{js(\cdot)} | W(\cdot)(\varphi_j\ast \vec{f})(\cdot) |\}_{j\in\mathbb{Z}_+} \|_{l^{q(\cdot)}(L^{p(\cdot)})} \lesssim 1$,
which further implies that
$\| \{2^{js(\cdot)} | W(\cdot)(\varphi_j\ast \vec{f})(\cdot) |\}_{j\in\mathbb{Z}_+} \|_{l^{q(\cdot)}(L^{p(\cdot)})} \lesssim 1. $
This finishes the proof of the above claim.

Now, we turn to prove \eqref{Ajtj 2}.
Let $r := \min\{1, p_-\}$ and hence $\frac{p(\cdot)}{r} \in \mathcal{P} \cap LH$.
Then, by Lemmas \ref{con f} and \ref{fg Lp} with $p(\cdot) := \frac{p(\cdot)}{r}$,
we find that, for any $j\in\mathbb{Z}_+$,
\begin{align}\label{Ajtj 3}
&\left\| \delta_j^{-\frac{1}{q(\cdot)}}2^{js(\cdot)} \left| W(\cdot)\left(\varphi_j\ast \vec{f}\right)(\cdot) \right| \right\|^r_{L^{p(\cdot)}}\nonumber\\
&\quad = \left\| \delta_j^{-\frac{r}{q(\cdot)}}2^{jrs(\cdot)} \left| W(\cdot)\left(\varphi_j\ast \vec{f}\right)(\cdot) \right|^r \right\|_{L^{\frac{p(\cdot)}{r}}}\nonumber\\
&\quad \sim \sup_{\|g\|_{L^{(\frac{p(\cdot)}{r})'}}\leq 1} \int_{\mathbb{R}^n} \delta_j^{-\frac{r}{q(x)}}2^{jrs(x)} \left| W(x)\left(\varphi_j\ast \vec{f}\right)(x) \right|^r |g(x)|\,dx
\end{align}
Now, let $g\in L^{(\frac{p(\cdot)}{r})'}$ be any given function with $\|g\|_{L^{(\frac{p(\cdot)}{r})'}}\leq 1$.
Then, from Lemma \ref{2js eq} with $p(\cdot) := -\frac{1}{q(\cdot)}$,
it follows that,
for any $j\in\mathbb{Z}_+$, $Q \in \mathcal{Q}_j$, and $x\in Q$,
\begin{align}\label{Ajtj 15}
\delta_j^{-\frac{1}{q(x)}} 2^{js(x)} \sim \delta_j^{-\frac{1}{q(x_Q)}} 2^{js(x_Q)}.
\end{align}
By this, the disjointness of the dyadic cubes in $\mathcal{Q}_j$,
\eqref{def supp f 1}, and Lemma \ref{Holder} with $p(\cdot) := \frac{p(\cdot)}{r}$,
we obtain
\begin{align}\label{Ajtj 16}
&\int_{\mathbb{R}^n} \delta_j^{-\frac{r}{q(x)}}2^{jrs(x)} \left| W(x)\left(\varphi_j\ast \vec{f}\right)(x) \right|^r |g(x)|\,dx\nonumber \\
&\quad  = \sum_{Q\in \mathcal{Q}_j} \int_{Q} \delta_j^{-\frac{r}{q(x)}}2^{jrs(x)} \left| W(x)\left(\varphi_j\ast \vec{f}\right)(x) \right|^r |g(x)|\,dx\nonumber\\
&\quad \sim \sum_{Q\in \mathcal{Q}_j} \delta_j^{-\frac{r}{q(x_Q)}}2^{jrs(x_Q)}\int_{Q}  \left| W(x)\left(\varphi_j\ast \vec{f}\right)(x) \right|^r g(x)\,dx\nonumber \\
&\quad \leq \sum_{Q\in \mathcal{Q}_j} \delta_j^{-\frac{r}{q(x_Q)}}2^{jrs(x_Q)} \int_{Q} \left\| W(x)A_Q^{-1}\right\|^r \left| A_Q \left(\varphi_j\ast \vec{f}\right)(x) \right|^r g(x)\,dx\nonumber \\
&\quad \leq \sum_{Q\in \mathcal{Q}_j} \delta_j^{-\frac{r}{q(x_Q)}}2^{jrs(x_Q)} \int_{Q} \left\| W(x)A_Q^{-1}\right\|^r |Q|^{-\frac r2} \left[\sup_{\mathbb{A},\varphi,Q}\left( \vec{f} \right)\right]^r g(x)\,dx\nonumber \\
&\quad \lesssim \sum_{Q\in \mathcal{Q}_j} \delta_j^{-\frac{r}{q(x_Q)}}2^{jrs(x_Q)} |Q|^{-\frac r2} \left[\sup_{\mathbb{A},\varphi,Q}\left( \vec{f} \right)\right]^r \left\|\, \left\| W(\cdot)A_Q^{-1}\right\|^r\mathbf{1}_Q \right\|_{L^{\frac{p(\cdot)}{r}}}  \left\|g\mathbf{1}_Q\right\|_{L^{(\frac{p(\cdot)}{r})'}}.
\end{align}
From Lemmas \ref{con f} and \ref{eq reduc M} with $M := A_Q^{-1}$,
we deduce that, for any cube $Q$ in $\mathbb{R}^n$,
\begin{align}\label{WAQ-1}
\left\|\, \left\| W(\cdot)A_Q^{-1}\right\|^r\mathbf{1}_Q \right\|_{L^{\frac{p(\cdot)}{r}}}
= \left\|\, \left\| W(\cdot)A_Q^{-1}\right\|\mathbf{1}_Q \right\|_{L^{p(\cdot)}}^{r}
\sim \left\|\mathbf{1}_Q\right\|_{L^{p(\cdot)}}^r \left\| A_Q A_Q^{-1} \right\| =  \left\|\mathbf{1}_Q\right\|_{L^{\frac{p(\cdot)}{r}}}.
\end{align}
Using this, \eqref{Ajtj 15}, \eqref{Ajtj 16},
and Lemma \ref{g pro} with $\mathcal{K} := \mathcal{Q}_j$,
we conclude that
\begin{align*}
&\int_{\mathbb{R}^n} \delta_j^{-\frac{r}{q(x)}}2^{jrs(x)} \left| W(x)\left(\varphi_j\ast \vec{f}\right)(x) \right|^r |g(x)|\,dx\nonumber\\
&\quad \lesssim \sum_{Q\in \mathcal{Q}_j} \delta_j^{-\frac{r}{q(x_Q)}}2^{jrs(x_Q)} |Q|^{-\frac r2} \left[\sup_{\mathbb{A},\varphi,Q}\left( \vec{f} \right)\right]^r \left\| \mathbf{1}_Q \right\|_{L^{\frac{p(\cdot)}{r}}}  \left\|g\mathbf{1}_Q\right\|_{L^{(\frac{p(\cdot)}{r})'}}\nonumber \\
&\quad \lesssim \sum_{Q\in \mathcal{Q}_j}  \left\| \left[\delta_j^{-\frac{1}{q(\cdot)}}2^{js(\cdot)} \sup_{\mathbb{A},\varphi,Q}\left( \vec{f} \right) \widetilde{\mathbf{1}}_Q\right]^r \right\|_{L^{\frac{p(\cdot)}{r}}}  \left\|g\mathbf{1}_Q\right\|_{L^{(\frac{p(\cdot)}{r})'}}\nonumber\\
&\quad \lesssim \left\| \left[\delta_j^{-\frac{1}{q(\cdot)}}2^{js(\cdot)} \sum_{Q\in \mathcal{Q}_j}  \sup_{\mathbb{A},\varphi,Q}\left( \vec{f} \right) \widetilde{\mathbf{1}}_Q\right]^r \right\|_{L^{\frac{p(\cdot)}{r}}}  \left\|g\right\|_{L^{(\frac{p(\cdot)}{r})'}},
\end{align*}
which, together with \eqref{Ajtj 1} and Lemma \ref{con f},
further implies that
\begin{align}\label{Ajtj 17}
\left\| \delta_j^{-\frac{1}{q(\cdot)}}2^{js(\cdot)} \left| W(\cdot)\left(\varphi_j\ast \vec{f}\right)(\cdot) \right| \right\|^r_{L^{p(\cdot)}}
&\leq \sup_{\|g\|_{L^{(\frac{p(\cdot)}{r})'}}\leq 1} \left\| \left[\delta_j^{-\frac{1}{q(\cdot)}}2^{js(\cdot)} \sum_{Q\in \mathcal{Q}_j}  \sup_{\mathbb{A},\varphi,Q}\left( \vec{f} \right) \widetilde{\mathbf{1}}_Q\right]^r \right\|_{L^{\frac{p(\cdot)}{r}}}  \left\|g\right\|_{L^{(\frac{p(\cdot)}{r})'}}\nonumber \\
&\leq \sup_{\|g\|_{L^{(\frac{p(\cdot)}{r})'}}\leq 1} \left\| \delta_j^{-\frac{r}{q(\cdot)}}2^{jrs(\cdot)} t_j^r \right\|_{L^{\frac{p(\cdot)}{r}}}
 = \left\| \delta_j^{-\frac{1}{q(\cdot)}}2^{js(\cdot)} t_j \right\|^r_{L^{p(\cdot)}}.
\end{align}
By \eqref{def deltaj}, we are easy to find that
$ \| \delta_j^{-1} 2^{js(\cdot)q(\cdot)} t_j^{q(\cdot)} \|_{L^{\frac{p(\cdot)}{q(\cdot)}}} \leq 1, $
which, combined with Lemma \ref{f pq},
further implies that
$ \| \delta_j^{-\frac{1}{q(\cdot)}}2^{js(\cdot)} t_j \|_{L^{p(\cdot)}} \leq 1. $
From this and \eqref{Ajtj 17},
we infer that
$$ \left\| \delta_j^{-\frac{1}{q(\cdot)}}2^{js(\cdot)} \left| W(\cdot)\left(\varphi_j\ast \vec{f}\right)(\cdot)\right| \right\|_{L^{p(\cdot)}} \lesssim 1. $$
This finishes the proof of \eqref{Ajtj 2} and hence Lemma \ref{W supp}.
\end{proof}
Finally, we show the last part of Theorem \ref{W aa supp}.
\begin{lemma}\label{W aa}
Let $p(\cdot)$, $q(\cdot)$, $s(\cdot)$, $\{\varphi_j\}_{j\in\mathbb{Z}_+}$,
$W$, and $\mathbb{A}$ be the same as in Theorem \ref{W aa supp}.
Then, for any $\vec{f} \in (\mathcal{S}')^m$,
$$\left\| \sup_{\mathbb{A},\varphi}\left( \vec{f} \right) \right\|_{b^{s(\cdot)}_{p(\cdot),q(\cdot)}}
\lesssim \left\| \vec{f} \right\|_{B^{s(\cdot)}_{p(\cdot),q(\cdot)}(W,\varphi)}, $$
where the implicit positive constant is independent of $\vec{f}$.
\end{lemma}

Before giving the proof of Lemma \ref{W aa},
we recall some necessary tools.
For any $N \in\mathbb{Z}_+$ and $\vec{f} \in (\mathcal{S}')^m$,
let
$$ \inf_{\mathbb{A}, \varphi, N}\left( \vec{f} \right) := \left\{ \inf_{\mathbb{A}, \varphi, Q, N}\left( \vec{f} \right) \right\}_{Q\in\mathcal{Q}_+}, $$
where, for any $Q \in \mathcal{Q}_+$,
\begin{align}\label{def inf}
\inf_{\mathbb{A}, \varphi, Q, N}\left( \vec{f} \right) := |Q|^\frac12
\max\left\{ \inf_{y\in\widetilde{Q}} \left| A_{\widetilde{Q}} \left( \varphi_{j_Q} \ast \vec{f} \right)(y) \right| :\
\widetilde{Q} \in \mathcal{Q}_{j_Q + N},\ \widetilde{Q} \subset Q \right\}.
\end{align}

For any sequence
$t := \{t_Q\}_{Q\in\mathcal{Q}_+} \subset \mathbb{C}$, $r\in(0,\infty]$,
and $\lambda \in (0,\infty)$,
let $t^\ast_{r,\lambda} := \{ (t^\ast_{r,\lambda})_Q \}_{Q\in\mathcal{Q}_+}$,
where, for any $Q\in\mathcal{Q}_+$,
\begin{align}\label{def tast}
\left(t^\ast_{r,\lambda}\right)_Q :=
\left[ \sum_{R\in\mathcal{Q}_+, l(R) = l(Q)} \frac{|t_R|^r}{\{1 + [l(R)]^{-1}|x_R - x_Q|\}^\lambda} \right]^{\frac1r}.
\end{align}
The following lemma is exactly \cite[Lemma 3.13]{d15}.
\begin{lemma}\label{tast bound}
Let $p(\cdot), q(\cdot)\in \mathcal{P}_0\cap LH$ and let $s(\cdot)\in LH$,
$ r\in (0,p_-) $,
$$\widetilde{R} := r \min\left\{2C_{\rm log}(q) + C_{\rm log}(s), 2\left(\frac{1}{q_-}
- \frac{1}{q_+}\right) + s_+-s_- \right\}, $$
and $\lambda \in (n + \widetilde{R},\infty) $.
Then, for any $t := \{t_Q\}_{Q\in \mathcal{Q}_+}$,
$ \| t^\ast_{r,\lambda} \|_{b^{s(\cdot)}_{p(\cdot), q(\cdot)}} \sim \| t \|_{b^{s(\cdot)}_{p(\cdot), q(\cdot)}}, $
where the positive equivalence constants are independent of $t$.
\end{lemma}

The following lemma is precisely \cite[Lemma 3.15]{bhyy24}.
\begin{lemma}\label{a b seq}
Let $j\in\mathbb{Z}_+$, $\vec{f} \in (\mathcal{S}')^m$ satisfy
${\mathop\mathrm{\,supp\,}} \widehat{\vec{f}} \subset \left\{ \xi\in\mathbb{R}^n:\ |\xi|\leq 2^{j+1} \right\}$,
$\mathbb{A} := \{A_Q\}_{Q\in\mathcal{Q}_+}$ be strongly doubling of order $(d_1, d_2)$
for some $d_1,d_2\in[0,\infty)$,
and $N\in\mathbb{Z}_+$ sufficiently large.
For any $Q\in\mathcal{Q}_+$, let $a_Q := |Q|^{\frac12}{\mathop\mathrm{\,supp\,}}_{y\in Q} |A_Q\vec{f}(y)|$ and
\begin{align}\label{def bQN}
b_{Q,N} := |Q|^{\frac12} \max\left\{ \inf_{y\in \widetilde{Q}} \left\{\left| A_{\widetilde{Q}} \vec{f}(y)\right|:\
\widetilde{Q} \in \mathcal{Q}_{j_Q + N},\ \widetilde{Q} \subset Q \right\} \right\}.
\end{align}
Let $a := \{ a_Q \}_{Q\in\mathcal{Q}_+}$, $b := \{b_{Q,N}\}_{Q\in\mathcal{Q}_+}$,
$r \in (0,\infty)$, and $\lambda \in (n,\infty)$.
Then, for any $Q\in\mathcal{Q}_j$, $(a^\ast_{r,\lambda})_Q \sim (b^\ast_{r,\lambda})_Q$,
where the positive equivalence constants are independent of $\vec{f}$, $j$, and $Q$.
\end{lemma}

The following lemma is exactly \cite[Lemma 2.8]{cp24}.
\begin{lemma}
Let $p(\cdot) \in \mathcal{P}_0\cap LH$.
Then, for any cube $Q$ in $\mathbb{R}^n$ and any $x,y \in Q$,
$|Q|^{-|p(x) - p(y)|} \lesssim 1, $
where the implicit positive constant depends only on $p(\cdot)$ and $n$.
\end{lemma}
The following result, in the case where $q$ is a constant, is covered
by \cite[Lemma 2.6]{mns13} and, in the case where $q(\cdot)$ is variable,
is precisely \cite[Lemma 11]{n16} with $3Q_{v,m}$ replaced by $dQ$
and hence we omit the details of its proof.
\begin{lemma}\label{Q EQ 1}
Let $p(\cdot), q(\cdot)\in \mathcal{P}_0\cap LH$ and $s(\cdot)\in LH$
and let $d \in [1,\infty)$ and $\delta\in (0,1)$.
If $\{E_Q\}_{Q\in \mathcal{Q}_+}$ is a sequence of measurable sets in $\mathbb{R}^n$
satisfying $E_Q \subset dQ$ and $|E_Q| \geq \delta |dQ|$ for any $Q \in \mathcal{Q}_+$,
then, for any sequence $t := \{t_Q\}_{Q\in \mathcal{Q}_+} \subset \mathbb{C}$,
\begin{align*}
\left\| t \right\|_{b^{s(\cdot)}_{p(\cdot),q(\cdot)}}
\sim \left\| \left\{ 2^{j[s(\cdot)+\frac n2]} \sum_{Q\in \mathcal{Q}_j} \left|t_Q\right|\mathbf{1}_{E_Q} \right\}_{j\in\mathbb{Z}_+} \right\|_{l^{q(\cdot)}(L^{p(\cdot)})},
\end{align*}
where the positive equivalence constants are independent of $t$.
\end{lemma}

The following lemma is precisely \cite[Lemma 3.25]{yyz25}
(see \cite[Corollary 3.9]{bhyy23} for the related result for matrix $A_{p,\infty}$ weights).
\begin{lemma}\label{Wpinfty}
Let $p(\cdot) \in \mathcal{P}_0\cap LH$ and
$W\in \mathscr{A}_{p(\cdot), \infty}$.
Then there exists a positive constant $C$, depending on $p(\cdot)$ and $n$,
such that, for any cube $Q$ in $\mathbb{R}^n$ and any $M\in (0,\infty)$,
\begin{align*}
\left|\left\{ y\in Q:\ \left\|A_QW^{-1}(y)\right\| \geq e^M \right\}\right| \leq \frac{\log(C[W]_{\mathscr{A}_{p(\cdot),\infty}})}{M}|Q|.
\end{align*}
\end{lemma}
Now, we give the proof of Lemma \ref{W aa}.
\begin{proof}[Proof of Lemma \ref{W aa}]
From Lemma \ref{a b seq} and
the fact that ${\mathop\mathrm{\,supp\,}} \widehat{\varphi_j\ast \vec{f}} \subset \{ \xi\in\mathbb{R}^n:\ |\xi| \leq 2^{j+1} \}$
for any $j\in\mathbb{Z}_+$,
we deduce that, for any $r\in (0,\infty)$, $\lambda \in (n,\infty)$,
$j\in\mathbb{Z}_+$, and $Q\in\mathcal{Q}_j$,
$ (a^\ast_{r,\lambda})_Q \sim (b^\ast_{r,\lambda})_Q $,
where $a := \{a_Q\}_{Q\in\mathcal{Q}_+}$ and $b := \{b_{Q,N}\}_{Q\in\mathcal{Q}_+}$
are the same as in Lemma \ref{a b seq}.
By this and Lemma \ref{tast bound},
we conclude that
\begin{align}\label{a b seq eq 1}
\left\| a \right\|_{b^{s(\cdot)}_{p(\cdot),q(\cdot)}}
\sim \left\| a^\ast_{r,\lambda} \right\|_{b^{s(\cdot)}_{p(\cdot),q(\cdot)}}
\sim  \left\| b^\ast_{r,\lambda} \right\|_{b^{s(\cdot)}_{p(\cdot),q(\cdot)}}
\sim \left\| b \right\|_{b^{s(\cdot)}_{p(\cdot),q(\cdot)}}.
\end{align}
Notice that, by \eqref{def inf} and \eqref{def bQN},
we find that, for any $\widetilde{Q} \in \mathcal{Q}_{j_Q+N}$ with $\widetilde{Q} \subset Q$
and for any $y\in \widetilde{Q}$,
\begin{align}\label{bQN est}
b_{Q,N} =  |Q|^{-\frac12} \inf_{\mathbb{A}, \varphi, Q, N}\left( \vec{f} \right)
= \inf_{y\in \widetilde{Q}} \left| A_{\widetilde{Q}} \left( \varphi_{j_Q} \ast \vec{f} \right)(y) \right|
\leq \inf_{y\in \widetilde{Q}} \left\| A_{\widetilde{Q}} W^{-1}(y)\right\|  \left| W(y) \left( \varphi_{j_Q} \ast \vec{f} \right)(y) \right|.
\end{align}
Let $E_Q := \{ y\in \widetilde{Q}:\ \|A_{\widetilde{Q}} W^{-1}(y)\| < (C[W]_{\mathscr{A}_{p(\cdot),\infty}})^2 \}$,
where $C$ is the same as in Lemma \ref{Wpinfty}.
Then it follows from \eqref{bQN est} and the assumption $E_Q\subset \widetilde{Q}$ that,
for any $Q\in\mathcal{Q}_+$,
\begin{align}\label{bQN est 1}
b_{Q,N} \lesssim \inf_{y\in E_Q} \left\| A_{\widetilde{Q}} W^{-1}(y)\right\|  \left| W(y) \left( \varphi_{j_Q} \ast \vec{f} \right)(y) \right|
\lesssim \inf_{y\in E_Q} \left| W(y) \left( \varphi_{j_Q} \ast \vec{f} \right)(y) \right|.
\end{align}
Observe that, by the definition of $E_Q$ and Lemma \ref{Wpinfty},
we obtain $$|E_Q| = |\widetilde{Q}| - |\widetilde{Q} \setminus E_Q| \geq \frac12|\widetilde{Q}| = 2^{-Nn-1}|Q|.$$
Using this, \eqref{a b seq eq 1}, and Lemma \ref{Q EQ 1} with $d := 1$ and $\delta := 2^{-Nn-1}$
and using \eqref{bQN est 1} and the definition of $\| \cdot\|_{B^{s(\cdot)}_{p(\cdot),q(\cdot)}(W,\varphi)}$,
we conclude that
\begin{align*}
\left\|a\right\|_{b^{s(\cdot)}_{p(\cdot),q(\cdot)}}
&\sim \left\|b\right\|_{b^{s(\cdot)}_{p(\cdot),q(\cdot)}}
\lesssim \left\| \left\{ 2^{j[s(\cdot)+\frac n2]} \sum_{Q\in \mathcal{Q}_j} \left|b_{Q,N}\right|\mathbf{1}_{E_Q} \right\}_{j\in\mathbb{Z}_+} \right\|_{l^{q(\cdot)}(L^{p(\cdot)})}\\
&\lesssim \left\| \left\{ 2^{j[s(\cdot)+\frac n2]} \sum_{Q\in \mathcal{Q}_j} \left|W(\cdot) \left( \varphi_j\ast \vec{f}\right)(y)\right|\mathbf{1}_{E_Q} \right\}_{j\in\mathbb{Z}_+} \right\|_{l^{q(\cdot)}(L^{p(\cdot)})}\\
&\leq \left\| \left\{ 2^{j[s(\cdot)+\frac n2]} \sum_{Q\in \mathcal{Q}_j} \left|W(\cdot) \left( \varphi_j\ast \vec{f}\right)(y)\right|\mathbf{1}_{Q} \right\}_{j\in\mathbb{Z}_+} \right\|_{l^{q(\cdot)}(L^{p(\cdot)})}
 = \left\| \vec{f} \right\|_{B^{s(\cdot)}_{p(\cdot),q(\cdot)}(W,\varphi)},
\end{align*}
which completes the proof of Lemma \ref{W aa}.
\end{proof}

Finally, we give the proof of Theorem \ref{W aa supp}.
\begin{proof}[Proof of Thorem \ref{W aa supp}]
By Lemmas \ref{W aa} and \ref{W supp},
we obtain
$$ \| \vec{f} \|_{B^{s(\cdot)}_{p(\cdot),q(\cdot)}(\mathbb{A},\varphi)}
\lesssim \| \vec{f} \|_{B^{s(\cdot)}_{p(\cdot),q(\cdot)}(\mathbb{A},\varphi)}
\lesssim \| \sup_{\mathbb{A},\varphi}( \vec{f} ) \|_{b^{s(\cdot)}_{p(\cdot),q(\cdot)}}, $$
which, together with Lemma \ref{aa supaa}, gives the equivalence of all above norms
and hence completes the proof of Theorem \ref{W aa supp}.
\end{proof}

\subsection{Matrix-Weighted Variable Besov Sequence Spaces}\label{sec Besov 2}

In this subsection,
we introduce two matrix-weighted variable Besov sequence spaces,
$$b^{s(\cdot)}_{p(\cdot),q(\cdot)}(W)\ \mathrm{and}\
b^{s(\cdot)}_{p(\cdot),q(\cdot)}(\mathbb{A}),$$
and establish their equivalences.
We begin with the following sequence spaces.
\begin{definition}
Let $p(\cdot), q(\cdot) \in \mathcal{P}_0$, $s(\cdot) \in L^\infty$,
and $W\in \mathscr{A}_{p(\cdot),\infty}$.
The \emph{(pointwise) matrix-weighted variable Besov sequence space $ b^{s(\cdot)}_{p(\cdot),q(\cdot)}(W) $}
is defined to be the set of all sequences
$\vec{t} := \{\vec{t}_Q\}_{Q\in \mathcal{Q}_+} \subset \mathbb{C}^m$
such that
$$ \left\| \vec{t} \right\|_{b^{s(\cdot)}_{p(\cdot),q(\cdot)}(W)}
:= \left\| \left\{ 2^{js(\cdot)}\left| W(\cdot) \vec{t}_j \right|\right\}_{j\in \mathbb{Z}_+} \right\|_{l^{q(\cdot)}(L^{p(\cdot)})} <\infty, $$
where, for any $j\in \mathbb{Z}_+$,
\begin{align}\label{def tj}
\vec{t}_j := \sum_{Q\in \mathcal{Q}_j} \vec{t}_Q \widetilde{\mathbf{1}}_Q.
\end{align}
\end{definition}
Next, we introduce the concept of averaging matrix-weighted variable Besov sequence spaces.
\begin{definition}
Let $p(\cdot), q(\cdot) \in \mathcal{P}_0$ and $s(\cdot) \in L^\infty$
and let $W\in \mathscr{A}_{p(\cdot),\infty}$ and
$\mathbb{A} := \{A_Q\}_{Q\in\mathcal{Q}_+}$ be reducing operators of order $p(\cdot)$ for $W$.
The \emph{averaging matrix-weighted variable Besov sequence space
$ b^{s(\cdot)}_{p(\cdot),q(\cdot)}(\mathbb{A}) $}
is defined to be the set of all sequences
$\vec{t} := \{\vec{t}_Q\}_{Q\in \mathcal{Q}_+} \subset \mathbb{C}^m$
such that
$$ \left\| \vec{t} \right\|_{b^{s(\cdot)}_{p(\cdot),q(\cdot)}(\mathbb{A})}
:= \left\| \left\{ 2^{js(\cdot)}\left| A_j \vec{t}_j \right|\right\}_{j\in \mathbb{Z}_+} \right\|_{l^{q(\cdot)}(L^{p(\cdot)})} <\infty, $$
where $A_j$ for any $j\in\mathbb{Z}_+$ is the same as in \eqref{def Aj}.
\end{definition}
Similarly to the equivalence between the (pointwise) matrix-weighted Besov space
and the averaging matrix-weighted one,
the above two types of  matrix-weighted variable Besov sequence spaces
are also equivalent,
which is exactly the following result.
\begin{theorem}\label{W aa 3}
Let $p(\cdot), q(\cdot)\in \mathcal{P}_0\cap LH$
and $s(\cdot)\in LH$ and let $W\in \mathscr{A}_{p(\cdot),\infty}$
and $\mathbb{A} := \{A_Q\}_{Q\in\mathcal{Q}_+}$ be reducing operators of order $p(\cdot)$ for $W$.
Then, for any sequence $\vec{t} := \{\vec{t}_Q\}_{Q\in \mathcal{Q}_+} \subset \mathbb{C}^m$,
\begin{align*}
\left\| \vec{t} \right\|_{b^{s(\cdot)}_{p(\cdot),q(\cdot)}(W)}
\sim \left\| \vec{t} \right\|_{b^{s(\cdot)}_{p(\cdot),q(\cdot)}(\mathbb{A})},
\end{align*}
where the positive equivalence constants are independent of $\vec{t}$.
\end{theorem}
\begin{proof}
We first prove
\begin{align}\label{W aa 1}
\left\| \vec{t} \right\|_{b^{s(\cdot)}_{p(\cdot),q(\cdot)}(W)}
\lesssim \left\| \vec{t} \right\|_{b^{s(\cdot)}_{p(\cdot),q(\cdot)}(\mathbb{A})}.
\end{align}
Similarly to the proof of \eqref{Ajtj 12},
to show \eqref{W aa 1},
it is sufficient to prove that,
for any $\vec{t}:= \{t_Q\}_{Q\in\mathcal{Q}_+} \subset \mathbb{C}$ satisfying
$\sum_{j\in\mathbb{Z}_+} \| 2^{js(\cdot)q(\cdot)} | A_j\vec{t}_j |^{q(\cdot)} \|_{L^{\frac{p(\cdot)}{q(\cdot)}}} = 1  $
and for any $j\in\mathbb{Z}_+$,
$$ \left\| \delta_j^{-\frac{1}{q(\cdot)}} 2^{js(\cdot)} \left| W(\cdot)\vec{t}_j \right| \right\|_{L^{p(\cdot)}} \lesssim 1, $$
where the implicit positive constant is independent of $\vec{t}$ and $j$ and
$$\delta_j := \left\| 2^{js(\cdot)q(\cdot)} \left| A_j\vec{t}_j \right|^{q(\cdot)} \right\|_{L^{\frac{p(\cdot)}{q(\cdot)}}} + 2^{-j}. $$
Let $r := \min\{1,p_-\}$ and hence $\frac{p(\cdot)}{r}\in \mathcal{P} \cap LH$.
By this and Lemmas \ref{con f} and \ref{fg Lp} with $p(\cdot) := \frac{p(\cdot)}{r}$,
we find that
\begin{align}\label{Ajtj 7}
\left\| \delta_j^{-\frac{1}{q(\cdot)}}2^{js(\cdot)} \left| W(\cdot)\vec{t}_j \right| \right\|_{L^{p(\cdot)}}^r
& = \left\| \delta_j^{-\frac{r}{q(\cdot)}}2^{jrs(\cdot)} \left| W(\cdot)\vec{t}_j \right|^r \right\|_{L^{\frac{p(\cdot)}{r}}}\nonumber\\
& \sim \sup_{\|g\|_{L^{(\frac{p(\cdot)}{r})'}}\leq 1} \int_{\mathbb{R}^n} \delta_j^{-\frac{r}{q(x)}}2^{jrs(x)} \left| W(x)\vec{t}_j \right|^r |g(x)|\,dx.
\end{align}
Now, let $g\in L^{(\frac{p(\cdot)}{r})'}$ be any given function.
Then, using the definition of $\vec{t}_j$, the disjointness of the dyadic cubes in $\mathcal{Q}_j$,
and \eqref{Ajtj 15}, we obtain
\begin{align*}
\int_{\mathbb{R}^n} \delta_j^{-\frac{r}{q(x)}}2^{jrs(x)} \left| W(x)\vec{t}_j \right|^r |g(x)|\,dx
& = \sum_{Q\in \mathcal{Q}_j} \int_{Q} \delta_j^{-\frac{r}{q(x)}}2^{jrs(x)} |Q|^{-\frac r2} \left| W(x)\vec{t}_Q \right|^r |g(x)|\,dx \nonumber\\
& \lesssim \sum_{Q\in \mathcal{Q}_j} \delta_j^{-\frac{r}{q(x_Q)}}2^{jrs(x_Q)} |Q|^{-\frac r2} \int_{Q} \left\| W(x)A_Q^{-1} \right\|^r \left| A_Q\vec{t}_Q \right|^r |g(x)|\,dx.
\end{align*}
From this and Lemma \ref{Holder} with $p(\cdot) := \frac{p(\cdot)}{r}$
and from \eqref{WAQ-1} and Lemma \ref{g pro} with $\mathcal{K} := \mathcal{Q}_j$
and the disjointness of the dyadic cubes in $\mathcal{Q}_j$,
and the definitions of $A_j$ and $\vec{t}_j$,
we infer that
\begin{align*}
&\int_{\mathbb{R}^n} \delta_j^{-\frac{r}{q(x)}}2^{jrs(x)} \left| W(x)\vec{t}_j \right|^r |g(x)|\,dx \nonumber\\
&\quad \lesssim  \sum_{Q\in \mathcal{Q}_j} \delta_j^{-\frac{r}{q(x_Q)}}2^{jrs(x_Q)} |Q|^{-\frac r2} \left| A_Q\vec{t}_Q \right|^r \left\| \left\|W(\cdot) A_Q^{-1} \right\| \mathbf{1}_Q \right\|_{L^{\frac{p(\cdot)}{r}}} \|g\mathbf{1}_Q\|_{L^{(\frac{p(\cdot)}{r})'}}\nonumber\\
&\quad \lesssim  \sum_{Q\in \mathcal{Q}_j} \delta_j^{-\frac{r}{q(x_Q)}}2^{jrs(x_Q)} |Q|^{-\frac r2} \left| A_Q\vec{t}_Q \right|^r \left\| \mathbf{1}_Q \right\|_{L^{\frac{p(\cdot)}{r}}} \|g\mathbf{1}_Q\|_{L^{(\frac{p(\cdot)}{r})'}}\nonumber\\
&\quad \sim  \sum_{Q\in \mathcal{Q}_j}  \left\| \delta_j^{-\frac{r}{q(\cdot)}}2^{jrs(\cdot)} |Q|^{-\frac r2} \left| A_Q\vec{t}_Q \right|^r \mathbf{1}_Q \right\|_{L^{\frac{p(\cdot)}{r}}} \|g\mathbf{1}_Q\|_{L^{(\frac{p(\cdot)}{r})'}}\nonumber\\
&\quad \lesssim \left\| \delta_j^{-\frac{r}{q(\cdot)}}2^{jrs(\cdot)} \left| A_j \vec{t}_j \right|^r  \right\|_{L^{\frac{p(\cdot)}{r}}} \|g\|_{L^{(\frac{p(\cdot)}{r})'}},
\end{align*}
which, combined with \eqref{Ajtj 7} and Lemma \ref{con f},
further implies that
\begin{align}\label{Ajtj 18}
\left\| \delta_j^{-\frac{1}{q(\cdot)}}2^{js(\cdot)} \left| W(\cdot)\vec{t}_j \right| \right\|_{L^{p(\cdot)}}^r
& \lesssim \sup_{\|g\|_{L^{(\frac{p(\cdot)}{r})'}}\leq 1} \left\| \delta_j^{-\frac{r}{q(\cdot)}}2^{jrs(\cdot)}  \left| A_j\vec{t}_j \right|^r \mathbf{1}_Q \right\|_{L^{\frac{p(\cdot)}{r}}} \|g\|_{L^{(\frac{p(\cdot)}{r})'}}\nonumber \\
&= \left\| \delta_j^{-\frac{1}{q(\cdot)}}2^{js(\cdot)} |A_j \vec{t}_j| \right\|^r_{L^{p(\cdot)}}
\end{align}
Using the definition of $\delta_j$,
we immediately find that
$ \|\delta_j^{-1} 2^{js(\cdot)q(\cdot)} | A_j\vec{t}_j |^{q(\cdot)} \|_{L^{\frac{p(\cdot)}{q(\cdot)}}}\leq 1, $
which, together with Lemma \ref{f pq}, further implies that
$ \|\delta_j^{-\frac{1}{q(\cdot)}} 2^{js(\cdot)} | A_j\vec{t}_j |\|_{L^{p(\cdot)}}\leq 1. $
By this and \eqref{Ajtj 18},
we conclude that
$  \| \delta_j^{-\frac{1}{q(\cdot)}}2^{js(\cdot)} | W(\cdot)\vec{t}_j |
\|_{L^{p(\cdot)}} \lesssim 1$,
which completes the proof of \eqref{W aa 1}.

Next, we prove the converse inequality of \eqref{W aa 1}, that is,
\begin{align}\label{W aa 2}
\left\| \vec{t} \right\|_{b^{s(\cdot)}_{p(\cdot),q(\cdot)}(\mathbb{A})}
\lesssim \left\| \vec{t} \right\|_{b^{s(\cdot)}_{p(\cdot),q(\cdot)}(W)}.
\end{align}
Similarly to the proof of \eqref{Ajtj 12},
to show \eqref{W aa 2},
it is sufficient to prove that,
for any $\vec{t}:= \{t_Q\}_{Q\in\mathcal{Q}_+} \subset \mathbb{C}$ satisfying
$\sum_{j\in\mathbb{Z}_+}^\infty \| 2^{js(\cdot)q(\cdot)} | W(\cdot)\vec{t}_j |^{q(\cdot)} \|_{L^{\frac{p(\cdot)}{q(\cdot)}}} = 1 $
and for any $j\in\mathbb{Z}_+$,
$$\left\| \delta_j^{-\frac{1}{q(\cdot)}}2^{js(\cdot)} \left| A_j \vec{t}_j \right| \right\|_{L^{p(\cdot)}} \lesssim 1, $$
where the implicit positive constant is independent of $\vec{t}$ and $j$
and where
$$ \delta_j := \| 2^{js(\cdot)q(\cdot)} | W(\cdot)\vec{t}_j |^{q(\cdot)} \|_{L^{\frac{p(\cdot)}{q(\cdot)}}} + 2^{-j}. $$

Using Lemmas \ref{con f} and \ref{fg Lp} with $p(\cdot) := \frac{p(\cdot)}{r}$,
we find that
\begin{align}\label{Ajtj 9}
\left\| \delta_j^{-\frac{1}{q(\cdot)}}2^{js(\cdot)} \left| A_j \vec{t}_j \right| \right\|^r_{L^{p(\cdot)}}
& = \left\| \delta_j^{-\frac{r}{q(\cdot)}}2^{jrs(\cdot)} \left| A_j \vec{t}_j \right|^r \right\|_{L^{\frac{p(\cdot)}{r}}}\nonumber\\
& \sim \sup_{\|g\|_{L^{(\frac{p(\cdot)}{r})'}}\leq 1} \int_{\mathbb{R}^n} \delta_j^{-\frac{r}{q(x)}}2^{jrs(x)} \left| A_j \vec{t}_j \right|^r |g(x)|\,dx.
\end{align}
Let $g\in L^{(\frac{p(\cdot)}{r})'}$ be any given function.
Then, by the disjointness of the dyadic cubes in $\mathcal{Q}_j$
and Lemma \ref{Holder} with $p(\cdot) := \frac{p(\cdot)}{r}$,
we obtain
\begin{align}\label{Ajtj 8}
\int_{\mathbb{R}^n} \delta_j^{-\frac{r}{q(x)}}2^{jrs(x)} \left| A_j \vec{t}_j \right|^r |g(x)|\,dx
& =  \sum_{Q\in\mathcal{Q}_j} \int_{Q} \delta_j^{-\frac{r}{q(x)}}2^{jrs(x)} |Q|^{-\frac r2} \left| A_Q \vec{t}_Q \right|^r g(x)\,dx \nonumber\\
& \lesssim  \sum_{Q\in\mathcal{Q}_j}  \left\| \delta_j^{-\frac{r}{q(\cdot)}}2^{jrs(\cdot)} |Q|^{-\frac r2} \left| A_Q \vec{t}_Q \right|^r \mathbf{1}_Q \right\|_{L^{\frac{p(\cdot)}{r}}}  \|g\mathbf{1}_Q\|_{L^{(\frac{p(\cdot)}{r})'}}.
\end{align}
Observe that, using the definition of the reducing operators and \eqref{Ajtj 15},
we have, for any $Q\in \mathcal{Q}_j$ and $x\in Q$,
\begin{align*}
\delta_j^{-\frac{r}{q(x)}}2^{jrs(x)} |Q|^{-\frac r2} \left| A_Q \vec{t}_Q \right|^r
&\sim \delta_j^{-\frac{r}{q(x_Q)}}2^{jrs(x_Q)} |Q|^{-\frac r2} \frac{1}{\|\mathbf{1}_{Q}\|^r_{L^{p(\cdot)}}}
 \left\| \left| W(\cdot)\vec{t}_Q \right| \mathbf{1}_Q \right\|^r_{L^{p(\cdot)}}\\
&\sim \frac{1}{\|\mathbf{1}_{Q}\|^r_{L^{p(\cdot)}}}
 \left\| \delta_j^{-\frac{1}{q(\cdot)}}2^{js(\cdot)} |Q|^{-\frac 12} \left| W(\cdot)\vec{t}_Q \right| \mathbf{1}_Q \right\|^r_{L^{p(\cdot)}}
\end{align*}
Combining this with \eqref{Ajtj 8} and Lemma \ref{con f},
we find that
\begin{align*}
&\int_{\mathbb{R}^n} \delta_j^{-\frac{r}{q(x)}}2^{jrs(x)} \left| A_j \vec{t}_j \right|^r |g(x)|\,dx\nonumber\\
&\quad \sim  \sum_{Q\in\mathcal{Q}_j}
\left\| \frac{1}{\|\mathbf{1}_{Q}\|^r_{L^{p(\cdot)}}}
 \left\| \delta_j^{-\frac{1}{q(\cdot)}}2^{js(\cdot)} |Q|^{-\frac 12} \left| W(\cdot)\vec{t}_Q \right| \mathbf{1}_Q \right\|^r_{L^{p(\cdot)}} \mathbf{1}_Q \right\|_{L^{\frac{p(\cdot)}{r}}}  \|g\mathbf{1}_Q\|_{L^{(\frac{p(\cdot)}{r})'}}\nonumber\\
&\quad =  \sum_{Q\in\mathcal{Q}_j}
\left\| \delta_j^{-\frac{r}{q(\cdot)}}2^{jrs(\cdot)} |Q|^{-\frac r2}\left| W(\cdot)\vec{t}_Q \right|^r \mathbf{1}_Q \right\|_{L^{\frac{p(\cdot)}{r}}}  \|g\mathbf{1}_Q\|_{L^{(\frac{p(\cdot)}{r})'}}.
\end{align*}
Applying this with the definition of $\vec{t}_j$ and Lemma \ref{g pro} with $\mathcal{K} := \mathcal{Q}_j$
and the disjointness of the dyadic cubes in $\mathcal{Q}_j$ yields that
\begin{align*}
\int_{\mathbb{R}^n} \delta_j^{-\frac{r}{q(x)}}2^{jrs(x)} \left| A_j \vec{t}_j \right|^r |g(x)|\,dx
 \lesssim \left\| \delta_j^{-\frac{r}{q(\cdot)}}2^{jrs(\cdot)} \left| W(\cdot) \vec{t}_j \right|^r \right\|_{L^{\frac{p(\cdot)}{r}}} \|g\|_{L^{(\frac{p(\cdot)}{r})'}},
\end{align*}
which, combined with \eqref{Ajtj 9} and Lemma \ref{con f},
further implies that
\begin{align}\label{Ajtj 6}
\left\| \delta_j^{-\frac{1}{q(\cdot)}}2^{js(\cdot)} \left| A_j \vec{t}_j \right| \right\|^r_{L^{p(\cdot)}}
&\lesssim \sup_{\|g\|_{L^{(\frac{p(\cdot)}{r})'}}\leq 1} \left\| \delta_j^{-\frac{r}{q(\cdot)}}2^{jrs(\cdot)} \left| W(\cdot) \vec{t}_j \right|^r \right\|_{L^{\frac{p(\cdot)}{r}}} \|g\|_{L^{(\frac{p(\cdot)}{r})'}} \nonumber\\
& = \left\| \delta_j^{-\frac{1}{q(\cdot)}}2^{js(\cdot)} \left| W(\cdot) \vec{t}_j \right| \right\|^r_{L^{p(\cdot)}}.
\end{align}
Notice that, by the definition of $\delta_j$,
we obtain
$\|\delta_j^{-1} 2^{js(\cdot)q(\cdot)} | W(\cdot)\vec{t}_j |^{q(\cdot)}\|_{L^{\frac{p(\cdot)}{q(\cdot)}}}\leq 1, $
which, together with Lemma \ref{f pq},
further implies that
$ \|\delta_j^{-\frac{1}{q(\cdot)}} 2^{js(\cdot)} | W(\cdot)\vec{t}_j |\|_{L^{p(\cdot)}}\leq 1. $
Using this and \eqref{Ajtj 6},
we obtain
$ \| \delta_j^{-\frac{1}{q(\cdot)}}2^{js(\cdot)} | A_j\vec{t}_j | \|_{L^{p(\cdot)}} \lesssim 1.$
This finishes the proof of \eqref{W aa 2}
and hence Theorem \ref{W aa 3}.
\end{proof}

\subsection{The $\varphi$-Transform Characterization}\label{sec varphi}
In this subsection, we establish
the $\varphi$-transform characterization of matrix-weighted variable Besov spaces.
We first recall some basic notions on the $\varphi$-transform and its properties.
Let $\{\varphi_j\}_{j\in\mathbb{Z}_+}$
be as in Definition \ref{phi pair}.
Then there exists $\{\psi_j\}_{j\in\mathbb{Z}_+}$
satisfying the same conditions as in Definition \ref{phi pair} such that,
for any $\xi\in\mathbb{R}^n$,
\begin{align}\label{varphi psi}
\sum_{j = 0}^\infty \widehat{\varphi_j}(\xi) \widehat{\psi_j}(\xi) = 1.
\end{align}
The \emph{$\varphi$-transform $S_\varphi$} is defined to be the map taking each $\vec{f} \in (\mathcal{S}')^m$
to the sequence $S_\varphi \vec{f} := \{ (S_\varphi \vec{f})_Q \}_{Q\in\mathcal{Q}_+}$,
where, for any $Q\in\mathcal{Q}_+$,
$(S_\varphi \vec{f})_Q := \langle \vec{f},\varphi_Q \rangle$
and
\begin{align}\label{def varphiQ}
\varphi_Q := |Q|^{\frac12}\varphi_j \left(\cdot-x_Q\right)
\end{align}
with $x_Q$ being the center of $Q$.
The \emph{inverse $\varphi$-transform $T_\psi$} is defined to be the map taking each sequence
$\vec{t} := \{\vec{t}_Q\}_{Q\in\mathcal{Q}_+}\subset \mathbb{C}^m$ to
$ T_\psi \vec{t} := \sum_{Q\in\mathcal{Q}_+} \vec{t}_Q \psi_Q $ in $ (\mathcal{S}')^m $.

The following theorem is the main result of this subsection.
In what follows, for any $x\in\mathbb{R}^n$, let $\widetilde{\varphi}(x) := \varphi(-x)$.
\begin{theorem}\label{phi bound}
Let $p(\cdot), q(\cdot)\in \mathcal{P}_0\cap LH$, $s(\cdot)\in LH$,
and $W \in \mathscr{A}_{p(\cdot),\infty}$
and let $\{\varphi_j\}_{j\in\mathbb{Z}_+}$ and $\{\psi_j\}_{j\in\mathbb{Z}_+}$
be as in Definition \ref{phi pair} satisfying \eqref{varphi psi}.
Then the operators $S_\varphi :\  B^{s(\cdot)}_{p(\cdot),q(\cdot)}(W,\widetilde{\varphi}) \rightarrow b^{s(\cdot)}_{p(\cdot),q(\cdot)}(W)$
and $T_\psi :\  b^{s(\cdot)}_{p(\cdot),q(\cdot)}(W) \rightarrow B^{s(\cdot)}_{p(\cdot),q(\cdot)}(W,\varphi)$
are bounded.
Furthermore, $T_\psi \circ S_\varphi $ is the identity on $ B^{s(\cdot)}_{p(\cdot),q(\cdot)}(W, \widetilde{\varphi}) $.
\end{theorem}
Before giving the proof of Theorem \ref{phi bound},
we first point out that Theorem \ref{phi bound} implies that
$B^{s(\cdot)}_{p(\cdot),q(\cdot)}(W, {\varphi})$
is independent of the choice of $(\varPhi, \varphi)$.
\begin{proposition}\label{ind choice}
Let $p(\cdot), q(\cdot)\in \mathcal{P}_0\cap LH$, $s(\cdot)\in LH$,
and $W \in \mathscr{A}_{p(\cdot),\infty}$
and let $\{\varphi_j\}_{j\in\mathbb{Z}_+}$ be as in Definition \ref{phi pair}.
Then $B^{s(\cdot)}_{p(\cdot),q(\cdot)}(W, \varphi)$
is independent of the choice of $\varphi$.
\end{proposition}
\begin{proof}
Let $\{\varphi_j^{(1)}\}_{j\in\mathbb{Z}_+}$ and $\{\varphi_j^{(2)}\}_{j\in\mathbb{Z}_+}$ be as in Definition \ref{phi pair}
and let $\{\psi_j^{(2)}\}_{j\in\mathbb{Z}_+}$ be as in \eqref{varphi psi}
such that \eqref{varphi psi} holds for
$\{\varphi_j^{(2)}\}_{j\in\mathbb{Z}_+}$ and $\{\psi_j^{(2)}\}_{j\in\mathbb{Z}_+}$.
Then, using Theorem \ref{phi bound},
we conclude that, for any $\vec{f}\in B^{s(\cdot)}_{p(\cdot),q(\cdot)}(W, \varphi^{(2)})$,
\begin{align*}
\left\| \vec{f} \right\|_{B^{s(\cdot)}_{p(\cdot),q(\cdot)}(W, \varphi^{(1)})}
 = \left\| T_{\widetilde{\psi^{(2)}}} \circ S_{\widetilde{\varphi^{(2)}}}  \vec{f} \right\|_{B^{s(\cdot)}_{p(\cdot),q(\cdot)}(W, \varphi^{(1)})}
\lesssim \left\| S_{\widetilde{\varphi^{(2)}}} \vec{f} \right\|_{b^{s(\cdot)}_{p(\cdot),q(\cdot)}(W)}
\lesssim \left\| \vec{f} \right\|_{B^{s(\cdot)}_{p(\cdot),q(\cdot)}(W, \varphi^{(2)})}.
\end{align*}
By symmetry,
we also obtain the reverse inequality.
This finishes the proof of Proposition \ref{ind choice}.
\end{proof}

Now, to prove Theorem \ref{phi bound},
we first recall several basic lemmas.
The following lemma is precisely \cite[(12.4)]{fj90}.
\begin{lemma}\label{CZ com}
Let $\{\varphi_j\}_{j\in\mathbb{Z}_+}$ and $\{\psi_j\}_{j\in\mathbb{Z}_+}$
be as in Definition \ref{phi pair} satisfying \eqref{varphi psi}.
Then, for any $f \in \mathcal{S}'$,
$$ f = \sum_{j} \sum_{Q\in\mathcal{Q}_j} \langle f, \varphi_Q \rangle \psi_Q
= \sum_{j} 2^{-jn} \sum_{k\in\mathbb{Z}^n} \left( \widetilde{\varphi}_j\ast f \right)(2^{-j}k) \psi_j(\cdot - 2^{-j}k)$$
in $\mathcal{S}'.$
\end{lemma}
The following lemma is exactly \cite[Lemma 2.4]{ysy10}.
\begin{lemma}\label{psiphi 1}
Let $M\in\mathbb{Z}_+$ and $\psi, \varphi \in \mathcal{S}$ satisfy
$\int_{\mathbb{R}^n} x^\gamma \psi(x)\,dx = 0$
for all multi-indices $\gamma \in \mathbb{Z}_+^n$ with $|\gamma| \leq M$.
Then, for any $j\in\mathbb{Z}_+$ and $x\in\mathbb{R}^n$,
$$ \left| \varphi \ast \psi_j(x) \right| \lesssim \|\psi\|_{\mathcal{S}_{M+1}} \|\varphi\|_{\mathcal{S}_{M+1}}
2^{-jM} \frac{1}{(1+|x|)^{n+M}},  $$
where the implicit positive constant depends only on $n$ and $M$.
\end{lemma}

The following lemma guarantees the convergence of $T_\psi \vec{t}$
for any $\vec{t}\in b^{s(\cdot)}_{p(\cdot),q(\cdot)}(W)$.
\begin{lemma}\label{Tpsi 1}
Let $p(\cdot), q(\cdot)\in \mathcal{P}_0\cap LH$, $s(\cdot)\in LH$,
and $W\in\mathscr{A}_{p(\cdot),\infty}$
and let $\{\psi_j\}_{j\in\mathbb{Z}_+}$ be as in Definition \ref{phi pair}.
Then, for any $\vec{t} := \{\vec{t}_Q\}_{Q\in\mathcal{Q}_+} \in b^{s(\cdot)}_{p(\cdot),q(\cdot)}(W)$,
$\sum_{Q\in \mathcal{Q}_+} \vec{t}_Q \psi_Q$ converges in $(\mathcal{S}')^m$.
Moreover, if $M\in\mathbb{Z}_+$ also satisfies
\begin{align}\label{M con}
M >\max\left\{d^{\rm upper}_{p(\cdot),\infty}(W) + \frac{n}{p_-} - s_- , \Delta \right\},
\end{align}
where $\Delta$ is the same as in Lemma \ref{QP5},
then, for any $\vec{t} \in b^{s(\cdot)}_{p(\cdot),q(\cdot)}(W)$ and $\phi \in \mathcal{S}$,
$$ \sum_{Q\in\mathcal{Q}_+} \left| \vec{t}_Q \right| \left|\left\langle \psi_Q,\phi \right\rangle \right|
\lesssim \left\| \vec{t} \right\|_{b^{s(\cdot)}_{p(\cdot),q(\cdot)}(W)}
\|\psi\|_{\mathcal{S}_{M+1}} \|\phi\|_{\mathcal{S}_{M+1}}, $$
where the implicit positive constant is independent of $\vec{t}$.
\end{lemma}
\begin{proof}
Let $\{A_Q\}_{Q\in\mathcal{Q}_+}$ be reducing operators of order $p(\cdot)$ for $W$.
Observe that, by the facts that $ \frac{1}{p_Q} \leq \frac{1}{p_-} $
and $|Q| \leq 1$ for any $Q\in \mathcal{Q}_+$,
we obtain $|Q|^{-\frac{1}{p_Q}} \leq |Q|^{-\frac{1}{p_-}}$.
From this, Lemma \ref{est Q}, and the definition of $\widetilde{\mathbf{1}}_Q$,
it follows that, for any $j\in\mathbb{Z}_+$ and $Q\in\mathcal{Q}_j$,
\begin{align*}
\left| \vec{t}_Q \right| &\leq \left\|A_Q^{-1}\right\| \left| A_Q \vec{t}_Q \right|
 = \left\|A_Q^{-1}\right\| \left\| \widetilde{\mathbf{1}}_Q \right\|^{-1}_{L^{p(\cdot)}} \left\| \left| A_Q \vec{t}_Q \right| \widetilde{\mathbf{1}}_Q \right\|_{L^{p(\cdot)}}
 \sim |Q|^{\frac12-\frac{1}{p_Q}} \left\|A_Q^{-1}\right\|  \left\| \left| A_Q \vec{t}_Q \right| \widetilde{\mathbf{1}}_Q \right\|_{L^{p(\cdot)}}\nonumber \\
& \lesssim |Q|^{\frac12-\frac{1}{p_Q}} 2^{-js_-} \left\|A_Q^{-1}\right\| \left\| 2^{js(\cdot)} \left| A_j \vec{t}_j \right| \right\|_{L^{p(\cdot)}}
 \leq |Q|^{\frac{s_-}{n} + \frac12-\frac{1}{p_-}} \left\|A_Q^{-1}\right\| \left\| 2^{js(\cdot)} \left| A_j \vec{t}_j \right| \right\|_{L^{p(\cdot)}},
\end{align*}
where $\vec{t}_j$ is the same as in \eqref{def tj} and $A_j$ the same as in \eqref{def Aj}.
Using this and Theorem \ref{W aa 3},
we conclude that, for any $j\in\mathbb{Z}_+$ and $Q\in\mathcal{Q}_j$,
\begin{align*}
\left| \vec{t}_Q \right|
 &\lesssim |Q|^{\frac{s_-}{n} + \frac12-\frac{1}{p_-}} \left\|A_Q^{-1}\right\| \left\| \vec{t} \right\|_{b^{s(\cdot)}_{p(\cdot),q(\cdot)}(\mathbb{A})}
 \lesssim |Q|^{\frac{s_-}{n} + \frac12-\frac{1}{p_-}} \left\|A_Q^{-1}\right\| \left\| \vec{t} \right\|_{b^{s(\cdot)}_{p(\cdot),q(\cdot)}(W)},
\end{align*}
which further implies that, for any $\phi \in\mathcal{S}$,
\begin{align}\label{tQ 2}
\sum_{Q\in \mathcal{Q}_+} \left|\vec{t}_Q\right| \left| \left\langle \psi_Q, \phi \right\rangle \right|
\lesssim \left\| \vec{t} \right\|_{b^{s(\cdot)}_{p(\cdot),q(\cdot)}(W)} \sum_{Q\in\mathcal{Q}_+} |Q|^{\frac{s_-}{n} + \frac12-\frac{1}{p_-}} \left\|A_Q^{-1}\right\|\left| \left\langle \psi_Q, \phi \right\rangle \right|.
\end{align}
By Lemmas \ref{pA norm} and \ref{QP5} with $Q := Q_{0,\mathbf{0}} := (0,1]^n$ and $R := Q$
and by the fact that $l(Q)\leq 1$ for any $Q\in\mathcal{Q}_+$,
we have
\begin{align}\label{tQ 4}
\left\|A_Q^{-1}\right\| \leq \left\| A_Q^{-1} A_{Q_{0,\mathbf{0}}}\right\| \left\| A_{Q_{0,\mathbf{0}}}^{-1} \right\|
\lesssim \left\| A_{Q_{0,\mathbf{0}}} A_Q^{-1}\right\|
\lesssim |Q|^{-\frac{d_2}{n}} \left( 1+ |x_Q| \right)^{\Delta},
\end{align}
where $d_2 \in [\![ d^{\rm upper}_{p(\cdot),\infty}(W),\infty )$
is a fixed parameter.
Let $M\in\mathbb{N}$ satisfy
$M > \max\{d_2 + \frac{n}{p_-} - s_- , \Delta \}$.
Then, if $j \geq 1$,
by Lemma \ref{psiphi 1} and the fact $\psi_j \in \mathcal{S}_\infty$,
we obtain, for any $\phi \in \mathcal{S}$ and $Q\in\mathcal{Q}_j$,
\begin{align*}
\left| \left\langle \psi_Q, \phi \right\rangle \right| =  \psi_Q \ast \phi(x_Q)
\lesssim \|\psi\|_{\mathcal{S}_{M+1}} \|\phi\|_{\mathcal{S}_{M+1}}
|Q|^{\frac{M}{n} + \frac{1}{2}} \left(1+|x_Q|\right)^{-n-M}.
\end{align*}
From this, \eqref{tQ 2}, and \eqref{tQ 4},
we deduce that
\begin{align}\label{tQ 5}
&\sum_{j = 1}^\infty \sum_{Q \in \mathcal{Q}_j} \left|\vec{t}_Q\right| \left| \left\langle \psi_Q, \phi \right\rangle \right|\nonumber\\
&\quad \lesssim \|\psi\|_{\mathcal{S}_{M+1}} \|\phi\|_{\mathcal{S}_{M+1}} \left\| \vec{t} \right\|_{b^{s(\cdot)}_{p(\cdot),q(\cdot)}(W)}
\sum_{j = 1}^\infty \sum_{Q \in \mathcal{Q}_j} |Q|^{\frac{M}{n} + \frac{s_-}{n} + 1 - \frac{1}{p_-} - \frac{d_2}{n}}
\left( 1+ |x_Q| \right)^{\Delta-n-M}\nonumber\\
&\quad \lesssim \|\psi\|_{\mathcal{S}_{M+1}} \|\phi\|_{\mathcal{S}_{M+1}} \left\| \vec{t} \right\|_{b^{s(\cdot)}_{p(\cdot),q(\cdot)}(W)}
\sum_{j = 1}^\infty 2^{-j(M + s_- + n - \frac{n}{p_-} -d_2)} \sum_{k \in\mathbb{Z}^n}
\left( 1+ 2^{-j}|k| \right)^{\Delta-n-M}\nonumber\\
&\quad \lesssim \|\psi\|_{\mathcal{S}_{M+1}} \|\phi\|_{\mathcal{S}_{M+1}} \left\| \vec{t} \right\|_{b^{s(\cdot)}_{p(\cdot),q(\cdot)}(W)}
\sum_{j = 1}^\infty
2^{-j(M + s_- - \frac{n}{p_-} - d_2)}
 \lesssim \|\psi\|_{\mathcal{S}_{M+1}} \|\phi\|_{\mathcal{S}_{M+1}} \left\| \vec{t} \right\|_{b^{s(\cdot)}_{p(\cdot),q(\cdot)}(W)}.
\end{align}
Now, if $j = 0$,
then, using the definition of $\|\cdot\|_{\mathcal{S}_{M+1}}$ and the fact $\Psi \in \mathcal{S}$,
we find that
\begin{align*}
\left| \left\langle \Psi_Q, \phi \right\rangle \right|
&= \left| \int_{\mathbb{R}^n} \Psi\left( x - x_Q \right) \phi(x)\,dx \right|
\lesssim \|\Psi\|_{\mathcal{S}_{M+1}} \|\phi\|_{\mathcal{S}_{M+1}}
\int_{\mathbb{R}^n} \frac{1}{\left(1 + |x - x_Q|\right)^{n+M+1}} \frac{1}{\left(1 + |x|\right)^{n+M+1}}\,dx  \\
&\lesssim \|\Psi\|_{\mathcal{S}_{M+1}} \|\phi\|_{\mathcal{S}_{M+1}} \frac{1}{\left(1 + |x_Q|\right)^{n+M+1}},
\end{align*}
which, combined with \eqref{tQ 2} and \eqref{tQ 4}, further implies that
\begin{align*}
\sum_{Q \in \mathcal{Q}_0} \left|\vec{t}_Q\right| \left| \left\langle \psi_Q, \phi \right\rangle \right|
& \lesssim \left\| \vec{t} \right\|_{b^{s(\cdot)}_{p(\cdot),q(\cdot)}(W)}\|\psi\|_{\mathcal{S}_{M+1}} \|\phi\|_{\mathcal{S}_{M+1}}
\sum_{k\in\mathbb{Z}^n }  \left( 1+|k| \right)^{-(M+n+1)+\Delta}\\
& \lesssim \left\| \vec{t} \right\|_{b^{s(\cdot)}_{p(\cdot),q(\cdot)}(W)}\|\psi\|_{\mathcal{S}_{M+1}} \|\phi\|_{\mathcal{S}_{M+1}}.
\end{align*}
From this and \eqref{tQ 5},
we infer that \eqref{tQ 2} converges absolutely.
Thus, $\sum_{Q\in\mathcal{Q}_+} \vec{t}_Q \psi_Q$ converges in $\mathcal{S}'$,
which completes the proof of Lemma \ref{Tpsi 1}.
\end{proof}

The following lemma is precisely \cite[Lemma 2.2]{yy08}.
\begin{lemma}\label{psiphi 2}
Let $M\in\mathbb{Z}_+$ and $\psi, \varphi \in \mathcal{S}_\infty$.
Then, for any $j,i\in\mathbb{Z}_+$ and $x\in\mathbb{R}^n$,
$$ \left| \varphi_i \ast \psi_j  (x) \right| \lesssim \|\psi\|_{\mathcal{S}_{M+1}} \|\varphi\|_{\mathcal{S}_{M+1}}
2^{-|i-j|M} \frac{2^{-(i\wedge j)M}}{(2^{-(i\wedge j)}+|x|)^{n+M}},  $$
where the implicit positive constant depends only on $n$ and $M$.
\end{lemma}

The following lemma gives a sufficient condition ensuring that
$\|\cdot\|_{l^{q(\cdot)}(L^{p(\cdot)})}$ is a norm,
which is precisely \cite[Theorems 3.6 and 3.8]{ah10}.
\begin{lemma}\label{seq norm}
Let $p(\cdot),q(\cdot)\in \mathcal{P}_0$.
Then $\|\cdot\|_{l^{q(\cdot)}(L^{p(\cdot)})}$ is a quasi-norm.
Moreover, if $p(\cdot), q(\cdot) \in \mathcal{P}$
satisfy either $\frac{1}{p(\cdot)}+ \frac{1}{q(\cdot)}\leq 1$ pointwise
or $q$ is a constant, then $\|\cdot\|_{l^{q(\cdot)}(L^{p(\cdot)})}$ is a norm.
\end{lemma}
Finally, we give the proof of Theorem \ref{phi bound}.
\begin{proof}[Proof of Theorem \ref{phi bound}]
We first prove the boundedness of $S_\varphi$.
For any $\vec{f} \in B^{s(\cdot)}_{p(\cdot),q(\cdot)}(W, \widetilde{\varphi})$,
letting $ \sup_{\mathbb{A},\widetilde{\varphi}} (\vec{f}) $ be as in \eqref{def supp f},
then, by the definition of $S_\varphi$,
we obtain, for any $j\in \mathbb{Z}_+$ and $Q\in\mathcal{Q}_j$,
$$ \left| A_Q \left( S_\varphi \vec{f} \right)_Q \right|
= \left| A_Q \left\langle \vec{f},\varphi_Q \right\rangle \right|
= |Q|^\frac12 \left| A_Q \left( \widetilde{\varphi_{j}} \ast \vec{f} \right)(x_Q) \right|
\leq \sup_{\mathbb{A},\widetilde{\varphi}, Q}\left( \vec{f} \right) $$
and hence, for any $j\in \mathbb{Z}_+$ and $x\in \mathbb{R}^n$,
\begin{align*}
2^{js(x)} \left| A_j \left( S_\varphi \vec{f} \right)_j (x) \right| \leq
2^{js(x)} \sum_{Q\in \mathcal{Q}_j} \sup_{\mathbb{A},\widetilde{\varphi}, Q}\left( \vec{f} \right) \widetilde{\mathbf{1}}_Q(x).
\end{align*}
This, together with the definitions of $\|\cdot\|_{b^{s(\cdot)}_{p(\cdot),q(\cdot)}(\mathbb{A})}$ and $\|\cdot\|_{b^{s(\cdot)}_{p(\cdot),q(\cdot)}}$
and Theorems \ref{W aa 3} and \ref{W aa supp},
further implies that
$$ \left\| S_\varphi \vec{f} \right\|_{b^{s(\cdot)}_{p(\cdot),q(\cdot)}(W)}
\sim \left\| S_\varphi \vec{f} \right\|_{b^{s(\cdot)}_{p(\cdot),q(\cdot)}(\mathbb{A})}
\leq \left\| \sup_{\mathbb{A},\widetilde{\varphi}} \left(\vec{f} \right) \right\|_{b^{s(\cdot)}_{p(\cdot),q(\cdot)}}
\sim \left\|  \vec{f} \right\|_{B^{s(\cdot)}_{p(\cdot),q(\cdot)}(W, \widetilde{\varphi})}, $$
which completes the proof of the boundedness of $S_\varphi$.

Next, we show the boundedness of $T_\psi$.
By Lemma \ref{Tpsi 1},
we find that $T_\psi$ is well-defined
for any $\vec{t} \in b^{s(\cdot)}_{p(\cdot),q(\cdot)}$
and hence, by the definition of $T_\psi$, we obtain,
for any $j\in\mathbb{Z}_+$, $Q\in \mathcal{Q}_j$, and $x\in Q$,
\begin{align}\label{AQ Ii 5}
\left|A_Q \left[ \varphi_j \ast T_\psi\vec{t} \right](x)\right|
& = \left| \sum_{i \in \mathbb{Z}_+} \sum_{R\in \mathcal{Q}_i} A_Q \vec{t}_R \left( \varphi_j\ast \psi_R \right)(x) \right|.
\end{align}
Notice that, for any $\{\psi_i\}_{i\in\mathbb{Z}_+}$ and $\{\varphi_j\}_{j\in\mathbb{Z}_+}$
as in Definition \ref{phi pair}
and any $i,j\in\mathbb{Z}_+$, if $|i-j| > 1$,
then $ \psi_i \ast \varphi_j = 0. $
Using this and \eqref{AQ Ii 5}, we conclude that,
for any $j\in\mathbb{Z}_+$, $Q\in \mathcal{Q}_j$, and $x\in Q$,
\begin{align}\label{AQ Ii 1}
\left|A_Q \left[ \varphi_j \ast T_\psi\vec{t} \right](x)\right|
& \leq \sum_{i\in\mathbb{Z}_+} \sum_{R\in \mathcal{Q}_i} \left\|A_Q \vec{t}_R\right\| \left|\left( \varphi_j\ast \psi_R \right)(x) \right|
 = \sum_{\substack{i\in\mathbb{Z}_+\\ |i-j| \leq 1}} \sum_{R\in \mathcal{Q}_i} \left|A_Q\vec{t}_R \right| \left| \left( \varphi_j\ast \psi_R \right)(x) \right|\nonumber\\
& \leq \sum_{\substack{i\in\mathbb{Z}_+\\ |i-j| \leq 1}} \sum_{R\in \mathcal{Q}_i} \left\| A_Q A_R^{-1}\right\| \left|A_R\vec{t}_R \right| \left| \left( \varphi_j\ast \psi_R \right)(x) \right|,
\end{align}
By Lemma \ref{QP5}, we find that, for any $j,i\in\mathbb{Z}_+$ with $|i - j| \leq 1$,
any $Q\in\mathcal{Q}_j$, and $R\in\mathcal{Q}_i$,
\begin{align}\label{AQ Ii 2}
\left\| A_QA_R^{-1} \right\| \
&\lesssim \max\left\{ \left[ \frac{l(R)}{l(Q)} \right]^{d_1},
\left[ \frac{l(Q)}{l(R)} \right]^{d_2} \right\}\left[ 1+ \frac{|x_Q - x_R|}{l(Q)\vee l(R)} \right]^{\Delta}
\sim \left\{ 1+ \left[ l(R) \right]^{-1} \left|x_Q - x_R\right| \right\}^{\Delta},
\end{align}
where $d_1,d_2$, and $\Delta$ are the same as in Lemma \ref{QP5}.
Let $M$ satisfy \eqref{M con}.
Observe that, for any $i,j\in\mathbb{Z}_+$,
$|i-j| + (i\wedge j) = i\vee j$.
Then, from this and Lemmas \ref{psiphi 1} and \ref{psiphi 2}
or, when both $j,i= 0$, from the fact that, for any $M \in(0,\infty)$ and $x\in\mathbb{R}^n$,
$|\varphi_0 \ast \psi_0(x)| \lesssim ( 1 + |x| )^{-(n+M)}$,
it follows that, for any $j,i\in\mathbb{Z}_+$ with $|i-j| \leq 1$,
any $R\in\mathcal{Q}_i$, and any $x\in\mathbb{R}^n$,
\begin{align}\label{AQ Ii 3}
\left| \left( \varphi_j\ast \psi_R \right)(x) \right|
& = |R|^\frac12 \left| \left( \varphi_j\ast \psi_i \right)(x-x_R) \right|
\lesssim |R|^\frac12 2^{-|i-j|M}\frac{2^{-(i\wedge j)M}}{[2^{-(i\wedge j)}+|x-x_R|]^{n+M}}\nonumber\\
& = |R|^\frac12 2^{(i \vee j)n} \frac{2^{-(i \vee j)(M + n)}}{[2^{-(i\wedge j)}+|x-x_R|]^{n+M}}
\sim |R|^{-\frac12} \left\{1+ [l(R)]^{-1} |x-x_R|\right\}^{-(n+M)}.
\end{align}
Let $u := \{u_Q\}_{Q\in\mathcal{Q}_+}$,
where $u_Q := |A_Q \vec{t}_Q|$ for any $Q\in\mathcal{Q}_+$.
Then, by \eqref{AQ Ii 1}, \eqref{AQ Ii 2}, and \eqref{AQ Ii 3},
we conclude that, for any $j\in\mathbb{Z}_+$, $Q\in\mathcal{Q}_j$, and $x\in Q$,
\begin{align}\label{AQ var 1}
\left| A_Q \left[ \varphi_j \ast T_\psi\vec{t} \right](x) \right|
\lesssim \sum_{\substack{i\in\mathbb{Z}_+\\ |i-j| \leq 1}} \sum_{R\in\mathcal{Q}_i}  |R|^{-\frac12} \frac{u_R}{ \{1+ [l(R)]^{-1} |x-x_R|\}^{n+M - \Delta} }
\sim |Q|^{-\frac12} \sum_{\substack{i\in\mathbb{Z}_+\\ |i-j| \leq 1}} {\rm I}_i(x),
\end{align}
where, for any $i\in\mathbb{Z}_+$,
\begin{align}\label{def Ii}
{\rm I}_i(x) := \sum_{R\in\mathcal{Q}_i} \frac{u_R}{ \{1+ [l(R)]^{-1} |x-x_R|\}^{n+M - \Delta} }.
\end{align}

Notice that, by the definition of dyadic cubes,
for any $x\in \mathbb{R}^n$ and $j\in\mathbb{Z}_+$,
there exists a unique cube $Q\in\mathcal{Q}_j$ such that $ x\in Q$.
Combining this with \eqref{def Ii} and Lemma \ref{QP4}, we obtain
\begin{align}\label{AQ Ii 4}
{\rm I}_i (x) \lesssim \sum_{R\in\mathcal{Q}_i} \frac{u_R}{ \{1+ [l(R)]^{-1} |x_Q-x_R|\}^{n+M - \Delta} } = \left(u_{1,n+M - \Delta}^\ast\right)_Q,
\end{align}
where $(u_{1,n+M - \Delta}^\ast)_Q$ is the same as in \eqref{def tast}.
In what follows, for simplicity of presentation,
let $ (u_{1,n+M - \Delta}^\ast)_{-1} := 0 $.
Applying this with \eqref{AQ var 1} and \eqref{AQ Ii 4},
we conclude that, for any $ j\in \mathbb{Z}_+$,
$$ \left| A_j \left[ \varphi_j \ast T_\psi\vec{t} \right] \right| \lesssim \sum_{i = -1}^{ 1} \left(u_{1,n+M - \Delta}^\ast\right)_{j + i}. $$
By this, Lemma \ref{W aa supp}, the definition of $\|\cdot\|_{{B^{s(\cdot)}_{p(\cdot),q(\cdot)}(\mathbb{A}, \varphi)}}$,
and Lemma \ref{seq norm},
\begin{align*}
\left\| T_\psi(\vec{t}) \right\|_{B^{s(\cdot)}_{p(\cdot),q(\cdot)}(W, \varphi)}
& \sim \left\| T_\psi(\vec{t}) \right\|_{B^{s(\cdot)}_{p(\cdot),q(\cdot)}(\mathbb{A}, \varphi)}
 = \left\| \left\{ 2^{js(\cdot)} \left| A_j \left[\varphi_j \ast T_\psi(\vec{t}) \right] \right|  \right\}_{j\in\mathbb{Z}_+} \right\|_{l^{q(\cdot)}(L^{p(\cdot)})}\\
& \lesssim \sum_{i = -1}^{1} \left\| \left\{ 2^{(j+i) s(\cdot)} \left(u_{1,n+M - \Delta}^\ast\right)_{j+i} \right\}_{j\in\mathbb{Z}_+} \right\|_{l^{q(\cdot)}(L^{p(\cdot)})}.
\end{align*}
Applying this with Lemma \ref{tast bound}
with $t := u$, $r := 1$, and $\lambda := n + M - \Delta$,
we conclude that
\begin{align*}
\left\| T_\psi(\vec{t}) \right\|_{B^{s(\cdot)}_{p(\cdot),q(\cdot)}(W, \varphi)}
& \lesssim \left\| \left\{ 2^{j s(\cdot)} \left(u_{1,n+M - \Delta}^\ast\right)_{j} \right\}_{j\in\mathbb{Z}_+} \right\|_{l^{q(\cdot)}(L^{p(\cdot)})}
 = \left\| u_{1,n+M - \Delta}^\ast \right\|_{b^{s(\cdot)}_{p(\cdot),q(\cdot)}}
 \lesssim \left\| u \right\|_{b^{s(\cdot)}_{p(\cdot),q(\cdot)}},
\end{align*}
which, together with the definition of $u$ and Theorem \ref{W aa 3},
further implies that
\begin{align*}
\left\| T_\psi(\vec{t}) \right\|_{B^{s(\cdot)}_{p(\cdot),q(\cdot)}(W, \varphi)}
\lesssim \left\| \vec{t} \right\|_{b^{s(\cdot)}_{p(\cdot),q(\cdot)}(\mathbb{A})}
\sim \left\| \vec{t} \right\|_{b^{s(\cdot)}_{p(\cdot),q(\cdot)}(W)}.
\end{align*}
This finishes the proof of the boundedness of $T_\psi$.

Finally, if $\{\varphi_j\}_{j\in\mathbb{Z}_+}$ and $\{\psi_j\}_{j\in\mathbb{Z}_+}$ satisfy \eqref{varphi psi},
then it follows immediately from Lemma \ref{CZ com}
that $T_\psi\circ S_\varphi $ is the identity on $B^{s(\cdot)}_{p(\cdot),q(\cdot)}(W, \widetilde{\varphi})$,
which completes the proof of Theorem \ref{phi bound}.
\end{proof}

\section{Almost Diagonal Operators}\label{sec abo}
In this section, we focus on the boundedness of almost diagonal operators
on $b^{s(\cdot)}_{p(\cdot),q(\cdot)}(W)$,
which is a key tool for establishing various real-variable characterizations of Besov spaces
and the boundedness of operators on them (see, for example, \cite{fj90,fr21,bhyy23 2}).
Let $B := \{b_{Q,R}\}_{Q,R\in\mathcal{Q}_+} \subset \mathbb{C}$.
For any sequence $\vec{t} := \{\vec{t}_R\}_{R\in\mathcal{Q}_+} \subset \mathbb{C}^m$,
we define $B\vec{t} := \{(B\vec{t})_Q\}_{Q\in\mathcal{Q}_+}$ by setting,
for any $Q\in\mathcal{Q}_+$,
\begin{align}\label{def abo seq}
\left(B\vec{t}\right)_Q := \sum_{R\in\mathcal{Q}_+} b_{Q,R} \vec{t}_R
\end{align}
if this given summation is absolutely convergent.
Then we recall the concept of almost diagonal operators,
which was first introduced by Frazier and Jawerth in \cite{fj90}.

\begin{definition}\label{def ado}
Let $D, E, F \in \mathbb{R}$.
Define the special infinite matrix
$B^{DEF} := \{b^{DEF}_{Q,R}\}_{Q,R \in \mathcal{Q}_+} \subset \mathbb{C}$
by setting, for any $Q,R \in\mathcal{Q}_+$,
\begin{align}\label{def bDEF}
b^{DEF}_{Q,R} := \left[ 1 + \frac{|x_Q - x_R|}{l(Q)\vee l(R)} \right]^{-D}
\begin{cases}
\displaystyle \left[ \frac{l(Q)}{l(R)} \right]^E & \text{if}\ l(Q)\leq l(R), \\
\displaystyle \left[ \frac{l(R)}{l(Q)} \right]^F & \text{if}\ l(Q)> l(R).
\end{cases}
\end{align}
An infinite matrix $B := \{b_{Q,R}\}_{Q,R\in\mathcal{Q}_+} \subset \mathbb{C}$ is said to be
\emph{$( D,E,F )$-almost diagonal} if there exists a positive constant $C$
such that, for any $Q,R\in\mathcal{Q}_+$, $|b_{Q,R}| \leq C b_{Q,R}^{DEF}$.
\end{definition}

\begin{remark}
\begin{itemize}
\item[{\rm (i)}] If $E+F >0$, which is usually the only case interested to us,
then the second factor on the right-hand side of \eqref{def bDEF}
is exactly
$$ \min\left\{ \left[ \frac{l(Q)}{l(R)} \right]^E, \left[ \frac{l(R)}{l(Q)} \right]^F \right\}. $$
\item[{\rm (ii)}] Clearly, the special infinite matrix
$B^{DEF}$ itself is $(D,E,F)$-almost diagonal.
\end{itemize}
\end{remark}

Recall that, in the setting of variable function spaces, the Fefferman--Stein type
vector-valued inequality involving the Hardy--Littlewood maximal operator
is known to be non-existent. As a suitable substitute, a vector-valued inequality
involving $\eta_{j,m}$ functions fits very well into this scheme; see, for example, \cite{ah10}.
The following is the matrix-weighted version of the vector-valued inequality
involving $\eta$ functions.
The proof of this result is similar to that of \cite[Lemma 4.7]{ah10}
with the boundedness of the convolution of $\eta_{j,m}$
replaced by Theorem \ref{bound eta}. We omit the details here.

\begin{proposition}\label{eta bound seq}
Let $p(\cdot), q(\cdot)\in \mathcal{P}_0\cap LH$ and $W\in\mathscr{A}_{p(\cdot),\infty}$.
Then there exists $\alpha \in (0,1]$, depending on $[W]_{\mathscr{A}_{p(\cdot),\infty}}$,
such that, for any $m\in (\frac{n}{\alpha} + C_{\rm log}(\frac1q),\infty)$,
where $C_{\rm log}(\frac1q)$ is the same as in \eqref{clogp},
and for any sequence of measurable functions $\{\vec{f}_j\}_{j\in\mathbb{N}}$,
\begin{align}\label{eq eta bound seq}
\left\| \left\{ \eta_{j,m,W}^{(\alpha)}(\vec{f}_j) \right\}_{j\in\mathbb{Z}_+} \right\|_{l^{q(\cdot)}(L^{p(\cdot)})}
\lesssim \left\| \left\{ \left| \vec{f}_j \right| \right\}_{j\in\mathbb{Z}_+} \right\|_{l^{q(\cdot)}(L^{p(\cdot)})},
\end{align}
where $\eta_{j,m,W}^{(\alpha)}$ is as in \eqref{def etaf} and
the implicit positive constant is independent of $\{f_j\}_{j\in\mathbb{Z}_+}$.
\end{proposition}
Let $p(\cdot) \in \mathcal{P}_0\cap LH$.
For any $W\in \mathscr{A}_{p(\cdot),\infty}$, let
\begin{align}\label{def alphaW}
\alpha_W := \sup\left\{\alpha \in (0,1]:\ \eta^{(\alpha)}_{j,m,W} \ \text{is bounded on}\ L^{p(\cdot)}\right\}.
\end{align}
\begin{remark}\label{rem eta seq}
\begin{itemize}
\item[{\rm (i)}] Let $p(\cdot),q(\cdot) \in \mathcal{P}_0\cap LH$
and $W\in \mathscr{A}_{p(\cdot),\infty}$.
Then, by both the proof of \cite[Lemma 4.7]{ah10} and
\cite[Example 3.4]{ah10}, we conclude that,
for any $\alpha\in (0,1]$, $m\in (\frac{n}{\alpha}+C_{\rm log}(\frac1q),\infty)$,
and $j\in \mathbb{Z}_+$,
$\eta_{j,m,W}^{(\alpha)}$ is bounded on $L^{p(\cdot)}$
if and only if $\{\eta_{j,m,W}^{(\alpha)}\}_{j\in\mathbb{Z}_+}$ is bounded on
$l^{q(\cdot)}(L^{p(\cdot)})$.
\item[{\rm (ii)}] Let $p(\cdot) \in \mathcal{P}_0\cap LH$ with $p_- > 1$,
$q(\cdot) \in \mathcal{P}_0\cap LH$, and $m\in (n+C_{\rm log}(\frac1q),\infty)$.
From Remarks \ref{rem eta} and \ref{rem eta seq}{\rm (i)},
we infer that, for any $W \in \mathscr{p(\cdot)}$,
$\{\eta_{j,m,W}^{(1)}\}_{j\in\mathbb{Z}_+}$ is bounded on $l^{q(\cdot)}(L^{p(\cdot)})$
and hence $\alpha_W = 1$.
\end{itemize}
\end{remark}

The following is the boundedness of almost diagonal operators
on matrix-weighted variable Besov sequence spaces.
We refer to \cite{bhyy23 2} for the known best results about almost diagonal operators on
matrix $A_{p}$ weighted Besov sequence spaces
and to \cite{bhyy24} on matrix $ A_{p,\infty} $ weighted Besov sequence spaces.

\begin{theorem}\label{abo bound 1}
Let $p(\cdot), q(\cdot)\in \mathcal{P}_0\cap LH$, $s(\cdot)\in LH$,
and $W\in \mathscr{A}_{p(\cdot),\infty}$.
If $B$ is $(D,E,F)$-almost diagonal,
then $B$ is bounded on $b^{s(\cdot)}_{p(\cdot),q(\cdot)}(W)$
whenever
\begin{align}\label{abo bound con}
 D > \frac{n}{\alpha_W} + C(s,q), \quad E > \frac{n}{2} + s_+, \quad \text{and}\quad F> \frac{n}{\alpha_W} - \frac{n}{2} - s_-,
\end{align}
where $C(s,q) := C_{\rm{log}}(s) + C_{\rm log}(q^{-1})$
with $C_{\rm{log}}(s)$ and $C_{\rm log}(q^{-1})$ being the same as in \eqref{clogp}.
\end{theorem}

\begin{remark}
\begin{itemize}
\item[{\rm (i)}] Let $p(\cdot)$, $q(\cdot)$, and $s(\cdot)$ all be constant exponents.
Then, if $W$ is an $\mathscr{A}_{p}$ matrix weight,
then \eqref{abo bound con} coincides with the sharp result obtained in
\cite[Theorem 4.1]{bhyy23 2} in the case $\tau := 0$.
Hence, in this sense, when $p(\cdot),q(\cdot) \in \mathcal{P}\cap LH$,
$s(\cdot) \in LH$, and $W\in \mathscr{A}_{p(\cdot)}$,
the ranges of $D$, $E$, and $F$ obtained in \eqref{abo bound con} are sharp.
Moreover, when we reduce to the scalar-valued case,
by \cite[Lemma 4.32]{bhyy24} and Remark \ref{rem eta seq},
there exists $w\in A_\infty\setminus A_p$ such that the ranges of $D$, $E$, and $F$ in \eqref{abo bound con}
are wider than the corresponding ones in \cite[Theorem 4.6]{bhyy24}
in the case $\tau := 0$.
\item[{\rm (ii)}] The ranges of $D$, $E$, and $F$ of Theorem \ref{abo bound 1}
in the unweighted scalar-valued variable Besov spaces case
are also wider than the corresponding ones in \cite[Theorem 2]{hsz24}.
\end{itemize}
\end{remark}

Using Theorem \ref{abo bound 1},
we introduce the concept of $b^{s(\cdot)}_{p(\cdot),q(\cdot)}(W)$-almost diagonal operators.

\begin{definition}
Let $p(\cdot), q(\cdot)\in \mathcal{P}_0\cap LH$,
$s(\cdot)\in LH$, and $W\in \mathscr{A}_{p(\cdot),\infty}$
and let $B$ be $(D,E,F)$-almost diagonal.
Then $B$ is called a \emph{$b^{s(\cdot)}_{p(\cdot),q(\cdot)}(W)$-almost diagonal operator}
if $D,E$, and $F$ satisfy \eqref{abo bound con}.
\end{definition}

The following is the property of $ \eta^{(\alpha)}_{j,m,W}$.
\begin{lemma}\label{eta alpha}
Let $s(\cdot) \in LH$, $\alpha\in (0,1]$, and $j,m\in\mathbb{N}$.
If $R \in (C_{log}(s),\infty)$, where $C_{log}(s)$ is the same as in \eqref{clogp},
then, for any $f\in L^1_{\rm loc}$,
\begin{align*}
2^{js(x)} \left[\eta^{(\alpha)}_{j,m+R,W}\ast f\right](x)
\lesssim \eta^{(\alpha)}_{j,m,W} \ast \left[ 2^{js(\cdot)} f\right] (x),
\end{align*}
where the implicit positive constant is independent of $x$, $j$, $W$, and $f$.
\end{lemma}
\begin{proof}
By the definition of $\eta^{(\alpha)}_{j,m+R,W}$ and Lemma \ref{s eta},
we find that
\begin{align*}
2^{js(x)} \left[\eta^{(\alpha)}_{j,m+R,W}\ast f\right](x)
& =  \left[\int_{\mathbb{R}^n} 2^{j\alpha s(x)} \frac{2^{jn}|W(x)W^{-1}(y)\vec{f}(y)|^{\alpha}}{(1 + 2^j|x-y|)^{\alpha (m+R)}}\,dy\right]^\frac{1}{\alpha}\\
&\lesssim \left[\int_{\mathbb{R}^n} 2^{j\alpha s(y)} \frac{2^{jn}|W(x)W^{-1}(y)\vec{f}(y)|^{\alpha}}{(1 + 2^j|x-y|)^{\alpha m}}\,dy\right]^\frac{1}{\alpha}
 = \eta^{(\alpha)}_{j,m,W} \ast \left[ 2^{js(\cdot)} f\right] (x),
\end{align*}
which completes the proof of Lemma \ref{eta alpha}.
\end{proof}
In what follows, for any $r\in\mathbb{R}$,
let $r^{(+)} := \max\{0,r\}$ and $r^{(-)} := \max\{0,-r\}$.
Now, we prove Theorem \ref{abo bound 1}.
\begin{proof}[Proof of Theorem \ref{abo bound 1}]
By Proposition \ref{eta bound seq},
there exists $\alpha \in (0,1]$ such that \eqref{eq eta bound seq} holds.
Observe that, for any $j\in\mathbb{Z}_+$, $Q\in\mathcal{Q}_j$, and $x\in Q$,
if $ (B\vec{t})_Q $ converges absolutely,
then $W(x)(B\vec{t})_Q = (B[W(x)\vec{t}])_Q$.
From this, \eqref{def abo seq}, Definition \ref{def ado}, and Lemma \ref{QP4},
we deduce that, for any $j\in\mathbb{Z}_+$, $Q\in\mathcal{Q}_j$, and $x\in Q$,
\begin{align*}
\left|W(x)\left( B \vec{t} \right)_Q\right| &= \left|\left( B \left(W(x)\vec{t}\right) \right)_Q\right|
 \leq \sum_{R\in\mathcal{Q}_+} \left|b_{Q,R}\right| \left| W(x)\vec{t}_R \right|\\
& \lesssim \sum_{i = 0}^\infty \sum_{R\in \mathcal{Q}_i} \min\left\{ \left[ \frac{l(Q)}{l(R)} \right]^E, \left[ \frac{l(R)}{l(Q)} \right]^F \right\}
\left[ 1 + \frac{|x_Q - x_R|}{l(Q)\vee l(R)} \right]^{-D} \left| W(x)\vec{t}_R \right|\\
& \sim \sum_{i = 0}^\infty 2^{-(j-i)^{(+)} E} 2^{-(i-j)^{(+)} F} \sum_{R\in \mathcal{Q}_i} \left( 1 + 2^{i\wedge j} \left|x - x_R\right| \right)^{-D} \left|W(x) \vec{t}_R \right|.
\end{align*}
This, together with the well-known inequality that, for any $\alpha\in (0,1]$
and any sequence $\{ a_k \}_{k\in\mathbb{Z}_+} \subset \mathbb{C}$,
$\sum_{k = 0}^\infty |a_k| \leq (\sum_{k = 0}^\infty |a_k|^\alpha)^\frac{1}{\alpha}$
and Lemma \ref{QP4}, further implies that
\begin{align*}
\left|W(x)\left( B \vec{t} \right)_Q\right|
& \lesssim \sum_{i = 0}^\infty 2^{-(j-i)^{(+)} E} 2^{-(i-j)^{(+)} F} \left[\sum_{R\in \mathcal{Q}_i} \left( 1 + 2^{i\wedge j} \left|x - x_R\right| \right)^{-\alpha D} \left|W(x) \vec{t}_R \right|^{\alpha}\right]^{\frac{1}{\alpha}}\nonumber\\
& = \sum_{i = 0}^\infty 2^{-(j-i)^{(+)} E} 2^{-(i-j)^{(+)} F} \left[\sum_{R\in \mathcal{Q}_i} \fint_{R} \frac{ | W(x) \vec{t}_R |^{\alpha}}{ (1 + 2^{i\wedge j} |x - x_R| )^{\alpha D}}\,dy\right]^{\frac{1}{\alpha}}.
\end{align*}
From this, the disjointness of the dyadic cubes in $\mathcal{Q}_i$,
and the definitions of $\vec{t}_j$ and $\eta^{(\alpha)}_{j,m,W}$,
it follows that
\begin{align*}
\left|W(x)\left( B \vec{t} \right)_Q\right|
& \lesssim \sum_{i = 0}^\infty 2^{-(j-i)^{(+)} E} 2^{-(i-j)^{(+)} F} \left[\sum_{R\in \mathcal{Q}_i} |R|^{\frac{\alpha}2 - 1} \int_{R} \frac{ |R|^{-\frac{\alpha}2} | W(x) \vec{t}_R |^{\alpha}}{ (1 + 2^{i\wedge j} |x - x_R| )^{\alpha D}}\,dy\right]^{\frac{1}{\alpha}} \nonumber\\
& = \sum_{i = 0}^\infty 2^{-(j-i)^{(+)} E} 2^{-(i-j)^{(+)} F} 2^{n(\frac{1}{\alpha} - \frac{1}{2}) i } 2^{-\frac{n}{\alpha}(i\wedge j) } \eta_{i\wedge j,m,W}^{(\alpha)}(\vec{t}_i)(x).
\end{align*}
Using this, the definition of $(  B (W(x)\vec{t}) )_j$,
and the disjointness of the dyadic cubes in $\mathcal{Q}_j$,
we conclude that, for any $j\in\mathbb{Z}_+$ and $x\in\mathbb{R}^n$,
\begin{align*}
\left|\left(  B \left(W(\cdot)\vec{t}\right) \right)_j(x)\right|
&\lesssim \sum_{i = 0}^\infty 2^{-(j-i)^{(+)} E} 2^{-(i-j)^{(+)} F} 2^{n(\frac{1}{\alpha} - \frac{1}{2}) i } 2^{\frac{n}{2} j} 2^{-\frac{n}{\alpha}(i\wedge j) } \eta_{i\wedge j,m,W}^{(\alpha)}(\vec{t}_i)(x).
\end{align*}
Thus, combining this with Lemma \ref{eta alpha}
and the facts that $i - (i\wedge i) = (i-j)^{(+)}$ and $ j- (i\wedge j) = (j-i)^{(+)} $,
we conclude that
\begin{align}\label{alm eq 2}
2^{js(x)} \left|\left( B \left(W(\cdot)\vec{t}\right) \right)_j(x)\right|
&\lesssim 2^{js(x)} \sum_{i = 0}^\infty 2^{-(j-i)^{(+)} E} 2^{-(i-j)^{(+)} F} 2^{n(\frac{1}{\alpha} - \frac{1}{2}) (i - i\wedge j) } 2^{\frac{n}{2} (j - i\wedge j)} \eta_{i\wedge j,m,W}^{(\alpha)}(\vec{t}_i)(x)\nonumber\\
&= \sum_{i = 0}^\infty 2^{-(j-i)^{(+)}(E - \frac{n}{2})} 2^{(j - i)^{+}s(x)} 2^{-(i-j)^{(+)} (F - \frac{n}{\alpha} + \frac{n}{2})} 2^{(j\wedge i)s(x)} \eta_{i\wedge j,m,W}^{(\alpha)}(\vec{t}_i)(x)\nonumber\\
&\lesssim \sum_{i = 0}^\infty 2^{-(j-i)^{(+)}(E - \frac{n}{2} - s_+)} 2^{-(i-j)^{(+)} (F - \frac{n}{\alpha} + \frac{n}{2})} \eta_{i\wedge j,m-R',W}^{(\alpha)}\left(2^{(j\wedge i)s(\cdot)}\vec{t}_i\right)(x),
\end{align}
where $R' \in (C_{\rm log}(s),\infty)$ is a fixed constant.

From Lemma \ref{seq norm}, Remark \ref{rem def seq}{\rm (iv)},
and the fact $p(\cdot),q(\cdot)\in\mathcal{P}_0$,
we deduce that there exists a positive constant $a\in (0,1]$
such that $\frac{a}{p(x)}+\frac{a}{q(x)} \leq 1$
and hence $\|\cdot\|_{l^{\frac{q(\cdot)}{a}}(L^{\frac{p(\cdot)}{a}})}$ is a norm.
Then, by \eqref{alm eq 2} with $ k := i-j $,
we find that, for any $j\in\mathbb{Z}_+$ and $x\in \mathbb{R}^n$,
\begin{align}\label{alm eq 4}
&2^{ajs(x)}\left|W(x)\left(B\vec{t}\right)_j\right|^a\nonumber\\
&\quad \lesssim  \sum_{i = 0}^\infty 2^{-a(j-i)^{(+)} (E-\frac n2-s_+)} 2^{-a(i-j)^{(+)} (F-\frac{n}{\alpha} + \frac n2)}
\left[\eta_{i\wedge j,m-R',W}^{(\alpha)}\left(2^{(j\wedge i)s(\cdot)}\vec{t}_i\right)(x)\right]^a \nonumber\\
&\quad = \sum_{k = -j}^\infty  2^{-ak^{(-)} (E-\frac n2-s_+)} 2^{-ak^{(+)} (F-\frac{n}{\alpha} + \frac n2)}
\left[\eta_{(k + j)\wedge j,m-R',W}^{(\alpha)}\left(2^{[(k + j)\wedge j]s(\cdot)}\vec{t}_{j+k}\right)(x)\right]^a .
\end{align}
Now, for any $j\in\mathbb{Z}_+$, $k\in\mathbb{Z}$,
and $x\in\mathbb{R}^n$,
let
\begin{align*}
r_{j,k}(x) :=
\begin{cases}
\displaystyle \left[\eta_{(k + j)\wedge j,m-R',W}^{(\alpha)}(2^{[(k + j)\wedge j]s(\cdot)}\vec{t}_{j+k})(x)\right]^a &\ \text{if}\ j\geq k^{(-)}, \\
\displaystyle 0 &\ \text{otherwise}.
\end{cases}
\end{align*}
From this and \eqref{alm eq 4},
we infer that, for any $j\in\mathbb{Z}_+$ and $x\in\mathbb{R}^n$,
$$2^{ajs(x)}\left|W(x)\left(B\vec{t}\right)_j\right|^a
\lesssim \sum_{k\in\mathbb{Z}}  2^{-ak^{(-)} (E-\frac n2-s_+)} 2^{-ak^{(+)} (F-\frac{n}{\alpha} + \frac n2)}  r_{j,k}(x).$$
Using this, Remark \ref{rem def seq}{\rm (iv)},
and the proved fact that
$\|\cdot\|_{l^{\frac{q(\cdot)}{a}}(L^{\frac{p(\cdot)}{a}})}$ is a norm,
we conclude that
\begin{align}\label{alm eq 13}
\left\| \left\{ 2^{js(\cdot)} \left| W(\cdot) \left(B \vec{t} \right)_j\right| \right\}_{j\in\mathbb{Z}_+} \right\|^a_{l^{q(\cdot)}(L^{p(\cdot)})}
& = \left\| \left\{ 2^{ajs(\cdot)} \left| W(\cdot) \left(B \vec{t} \right)_j\right|^a \right\}_{j\in\mathbb{Z}_+} \right\|_{l^{\frac{q(\cdot)}{a}}(L^{\frac{p(\cdot)}{a}})}\nonumber\\
& \leq \sum_{k\in\mathbb{Z}} 2^{-ak^{(-)} (E-\frac n2-s_+)}2^{-ak^{(+)} (F-\frac{n}{\alpha} + \frac n2)} \left\| \left\{  r_{j,k} \right\}_{j\in\mathbb{Z}_+} \right\|_{l^{\frac{q(\cdot)}{a}}(L^{\frac{p(\cdot)}{a}})}.
\end{align}
Observe that, for any $k\in\mathbb{Z}$ and $j\in\mathbb{Z}_+$ with $j\in [0,-k^{(-)})$,
$r_{j,k} \equiv 0$ and hence, combining this with the definition of $\|\cdot\|_{l^{\frac{q(\cdot)}{a}}(L^{\frac{p(\cdot)}{a}})}$,
we find that, for any $k\in\mathbb{Z}$,
$$ \left\| \left\{  r_{j,k} \right\}_{j\in\mathbb{Z}_+} \right\|_{l^{\frac{q(\cdot)}{a}}(L^{\frac{p(\cdot)}{a}})}
 = \left\| \left\{  r_{j + k^{(-)},k} \right\}_{j\in\mathbb{Z}_+} \right\|_{l^{\frac{q(\cdot)}{a}}(L^{\frac{p(\cdot)}{a}})}. $$
Applying this with the definition of $r_{j,k}$ yields that,
for any $k\in\mathbb{Z}$ with $k\leq -1$,
$$ \left\| \left\{  r_{j,k} \right\}_{j\in\mathbb{Z}_+} \right\|_{l^{\frac{q(\cdot)}{a}}(L^{\frac{p(\cdot)}{a}})}
= \left\| \left\{ \eta_{j,m-R',W}^{(\alpha)}\left(2^{js(\cdot)}\vec{t}_{j}\right) \right\}_{j\in\mathbb{Z}_+} \right\|_{l^{\frac{q(\cdot)}{a}}(L^{\frac{p(\cdot)}{a}})} $$
and, moreover, for any $k\in\mathbb{Z}_+$,
$$ \left\| \left\{  r_{j,k} \right\}_{j\in\mathbb{Z}_+} \right\|_{l^{\frac{q(\cdot)}{a}}(L^{\frac{p(\cdot)}{a}})}
= \left\| \left\{ \eta_{j,m-R',W}^{(\alpha)}\left(2^{js(\cdot)}\vec{t}_{j+k}\right) \right\}_{j\in\mathbb{Z}_+} \right\|_{l^{\frac{q(\cdot)}{a}}(L^{\frac{p(\cdot)}{a}})}.  $$
From these and \eqref{alm eq 13}, we infer that
\begin{align*}
&\left\| \left\{ 2^{js(\cdot)} \left| W(\cdot) \left(B \vec{t} \right)_j\right| \right\}_{j\in\mathbb{Z}_+} \right\|^a_{l^{q(\cdot)}(L^{p(\cdot)})}\nonumber\\
&\quad \lesssim \sum_{k\in\mathbb{Z}} 2^{-ak^{(-)} (E-\frac n2-s_+)} 2^{-ak^{(+)} (F-\frac{n}{\alpha} + \frac n2)}
\left\| \left\{ \eta_{j,m-R',W}^{(\alpha)}\left(2^{js(\cdot)}\vec{t}_{j+k^{(+)}}\right) \right\}_{j\in\mathbb{Z}_+} \right\|^{a}_{l^{q(\cdot)}(L^{p(\cdot)})}\nonumber\\
&\quad \lesssim \sum_{k\in\mathbb{Z}} 2^{-ak^{(-)} (E-\frac n2-s_+)} 2^{-ak^{(+)} (F-\frac{n}{\alpha} + \frac n2 + s_-)}
\left\| \left\{ \eta_{j,m-R',W}^{(\alpha)}\left(2^{[j+k^{(+)}]s(\cdot)}\vec{t}_{j+k^{(+)}}\right) \right\}_{j\in\mathbb{Z}_+} \right\|^{a}_{l^{q(\cdot)}(L^{p(\cdot)})}.
\end{align*}
Using this and Proposition \ref{eta bound seq},
we conclude that, if $ E\in (\frac n2 + s_+,\infty) $, $F\in (\frac{n}{\alpha} - \frac n2 -s_-, \infty)$,
and $D\in (\frac{n}{\alpha} + C_{\rm log}(\frac1q) + R', \infty)$,
then
\begin{align*}
&\left\| \left\{ 2^{js(\cdot)} \left| W(\cdot) \left(B \vec{t} \right)_j\right| \right\}_{j\in\mathbb{Z}_+} \right\|^a_{l^{q(\cdot)}(L^{p(\cdot)})} \\
&\quad \lesssim \sum_{k\in\mathbb{Z}} 2^{-ak^{(-)} (E-\frac n2-s_+)} 2^{-ak^{(+)} (F-\frac{n}{\alpha} + \frac n2 + s_-)}
\left\| \left\{ 2^{js(\cdot)}\left|W(\cdot)\vec{t}_{j}\right| \right\}_{j\in\mathbb{Z}_+} \right\|^{a}_{l^{q(\cdot)}(L^{p(\cdot)})}
\lesssim \left\|\vec{t}\right\|_{b^{s(\cdot)}_{p(\cdot),q(\cdot)}}.
\end{align*}
This finishes the proof of Theorem \ref{abo bound 1}.
\end{proof}

\begin{remark}\label{abo of abo}
\begin{itemize}
\item[{\rm (i)}]
Inspired by the proof of Theorem \ref{abo bound 1},
as a slight stronger result than Theorem \ref{abo bound 1},
with  $ |B(W(\cdot)\vec{t})|(x) $ replaced by
$ \widetilde{\mathbf{1}}_{Q} \sum_{R\in\mathcal{Q}_+} |W(x)b_{Q,R} \vec{t}_{R}| $,
we find that, for any $\vec{t} \in b^{s(\cdot)}_{p(\cdot),q(\cdot)}(W)$,
\begin{align*}
\left\| \left\{ \sum_{Q\in\mathcal{Q}_j} \widetilde{\mathbf{1}}_{Q} \sum_{R\in\mathcal{Q}_+} \left|W(\cdot) b_{Q,R}\vec{t}_{R}\right| \right\} \right\|_{B^{s(\cdot)}_{p(\cdot),q(\cdot)}(W)} < \infty,
\end{align*}
which further implies that,
for any $Q\in\mathcal{Q}_+$ and almost every $x\in Q$,
$\sum_{R\in\mathcal{Q}_+} |W(x)b_{Q,R} \vec{t}_{R}|$ is finite
and hence
$$\sum_{R\in\mathcal{Q}_+} \left|b_{Q,R} \vec{t}_{R}\right| \leq \left\| W^{-1}(x) \right\| \sum_{R\in\mathcal{Q}_+}  \left|W(x)b_{Q,R} \vec{t}_{R}\right| < \infty. $$
That is, $\sum_{R\in\mathcal{Q}_+} |b_{Q,R} \vec{t}_{R}|$ converges absolutely
and hence, for any $\vec{t} \in b^{s(\cdot)}_{p(\cdot),q(\cdot)}(W)$
and any bounded almost diagonal operator $B$,
$B\vec{t}$ is well-defined.
\item[{\rm (ii)}] Let $B^{(1)} := \{b^{(1)}_{Q,R}\}_{Q,R\in\mathcal{Q}_+}$
and $B^{(2)} := \{b^{(2)}_{Q,R}\}_{Q,R\in\mathcal{Q}_+}$ be
bounded almost diagonal operators on $b^{s(\cdot)}_{p(\cdot),q(\cdot)}(W)$.
Then, by the definition of the boundedness of almost diagonal operators,
it is easy to find that
the operator $B := B^{(1)} \circ B^{(2)}$ is also a bounded almost diagonal operator on $b^{s(\cdot)}_{p(\cdot),q(\cdot)}(W)$.
Moreover, if let $B := \{b_{Q,R}\}_{Q,R\in\mathcal{Q}_+}$,
then $b_{Q,R} = \sum_{P\in\mathcal{Q}_+} b^{(1)}_{Q,P}b^{(2)}_{P,R}$.
Indeed, from Remark \ref{abo of abo}(i),
it follows that, for any $\vec{t} \in b^{s(\cdot)}_{p(\cdot),q(\cdot)}(W)$,
$B^{(1)}\vec{t}$ and $B^{(2)}\vec{t}$ are well-defined.
Hence, for any $Q \in \mathcal{Q}_+$,
$$\left(B\vec{t}\right)_Q = \sum_{P\in\mathcal{Q}_+} b^{(1)}_{Q,P} \left( B^{(2)} \vec{t} \right)_P
= \sum_{P\in\mathcal{Q}_+} b^{(1)}_{Q,P} \sum_{R\in\mathcal{Q}_+} b^{(2)}_{P,R} \vec{t}_R
= \sum_{R\in\mathcal{Q}_+} \sum_{P\in\mathcal{Q}_+} b^{(1)}_{Q,P} b^{(2)}_{P,R}\vec{t}_R,$$
which further implies that
$ b_{Q,R} = \sum_{P\in\mathcal{Q}_+} b^{(1)}_{Q,P}b^{(2)}_{P,R} $.
\end{itemize}
\end{remark}

\section{Decomposition Characterizations and Their Applications}\label{sec mole wave}

In this section, we focus on various decomposition characterizations of
the  matrix-weighted variable Besov space.
In Subsection \ref{sec mole},
we establish the molecular characterization of $B^{s(\cdot)}_{p(\cdot),q(\cdot)}(W)$
and, applying this, in Subsection \ref{sec wave}
we show the wavelet and the atomic characterizations
of $B^{s(\cdot)}_{p(\cdot),q(\cdot)}(W)$.
\subsection{Molecular Characterization}\label{sec mole}
In this subsection,
we establish the molecular characterization of the matrix-weighted variable Besov space.
First, we give some symbols.
For any $r\in\mathbb{R}$, let
\begin{align}\label{def int}
&\lfloor r \rfloor := \max\{k\in\mathbb{Z}:\ k\leq r\},\ \
\lfloor\!\lfloor r\rfloor\!\rfloor := \max\{k\in\mathbb{Z}:\ k<r\}, \nonumber \\
&\lceil r \rceil := \min\{k\in\mathbb{Z}:\ k\geq r\},\ \
\lceil\!\lceil r\rceil\!\rceil := \min\{k\in\mathbb{Z}:\ k > r\},
\end{align}
$r^\ast := r-\lfloor r\rfloor$, and $r^{\ast\ast} := r-\lfloor\!\lfloor r\rfloor\!\rfloor$.

Next, we recall the concept of molecules.
\begin{definition}\label{def mole}
Let $K,M\in [0,\infty)$ and $L,N\in\mathbb{R}$.
For any $K\in [0,\infty)$, $Q\in\mathcal{Q}_+$ with $l(Q) \leq 1$, and $x\in\mathbb{R}^n$,
let $ u_K(x) := ( 1+ |x| )^{-K}$ and
$$\left(u_K\right)_Q(x) := |Q|^{-\frac12} u_K\left( \frac{x-x_Q}{l(Q)} \right). $$
A function $m_Q\in \mathscr{M}$ is called a
\emph{ (smooth) $(K,L,M,N)$-molecule on a cube $Q$} if,
for any $x,y\in\mathbb{R}^n$
and any multi-index $\gamma\in\mathbb{Z}_+^n$,
\begin{itemize}
\item[{\rm (i)}]$ | m_Q(x) | \leq (u_K)_Q(x),$
\item[{\rm (ii)}]$\int_{\mathbb{R}^n} x^\gamma m_Q(x)\,dx = 0 \ \text{when}\ |\gamma|\leq L\ \text{and}\ l(Q) < 1,$
\item[{\rm (iii)}] $ | \partial^\gamma m_Q(x) | \leq [ l(Q) ]^{-|\gamma|} (u_M)_Q(x) \ \text{when}\ |\gamma| < N, $
\item[{\rm (iv)}]$ | \partial^\gamma m_Q(x) - \partial^\gamma m_Q(y) |
\leq [ l(Q) ]^{-|\gamma|} [ \frac{|x-y|}{l(Q)} ]^{N^{\ast\ast}}
\sup\limits_{|z|\leq |x-y|} (u_M)_Q(x+z)\ \
\text{when}\ \  |\gamma| = \lfloor\!\lfloor N \rfloor\!\rfloor$.
\end{itemize}
\end{definition}

The following is the relationship between molecules and almost diagonal operators.
\begin{theorem}\label{mole abo}
Let $p(\cdot), q(\cdot)\in \mathcal{P}_0\cap LH$,
$s(\cdot)\in LH$, and $W\in \mathscr{A}_{p(\cdot),\infty}$.
Let $\{m_Q\}_{Q\in\mathcal{Q}_+}$ be a family of $(K_m,L_m,M_m,N_m)$-molecules
and let $\{b_Q\}_{Q\in\mathcal{Q}_+}$ be a family of $(K_b,L_b,M_b,N_b)$-molecules
with $K_m$, $L_m$, $M_m$, $N_m$, $K_b$, $L_b$, $M_b$, and $N_b$ being the same as in Definition \ref{def mole}.
Then the infinite matrix $\{\langle m_Q,b_Q \rangle\}_{Q\in\mathcal{Q}_+}$ is
a $b^{s(\cdot)}_{p(\cdot),q(\cdot)}(W)$-almost diagonal operator if
\begin{align}\label{ana index}
K_m > \left( n + s_+ \right)\vee \left[ \frac{n}{\alpha_W} + C(s,q) \right],
\  L_m \geq s_+,\  M_m > \frac{n}{\alpha_W} + C(s,q),\  N_m >\frac{n}{\alpha_W} - n -s_-
\end{align}
and
\begin{align}\label{syn index}
K_b > \left( \frac{n}{\alpha_W} - s_- \right)\vee \left[ \frac{n}{\alpha_W} + C(s,q) \right],\
L_b \geq \frac{n}{\alpha_W} - n -s_-,\ M_b > \frac{n}{\alpha_W} + C(s,q),\
N_b > s_+,
\end{align}
where $\alpha_W$ is as in \eqref{def alphaW} and $C(s,q)$ as in \eqref{abo bound con}.
\end{theorem}
\begin{remark}
If $p(\cdot)$, $q(\cdot)$, and $s(\cdot)$ are constant exponents
and $W$ is an $\mathscr{A}_{p}$ matrix weight,
then the ranges of $K_m$, $L_m$, $M_m$, $N_m$, $K_b$, $L_b$, $M_b$, and $N_b$
of Theorem \ref{mole abo} coincide with the corresponding ranges of
\cite[Theorem 3.8]{bhyy23 3} in the case $\tau := 0$.
\end{remark}

To prove this theorem,
we need the following property of molecules, which is from \cite[Lemma 5.2]{bhyy24}.

\begin{lemma}\label{mole abo 1}
Let $m_Q$ be a $(K_m, L_m, M_m, N_m)$-molecule on cube $Q$
and let $b_p$ be a $(K_b, L_b, M_b, N_b)$-molecule on cube $P$,
where $K_m,M_m,K_b,M_b \in (n,\infty)$ and $L_m$, $N_m$, $L_b$, and $N_b$ are real numbers.
Then, for any $\alpha\in (0,\infty)$, there exists a positive constant $C$
such that $|\langle m_Q,b_P\rangle| \leq C b^{MGH}_{Q,P},$
where $b^{MGH}_{Q,P}$ is the same as in \eqref{def bDEF}
with $M := K_m \wedge M_m \wedge K_b \wedge M_b \in (n,\infty)$,
$$ G := \frac n2 + \left[N_b \wedge \lceil\!\lceil L_m \rceil\!\rceil \wedge \left( K_m-n-\alpha \right) \right]^{(+)},
\ \ \text{and}\ \  H := \frac n2 + \left[N_m \wedge \lceil\!\lceil L_b \rceil\!\rceil \wedge \left( K_b-n-\alpha \right) \right]^{(+)}. $$
\end{lemma}
Now, we give the proof of Theorem \ref{mole abo}.

\begin{proof}[Proof of Theorem \ref{mole abo}]
It follows from Lemma \ref{mole abo 1} and Theorem \ref{abo bound 1}
that, to show the desired estimates of $\{\langle m_Q,b_P\rangle\}_{Q,P\in\mathcal{Q}_+}$,
it is sufficient to prove that
\begin{align*}
 M > \frac{n}{\alpha_W} + C(s,q),
\quad G> \frac{n}{2} + s_+,
\quad \text{and}\quad H >\frac{n}{\alpha_W} - \frac{n}{2} - s_- ,
\end{align*}
where $M,G$, and $H$ are the same as in Lemma \ref{mole abo 1}.
Applying this with Lemma \ref{mole abo 1},
we find that these conditions are equivalent to the following ones:
\begin{itemize}
\item [{\rm (i)}] $K_m \wedge M_m \wedge K_b \wedge M_b > \frac{n}{\alpha_W} + C(s,q)$,
\item [{\rm (ii)}] $[N_b \wedge \lceil\!\lceil L_m \rceil\!\rceil \wedge ( K_m-n-\alpha ) ]^{(+)}> s_+$,
\item [{\rm (iii)}] $[N_m \wedge \lceil\!\lceil L_b \rceil\!\rceil \wedge ( K_b-n-\alpha ) ]^{(+)} >\frac{n}{\alpha_W} - n- s_-$.
\end{itemize}
Thus, we only need to show that {\rm (i)}, {\rm (ii)}, and {\rm (iii)} hold
for any $K_m$, $L_m$, $M_m$, $N_m$, $K_b$, $L_b$, $M_b$, and $N_b$ satisfying \eqref{ana index} and \eqref{syn index}.

First, by \eqref{ana index} and \eqref{syn index},
we obtain $K_m,M_m,K_b,M_b \in (\frac{n}{\alpha_W} + C(s,q),\infty)$,
which further implies that {\rm (i)} holds.
Next, we prove {\rm (ii)} holds.
It follows immediately from \eqref{ana index} that
$K_m > n+s_+$ and $L_m \geq s_+$.
Since the arbitrariness of $\alpha$,
we infer that, via choosing a small enough $\alpha$,
we obtain $K_m - n - \alpha > s_+$.
Then, by the definition of $\lceil\!\lceil \cdot \rceil\!\rceil$,
we find that, for any $y\in\mathbb{R}$,
$\lceil\!\lceil y \rceil\!\rceil > y$,
which, together with condition $L_m \geq s_+$,
further implies that $\lceil\!\lceil L_m \rceil\!\rceil > s_+$.
Finally, it follows immediately from \eqref{syn index} that $N_b >s_+$.
Thus, summarizing above estimates about $K_m$, $L_m$, and $N_b$,
we conclude that $ N_b \wedge \lceil\!\lceil L_m \rceil\!\rceil \wedge ( K_m-n-\alpha ) > s_+ $
and hence {\rm (ii)} holds.

Finally, similarly to the above estimation about {\rm (ii)},
with $N_b$, $L_m$, and $K_m$ replaced, respectively, by $N_m$, $L_b$, and $K_b$,
we find that
$N_m \wedge \lceil\!\lceil L_b \rceil\!\rceil \wedge ( K_b-n-\alpha ) > \frac{n}{\alpha_W} - n- s_-$
and hence {\rm (ii)} holds.
This finishes the proof of Theorem \ref{mole abo}.
\end{proof}
Next, by using Theorem \ref{mole abo},
we introduce the concepts of synthesis molecules and analysis molecules
of $B^{s(\cdot)}_{p(\cdot),q(\cdot)}(W)$
(see \cite{bhyy23 3} for those molecules of matrix $A_{p}$ weighted Besov spaces
and \cite{bhyy24} for molecules of matrix $A_{p,\infty}$ weighted Besov spaces).
\begin{definition}
Let $p(\cdot), q(\cdot)\in \mathcal{P}_0\cap LH$, $s(\cdot)\in LH$,
and $W\in \mathscr{A}_{p(\cdot),\infty}$.
A $(K,L,M,N)$-molecule $m_Q$ is called a
\emph{$B^{s(\cdot)}_{p(\cdot),q(\cdot)}(W)$-analysis molecule} on $Q$
if $K$, $L$, $M$, and $N$ satisfy \eqref{ana index}.
Moreover, a $(K,L,M,N)$-molecule $m_Q$ is called a
\emph{$B^{s(\cdot)}_{p(\cdot),q(\cdot)}(W)$-synthesis molecule} on $Q$
if $K,L,M$, and $N$ satisfy \eqref{syn index}.
\end{definition}
Using Theorems \ref{mole abo} and \ref{abo bound 1},
we obtain the following results.
\begin{lemma}\label{mbpsi abo}
Let $p(\cdot), q(\cdot)\in \mathcal{P}_0\cap LH$, $s(\cdot)\in LH$,
and $W\in \mathscr{A}_{p(\cdot),\infty}$
and let $\{\varphi_j\}_{j\in\mathbb{Z}_+}$ and $\{\psi_j\}_{j\in\mathbb{Z}_+}$
be as in Definition \ref{phi pair} satisfying \eqref{varphi psi}.
Suppose that $\{\{m^{(i)}_Q\}_{Q\in\mathcal{Q}_+}\}_{i = 1}^2$
are families of $B^{s(\cdot)}_{p(\cdot),q(\cdot)}(W)$-analysis molecules
and $\{\{b^{(i)}_Q\}_{Q\in\mathcal{Q}_+}\}_{i = 1}^2$ are
families of $B^{s(\cdot)}_{p(\cdot),q(\cdot)}(W)$-synthesis molecules.
Then
\begin{itemize}
\item[{\rm (i)}] for any $i \in \{1,2\}$, the infinity matrices
$$\left\{ \left\langle m^{(i)}_P,b^{(i)}_Q \right\rangle\right\}_{P,Q\in\mathcal{Q}_+},
\ \ \left\{ \left\langle m^{(i)}_P,\psi_Q \right\rangle \right\}_{P,Q\in\mathcal{Q}_+},
\ \ \text{and}\ \  \left\{ \left\langle \varphi_P,b^{(i)}_Q \right\rangle \right\}_{P,Q\in\mathcal{Q}_+} $$
are $b^{s(\cdot)}_{p(\cdot),q(\cdot)}(W)$-almost diagonal operators,
where $\psi_Q$ for any $Q\in \mathcal{Q}_+$ and $\varphi_P$ for any $P\in \mathcal{Q}_+$ are the same as in \eqref{def varphiQ}.
\item[{\rm (ii)}] if $\vec{t} := \{\vec{t}_Q\}_{Q\in\mathcal{Q}_+} \in b^{s(\cdot)}_{p(\cdot),q(\cdot)}(W)$,
then
$\vec{s}_P := \sum_{Q,R\in\mathcal{Q}_+} \langle m^{(1)}_P, b^{(1)}_Q \rangle
\langle m^{(2)}_Q, b^{(2)}_R \rangle \vec{t}_R $
converges unconditionally for any $P\in\mathcal{Q}_+$
and, moreover, $\vec{s} := \{\vec{s}_P\}_{P\in\mathcal{Q}_+}$ satisfies
$\| \vec{s} \|_{b^{s(\cdot)}_{p(\cdot),q(\cdot)}(W)} \lesssim \| \vec{t} \|_{b^{s(\cdot)}_{p(\cdot),q(\cdot)}(W)}, $
where the implicit positive constant is independent of $\vec{t}$, $\{m^{(i)}_Q\}_{Q\in\mathcal{Q}_+}$,
and $\{b^{(i)}_Q\}_{Q\in\mathcal{Q}_+}$.
\end{itemize}
\end{lemma}
\begin{proof}
Notice that, for any pair $\{\varphi_j\}_{j\in\mathbb{Z}_+}$
and $\{\psi_j\}_{j\in\mathbb{Z}_+}$ satisfying \eqref{varphi psi},
$\{\varphi_R\}_{R\in\mathcal{Q}_+}$ (resp. $\{\psi_R\}_{R\in\mathcal{Q}_+}$)
is a family of $B^{s(\cdot)}_{p(\cdot),q(\cdot)}(W)$-synthesis
(resp. $B^{s(\cdot)}_{p(\cdot),q(\cdot)}(W)$-analysis) molecules
with harmless constant multiples.
Combining this with Theorem \ref{mole abo},
we conclude that matrices $\{ \langle m^{(i)}_P,b^{(i)}_Q \rangle\}_{P,Q\in\mathcal{Q}_+}$,
$\{ \langle m^{(i)}_P,\psi_Q \rangle \}_{P,Q\in\mathcal{Q}_+}, $
and $\{ \langle \varphi_P,b^{(i)}_Q \rangle \}_{P,Q\in\mathcal{Q}_+}$
with $i\in \{1,2\}$ are bounded almost diagonal operators on $b^{s(\cdot)}_{p(\cdot),q(\cdot)}(W)$,
which completes the proof of {\rm (i)}.

Next, we give the proof of {\rm(ii)}.
By Theorem \ref{mole abo}, we find that
$\{ \langle m^{(i)}_P,b^{(i)}_Q \rangle \}_{P,Q\in\mathcal{Q}_+}$ with $i \in \{1,2\}$
are $b^{s(\cdot)}_{p(\cdot),q(\cdot)}(W)$-almost diagonal.
Using this and Remark \ref{abo of abo}(ii),
we conclude that $B:= \{b_{P,R}\}_{P,R\in\mathcal{Q}_+}$ with
$$b_{P,Q} := \sum_{Q\in\mathcal{Q}_+} \left|\left\langle m^{(1)}_P,b^{(1)}_Q \right\rangle\right| \left|\left\langle m^{(2)}_Q,b^{(2)}_R \right\rangle\right| $$
is also a $b^{s(\cdot)}_{p(\cdot),q(\cdot)}(W)$-almost diagonal operator.
From this, the assumption $\vec{t} \in b^{s(\cdot)}_{p(\cdot),q(\cdot)}(W)$,
and Remark \ref{abo of abo}(i),
we infer that, for any $P\in \mathcal{Q}_+$,
$$ \left| \vec{s}_P \right| \leq \sum_{Q,R\in\mathcal{Q}_+} \left|\left\langle m^{(1)}_P, b^{(1)}_Q \right\rangle \right|
\left|\left\langle m^{(2)}_Q, b^{(2)}_R \right\rangle\right| \left|\vec{t}_R \right|
= \sum_{R\in\mathcal{Q}_+} b_{P,R} \left| \vec{t}_R \right| < \infty. $$
This finishes the proof of {\rm (ii)} and hence Lemma \ref{mbpsi abo}.
\end{proof}

The following lemma gives the definition of $\langle \vec{f},m_Q \rangle$
and guarantees that it is well-defined.
Its proof is similar to that of \cite[Lemma 3.16]{bhyy23 3}
with \cite[Corollary 3.15]{bhyy23 3} replaced by Lemma \ref{mbpsi abo};
we omit the details here.
\begin{lemma}\label{def fm}
Let $p(\cdot), q(\cdot)\in \mathcal{P}_0\cap LH$, $s(\cdot)\in LH$,
and $W\in \mathscr{A}_{p(\cdot),\infty}$.
If $\vec{f} \in B^{s(\cdot)}_{p(\cdot),q(\cdot)}(W)$
and $m_Q$ is a $B^{s(\cdot)}_{p(\cdot),q(\cdot)}(W)$-analysis molecule on cube $Q$,
then, for any pair of $\{\varphi_j\}_{j\in\mathbb{Z}_+}$ and $\{\psi_j\}_{j\in\mathbb{Z}_+}$ as in \eqref{varphi psi},
the pairing
\begin{align}\label{fm pair}
\left\langle \vec{f}, m_Q \right\rangle := \sum_{R\in\mathcal{Q}_+} \left\langle \vec{f}, \varphi_R \right\rangle \left\langle \psi_R, m_Q \right\rangle
\end{align}
is well-defined;
moreover, the series above converges absolutely and its value is independent of the choices of
$\{\varphi_R\}_{R\in\mathcal{Q}_+}$ and $\{\psi_R\}_{R\in\mathcal{Q}_+}$.
\end{lemma}

The following result is the molecular characterization of matrix-weighted variable Besov spaces
(see \cite[Theorem 4.7]{wgx24} for the corresponding characterization of
scalar-valued  weighted variable Besov spaces).

\begin{theorem}\label{mole com}
Let $p(\cdot), q(\cdot)\in \mathcal{P}_0\cap LH$, $s(\cdot)\in LH$,
and $W\in \mathscr{A}_{p(\cdot),\infty}$.
\begin{itemize}
\item[{\rm (i)}] If $\{m_Q\}_{Q\in\mathcal{Q}_+}$ is a
family of $B^{s(\cdot)}_{p(\cdot),q(\cdot)}(W)$-analysis molecules,
then, for any $\vec{f}\in B^{s(\cdot)}_{p(\cdot),q(\cdot)}(W)$,
$$ \left\| \left\{ \left\langle \vec{f},m_Q \right\rangle \right\}_{Q\in\mathcal{Q}_+} \right\|_{b^{s(\cdot)}_{p(\cdot),q(\cdot)}(W)}
\lesssim \| \vec{f} \|_{B^{s(\cdot)}_{p(\cdot),q(\cdot)}(W)}, $$
where the implicit positive constant is independent of $\vec{f}$.
\item[{\rm (ii)}] If $\{b_Q\}_{Q\in\mathcal{Q}_+}$ is a
family of $B^{s(\cdot)}_{p(\cdot),q(\cdot)}(W)$-synthesis molecules,
then, for any $\vec{t} \in b^{s(\cdot)}_{p(\cdot),q(\cdot)}(W)$,
$\vec{f} := \sum_{R\in\mathcal{Q}_+} b_R \vec{t}_R$ converges in $(\mathcal{S}')^m$ and
$$ \left\| \vec{f} \right\|_{B^{s(\cdot)}_{p(\cdot),q(\cdot)}(W)}
\lesssim \left\| \vec{t} \right\|_{b^{s(\cdot)}_{p(\cdot),q(\cdot)}(W)},  $$
where the implicit positive constant is independent of $\vec{t}$.
\end{itemize}
\end{theorem}
\begin{proof}
Let $\{\varphi_j\}_{j\in\mathbb{Z}_+}$ be as in Definition \ref{phi pair}
and $\{\psi_R\}_{R\in\mathcal{Q}_+}$ satisfy \eqref{varphi psi}.
By Lemma \eqref{def fm}, we obtain, for any cube $Q \in \mathcal{Q}_+$,
\begin{align}\label{fmQ}
\left\langle \vec{f}, m_Q \right\rangle = \sum_{R\in\mathcal{Q}_+} \left\langle \vec{f},\varphi_R \right\rangle \left\langle \psi_R, m_Q \right\rangle
= \sum_{R\in\mathcal{Q}_+} \left\langle \psi_R, m_Q \right\rangle \left( S_\varphi \vec{f} \right)_R.
\end{align}
Let $b_{R,Q} := \langle \psi_R, m_Q \rangle$ and $B := \{ b_{R,Q} \}_{Q,R\in\mathcal{Q}_+}$.
Then, from Lemma \ref{mbpsi abo}(i) with $m_P^{(1)}$ and $\psi_Q$ replaced by $m_Q$ and $\psi_R$,
it follows that $B$ is bounded on $b^{s(\cdot)}_{p(\cdot),q(\cdot)}(W)$.
Using this, \eqref{fmQ}, and Theorem \ref{phi bound},
we conclude that
$$ \left\| \left\{ \left\langle f,m_Q \right\rangle \right\}_{Q\in\mathcal{Q}_+} \right\|_{b^{s(\cdot)}_{p(\cdot),q(\cdot)}(W)}
= \left\| B\left( S_\varphi \vec{f} \right) \right\|_{b^{s(\cdot)}_{p(\cdot),q(\cdot)}(W)}
\lesssim \left\| S_\varphi \vec{f}  \right\|_{b^{s(\cdot)}_{p(\cdot),q(\cdot)}(W)}
\lesssim \left\| \vec{f} \right\|_{B^{s(\cdot)}_{p(\cdot),q(\cdot)}(W)}. $$
This finishes the proof of (i).

Now, we prove {\rm (ii)}. Observe that, for any $\phi\in\mathcal{S}$,
$\phi$ is a $B^{s(\cdot)}_{p(\cdot),q(\cdot)}(W)$-synthesis molecule
with harmless constant multiple.
Let $\{\varphi_j\}_{j\in\mathbb{Z}_+}$ be as in Definition \ref{phi pair}
and $\{\psi_R\}_{R\in\mathcal{Q}_+}$ satisfy \eqref{varphi psi}.
Then, to prove that $\sum_{R\in\mathcal{Q}_+} b_R \vec{t}_R$ converges in $(\mathcal{S}')^m$,
it is sufficient to show that \eqref{fm eq 1} converges absolutely.
Using Lemma \ref{def fm} with $\vec{f}$ and $m_Q$ replaced by $b_R$ and $\phi$, we obtain
\begin{align}\label{fm eq 1}
 \sum_{R\in\mathcal{Q}_+} \left\langle b_R, \phi \right\rangle \vec{t}_R
= \sum_{Q,R\in\mathcal{Q}_+} \left\langle \psi_Q, \phi \right\rangle
\left\langle b_R, \varphi_Q \right\rangle \vec{t}_R.
\end{align}
Observe that, by Remark \ref{abo of abo}{\rm (i)}, and Lemma \ref{mbpsi abo}{\rm (i)}
with $m_P^{(1)}$, $b^{(1)}_Q$, and $m^{(2)}_Q$
replaced, respectively, by $\phi\in \mathcal{S}$, $\varphi_Q$, and $\psi_Q$,
we find that $\{\langle \psi_Q, \phi \rangle\langle b_R, \varphi_Q \rangle\}_{R\in \mathcal{Q}_+}$
is a bounded almost diagonal operator on $b^{s(\cdot)}_{p(\cdot),q(\cdot)}(W)$.
Applying this with the assumption that $\vec{t}\in b^{s(\cdot)}_{p(\cdot),q(\cdot)}(W)$,
we conclude that \eqref{fm eq 1} converges absolutely
and hence $\vec{f} := \sum_{R\in\mathcal{Q}_+} b_R \vec{t}_R$ converges in $(\mathcal{S}')^m$.

For any $P,R\in \mathcal{Q}_+$, let
$b_{P,R} := \langle \psi_Q, \varphi_P \rangle \langle b_R, \varphi_Q \rangle$
and $B:= \{b_{P,R}\}_{P,R\in\mathcal{Q}_+}$.
By Lemma \ref{mbpsi abo}(i) with $m_P^{(1)}$, $\psi_Q$, $\varphi_P$, and $b^{(1)}_Q$
replaced, respectively, by $\varphi_P$, $\psi_Q$, $\varphi_Q$, and $b_R$
and by Remark \ref{abo of abo}(ii),
we find that $B$ is a $b^{s(\cdot)}_{p(\cdot),q(\cdot)}(W)$-almost diagonal operator
and
$$ \left( S_\varphi \vec{f} \right)_P = \left\langle \vec{f} , \varphi_P \right\rangle = \sum_{Q,R\in\mathcal{Q}_+} \left\langle \psi_Q, \varphi_P  \right\rangle
\left\langle b_R, \varphi_Q \right\rangle \vec{t}_R = \left(B\vec{t}\right)_P. $$
From this and Theorems \ref{phi bound} and \ref{abo bound 1},
we deduce that
\begin{align*}
\left\| \vec{f} \right\|_{B^{s(\cdot)}_{p(\cdot),q(\cdot)}(W)}
\lesssim \left\| S_\varphi \vec{f}  \right\|_{b^{s(\cdot)}_{p(\cdot),q(\cdot)}(W)}
= \left\| B\vec{t} \right\|_{b^{s(\cdot)}_{p(\cdot),q(\cdot)}(W)}
\lesssim \left\|\vec{t} \right\|_{b^{s(\cdot)}_{p(\cdot),q(\cdot)}(W)},
\end{align*}
which completes the proof of (ii) and hence Theorem \ref{mole com}.
\end{proof}

\subsection{Wavelet  and Atomic Characterizations}\label{sec wave}

We   begin with the concept of Daubechies wavelets
(see, for example, \cite{d88}).
In what follows, for any $\mathcal{N}\in \mathbb{N}$,
we use the \emph{symbol $C^N$} to denote the set of all $\mathcal{N}$ times continuously
differentiable functions on $\mathbb{R}^n$.
\begin{definition}
Let $\mathcal{N}\in \mathbb{N}$ and $\Lambda := \{0,1\}^n\setminus \{\mathbf{0}\}$
and let $\theta^{(\lambda)} \in C^{\mathcal{N}}$ for any $\lambda \in \{0,1\}^n$.
Then $\{\theta^{(\lambda)}\}_{\lambda\in\{0,1\}^n}$
are called the \emph{Daubechies wavelets of class $C^{\mathcal{N}}$}
if all $\{\theta^{(\lambda)}\}_{\lambda \in \{0,1\}^n}$
are real-valued functions with compact support and
$$\left\{\theta^{(\mathbf{0})}_{P}:\ P\in\mathcal{Q}_0\right\} \cup \left\{\theta^{(\lambda)}_{Q}:\ Q\in \mathcal{Q}_+,\ \text{and}\ \lambda\in \Lambda\right\}$$
is an orthonormal basis of $L^2$.
\end{definition}
The following wavelet basis was constructed by Daubechies
(see, for instance, \cite{d88} and \cite[Chapter 3.9]{m93}).
\begin{lemma}\label{ext wavelet}
Let $\Lambda := \{0,1\}^n\setminus \{\mathbf{0}\}$.
For any $\mathcal{N}\in\mathbb{N}$,
there exist functions $\{\theta^{(\lambda)}\}_{\lambda\in\{0,1\}^n} \subset C^{\mathcal{N}}$
having the following properties:
\begin{itemize}
\item[{\rm (i)}] there exists a positive constant $\gamma \in (1,\infty)$
such that, for any $\lambda \in \{0,1\}^n$, $\theta^{(\lambda)}$ supports in $\gamma Q(\mathbf{0},1)$;
\item[{\rm (ii)}] for any $\alpha\in\mathbb{Z}_+^n$ with $|\alpha|\leq\mathcal{N}$
and for any $\lambda \in \Lambda$,
$ \int_{\mathbb{R}^n} x^\alpha \theta^{(\lambda)} (x)\,dx = 0$;
\item[{\rm (iii)}] the system generated by $\{\theta^{(\lambda)}\}_{\lambda\in\{0,1\}^n}$,
namely $\{\theta^{(\mathbf{0})}_{P}:\ P\in\mathcal{Q}_0 \} \cup \{ \theta^{(\lambda)}_{Q}:\ Q\in \mathcal{Q}_+\ \text{and}\ \lambda\in \Lambda\}$,
is an orthonormal basis of $L^2$.
\end{itemize}
\end{lemma}

The following theorem is the Daubechies wavelet characterization of
matrix-weighted variable Besov spaces
(see \cite[Theorem 5.12]{wgx24} for the wavelet characterization of
scalar-valued weighted variable Besov spaces).

\begin{theorem}\label{dep wave 1}
Let $p(\cdot), q(\cdot)\in \mathcal{P}_0\cap LH$, $s(\cdot)\in LH$,
and $W\in \mathscr{A}_{p(\cdot),\infty}$
and let $\{\theta^{(\lambda)}\}_{\lambda\in\{0,1\}^n}$
be a class of $C^N$ Daubechies wavelets with
\begin{align}\label{con cn}
\mathcal{N} > \max\left\{ s_+ , \frac{n}{\alpha_W} - n - s_- \right\},
\end{align}
and let $\{\varphi_j\}_{j\in\mathbb{Z}_+}$ and $\{\psi_j\}_{j\in\mathbb{Z}_+}$
be the same as in \eqref{varphi psi}.
Then the following statements hold.
\begin{itemize}
\item[{\rm (i)}] For any $\vec{f} \in B^{s(\cdot)}_{p(\cdot),q(\cdot)}(W)$,
\begin{align}\label{dep wave}
\vec{f} = \sum_{P\in\mathcal{Q}_0} \left\langle \vec{f}, \theta^{(\mathbf{0})}_P \right\rangle \theta^{(\mathbf{0})}_P +
\sum_{\lambda \in \Lambda} \sum_{Q\in \mathcal{Q}_+} \left\langle \vec{f}, \theta^{(\lambda)}_Q \right\rangle \theta^{(\lambda)}_Q
\end{align}
in $(\mathcal{S}')^m$,
where $\langle \vec{f}, \theta^{(\mathbf{0})}_P \rangle$
and $\langle \vec{f}, \theta^{(\lambda)}_Q \rangle$ are defined as in
\eqref{fm pair} with $m_Q$ replaced, respectively, by $\theta^{(\mathbf{0})}_P$
and $\theta^{(\lambda)}_Q$, satisfying
$$
\left\| \vec{f} \right\|_{B^{s(\cdot)}_{p(\cdot),q(\cdot)}(W)}
\gtrsim \left\| \vec{f} \right\|_{B^{s(\cdot)}_{p(\cdot),q(\cdot)}(W)_w}
:= \left\| \left\{\left\langle \vec{f}, \theta^{(\mathbf{0})}_P\right\rangle \right\}_{P\in \mathcal{Q}_0} \right\|_{b^{s(\cdot)}_{p(\cdot),q(\cdot)}(W)}
 + \sum_{\lambda\in \Lambda}\left\| \left\{\left\langle \vec{f}, \theta^{(\lambda)}_Q\right\rangle \right\}_{Q\in \mathcal{Q}_+} \right\|_{b^{s(\cdot)}_{p(\cdot),q(\cdot)}(W)}
$$
with the implicit positive constant independent of $\vec{f}$.

\item[{\rm (ii)}] Conversely, for any $\vec{t} := \{\vec{t}^{(\mathbf{0})}_P\}_{P\in \mathcal{Q}_0} \cup (\bigcup_{\lambda \in \Lambda} \{\vec{t}^{(\lambda)}_Q\}_{Q\in \mathcal{Q}_+})$
satisfying
$$ \left\| \left\{\vec{t}^{(\mathbf{0})}_P \right\}_{P\in \mathcal{Q}_0} \right\|_{b^{s(\cdot)}_{p(\cdot),q(\cdot)}(W)}
 + \sum_{\lambda\in \Lambda}\left\| \left\{\vec{t}^{(\lambda)}_Q \right\}_{Q\in \mathcal{Q}_+} \right\|_{b^{s(\cdot)}_{p(\cdot),q(\cdot)}(W)} < \infty, $$
$\vec{f} := \sum_{P\in \mathcal{Q}_0} \vec{t}^{(\mathbf{0})}_P \theta^{(\mathbf{0})}_P
+ \sum_{\lambda \in \Lambda} \sum_{Q\in \mathcal{Q}_+} \vec{t}^{(\lambda)}_Q \theta^{(\mathbf{\lambda})}_Q$
converges in $(\mathcal{S}')^m$ and $\vec{f} \in B^{s(\cdot)}_{p(\cdot),q(\cdot)}(W)$
with
\begin{align*}
\left\| \vec{f} \right\|_{B^{s(\cdot)}_{p(\cdot),q(\cdot)}(W)}
\lesssim \left\| \left\{\vec{t}^{(\mathbf{0})}_P \right\}_{P\in \mathcal{Q}_0} \right\|_{b^{s(\cdot)}_{p(\cdot),q(\cdot)}(W)}
 + \sum_{\lambda\in \Lambda}\left\| \left\{\vec{t}^{(\lambda)}_Q \right\}_{Q\in \mathcal{Q}_+} \right\|_{b^{s(\cdot)}_{p(\cdot),q(\cdot)}(W)},
\end{align*}
where the implicit positive constant is independent of $\vec{t}$.
\end{itemize}
\end{theorem}
\begin{remark}
When $p(\cdot)$, $q(\cdot)$, and $s(\cdot)$ all are constant exponents
and $W\in \mathscr{A}_{p}$, the range of $\mathcal N$ in
Theorem \ref{dep wave 1} coincides with the range in \cite[Theorem 4.10]{bhyy23 3}
in the case $\tau =0$.
This result about the wavelet characterization is new
even when $W$ is a scalar variable weight.
\end{remark}

The following is the relationship between molecules and wavelets.

\begin{lemma}\label{wave mole}
Let $p(\cdot), q(\cdot)\in \mathcal{P}_0\cap LH$, $s(\cdot)\in LH$,
and $W\in \mathscr{A}_{p(\cdot),\infty}$
and let $\mathcal{N}\in\mathbb{N}$ and
$\{\theta^{(\lambda)}\}_{\lambda\in\{0,1\}^n}$
be a class of $C^{\mathcal{N}}$ Daubechies wavelets.
If $\mathcal{N}$ satisfies \eqref{con cn},
then both $\{\theta^{(\mathbf{0})}_{P}\}_{P\in\mathcal{Q}_0}$
and $\{ \theta^{(\lambda)}_{Q}\}_{Q\in \mathcal{Q}_+}$ for any $\lambda \in \Lambda$
are both a family of $B^{s(\cdot)}_{p(\cdot),q(\cdot)}(W)$-analysis molecules
and a family of $B^{s(\cdot)}_{p(\cdot),q(\cdot)}(W)$-synthesis molecules
with harmless constant multiples.
\end{lemma}
\begin{proof}
We first show that, for any $\lambda \in \{0,1\}^n$,
$\theta^{(\lambda)}$ is a $(K,L,M,N)$-molecule on $[0,1]^n$
with the index $K$, $L$, $M$, and $N$ satisfying both \eqref{syn index} and \eqref{ana index}.
Observe that, for any $\lambda \in \{0,1\}^n$, $\theta^{(\lambda)}$ has compact support.
Using this and the definition of molecules, we find that, for any $\lambda \in \{0,1\}^n$,
$\theta^{(\lambda)}$ satisfies (i) and (iii) of Definition \ref{def mole}
with any $K\in [0,\infty)$ and $M\in \mathbb{R}$.
Then, by Lemma \ref{ext wavelet}(ii),
we find that, if $L\in\mathbb{Z}_+$ with $L < \mathcal{N}$,
then, for any $\lambda\in \Lambda$ and $\alpha\in\mathbb{Z}^n$ with $|\alpha| \leq L$,
\begin{align}\label{wave mole eq 1}
\int_{\mathbb{R}^n} x^\alpha \theta^{(\lambda)}(x)\,dx = 0.
\end{align}
Moreover, using Lemma \ref{ext wavelet},
we obtain, for any $\lambda\in \{0,1\}^n$, $\theta^{(\lambda)} \in C^{\mathcal{N}}$
and hence $\theta^{(\lambda)}$ satisfies Definition \ref{def mole}(iv)
with any $N\in (0,\mathcal{N})$.
Thus, summarizing all the above discussions,
we conclude that, for any $K\in [0,\infty)$, $L\in \mathbb{Z}$, and $M,N\in\mathbb{R}$
with $\max\{L,N\} < \mathcal{N}$,
$\{\theta^{(\lambda)}\}_{\lambda \in \{0,1\}^n}$ is a family of $(K,L,M,N)$-molecules on $[0,1]^n$
with harmless constant multiples
and, moreover, for any $\lambda\in \Lambda$ and
$\alpha\in\mathbb{Z}^n$ with $|\alpha| \leq L$,
\eqref{wave mole eq 1} holds.
These, together with the definition of $\theta^{(\lambda)}_Q$,
further implies that, for any $Q\in \mathcal{Q}_+$ (resp. $P\in \mathcal{Q}_0$),
$\theta^{(\lambda)}_Q$ with $\lambda \in \Lambda$ (resp. $\theta^{(\mathbf{0})}_P$)
is a $(K,L,M,N)$-molecule on $Q$ (resp. $P$) and satisfies
both \eqref{ana index} and \eqref{syn index}.
This finishes the proof of Lemma \ref{wave mole}.
\end{proof}

Now, we give the proof of Theorem \ref{dep wave 1}.
\begin{proof}[Proof of Theorem \ref{dep wave 1}]
We first prove that {\rm (i)}.
By \eqref{varphi psi}, to show that \eqref{dep wave} converges in $(\mathcal{S}')^m$,
it is sufficient to prove that, for any $\phi\in\mathcal{S}$,
\begin{align*}
{\rm D} := \sum_{P\in\mathcal{Q}_0} \sum_{R\in \mathcal{Q}_+} \left\langle \vec{f}, \varphi_R \right\rangle
\left\langle \psi_R, \theta^{(\mathbf{0})}_P \right\rangle
\left\langle \theta^{(\mathbf{0})}_P, \phi \right\rangle +
\sum_{\lambda \in \Lambda} \sum_{Q\in \mathcal{Q}_+} \sum_{R\in \mathcal{Q}_+}
\left\langle \vec{f}, \varphi_R \right\rangle \left\langle \psi_R, \theta^{(\lambda)}_Q \right\rangle
\left\langle \theta^{(\lambda)}_Q, \phi \right\rangle
\end{align*}
converges absolutely.
Let $\vec{t} := \{\vec{t}_R\}_{R\in \mathcal{Q}_+}$
with $\vec{t}_R := \langle \vec{f}, \varphi_R \rangle$ for any $R\in \mathcal{Q}_+$.
From Theorem \ref{phi bound} and the assumption $\vec{f} \in B^{s(\cdot)}_{p(\cdot),q(\cdot)}(W)$,
we infer that $ \vec{t} \in b^{s(\cdot)}_{p(\cdot),q(\cdot)}(W)$.
Since $\phi\in \mathcal{S}$,
it is easy to see that $\phi$ is both analysis and synthesis molecules
with harmless constant multiples supported in some cube $Q_0 \in \mathcal{Q}_0$.
By this and Lemma \ref{mbpsi abo} with $m_P^{(1)}$, $b_Q^{(1)}$, and $\psi_Q$
replaced, respectively, by $\theta^{(\mathbf{0})}$ (or $\theta^{(\lambda)}$),
$\psi_R$, and $\phi$ and by Remark \ref{abo of abo}{\rm (ii)},
we conclude that $ \{ \langle \psi_R, \theta^{(\mathbf{0})}_P \rangle
\langle \theta^{(\mathbf{0})}_P, \phi \rangle \}_{R\in \mathcal{Q}_0}$
and $\{ \langle \psi_R, \theta^{(\lambda)}_Q \rangle
\langle \theta^{(\lambda)}_Q, \phi \rangle \}_{R\in \mathcal{Q}_+} $
with $\lambda \in \Lambda$
are $b^{s(\cdot)}_{p(\cdot),q(\cdot)}(W)$-almost diagonal operators.
This, together with Remark \ref{abo of abo}{\rm (i)},
further implies that {\rm D} converges absolutely.
Applying this with Lemma \ref{ext wavelet}(iii),
we find that
\begin{align*}
{\rm D} &= \sum_{R\in \mathcal{Q}_+}\left\langle \vec{f}, \varphi_R \right\rangle \left[ \sum_{P\in\mathcal{Q}_0}  \left\langle \psi_R, \theta^{(\mathbf{0})}_P \right\rangle \left\langle \theta^{(\mathbf{0})}_P, \phi \right\rangle +
\sum_{\lambda \in \Lambda} \sum_{Q\in \mathcal{Q}_+} \left\langle \psi_R, \theta^{(\lambda)}_Q \right\rangle \left\langle \theta^{(\lambda)}_Q, \phi \right\rangle \right]\\
& = \sum_{R\in \mathcal{Q}_+}\left\langle \vec{f}, \varphi_R \right\rangle \left\langle \psi_R, \phi \right\rangle
 = \left\langle \vec{f}, \phi \right\rangle.
\end{align*}
This finishes the proof of \eqref{dep wave}.

Next, by Lemma \ref{wave mole}, we find that
both $\{ \theta^{(\mathbf{0})}_{P}\}_{P\in \mathcal{Q}_0}$ and
$\{ \theta^{(\lambda)}_{Q}\}_{Q\in \mathcal{Q}_+}$ with any $\lambda \in \Lambda$
are both  families of $B^{s(\cdot)}_{p(\cdot),q(\cdot)}(W)$-analysis
and $B^{s(\cdot)}_{p(\cdot),q(\cdot)}(W)$-synthesis molecules
with harmless constants multiples.
Hence, using this and Theorem \ref{mole com},
we conclude that, for any $\lambda \in \Lambda$,
\begin{align*}
\left\| \left\{\left\langle \vec{f}, \theta^{(\lambda)}_Q\right\rangle \right\}_{Q\in \mathcal{Q}_+} \right\|_{b^{s(\cdot)}_{p(\cdot),q(\cdot)}(W)}
\lesssim \left\| \vec{f} \right\|_{B^{s(\cdot)}_{p(\cdot),q(\cdot)}(W)}
\ \ {\rm and}\ \
 \left\| \left\{\left\langle \vec{f}, \theta^{(\mathbf{0})}_P\right\rangle \right\}_{P\in \mathcal{Q}_0} \right\|_{b^{s(\cdot)}_{p(\cdot),q(\cdot)}(W)}
\lesssim \left\| \vec{f} \right\|_{B^{s(\cdot)}_{p(\cdot),q(\cdot)}(W)},
\end{align*}
which further implies that
$\| \vec{f} \|_{B^{s(\cdot)}_{p(\cdot),q(\cdot)}(W)_w}
\lesssim \| \vec{f} \|_{B^{s(\cdot)}_{p(\cdot),q(\cdot)}(W)}$.
This finishes the proof of {\rm (i)}.

Now, we show {\rm (ii)}.
Let $\vec{f}^{(0)} := \sum_{P\in \mathcal{Q}_0} \vec{t}^{(\mathbf{0})}_P \theta^{(\mathbf{0})}_P$
and, for any $\lambda \in \Lambda$,
$\vec{f}^{(\lambda)} := \sum_{Q\in \mathcal{Q}_+} \vec{t}^{(\lambda)}_Q \theta^{(\lambda)}_Q$
and hence $\vec{f} = \vec{f}^{(0)} + \sum_{\lambda\in \Lambda} \vec{f}^{(\lambda)}$.
Observe that $\{\theta^{(\mathbf{0})}_{P}\}_{P\in\mathcal{Q}_0}$ is a family of
$B^{s(\cdot)}_{p(\cdot),q(\cdot)}(W)$-synthesis molecules and, for any $\lambda \in \Lambda$,
$\{ \theta^{(\lambda)}_{Q}\}_{Q\in \mathcal{Q}_+}$ is also a family of
$B^{s(\cdot)}_{p(\cdot),q(\cdot)}(W)$-synthesis molecules.
Using these and Theorem \ref{mole com}(ii),
we conclude that $\vec{f}^{(\mathbf{0})}$ and $\vec{f}^{(\lambda)}$ with $\lambda \in \Lambda$
all converge in $(\mathcal{S}')^m$ and, moreover,
$\| \vec{f}^{(0)} \|_{B^{s(\cdot)}_{p(\cdot),q(\cdot)}(W)} \lesssim
\| \{\vec{t}^{(\mathbf{0})}_P\}_{P\in \mathcal{Q}_0} \|_{b^{s(\cdot)}_{p(\cdot),q(\cdot)}(W)}$
and, for any $\lambda \in \Lambda$,
$\| \vec{f}^{(\lambda)} \|_{B^{s(\cdot)}_{p(\cdot),q(\cdot)}(W)}
\lesssim \| \{ \vec{t}^{(\lambda)}_Q \}_{Q\in \mathcal{Q}_+} \|_{b^{s(\cdot)}_{p(\cdot),q(\cdot)}(W)}.$
From these, we infer that $\vec{f}$ converges in $(\mathcal{S}')^m$ and
\begin{align*}
\left\| \vec{f} \right\|_{B^{s(\cdot)}_{p(\cdot),q(\cdot)}(W)}
&\lesssim \left\| \vec{f}^{(0)} \right\|_{B^{s(\cdot)}_{p(\cdot),q(\cdot)}(W)}
+ \sum_{\lambda\in \Lambda}  \left\| \vec{f}^{(\lambda)} \right\|_{B^{s(\cdot)}_{p(\cdot),q(\cdot)}(W)}\\
&\lesssim \left\| \left\{\vec{t}^{(\mathbf{0})}_P \right\}_{P\in \mathcal{Q}_0} \right\|_{b^{s(\cdot)}_{p(\cdot),q(\cdot)}(W)}
+ \sum_{\lambda \in \Lambda} \left\| \left\{\vec{t}^{(\lambda)}_Q \right\}_{Q\in \mathcal{Q}_+} \right\|_{b^{s(\cdot)}_{p(\cdot),q(\cdot)}(W)},
\end{align*}
which completes the proof of {\rm (ii)} and hence Theorem \ref{dep wave 1}.
\end{proof}

Now, using the above obtained wavelet characterization,
we establish the atomic characterization of matrix-weighted variable Besov spaces.
We first recall the concept of $(r,L,N)$-atoms.
\begin{definition}\label{def atom}
Let $r\in (0,\infty)$ and $L, N\in \mathbb{R}$.
A function $a_Q$ is called an \emph{$(r, L, N)$-atom} on a cube $Q$
if
\begin{itemize}
\item[{\rm (i)}] ${\mathop\mathrm{\,supp\,}} a_Q \subset rQ$,
\item[{\rm (ii)}] $\int_{\mathbb{R}^n} x^\gamma a_Q(x)\,dx = 0$ if $l(Q) < 1$ and
$\gamma \in \mathbb{Z}_+^n$ with $|\gamma| \leq L$,
\item[{\rm (iii)}] $|D^\gamma a_Q(x)| \leq |Q|^{-\frac12 - \frac{|\gamma|}{n}}$ if
$\gamma \in \mathbb{Z}_+^n$ with $|\gamma|\leq N$.
\end{itemize}
\end{definition}
The following theorem is the atomic characterization of matrix-weighted variable Besov spaces
(see \cite[Corollary 4.8]{wgx24} for the atomic characterization of scalar weighted variable Besov spaces).
\begin{theorem}\label{dep atom}
Let $p(\cdot), q(\cdot)\in \mathcal{P}_0\cap LH$, $s(\cdot)\in LH$,
$W\in \mathscr{A}_{p(\cdot),\infty}$,
and $L, N\in\mathbb{R}$ with $ L \in (\frac{n}{\alpha_W} - n -s_-,\infty)$ and $N \in (s_+,\infty). $
Then there exists $r\in (0,\infty)$,
depending only on $L$ and $N$,
such that the following statements hold.
\begin{itemize}
\item[{\rm (i)}] For any $\vec{f}\in B^{s(\cdot)}_{p(\cdot),q(\cdot)}(W)$,
there exist $\vec{t} := \{\vec{t}_R\}_{R\in\mathcal{Q}_+} \in b^{s(\cdot)}_{p(\cdot),q(\cdot)}(W)$
and a family of $(r, L,N)$-atoms $\{a_Q\}_{Q\in\mathcal{Q}_+}$,
each on the cube indicated by its subscript,
such that $\vec{f} = \sum_{Q\in\mathcal{Q}_+} \vec{t}_Q a_Q$ in $(\mathcal{S}')^m$ and,
moreover, $ \| \vec{t} \|_{b^{s(\cdot)}_{p(\cdot),q(\cdot)}(W)} \lesssim \| \vec{f} \|_{B^{s(\cdot)}_{p(\cdot),q(\cdot)}(W)}, $
where the implicit positive constant is independent of $\vec{f}$.
\item[{\rm (ii)}] If $\{a_Q\}_{Q\in\mathcal{Q}_+}$ is a family of $(r, L, N)$-atoms,
then, for any $\vec{t} := \{\vec{t}_Q\}_{Q\in\mathcal{Q}_+} \in b^{s(\cdot)}_{p(\cdot),q(\cdot)}(W)$,
$\vec{f} := \sum_{Q\in\mathcal{Q}_+} \vec{t}_Q a_Q$ converges in $(\mathcal{S}')^m$
and, moreover,
$ \| \vec{f} \|_{B^{s(\cdot)}_{p(\cdot),q(\cdot)}(W)} \lesssim \| \vec{t} \|_{b^{s(\cdot)}_{p(\cdot),q(\cdot)}(W)}, $
where the implicit positive constant is independent of $\vec{t}$ and $\{a_Q\}_{Q\in\mathcal{Q}_+}$.
\end{itemize}
\end{theorem}
\begin{remark}
When $p(\cdot)$, $q(\cdot)$, and $s(\cdot)$ all are constant exponents
and $W$ is an $\mathscr{A}_p$ matrix weight,
Theorem \ref{dep atom} reduces to \cite[Theorem 4.13]{bhyy23 3}.
Moreover, Theorem \ref{dep atom} in the unweighted scalar-valued
case coincides with \cite[Theorem 3]{d12}
(see also, for example, \cite{yzy15}).
\end{remark}
Now, we give the proof of Theorem \ref{dep atom}.
\begin{proof}[Proof of Theorem \ref{dep atom}]
Notice that, by the definition of atoms, we find that, for any $L\in \mathbb{Z}$ and $N\in \mathbb{R}$,
an $(r, L, N)$-atom is a $(K,L,M,N)$-molecule with any $K,M\in \mathbb{R}$.
Using this, \eqref{syn index}, and the assumptions that
$ L \in (\frac{n}{\alpha_W} - n -s_-,\infty)$ and $N \in ( s_+,\infty) $,
we obtain $\{a_Q\}_{Q\in\mathcal{Q}_+}$ is a family of
$B^{s(\cdot)}_{p(\cdot),q(\cdot)}$-synthesis molecules,
which, combined with Theorem \ref{mole com}(ii),
further implies that {\rm (ii)} holds.

Next, we prove {\rm (i)}.
Let $\mathcal{N} \in \mathbb{Z}_+$ be such that $\mathcal{N} > \max\{L, N\}$.
Then, applying this with Theorem \ref{dep wave 1},
we find that there exists a class of $C^{\mathcal{N}}$ Daubechies wavelets
$\{\theta^{(\mathbf{0})}, \theta^{(\lambda)}:\ \lambda\in \Lambda\}$
such that
\begin{align}\label{dep atom eq 1}
\vec{f} = \sum_{P\in\mathcal{Q}_0} \left\langle \vec{f}, \theta^{(\mathbf{0})}_P \right\rangle \theta^{(\mathbf{0})}_P +
\sum_{\lambda \in \Lambda} \sum_{Q\in \mathcal{Q}_+} \left\langle \vec{f}, \theta^{(\lambda)}_Q \right\rangle \theta^{(\lambda)}_Q
\ \ \text{in}\ \ (\mathcal{S}')^m
\end{align}
and
\begin{align}\label{dep atom eq 2}
\left\| \vec{f} \right\|_{B^{s(\cdot)}_{p(\cdot),q(\cdot)}(W)} \sim \left\| \left\{\left\langle \vec{f}, \theta^{(\mathbf{0})}_P\right\rangle \right\}_{P\in \mathcal{Q}_0} \right\|_{b^{s(\cdot)}_{p(\cdot),q(\cdot)}(W)}
 + \sum_{\lambda\in \Lambda}\left\| \left\{\left\langle \vec{f}, \theta^{(\lambda)}_Q\right\rangle \right\}_{Q\in \mathcal{Q}_+} \right\|_{b^{s(\cdot)}_{p(\cdot),q(\cdot)}(W)}.
\end{align}
Observe that, for any $\lambda\in \{0,1\}^n$,
$\theta^{(\lambda)}$ has compact support.
Using this and the definition of $\theta^{(\lambda)}_Q$,
we find that there exists $r\in (0,\infty)$ such that,
for any $\lambda\in \Lambda$ and $Q\in \mathcal{Q}_+$
(resp. $P\in \mathcal{Q}_0$), ${\mathop\mathrm{\,supp\,}} \theta^{(\lambda)}_Q \subset rQ$
(resp. ${\mathop\mathrm{\,supp\,}} \theta^{(\mathbf{0})}_P \subset rP$).
This, together with Lemma \ref{ext wavelet}, further implies that,
for any $Q \in \mathcal{Q}_+$ (resp. $P\in \mathcal{Q}_0$),
$L\in \mathbb{Z}_+$, and $N\in \mathbb{R}$ satisfying $\max\{L,N\} < \mathcal{N}$,
$\theta^{(\lambda)}_Q$ with $\lambda \in \Lambda$ (resp. $\theta^{(\mathbf{0})}_P$)
is an $(r,L,N)$-atom on $Q$ (resp. $P$) with harmless constant multiple.
From these, \eqref{dep atom eq 1}, and \eqref{dep atom eq 2},
we infer that, to show {\rm (i)},
it is sufficient to rearrange a new suitable order of $\{\theta^{(\mathbf{0})}_{P}:\ P\in\mathcal{Q}_0 \}
\cup \{ \theta^{(\lambda)}_{Q}:\ Q\in \mathcal{Q}_+\ \text{and}\ \lambda\in \Lambda\}$
such that, for any $\theta \in \{\theta^{(\mathbf{0})}_{P}:\ P\in\mathcal{Q}_0 \}
\cup \{ \theta^{(\lambda)}_{Q}:\ Q\in \mathcal{Q}_+\ \text{and}\ \lambda\in \Lambda\}$,
there exists a unique $a_R$ with $R\in \mathcal{Q}_+$
satisfying $a_R = \theta^{(\lambda)}_{Q}$ and
$ {\mathop\mathrm{\,supp\,}} \theta \subset r'R$ for some given $r' \in (1,\infty)$.

Now, we first give a new order of $\{\theta^{(\mathbf{0})}_P\}_{P\in \mathcal{Q}_0}$.
For any $Q\in \mathcal{Q}_0$, let $a_Q := c_1 \theta^{(\mathbf{0})}_Q$
and $\vec{t}_Q := c_1^{-1}\langle \vec{f}, \theta^{(\mathbf{0})}_Q \rangle$,
where $c_1$ is a harmless constant such that $\theta^{(\mathbf{0})}_Q$ is an $(r,L,N)$-atom on $Q$.
Then, we give a new order of $\bigcup_{\lambda\in \Lambda}\{\theta^{(\lambda)}_Q\}_{Q\in \mathcal{Q}_+}$.
For any $Q \in \mathcal{Q}_+$,
let $\{Q_i\}_{i = 1}^{2^n}$ be the enumeration of the dyadic child cubes of $Q$.
Then, by the previous obtained result that $\theta^{(\lambda)}_Q$ is
an $(r,L,N)$-atom on $Q$ with harmless constant multiple
and by the fact that, for any $Q\in \mathcal{Q}_+$ and any dyadic child cube $Q'$ of $Q$,
$Q\subset 3Q'$, we conclude that, for any $\lambda \in \Lambda$,
there exist constants $c_2$ and $r_2$ such that
$c_2 \theta^{(\lambda)}_Q$ is an $(r_2,L,N)$-atom on $Q_i$.
Rearranging $\theta^{(\lambda)}$ with $\lambda\in \Lambda$
by $\theta^{(i)}$ with $i\in \{1,\dots, 2^n-1\}$,
then let
\begin{align*}
a_{Q_i} :=
\begin{cases}
\displaystyle c_2 \theta_Q^{(i)} &\text{if}\ i\in\{1,\dots, 2^n-1\}, \\
\displaystyle 0 &\text{if}\ i= 2^n
\end{cases}
\end{align*}
and
\begin{align*}
\vec{t}_{Q_i} :=
\begin{cases}
\displaystyle  c^{-1}_2 \left\langle \vec{f}, \theta_Q^{(i)}\right\rangle &\text{if}\  i\in\{1,\dots, 2^n-1\}, \\
\displaystyle \mathbf{0} &\text{if}\ i= 2^n.
\end{cases}
\end{align*}
By this and \eqref{dep atom eq 1}, we obtain immediately
$\vec{f} = \sum_{Q\in\mathcal{Q}_+} \vec{t}_Q a_Q$ in $(\mathcal{S}')^m$.
Moreover, since $\vec{t}:= \{\vec{t}_Q\}_{Q\in\mathcal{Q}_+}$ is exactly
$\{ \langle \vec{f}, \theta^{(\mathbf{0})}_P \rangle \}_{P\in\mathcal{Q}_0} \cup
\{ \langle \vec{f}, \theta^{(\lambda)}_Q\rangle:\ Q\in \mathcal{Q}_+\ \text{and}\  \lambda\in \Lambda \}$
with shifted by one level at most,
it follows from the definition of the norm of $b^{s(\cdot)}_{p(\cdot),q(\cdot)}(W)$
that this changes the norm at most by a positive constant $C$,
where $C$ is independent of $\vec{t}$.
This finishes the proof of {\rm (i)} and hence Theorem \ref{dep atom}.
\end{proof}

\section{Boundedness of Classical Operators}\label{sec app}
In this section, we study the boundedness of some classical operators
on  matrix-weighted variable Besov spaces.
In Subsection \ref{sec trace}, we prove the boundedness of trace operators
and then, in Subsection \ref{sec CZ},
we show the boundedness of Calder\'on--Zygmund operators.
\subsection{Trace Operators}\label{sec trace}
In this subsection, we establish the trace theorem of matrix-weighted Besov spaces.
Since the trace operators map the factor from $\mathbb{R}^n$ to $\mathbb{R}^{n-1}$,
to avoid the confusion,
we keep the underlying space symbols $\mathbb{R}^n$ and $\mathbb{R}^{n-1}$ in this subsection.

We first recall some basic symbols.
For any $x\in\mathbb{R}^n$, let $x := (x', x_n)$,
where $x'\in \mathbb{R}^{n-1}$ and $x_n \in\mathbb{R}$.
We also denote $\lambda \in \{0,1\}^n$ by
$\lambda = (\lambda',\lambda_n)$ with $ \lambda'\in \{0,1\}^{n-1} $
and $\lambda_n \in \{0,1\}$.
In what follows,
for any variable exponent $p(\cdot)\in \mathcal{P}_0(\mathbb{R}^{n})$ and any $x'\in \mathbb{R}^{n-1}$,
let
\begin{align}\label{def wp}
\widetilde{p}(x') := p(x',0).
\end{align}
Let $\mathbf{0}_n$ (resp. $\mathbf{0}_{n-1}$)
be the origin of $\mathbb{R}^n$ (resp. $\mathbb{R}^{n-1}$)
and let $\Lambda_n := \{0,1\}^n\setminus \{\mathbf{0}_n\}$
and $\Lambda_{n-1} := \{0,1\}^{n-1}\setminus \{\mathbf{0}_{n-1}\}$.
To recall the concept of trace operators,
we first recall some properties of Daubechies wavelets
(see, for example, \cite{d88}).
\begin{lemma}\label{mul wave}
For any $\mathcal{N}\in\mathbb{N}$,
there exist two real-valued $C^{\mathcal{N}}(\mathbb{R})$ functions $\varphi$ and $\psi$
with compact support such that, for any $n\in\mathbb{N}$,
$\{ \theta^{(\mathbf{0}_n)}_Q\}_{Q\in \mathcal{Q}_0} \cup
\{ \theta_Q^{(\lambda)}:\ Q\in\mathcal{Q}_+(\mathbb{R}^n)\ \text{and}\ \lambda \in \Lambda_n\} $
form an orthonormal basis of $L^2(\mathbb{R}^n)$,
where, for any $\lambda:= (\lambda_1,\dots,\lambda_n)\in \{0,1\}^n$
and $x := (x_1,\dots, x_n)\in\mathbb{R}^n$,
\begin{align}\label{def theta}
\theta^{(\lambda)}(x):=  \prod_{i = 1}^n \phi^{(\lambda_i)}(x_i)
\end{align}
with $\phi^{(0)}:= \varphi$ and $\phi^{(1)}:= \psi$.
\end{lemma}
\begin{remark}\label{wave k0}
By \cite[Remark 5.2]{bhyy23 3},
we find that, for $\varphi$ as in Lemma \ref{mul wave},
there exists $k_0\in \mathbb{Z}$ such that $\varphi(-k_0) \neq 0$.
\end{remark}
For any $I\in \mathcal{Q}_+({\mathbb{R}}^{n-1})$ and $k\in\mathbb{Z}$,
let $Q(I,k) := I \times [l(I)k, l(I)(k+1)).$
By the construction of $Q(I,k)$,
it is easy to find that,
for any cube $Q\in \mathcal{Q}_+(\mathbb{R}^n)$,
there exist a unique $I\in \mathcal{Q}_+({\mathbb{R}}^{n-1})$
and a unique $k\in\mathbb{Z}$ such that $Q = Q(I,k)$
and we denote such $I$ by $I(Q)$.
Let $p(\cdot),q(\cdot) \in \mathcal{P}_0(\mathbb{R}^n)\cap LH(\mathbb{R}^n)$,
$s(\cdot)\in LH(\mathbb{R}^n)$, $W\in \mathscr{A}_{p(\cdot),\infty}(\mathbb{R}^n)$,
and $V\in \mathscr{A}_{\widetilde{p}(\cdot),\infty}({\mathbb{R}}^{n-1})$.
In what follows, let $\mathbb{A}(W) := \{A_{Q,W}\}_{Q\in \mathcal{Q}_+(\mathbb{R}^n)}$
[resp. $\mathbb{A}(V) := \{A_{I,W}\}_{I\in \mathcal{Q}_+(\mathbb{R}^{n-1})}$]
be the reducing operators of order $p(\cdot)$ for $W$
(resp. $\widetilde{p}(\cdot)$ for $V$).
Assume that $\mathcal{N}$ is large enough such that \eqref{con cn} holds
for $B^{s(\cdot)}_{p(\cdot),q(\cdot)}(W)$ and
$B^{\widetilde{s}(\cdot)-\frac{1}{\widetilde{p}(\cdot)}}_{\widetilde{p}(\cdot),\widetilde{q}(\cdot)}(V,\mathbb{R}^{n-1})$.
Then, by Theorem \ref{dep wave 1} and Lemma \ref{mul wave},
we find that there exist real-valued functions
$\varphi,\psi \in C^{\mathcal{N}}(\mathbb{R})$
such that $\{\theta^{(\lambda)}\}_{\lambda\in \{0,1\}^n} \subset C^{\mathcal{N}}(\mathbb{R}^n)$
is the Daubechies wavelets of $B^{s(\cdot)}_{p(\cdot),q(\cdot)}(W)$,
where, for any $\lambda \in \{0,1\}^n$,
$\theta^{(\lambda)}$ is as in \eqref{def theta}.

We now introduce the trace operator by using Daubechies wavelets.
For any $\lambda := (\lambda', \lambda_n)\in \{0,1\}^n$ and
any cube $Q:= Q(I,k) \in \mathcal{Q}_+(\mathbb{R}^n)$
with $I\in\mathcal{Q}_+({\mathbb{R}}^{n-1})$ and $k\in\mathbb{Z}$
and for any $x'\in\mathbb{R}^{n-1}$,
let
\begin{align}\label{def tr 1}
\left[ {\mathop \mathrm{\,Tr\,}}\theta^{(\lambda)}_Q \right](x') := \theta^{(\lambda)}_Q(x',0)
= [l(Q)]^{-\frac12} \theta^{(\lambda')}_{I(Q)}(x') \phi^{(\lambda_n)}(-k),
\end{align}
where $\theta^{(\lambda')}$ is defined as in \eqref{def theta} with $n$ replaced by $n-1$.
Observe that it follows from Theorem \ref{dep wave 1} that,
for any $\vec{f}\in B^{s(\cdot)}_{p(\cdot),q(\cdot)}(W)$,
$$ \vec{f} = \sum_{Q \in\mathcal{Q}_0} \left\langle \vec{f}, \theta^{(\mathbf{0}_n)}_Q \right\rangle \theta^{(\mathbf{0}_n)}_Q +
\sum_{\lambda \in \Lambda_n } \sum_{Q\in \mathcal{Q}_+} \left\langle \vec{f}, \theta^{(\lambda)}_Q \right\rangle \theta^{(\lambda)}_Q
\ \ \text{in}\ \ [\mathcal{S}'(\mathbb{R}^n)]^m. $$
Using this, for any $\vec{f}\in B^{s(\cdot)}_{p(\cdot),q(\cdot)}(W)$,
we define ${\mathop \mathrm{\,Tr\,}} \vec{f}$ as
\begin{align}\label{def tr}
{\mathop \mathrm{\,Tr\,}} \vec{f} := \sum_{Q \in\mathcal{Q}_0} \left\langle \vec{f}, \theta^{(\mathbf{0}_n)}_Q \right\rangle {\mathop \mathrm{\,Tr\,}} \theta^{(\mathbf{0}_n)}_Q +
\sum_{\lambda \in \Lambda_n } \sum_{Q\in \mathcal{Q}_+} \left\langle \vec{f}, \theta^{(\lambda)}_Q \right\rangle {\mathop \mathrm{\,Tr\,}} \theta^{(\lambda)}_Q.
\end{align}
Since, for any $\lambda \in \{0,1\}^n$, $\theta^{(\lambda)}$ has compact support,
it follows that there exists $N\in\mathbb{N}$ such that,
for any $\lambda\in \{0,1\}^n$,
${\mathop\mathrm{\,supp\,}} \theta^{(\lambda)} \subset B(\mathbf{0}_n, N)$.
This, together with the definition of $\theta_Q^{(\lambda)}$,
further implies that, for any $I\in\mathcal{Q}_+(\mathbb{R}^{n-1})$ and
$k\in\mathbb{Z}$ with $|k|>N$ and for any $\lambda\in \{0,1\}^n$ and $x'\in\mathbb{R}^{n-1}$,
\begin{align*}
\theta^{(\lambda)}_{Q(I,k)}(x',0)
= \left[ l(I) \right]^{-\frac n2} \theta^{(\lambda)}\left( \frac{x'-x_I}{l(I)}, -k \right) = 0.
\end{align*}
Applying this with both \eqref{def tr 1} and \eqref{def tr},
we conclude that, for any $\vec{f}\in B^{s(\cdot)}_{p(\cdot),q(\cdot)}(W)$,
\begin{align*}
{\mathop \mathrm{\,Tr\,}} \vec{f} &= \sum_{k=-N}^{N} \sum_{I \in\mathcal{Q}_0(\mathbb{R}^{n-1})} \vec{t}^{(\mathbf{0})}_{Q(I,k)} \psi(-k) \theta^{(\mathbf{0}_{n-1})}_{I}
 +  \sum_{\lambda \in \Lambda_n } \sum_{k=-N}^{N} \sum_{I\in \mathcal{Q}_+(\mathbb{R}^{n-1})} \vec{t}^{(\lambda)}_{Q(I,k)} \phi^{(\lambda_n)}(-k) \theta^{(\lambda')}_{I}.
\end{align*}
Next, we introduce the extension operator.
For any function $g$ on $\mathbb{R}^{n-1}$, any function $h$ on $\mathbb{R}$,
and any $x := (x',x_n) \in\mathbb{R}^n$, let $g\otimes h(x) := g(x')h(x_n)$.
Then, for any $\lambda' \in \Lambda_{n-1}$, $I\in\mathcal{Q}_+(\mathbb{R}^{n-1})$,
and $x:= (x',x_n)\in\mathbb{R}^n$, let
\begin{align}\label{def ext 1}
\left[ {\mathop \mathrm{\,Ext\,}} \theta_I^{(\lambda')} \right](x)
&:= \frac{ [l(Q)]^\frac12 }{\varphi(-k_0)} \left[ \theta^{(\lambda')} \otimes \varphi \right]_{Q(I,k_0)}(x)
 = \frac{ [l(Q)]^\frac12 }{\varphi(-k_0)} \theta^{(\lambda',0)}_{Q(I,k_0)}(x)\nonumber\\
& = \frac{ 1 }{\varphi(-k_0)} \theta^{(\lambda')}_{I}(x') \varphi\left( \frac{x_n}{l(I)} - k_0 \right),
\end{align}
where $\varphi$ is the same as in Lemma \ref{mul wave} and $k_0$ is the same as in Remark \ref{wave k0}.
By Lemma \ref{dep wave 1},
we find that, for any $\vec{g} \in B^{\widetilde{s}(\cdot)-\frac{1}{\widetilde{p}(\cdot)}}_{\widetilde{p}(\cdot),\widetilde{q}(\cdot)}(V,\mathbb{R}^{n-1})$,
\begin{align*}
\vec{g} = \sum_{I \in\mathcal{Q}_0(\mathbb{R}^{n-1})} \left\langle \vec{g}, \theta^{(\mathbf{0}_{n-1})}_I \right\rangle \theta^{(\mathbf{0}_{n-1})}_I +
\sum_{\lambda' \in \Lambda_{n-1} } \sum_{I\in \mathcal{Q}_+(\mathbb{R}^{n-1})} \left\langle \vec{g}, \theta^{(\lambda')}_I \right\rangle \theta^{(\lambda')}_I
\end{align*}
in $[\mathcal{S}'(\mathbb{R}^{n-1})]^m$.
Then, define the extension operator ${\mathop \mathrm{\,Ext\,}}$ by setting, for any
$\vec{g} \in B^{\widetilde{s}(\cdot)-\frac{1}{\widetilde{p}(\cdot)}}_{\widetilde{p}(\cdot),\widetilde{q}(\cdot)}(V,\mathbb{R}^{n-1})$,
\begin{align}\label{def ext}
{\mathop \mathrm{\,Ext\,}} \vec{g} := \sum_{I \in\mathcal{Q}_0(\mathbb{R}^{n-1})} \left\langle \vec{g}, \theta^{(\mathbf{0}_{n-1})}_I \right\rangle {\mathop \mathrm{\,Ext\,}} \theta^{(\mathbf{0}_{n-1})}_I +
\sum_{\lambda' \in \Lambda_{n-1} } \sum_{I\in \mathcal{Q}_+(\mathbb{R}^{n-1})} \left\langle \vec{g}, \theta^{(\lambda')}_I \right\rangle {\mathop \mathrm{\,Ext\,}} \theta^{(\lambda')}_I.
\end{align}

As the main result of this subsection, we have the following trace theorem
(see \cite[Theorem 6.1]{yzy15} for the trace theorem on scalar variable Besov-type spaces).

\begin{theorem}\label{trace 1}
Let $p(\cdot), q(\cdot)\in \mathcal{P}_0(\mathbb{R}^n)\cap LH(\mathbb{R}^n)$ and $s(\cdot)\in LH(\mathbb{R}^n)$
and let $W\in \mathscr{A}_{p(\cdot),\infty}(\mathbb{R}^n)$ and $V\in \mathscr{A}_{\widetilde{p}(\cdot),\infty}({\mathbb{R}}^{n-1})$
satisfy $(\widetilde{s}-\frac{1}{\widetilde{p}})_- > (n-1)(\frac{1}{\alpha_V} - 1)$,
where $\alpha_V$ is the same as in \eqref{def alphaW} with $W$ replaced by $V$.
Then the trace operator
$$ {\mathop \mathrm{\,Tr\,}}:\ B^{s(\cdot)}_{p(\cdot),q(\cdot)}(W,\mathbb{R}^n) \rightarrow B^{\widetilde{s}(\cdot)-\frac{1}{\widetilde{p}(\cdot)}}_{\widetilde{p}(\cdot),\widetilde{q}(\cdot)}(V,\mathbb{R}^{n-1}) $$
defined as in \eqref{def tr} is well-defined and bounded
if and only if, for any $I\in \mathcal{Q}_+(\mathbb{R}^{n-1})$ and $\vec{z}\in\mathbb{C}^m$,
\begin{align}\label{WV nec}
\frac{1}{\|\mathbf{1}_I\|_{L^{\widetilde{p}(\cdot)}(\mathbb{R}^{n-1})}} \left\|\, \left|V(\cdot) \vec{z}\right| \mathbf{1}_{I} \right\|_{L^{\widetilde{p}(\cdot)}(\mathbb{R}^{n-1})}
\lesssim \frac{1}{\|\mathbf{1}_{Q(I,0)}\|_{L^{p(\cdot)}(\mathbb{R}^n)}} \left\|\,\left| W(\cdot) \vec{z}\right| \mathbf{1}_{Q(I,0)}\right\|_{L^{p(\cdot)}(\mathbb{R}^n)},
\end{align}
where the implicit positive constant is independent of $I$ and $\vec{z}$.
\end{theorem}

\begin{remark}
When $p(\cdot)$, $q(\cdot)$, and $s(\cdot)$ are all constant exponents and $W\in \mathscr{A}_p$,
Theorem \ref{trace 1} coincides with \cite[Theorem 5.6]{bhyy23 3} in the case $\tau = 0$.
Moreover, Theorem \ref{trace 1} in the scalar-valued weighted case
is new, and   in the scalar-valued unweighted case it
coincides with \cite[Theorem 6.1]{yzy15} with $\phi \equiv 0$ therein.
\end{remark}

To prove Theorem \ref{trace 1},
we first give several necessary tools.
The following is the relationship between reducing operators of $V$ and $W$.
\begin{lemma}\label{WV k}
Let $p(\cdot), q(\cdot)\in \mathcal{P}_0(\mathbb{R}^n)\cap LH(\mathbb{R}^n)$, $s(\cdot)\in LH(\mathbb{R}^n)$,
$W\in \mathscr{A}_{p(\cdot),\infty}(\mathbb{R}^n)$, and $V\in \mathscr{A}_{\widetilde{p}(\cdot),\infty}({\mathbb{R}}^{n-1}) $.
If \eqref{WV nec} holds,
then, for any $I\in\mathcal{Q}_+(\mathbb{R}^{n-1})$, $k\in\mathbb{Z}$, and $\vec{z}\in\mathbb{C}^m$,
$$ \left| A_{I,V} \vec{z} \right| \lesssim \left( 1+ |k| \right)^{\Delta_W} \left| A_{Q(I,k),W} \vec{z} \right|,  $$
where $A_{I,V}$ (resp. $A_{Q(I,k),W}$) is the reducing operator of order $\widetilde{p}(\cdot)$ for $V$
(resp. $p(\cdot)$ for $W$), $\Delta_W$ is the same as in Lemma \ref{QP5},
and the implicit positive constant is independent of $I$, $k$, and $\vec{z}$.
\end{lemma}
\begin{proof}
For any $k\in\mathbb{Z}$, using the definition of $Q(I,k)$,
we obtain $l(Q(I,0)) = l(Q(I,k)) = l(I) $ and consequently,
for any $x\in Q(I,0)$ and $y\in Q(I,k)$,
$\frac{|x - y|}{l(I)} \lesssim |k|$.
From this and Lemma \ref{QP5} with $ Q := Q(I,0)$, $R := Q(I,k)$, and $\Delta := \Delta_W$,
we infer that
\begin{align*}
\left\|A_{Q(I,0),W}A_{Q(I,k),W}^{-1}\right\|
&\lesssim \max\left\{ \left[ \frac{l(Q(I,k))}{l(Q(I,0))} \right]^{d_1},
\left[ \frac{l(Q(I,0))}{l(Q(I,k))} \right]^{d_2} \right\}\left[ 1+ \frac{|x - y|}{l(Q(I,k))\vee l(Q(I,0))} \right]^{\Delta_W}\\
&\lesssim \left( 1 + |k| \right)^{\Delta_W} .
\end{align*}
Applying this, Definition \ref{def reducing operator}, and \eqref{WV nec},
we obtain, for any $\vec{z} \in \mathbb{C}^m$ and $k\in\mathbb{Z}$,
\begin{align*}
\left|A_{I,V} \vec{z}\right| &\sim \frac{1}{\|\mathbf{1}_I\|_{L^{\widetilde{p}(\cdot)}(\mathbb{R}^{n-1})}} \left\|\, \left|V(\cdot) \vec{z}\right| \mathbf{1}_{I} \right\|_{L^{\widetilde{p}(\cdot)}(\mathbb{R}^{n-1})}
\lesssim \frac{1}{\|\mathbf{1}_{Q(I,0)}\|_{L^{p(\cdot)}(\mathbb{R}^n)}} \left\|\,\left| W(\cdot) \vec{z}\right| \mathbf{1}_{Q(I,0)}\right\|_{L^{p(\cdot)}(\mathbb{R}^n)}\\
&\sim \left|A_{Q(I,0),W}\vec{z}\right|
\leq \left\|A_{Q(I,0),W}A_{Q(I,k),W}^{-1}\right\| \left|A_{Q(I,k),W}\vec{z}\right|
\lesssim \left( 1 + |k| \right)^{\Delta_W}\left|A_{Q(I,k),W}\vec{z}\right|,
\end{align*}
which completes the proof of Lemma \ref{WV k}.
\end{proof}
\begin{lemma}\label{lj s}
Let $s(\cdot) \in LH(\mathbb{R}^n)$.
Then, for any $j,l \in \mathbb{Z}_+$ and $x,y\in \mathbb{R}^n$ with $|x - y| \leq 2^{l-j}$,
\begin{align*}
2^{js(x)} \lesssim 2^{lC_{\rm log}(s)} 2^{js(y)},
\end{align*}
where the implicit positive constant is independent of $j$ and $l$.
\end{lemma}
\begin{proof}
By \eqref{clogp}, we find that, for any $x,y\in\mathbb{R}^n$ with $|x-y|\leq 2^{l-j}$,
\begin{align}\label{sclog}
\left| s(x)-s(y) \right| \leq C_{\rm log }(s) \frac{1}{\log(e + \frac{1}{|x-y|})} \leq C_{\rm{log}}(s).
\end{align}
If $j \leq l$, then, by \eqref{sclog}, we obtain
\begin{align*}
2^{js(y)} = 2^{js(x)} 2^{j[s(x)-s(y)]} \leq 2^{js(x)} 2^{j|s(x)-s(y)|}
\leq 2^{js(x)} 2^{jC_{\rm log}(s)} \leq 2^{lC_{\rm log}(s)} 2^{js(x)}.
\end{align*}
If $j > l$, then $2^{l-j} \leq 1$.
From this, \eqref{sclog}, and Lemma \ref{2js eq} with $p(\cdot) := s(\cdot)$, we deduce that
$$ 2^{js(y)} = 2^{j[s(x)- s(y)]} 2^{js(x)}
= 2^{l[s(x) - s(y)]} 2^{(j-l)[s(x) - s(y)]} 2^{js(x)}
\lesssim 2^{lC_{\rm log}(s)} 2^{js(x)}, $$
which completes the proof of Lemma \ref{lj s}.
\end{proof}

In what follows, for any $t := \{ t_I \}_{I\in\mathcal{Q}_+(\mathbb{R}^{n-1})}$
and $k\in\mathbb{Z}$,
let $t^{(k)} := \{ t^{(k)}_{Q} \}_{Q\in \mathcal{Q}_+(\mathbb{R}^n)}$,
where, for any cube $Q$ in $\mathbb{R}^n$,
\begin{align}\label{def tk}
t^{(k)}_{Q} :=
\begin{cases}
\displaystyle [l(I)]^\frac12 t_I &\ \text{if}\ Q = Q(I,k), \\
\displaystyle 0&\ \text{otherwise}.
\end{cases}
\end{align}
\begin{lemma}\label{tk}
Let $p_1(\cdot),p_2(\cdot),q_1(\cdot),q_2(\cdot) \in \mathcal{P}_0(\mathbb{R}^n)\cap LH(\mathbb{R}^n)$,
$s_1(\cdot),s_2(\cdot) \in LH(\mathbb{R}^n)$, and $N\in \mathbb{Z}_+$.
Assume that $p_1 = p_2$, $q_1=q_2$, and $s_1 = s_2$
on $\mathbb{R}^{n-1}\times [0,\infty)$ or $\mathbb{R}^{n-1}\times (-\infty,0]$.
Then, for any $k \in \{-N,\dots,N\}$ and
$t:= \{ t_I \}_{I \in\mathcal{Q}_+(\mathbb{R}^{n-1})}
\subset \mathbb{C}$,
\begin{align*}
\left\| t^{(k)}  \right\|_{b^{s_1(\cdot)}_{p_1(\cdot),q_1(\cdot)}(\mathbb{R}^{n})}
\sim \left\| t^{(k)}  \right\|_{b^{s_2(\cdot)}_{p_2(\cdot),q_2(\cdot)}(\mathbb{R}^{n})},
\end{align*}
where the implicit positive constant is independent of $t$ and $k$.
\end{lemma}
\begin{proof}
Without loss of generality,
we may assume that $p_1 = p_2$, $q_1 = q_2$, and $s_1 = s_2$
on $\mathbb{R}^{n-1}\times [0,\infty)$.
Then it is obvious that, for any $k \in \{0,\dots, N\}$,
\begin{align*}
\left\| t^{(k)}  \right\|_{b^{s_1(\cdot)}_{p_1(\cdot),q_1(\cdot)}(\mathbb{R}^{n})}
= \left\| t^{(k)}  \right\|_{b^{s_2(\cdot)}_{p_2(\cdot),q_2(\cdot)}(\mathbb{R}^{n})}.
\end{align*}

Next, fix $k\in \{-N,\dots,-1\}$.
Observe that, for any $j\in\mathbb{Z}_+$ and any cube $I\in \mathcal{Q}_+(\mathbb{R}^{n-1})$,
$Q(I,0) \subset (2N+1)Q(I,k)$ and $|Q(I,0)| = (2N+1)^{-n} |(2N+1)Q(I,k)|$.
Using this, Lemma \ref{Q EQ 1},
and the assumptions that $p_1 = p_2$, $q_1 = q_2$, and $s_1 = s_2$
on $\mathbb{R}^{n-1}\times [0,\infty)$,
we obtain
\begin{align*}
\left\| t^{(k)}  \right\|_{b^{s_1(\cdot)}_{p_1(\cdot),q_1(\cdot)}(\mathbb{R}^{n})}
&\sim \left\| \left\{ 2^{js_1(\cdot)} \sum_{I\in \mathcal{Q}_j(\mathbb{R}^{n-1})}  t_I \widetilde{\mathbf{1}}_{Q(I,0)} \right\}_{j\in\mathbb{Z}_+}  \right\|_{l^{q_1(\cdot)}(L^{p_1(\cdot)})(\mathbb{R}^{n})}\nonumber \\
&= \left\| \left\{ 2^{js_2(\cdot)} \sum_{I\in \mathcal{Q}_j(\mathbb{R}^{n-1})}  t_I \widetilde{\mathbf{1}}_{Q(I,0)} \right\}_{j\in\mathbb{Z}_+}  \right\|_{l^{q_2(\cdot)}(L^{p_2(\cdot)})(\mathbb{R}^{n})}
\end{align*}
and
\begin{align*}
\left\| t^{(k)}  \right\|_{b^{s_2(\cdot)}_{p_2(\cdot),q_2(\cdot)}(\mathbb{R}^{n})}
&\sim \left\| \left\{ 2^{js_2(\cdot)} \sum_{I\in \mathcal{Q}_j(\mathbb{R}^{n-1})}  t_I \widetilde{\mathbf{1}}_{Q(I,0)} \right\}_{j}  \right\|_{l^{q_2(\cdot)}(L^{p_2(\cdot)})(\mathbb{R}^{n})},
\end{align*}
which further implies that
$\| t^{(k)}  \|_{b^{s_1(\cdot)}_{p_1(\cdot),q_1(\cdot)}(\mathbb{R}^{n})}
\sim \| t^{(k)} \|_{b^{s_2(\cdot)}_{p_2(\cdot),q_2(\cdot)}(\mathbb{R}^{n})} $
and hence completes the proof of Lemma \ref{tk}.
\end{proof}

Repeating the proof of \cite[Corollary 4]{n16}
with \cite[Lemma 12]{n16} replaced by Lemma \ref{tk}
(see, for instance, \cite[Corollary 6.6]{yzy15} for a similar result),
we obtain the following result, the details being omitted here.

\begin{lemma}\label{wp p}
Let $p_1(\cdot),p_2(\cdot),q_1(\cdot),q_2(\cdot) \in \mathcal{P}_0(\mathbb{R}^n)\cap LH(\mathbb{R}^n)$,
$s_1(\cdot),s_2(\cdot) \in LH(\mathbb{R}^n)$, and $N\in \mathbb{Z}_+$.
Assume that $p_1 = p_2$, $q_1 = q_2$, and $s_1 = s_2$ on $\mathbb{R}^{n-1} \times \{0\}$.
Then, for any $t:= \{ t_I \}_{I \in\mathcal{Q}_+(\mathbb{R}^{n-1})}
\subset \mathbb{C}$ and $k\in\{-N,\dots,N\}$,
\begin{align*}
\left\| t^{(k)}  \right\|_{b^{s_1(\cdot)}_{p_1(\cdot),q_1(\cdot)}(\mathbb{R}^{n})}
\sim \left\| t^{(k)}  \right\|_{b^{s_2(\cdot)}_{p_2(\cdot),q_2(\cdot)}(\mathbb{R}^{n})},
\end{align*}
where the positive equivalence constants are independent of $t$ and $k$.
\end{lemma}
\begin{lemma}\label{wp p 1}
Let $p(\cdot),q(\cdot)\in \mathcal{P}_0(\mathbb{R}^n)\cap LH(\mathbb{R}^n)$,
$s(\cdot) \in LH(\mathbb{R}^n)$, $N\in \mathbb{Z}_+$.
Then, for any $k \in \{-N,\dots,N\}$ and
$t:= \{ t_I \}_{I \in\mathcal{Q}_+(\mathbb{R}^{n-1})}
\subset \mathbb{C}$,
\begin{align}\label{tr bound 8}
\left\| t  \right\|_{b^{\widetilde{s}(\cdot) - \frac{1}{\widetilde{p}(\cdot)}}_{\widetilde{p}(\cdot),\widetilde{q}(\cdot)}(\mathbb{R}^{n-1})}
\sim \left\| t^{(k)}  \right\|_{b^{s(\cdot)}_{p(\cdot),q(\cdot)}(\mathbb{R}^{n})},
\end{align}
where $\widetilde{p}(\cdot)$, $\widetilde{q}(\cdot)$, and $\widetilde{s}(\cdot)$
are the same as in \eqref{def wp} and
the positive equivalence constants are independent of $t$ and $k$.
\end{lemma}
\begin{proof}
For any $x:= (x',x_n) \in \mathbb{R}^n$,
let $\widehat{p}(x) := p(x',0)$,
$\widehat{q}(x) := q(x',0)$,
and $\widehat{s}(x) := s(x',0)$.
Then  $\widehat{p} = p$, $\widehat{q} = q$, and $\widehat{s} = s$
on  $\mathbb{R}^{n-1} \times \{0\}$.
Using this and Lemma \ref{wp p} yields that,
to show \eqref{tr bound 8},
we only need to prove that
\begin{align}\label{tr bound 1}
\left\| t  \right\|_{b^{\widetilde{s}(\cdot) - \frac{1}{\widetilde{p}(\cdot)}}_{\widetilde{p}(\cdot),\widetilde{q}(\cdot)}(\mathbb{R}^{n-1})}
\sim \left\| t^{(k)}  \right\|_{b^{\widehat{s}(\cdot)}_{\widehat{p}(\cdot),\widehat{q}(\cdot)}(\mathbb{R}^{n})}.
\end{align}

We first show that the left-hand side of \eqref{tr bound 1}
is not more than the right-hand one.
Similarly to the argument used in the proof of \eqref{Ajtj 12},
it is sufficient to prove that, for any $j\in\mathbb{Z}_+$,
\begin{align}\label{tr bound 11}
\left\| \delta_j^{-\frac{1}{\widetilde{q}(\cdot)}} 2^{j[\widetilde{s}(\cdot)-\frac{1}{\widetilde{p}(\cdot)}]} \sum_{I\in\mathcal{Q}_{j}(\mathbb{R}^{n-1})} |t_I| \widetilde{\mathbf{1}}_{I} \right\|_{L^{\widetilde{p}(\cdot)}(\mathbb{R}^{n-1})} \lesssim 1,
\end{align}
where
$$\delta_j := 2^{-j} + \left\| 2^{j\widehat{q}(\cdot)\widehat{s}(\cdot)} \sum_{I\in\mathcal{Q}_j(\mathbb{R}^{n-1})} \left|t_{Q(I,k)}^{(k)}\right|^{\widehat{q}(\cdot)} \left|Q(I,k)\right|^{-\frac{\widehat{q}(\cdot)}{2}} \mathbf{1}_{Q(I,k)} \right\|_{L^{\frac{\widehat{p}(\cdot)}{\widehat{q}(\cdot)}}(\mathbb{R}^n)}$$
and
$$  \sum_{j\in\mathbb{Z}_+} \left\| 2^{j\widehat{q}(\cdot)\widehat{s}(\cdot)} \sum_{I\in\mathcal{Q}_j(\mathbb{R}^{n-1})} \left|t_{Q(I,k)}^{(k)}\right|^{\widehat{q}(\cdot)} \left|Q(I,k)\right|^{-\frac{\widehat{q}(\cdot)}{2}} \mathbf{1}_{Q(I,k)} \right\|_{L^{\frac{\widehat{p}(\cdot)}{\widehat{q}(\cdot)}}(\mathbb{R}^n)} = 1. $$
To this end, by Lemma \ref{lj s},
we obtain, for any $j\in\mathbb{Z}_+$, $I \in \mathcal{Q}_j(\mathbb{R}^{n-1})$,
and $x \in Q(I,k)$,
$2^{j\widetilde{s}(x)} \sim 2^{j\widetilde{s}(I_0)}$.
From this, the disjointness of the dyadic cubes in $\mathcal{Q}_j(\mathbb{R}^{n-1})$,
and the definitions of $\widehat{p}$, $\widehat{q}$, and $\widehat{s}$,
we infer that
\begin{align}\label{tr bound 10}
&\rho_{L^{\widetilde{p}(\cdot)}(\mathbb{R}^{n-1})} \left( \delta_j^{-\frac{1}{\widetilde{q}(\cdot)}} 2^{j[\widetilde{s}(\cdot)-\frac{1}{\widetilde{p}(\cdot)}]} \sum_{I\in\mathcal{Q}_j(\mathbb{R}^{n-1})} |t_I|  \widetilde{\mathbf{1}}_{I} \right)\nonumber\\
&\quad = \int_{\mathbb{R}^{n-1}}  \delta_j^{-\frac{\widetilde{p}(x')}{\widetilde{q}(x')}} 2^{j[\widetilde{p}(x')\widetilde{s}(x')-1]} \sum_{I\in\mathcal{Q}_j(\mathbb{R}^{n-1})} |t_I|^{\widetilde{p}(x')} \left|I\right|^{-\frac{\widetilde{p}(x')}{2}} \mathbf{1}_{I}(x')\,dx'\nonumber\\
&\quad \sim \sum_{I\in\mathcal{Q}_j(\mathbb{R}^{n-1})} \int_{I}  \delta_j^{-\frac{\widetilde{p}(x')}{\widetilde{q}(x')}} 2^{j[\widetilde{p}(x')\widetilde{s}(x')-1]} \left|t_I\right|^{\widetilde{p}(x')} \left|I\right|^{-\frac{\widetilde{p}(x')}{2}}\,dx'\nonumber\\
&\quad = \sum_{I\in\mathcal{Q}_j(\mathbb{R}^{n-1})} \int_{kl(I)}^{(k+1)l(I)} \int_{I} \delta_j^{-\frac{\widetilde{p}(x')}{\widetilde{q}(x')}} 2^{j\widetilde{p}(x')\widetilde{s}(x')} \left|t_I\right|^{\widetilde{p}(x')} \left|I\right|^{-\frac{\widetilde{p}(x')}{2}} \,dx'\,dx_n\nonumber\\
&\quad = \sum_{I\in\mathcal{Q}_j(\mathbb{R}^{n-1})} \int_{Q(I,k)} \delta_j^{-\frac{\widehat{p}(x)}{\widehat{q}(x)}} 2^{j\widehat{p}(x)\widehat{s}(x)} \left|t_{Q(I,k)}^{k}\right|^{\widehat{p}(x)} \left|Q(I,k)\right|^{-\frac{\widehat{p}(x)}{2}}\,dx\nonumber\nonumber\\
&\quad = \rho_{L^{\widehat{p}(\cdot)}(\mathbb{R}^n)} \left(  \delta_j^{-\frac{1}{\widehat{q}(\cdot)}} 2^{j\widehat{s}(\cdot)} \sum_{I\in\mathcal{Q}_j(\mathbb{R}^{n-1})} \left|t_{Q(I,k)}^{(k)}\right| \widetilde{\mathbf{1}}_{Q(I,k)} \right).
\end{align}
Notice that, by the definition of $\delta_j$ and Lemma \ref{f pq}, we have
\begin{align*}
\left\| \delta_j^{-\frac{1}{\widehat{q}(\cdot)}} 2^{j\widehat{s}(\cdot)} \sum_{I\in\mathcal{Q}_j(\mathbb{R}^{n-1})} \left|t_{Q(I,k)}^{(k)}\right| \widetilde{\mathbf{1}}_{Q(I,k)} \right\|_{L^{\widehat{p}(\cdot)}(\mathbb{R}^n)} \leq 1,
\end{align*}
which, combined with Lemma \ref{f rho},
further implies that
\begin{align*}
\rho_{L^{\widehat{p}(\cdot)}(\mathbb{R}^n)} \left(  \delta_j^{-\frac{1}{\widehat{q}(\cdot)}} 2^{j\widehat{s}(\cdot)} \sum_{I\in\mathcal{Q}_j(\mathbb{R}^{n-1})} \left|t_{Q(I,k)}^{(k)}\right| \widetilde{\mathbf{1}}_{Q(I,k)} \right) \leq 1.
\end{align*}
Applying this with \eqref{tr bound 10} and Lemma \ref{f rho}
yields \eqref{tr bound 11},
which further implies that the left-hand side of \eqref{tr bound 1}
is not more than the right-hand one.

Next, we show that the right-hand side of \eqref{tr bound 1}
is not more than the left-hand one.
Similarly to the argument used in the proof of \eqref{Ajtj 12},
we only need to prove that, for any $j\in\mathbb{Z}_+$,
\begin{align}\label{tr bound 2}
\left\| \delta^{-\frac{1}{\widehat{q}(\cdot)}} 2^{j\widehat{s}(\cdot)} \sum_{I\in\mathcal{Q}_j(\mathbb{R}^{n-1})} \left|t_{Q(I,k)}^{(k)}\right| \left|Q(I,k)\right|^{-\frac12} \mathbf{1}_{Q(I,k)} \right\|_{L^{\widehat{p}(\cdot)}(\mathbb{R}^n)}
 \lesssim 1,
\end{align}
where
$$\delta_j := 2^{-j} +
\left\| 2^{j\widetilde{q}(\cdot)[\widetilde{s}(\cdot)-\frac{1}{\widetilde{p}(\cdot)}]} \sum_{I\in\mathcal{Q}_{j}(\mathbb{R}^{n-1})} |t_I|^{\widetilde{q}(\cdot)} \left|I\right|^{-\frac{\widetilde{q}(\cdot)}{2}} \mathbf{1}_{I} \right\|_{L^{\frac{\widetilde{p}(\cdot)}{\widetilde{q}(\cdot)}}(\mathbb{R}^{n-1})}$$
and
$$  \sum_{j\in\mathbb{Z}_+} \left\| 2^{j\widetilde{q}(\cdot)[\widetilde{s}(\cdot)-\frac{1}{\widetilde{p}(\cdot)}]} \sum_{I\in\mathcal{Q}_{j}(\mathbb{R}^{n-1})} |t_I|^{\widetilde{q}(\cdot)} \left|I\right|^{-\frac{\widetilde{q}(\cdot)}{2}} \mathbf{1}_{I} \right\|_{L^{\frac{\widetilde{p}(\cdot)}{\widetilde{q}(\cdot)}}(\mathbb{R}^{n-1})} = 1. $$
By the definitions of $t^{(k)}_Q$, $Q(I,k)$, $\widetilde{p}$, $\widetilde{q}$, and $\widetilde{s}$
and the disjointness of the dyadic cubes in $\mathcal{Q}_j(\mathbb{R}^{n-1})$,
we find that, for any $j\in \mathbb{Z}_+$,
\begin{align*}
&\rho_{L^{\widehat{p}(\cdot)}(\mathbb{R}^n)} \left( \delta^{-\frac{1}{\widehat{q}(\cdot)}} 2^{j\widehat{s}(\cdot)} \sum_{I\in\mathcal{Q}_j(\mathbb{R}^{n-1})} \left|t_{Q(I,k)}^{(k)}\right| \left|Q(I,k)\right|^{-\frac12} \mathbf{1}_{Q(I,k)} \right)\\
&\quad = \int_{\mathbb{R}^n} \delta^{-\frac{\widehat{p}(x)}{\widehat{q}(x)}} 2^{j\widehat{p}(x)\widehat{s}(x)} \sum_{I\in\mathcal{Q}_j(\mathbb{R}^{n-1})} \left|t_{Q(I,k)}^{(k)}\right|^{\widehat{p}(x)} \left|Q(I,k)\right|^{-\frac{\widehat{p}(x)}{2}} \mathbf{1}_{Q(I,k)}\,dx\\
&\quad = \sum_{I\in\mathcal{Q}_j(\mathbb{R}^{n-1})} \int_{Q(I,k)} \delta^{-\frac{\widehat{p}(x)}{\widehat{q}(x)}} 2^{j\widehat{p}(x)\widehat{s}(x)}  \left|t_I\right|^{\widehat{p}(x)} \left|l(I)\right|^{-(n-1)\frac{\widehat{p}(x)}{2}} \,dx\\
&\quad = \sum_{I\in\mathcal{Q}_j(\mathbb{R}^{n-1})} \int_{kl(I)}^{(k+1)l(I)} \int_{I} \delta^{-\frac{\widetilde{p}(x')}{\widetilde{q}(x')}} 2^{j\widetilde{p}(x')\widetilde{s}(x')}  \left|t_I\right|^{\widetilde{p}(x')} \left|I\right|^{-\frac{\widetilde{p'}(x)}{2}} \,dx'\,dx_n\\
&\quad =  \int_{\mathbb{R}^{n-1}} \delta^{-\frac{\widetilde{p}(x')}{\widetilde{q}(x')}} 2^{j\widetilde{p}(x')[\widetilde{s}(x')-\frac{1}{\widetilde{p}(x')}]} \sum_{I\in\mathcal{Q}_j(\mathbb{R}^{n-1})} \left|t_I\right|^{\widetilde{p}(x')} \left|l(I)\right|^{-(n-1)\frac{\widetilde{p'}(x)}{2}} \mathbf{1}_I \,dx'\\
&\quad = \rho_{L^{\widetilde{p}(\cdot)}(\mathbb{R}^{n-1})} \left( \delta_j^{-\frac{1}{\widetilde{q}(\cdot)}} 2^{j[\widetilde{s}(\cdot)-\frac{1}{\widetilde{p}(\cdot)}]} \sum_{I\in\mathcal{Q}_j(\mathbb{R}^{n-1})} |t_I|  \widetilde{\mathbf{1}}_{I} \right),
\end{align*}
which, together with the definition of $\delta_j$ and Lemma \ref{f pq},
further implies \eqref{tr bound 2} and hence \eqref{tr bound 8}.
This finishes the proof of Lemma \ref{wp p 1}.
\end{proof}

Now, we prove Theorem \ref{trace 1}.

\begin{proof}[Proof of Theorem \ref{trace 1}]
First, we show the necessity.
Suppose that the trace operator ${\mathop \mathrm{\,Tr\,}}$ is well-defined and bounded.
Then, for any fixed cube $I_0 \in \mathcal{Q}_+(\mathbb{R}^{n-1})$,
assume that $I_0\in \mathcal{Q}_{j_{I_0}}(\mathbb{R}^{n-1})$ and
$x_{I_0}'$ is the center of $I_0$,
and, for any $\vec{z} \in\mathbb{C}^m$,
let $\vec{t} := \{\vec{t}_I\}_{I\in\mathcal{Q}_+(\mathbb{R}^{n-1})}$
with
\begin{align*}
\vec{t}_{I} :=
\begin{cases}
\displaystyle \left[ l(I_0) \right]^{[\widetilde{s}(x_{I_0}')-\frac{1}{\widetilde{p}(x_{I_0}')}] - (n-1)(\frac{1}{\widetilde{p}(x_{I_0}')}-\frac12)} \vec{z} &\ \text{if}\ I= I_0, \\
\displaystyle \mathbf{0}_m &\ \text{otherwise}.
\end{cases}
\end{align*}
Now, for any given $\lambda' \in \Lambda_{n-1}$,
let $\vec{g} := \vec{t}_{I_0} \theta^{(\lambda')}_{I_0}$.
Then, from Theorems \ref{dep wave 1} and \ref{W aa 3}, \cite[Example 3.4]{ah10},
Lemma \ref{2js eq}, and the assumption that $\widetilde{s}(\cdot), \frac{1}{\widetilde{p}(\cdot)}\in LH(\mathbb{R}^{n-1})$,
we infer that
\begin{align}\label{WV k eq1}
\left\| \vec{g} \right\|_{B^{\widetilde{s}(\cdot)-\frac{1}{\widetilde{p}(\cdot)}}_{\widetilde{p}(\cdot),\widetilde{q}(\cdot)}(V,\mathbb{R}^{n-1})}
&\sim \left\| \vec{t} \right\|_{b^{\widetilde{s}(\cdot)-\frac{1}{\widetilde{p}(\cdot)}}_{\widetilde{p}(\cdot),\widetilde{q}(\cdot)}(\mathbb{A}(V),\mathbb{R}^{n-1})}
 = \left\| 2^{j[\widetilde{s}(\cdot)-\frac{1}{\widetilde{p}(\cdot)}]} |I_0|^{-\frac12} \left| A_{I_0,V} \vec{t}_{I_0} \right| \mathbf{1}_{I_0} \right\|_{L^{\widetilde{p}(\cdot)}(\mathbb{R}^{n-1})}\nonumber\\
& \sim 2^{j[\widetilde{s}(x_{I_0}') - \frac{1}{\widetilde{p}(x_{I_0}')}]} |I_0|^{-\frac12} \left| A_{I_0,V} \vec{t}_{I_0} \right| \left[l\left(I_0\right)\right]^{\frac{n-1}{\widetilde{p}(x_{I_0}')}}
\sim \left| A_{I_0,V} \vec{z} \right|.
\end{align}
Now, take $\vec{u}:= \{\vec{u}_Q\}_{Q\in\mathcal{Q}_+(\mathbb{R}^n)}$ with
\begin{align*}
\vec{u}_Q :=
\begin{cases}
\displaystyle  \left[ l(I_0) \right]^\frac12 \vec{t}_{I_0} &\ \text{if}\ Q = Q(I_0,k_0),\\
\displaystyle \mathbf{0}_m &\ \text{otherwise},
\end{cases}
\end{align*}
where $\mathbf{0}_m$ is the origin of $\mathbb{C}^m$,
and let, for any $x := (x',x_n)\in\mathbb{R}^n$,
$$ \vec{f}(x) := \vec{t}_{I_0}  \left[{\mathop \mathrm{\,Ext\,}}  \theta^{(\lambda')}_{I_0} \right](x)
 = \vec{t}_{I_0} \frac{ [l(Q)]^\frac12 }{\varphi(-k_0)} \left[ \theta^{(\lambda')} \otimes \varphi \right]_{Q(I_0,k_0)}(x), $$
where $\theta^{(\lambda')}$ and $\varphi$ are the same as in \eqref{def theta}.
Then, from this and \eqref{def tr},
it follows that, for any $x'\in\mathbb{R}^{n-1}$,
\begin{align}\label{WV k eq2}
\left( {\mathop \mathrm{\,Tr\,}} \vec{f} \right)(x') = \vec{t}_{I_0} \theta^{(\lambda')}_{I_0}(x') = g(x').
\end{align}
Observe that, for any $x:= (x',x_n) \in Q(I_0,k_0)$
with the fixed $k_0$ being the same as in Remark \ref{wave k0},
$s(x) \sim s(x_{I_0}',0) = \widetilde{s}(x_{I_0}')$
and $ \frac{1}{p(x)} \sim \frac{1}{p(x_{I_0}',0)} = \frac{1}{\widetilde{p}(x_{I_0}')} $,
where the equivalence positive constants depend on $k_0$.
Using this and Lemma \ref{est Q},
we find that
\begin{align}\label{WV eq 3}
\|\mathbf{1}_{Q(I_0,k_0)}\|_{L^{p(\cdot)}} \sim [l(I_0)]^{\frac{n}{p(x_{I_0}',0)}}
= [l(I_0)]^{\frac{n}{\widetilde{p}(x_{I_0}')}}.
\end{align}
Also, by Lemma \ref{mul wave}, we obtain
$ \{\theta^{(\lambda')} \otimes \psi, \theta^{(\lambda')} \otimes \varphi:\ \lambda'\in \Lambda_{n-1}\} $
is a subset of the Daubechies wavelets appearing in the wavelet characterization of $ B^{s(\cdot)}_{p(\cdot),q(\cdot)}(W,\mathbb{R}^n) $.
From this, \cite[Example 3.4]{ah10}, Theorems \ref{dep wave 1} and \ref{W aa 3},
and \eqref{WV eq 3},
we infer that
\begin{align*}
\left\| \vec{f} \right\|_{B^{s(\cdot)}_{p(\cdot),q(\cdot)}(W,\mathbb{R}^n)}
&\sim \left\| \vec{u} \right\|_{b^{s(\cdot)}_{p(\cdot),q(\cdot)}(\mathbb{A}(W),\mathbb{R}^n)}
 = \left\| 2^{j_{I_0}s(\cdot)} \left|Q(I_0,k_0)\right|^{-\frac12} \left| A_{Q(I_0,k_0),W} \vec{u}_{Q(I_0,k_0)} \right| \mathbf{1}_{Q(I_0,k_0)} \right\|_{L^{p(\cdot)}(\mathbb{R}^n)}\\
& \sim 2^{j_{I_0}\widetilde{s}(x_{I_0}')} \left|Q(I_0,k_0)\right|^{-\frac12} \left| A_{Q(I_0,k_0),W} \vec{u}_{Q(I_0,k_0)} \right| \left\| \mathbf{1}_{Q(I_0,k_0)} \right\|_{L^{p(\cdot)}(\mathbb{R}^n)}\\
&\sim \left[ l(I_0) \right]^{-\widetilde{s}(x_{I_0}') - \frac{n}{2} + \frac{n}{\widetilde{p}(x_{I_0}')} }\left| A_{Q(I_0,k_0),W} \vec{u}_{Q(I_0,k_0)} \right|
 \sim \left| A_{Q(I_0,k_0),W} \vec{z} \right|.
\end{align*}
Combining this with \eqref{WV k eq1}, \eqref{WV k eq2},
and the assumption that ${\mathop \mathrm{\,Tr\,}}$ is bounded,
we conclude that
\begin{align*}
\left| A_{I_0,V} \vec{z} \right|
\sim \left\| \vec{g} \right\|_{B^{\widetilde{s}(\cdot)-\frac{1}{\widetilde{p}(\cdot)}}_{\widetilde{p}(\cdot),\widetilde{q}(\cdot)}(V,\mathbb{R}^{n-1})}
\lesssim \left\| \vec{f} \right\|_{B^{s(\cdot)}_{p(\cdot),q(\cdot)}(W,\mathbb{R}^n)}
\sim \left| A_{Q(I_0,k_0),W} \vec{z} \right|,
\end{align*}
which, together with Lemma \ref{WV k}, further implies that
\begin{align*}
\left| A_{I_0,V} \vec{z} \right| \lesssim \left| A_{Q(I_0,k_0),W} \vec{z} \right|
\lesssim \left( 1+ |k_0| \right)^{\Delta_W} \left| A_{Q(I_0,0),W} \vec{z} \right|.
\end{align*}
Using this and Definition \ref{def reducing operator},
we obtain \eqref{WV nec}, which hence completes the proof of the necessity.

Next, we prove the sufficiency.
To this end, let $\vec{f} \in B^{s(\cdot)}_{p(\cdot),q(\cdot)}(W)$.
For any $\lambda \in \{0,1\}^n$
let $ \vec{u}^{(\lambda)} := \{ \vec{u}^{(\lambda)}_{Q} \}_{Q\in\mathcal{Q}_+(\mathbb{R}^n)}$,
where $ \vec{u}^{(\lambda)}_{Q} := \langle \vec{f}, \theta^{(\lambda)}_{Q} \rangle $
for any $Q\in\mathcal{Q}_+(\mathbb{R}^n)$
with the analogous definition when $\lambda = \mathbf{0}_n$.
Then let $\vec{t}^{(\lambda)} := \{ \vec{t}^{(\lambda)}_{Q} \}_{Q\in\mathcal{Q}_+(\mathbb{R}^n)}$
with $ \vec{t}^{(\lambda)}_{Q} := [l(Q)]^{-\frac12} \vec{u}^{(\lambda)}_Q $
for any $Q\in\mathcal{Q}_+(\mathbb{R}^n)$ and, moreover, for any $k\in\{-N,\dots,N\}$,
let $ \vec{t}_k^{(\lambda)} := \{ \vec{t}^{(\lambda)}_{Q(I,k)} \}_{I\in \mathcal{Q}_+(\mathbb{R}^{n-1})} $.
From the assumption that $(\widetilde{s}-\frac{1}{\widetilde{p}})_- > (n-1)(\frac{1}{\alpha_V} - 1)$,
we infer that, for any $(K_b,L_b,M_b,N_b)$-molecule $b$,
one necessary condition such that $b$ is a
$B^{\widetilde{s}(\cdot)-\frac{1}{\widetilde{p}(\cdot)}}_{\widetilde{p}(\cdot),\widetilde{q}(\cdot)}(V,\mathbb{R}^{n-1})$-synthesis molecules
is $L_b < 0$,
which holds naturally for any $\{\theta^{(\mathbf{0}_{n-1})}_{I}:\ I\in \mathcal{Q}_{0}(\mathbb{R}^{n-1})\}$
and $\{ \theta^{(\lambda')}_{I}:\ I\in \mathcal{Q}_+(\mathbb{R}^{n-1}) \}$ with $\lambda \in \Lambda_{n-1}$.
Using this, Theorem \ref{wave mole}, and the assumption that $\mathcal{N}$
satisfies \eqref{con cn} for $B^{\widetilde{s}(\cdot)-\frac{1}{\widetilde{p}(\cdot)}}_{\widetilde{p}(\cdot),\widetilde{q}(\cdot)}(V,\mathbb{R}^{n-1})$,
we find that both
$\{\theta^{(\mathbf{0}_{n-1})}_{I}:\ I\in \mathcal{Q}_{0}(\mathbb{R}^{n-1})\}$
and $\{ \theta^{(\lambda')}_{I}:\ I\in \mathcal{Q}_+(\mathbb{R}^{n-1}) \}$ with $\lambda \in \Lambda_{n-1}$
are families of $B^{\widetilde{s}(\cdot)-\frac{1}{\widetilde{p}(\cdot)}}_{\widetilde{p}(\cdot),\widetilde{q}(\cdot)}(V,\mathbb{R}^{n-1})$-synthesis molecules.
Applying this with Theorems \ref{mole com} and \ref{dep wave 1}
and the assumption that $ \{ \theta^{(\lambda)}\}_{\lambda\in \{0,1\}^{n}}$
is a family of Daubechies wavelets, we conclude that we only need to show,
for any $\lambda \in \{0,1\}^n$ and $k\in\{-N,\dots,N\}$,
\begin{align}\label{tr bound 7}
\left\| \vec{t}^{(\lambda)}_k \right\|_{b^{\widetilde{s}(\cdot)-\frac{1}{\widetilde{p}(\cdot)}}_{\widetilde{p}(\cdot),\widetilde{q}(\cdot)}(V,\mathbb{R}^{n-1})}
\lesssim \left\| \left\{ \vec{u}^{(\lambda)}_{Q(I,k)} \right\}_{I\in \mathcal{Q}_+(\mathbb{R}^{n-1})} \right\|_{b^{s(\cdot)}_{p(\cdot),q(\cdot)}(W)}.
\end{align}
Now, fix $\lambda \in \{0,1\}^n$ and $k \in \{-N,\dots,N\}$.
Then let $\widetilde{t} := \{\widetilde{t}_I\}_{I\in \mathcal{Q}_+(\mathbb{R}^{n-1})}$
with $\widetilde{t}_I := |A_{Q(I,k),W} \vec{t}_{Q(I,k)}|$ for any $I\in \mathcal{Q}_+(\mathbb{R}^{n-1})$
and let $\widetilde{u}_k := \{\widetilde{u}_{Q}\}_{Q\in \mathcal{Q}_+(\mathbb{R}^{n})}$
with
\begin{align*}
\widetilde{u}_{Q} :=
\begin{cases}
\displaystyle |A_{Q,W}\vec{u}^{(\lambda)}_{Q}| &\ \text{if}\ Q = Q(I,k), \\
\displaystyle 0 &\ \text{otherwise}.
\end{cases}
\end{align*}
Observe that $\widetilde{u}_{Q(I,k)} = t^{(k)}_{Q(I,k)}$,
where $t^{(k)}$ is the same as in \eqref{def tk}.
Using Lemma \ref{WV k} and the inequality that $(1+|k|)^{\Delta_W} \leq (1+|N|)^{\Delta_W}$
for any $|k| \leq N$,
we find that, for any $I\in \mathcal{Q}_+(\mathbb{R}^{n-1})$,
$ |A_{I,V}\vec{t}^{(\lambda)}_{Q(I,k)}| \lesssim |A_{Q(I,k),W}\vec{t}^{(\lambda)}_{Q(I,k)}|$.
Combining these with Lemma \ref{wp p 1},
we conclude that
\begin{align*}
\left\| \vec{t}^{(\lambda)}_k \right\|_{b^{\widetilde{s}(\cdot)-\frac{1}{\widetilde{p}(\cdot)}}_{\widetilde{p}(\cdot),\widetilde{q}(\cdot)}(V,\mathbb{R}^{n-1})}
\lesssim \left\| \widetilde{t} \right\|_{b^{\widetilde{s}(\cdot)-\frac{1}{\widetilde{p}(\cdot)}}_{\widetilde{p}(\cdot),\widetilde{q}(\cdot)}(\mathbb{R}^{n-1})}
\lesssim \left\| \widetilde{t}^{(k)} \right\|_{b^{s(\cdot)}_{p(\cdot),q(\cdot)}(\mathbb{R}^{n})}
 = \left\| \widetilde{u} \right\|_{b^{s(\cdot)}_{p(\cdot),q(\cdot)}(\mathbb{R}^{n})},
\end{align*}
which completes the proof of \eqref{tr bound 7}.
Thus, we have ${\mathop \mathrm{\,Tr\,}} f$ converges in $[\mathcal{S}'(\mathbb{R}^{n-1})]^m$
and hence ${\mathop \mathrm{\,Tr\,}}$ is well-defined.
Moreover, using Lemma \ref{seq norm} and Theorem \ref{mole com},
we conclude that
\begin{align*}
\left\| {\mathop \mathrm{\,Tr\,}} \vec{f} \right\|_{B^{s(\cdot)-\frac{1}{p(\cdot)}}_{p(\cdot),q(\cdot)}(V,\mathbb{R}^{n-1})}
&\lesssim \sum_{\lambda\in\{0,1\}^n} \sum_{k=-N}^{N} \left\| \vec{t}^{(\lambda)}_k \right\|_{b^{s(\cdot)-\frac{1}{p(\cdot)}}_{p(\cdot),q(\cdot)}(V,\mathbb{R}^{n-1})}
\lesssim \left\| \vec{f} \right\|_{B^{s(\cdot)}_{p(\cdot),q(\cdot)}(W,\mathbb{R}^n)},
\end{align*}
which further implies that ${\mathop \mathrm{\,Tr\,}}$ is bounded.
This finishes the proof of Theorem \ref{trace 1}.
\end{proof}

Next, we establish the extension theorem for matrix-weighted variable Besov spaces.
\begin{theorem}\label{WV ext}
Let $p(\cdot), q(\cdot)\in \mathcal{P}_0(\mathbb{R}^n) \cap LH(\mathbb{R}^n)$, $s(\cdot)\in LH(\mathbb{R}^n)$,
$W\in \mathscr{A}_{p(\cdot),\infty}(\mathbb{R}^n)$, and $V\in \mathscr{A}_{\widetilde{p}(\cdot),\infty}({\mathbb{R}}^{n-1}) $.
If, for any $I\in \mathcal{Q}_+(\mathbb{R}^{n-1})$ and $\vec{z} \in \mathbb{C}^m$,
\begin{align}\label{WV nec ext}
\frac{1}{\|\mathbf{1}_{Q(I,0)}\|_{L^{p(\cdot)}(\mathbb{R}^n)}} \left\|\, \left|W(\cdot) \vec{z}\right| \mathbf{1}_{Q(I,0)}\right\|_{L^{p(\cdot)}(\mathbb{R}^n)}
\lesssim \frac{1}{\|\mathbf{1}_I\|_{L^{\widetilde{p}(\cdot)}(\mathbb{R}^{n-1})}} \left\|\, \left|V(\cdot) \vec{z}\right| \mathbf{1}_{I}\right\|_{L^{\widetilde{p}(\cdot)}(\mathbb{R}^{n-1})},
\end{align}
where the implicit positive constant is independent of $I$ and $\vec{z}$,
then the extension operator
$$ {\mathop \mathrm{\,Ext\,}}:\ B^{\widetilde{s}(\cdot)-\frac{1}{\widetilde{p}(\cdot)}}_{\widetilde{p}(\cdot),\widetilde{q}(\cdot)}(V,\mathbb{R}^{n-1}) \rightarrow B^{s(\cdot)}_{p(\cdot),q(\cdot)}(W,\mathbb{R}^n) $$
defined as in \eqref{def ext} is well-defined and bounded.
Moreover, if $\widetilde{s}(\cdot)$ satisfies $(\widetilde{s}-\frac{1}{\widetilde{p}})_- > (n-1)(\frac{1}{\alpha_V} - 1) $
and \eqref{WV nec} holds, then ${\mathop \mathrm{\,Tr\,}} \circ {\mathop \mathrm{\,Ext\,}}$ is the identity on
$B^{\widetilde{s}(\cdot)-\frac{1}{\widetilde{p}(\cdot)}}_{\widetilde{p}(\cdot),\widetilde{q}(\cdot)}(V,\mathbb{R}^{n-1})$.
\end{theorem}
\begin{remark}
We note that, when $p(\cdot)$, $q(\cdot)$, and $s(\cdot)$ are constant exponents and $W\in \mathscr{A}_p$,
the range of $s$ in Theorem \ref{WV ext} in this case coincides with the corresponding one in \cite[Theorem 5.10]{bhyy23 3}.
\end{remark}

To prove Theorem \ref{WV ext},
we  need the following result, which is the converse estimate of Lemma \ref{WV k}.
\begin{lemma}\label{W Wk 1}
Let $p(\cdot), q(\cdot)\in \mathcal{P}_0(\mathbb{R}^n)\cap LH(\mathbb{R}^n)$,
$s(\cdot)\in LH(\mathbb{R}^n)$, $W\in \mathscr{A}_{p(\cdot),\infty}(\mathbb{R}^n)$,
and $V\in \mathscr{A}_{p(\cdot),\infty}({\mathbb{R}}^{n-1}) $.
If \eqref{WV nec ext} holds,
then, for any $I\in\mathcal{Q}_+(\mathbb{R}^{n-1})$, $k\in\mathbb{Z}$,
and $\vec{z}\in\mathbb{C}^m$,
$$ \left| A_{Q(I,k),W} \vec{z} \right| \lesssim \left( 1+ |k| \right)^{\Delta_W} \left| A_{I,V} \vec{z} \right|, $$
where the implicit positive constant is independent of $\vec{z}$, $I$, and $k$.
\end{lemma}
\begin{proof}
For any $k\in\mathbb{Z}$, by the definition of $Q(I,k)$,
we obtain $l(Q(I,0)) = l(Q(I,k)) = l(I) $ and consequently,
for any $x\in Q(I,0)$ and $y\in Q(I,k)$,
$\frac{|x - y|}{l(I)} \lesssim |k|$.
From this and Lemma \ref{QP5} with $ Q := Q(I,k)$, $R := Q(I,0)$, and $\Delta := \Delta_W$,
we deduce that, for any $k\in \mathbb{Z}$,
\begin{align*}
\left\|A_{Q(I,k),W}A_{Q(I,0),W}^{-1}\right\|
&\lesssim \max\left\{ \left[ \frac{l(Q(I,0))}{l(Q(I,k))} \right]^{d_1},
\left[ \frac{l(Q(I,k))}{l(Q(I,0))} \right]^{d_2} \right\}\left[ 1+ \frac{|x - y|}{l(Q(I,k))\vee l(Q(I,0))} \right]^{\Delta_W}\\
&\lesssim \left( 1 + |k| \right)^{\Delta_W} .
\end{align*}
Using this, Definition \ref{def reducing operator}, and \eqref{WV nec ext},
we obtain, for any $\vec{z} \in \mathbb{C}^m$ and $k\in\mathbb{Z}$,
\begin{align*}
\left|A_{Q(I,k),W}\vec{z}\right|
&\leq \left\|A_{Q(I,k),W}A_{Q(I,0),W}^{-1}\right\| \left|A_{Q(I,0),W}\vec{z}\right|
\lesssim \left( 1 + |k| \right)^{\Delta_W}\left|A_{Q(I,0),W}\vec{z}\right|\\
&\sim \frac{1}{\|\mathbf{1}_{Q(I,0)}\|_{L^{p(\cdot)}(\mathbb{R}^n)}} \left\|\,\left| W(\cdot) \vec{z}\right| \mathbf{1}_{Q(I,0)}\right\|_{L^{p(\cdot)}(\mathbb{R}^n)}
\lesssim \frac{1}{\|\mathbf{1}_I\|_{L^{\widetilde{p}(\cdot)}(\mathbb{R}^{n-1})}} \left\|\, \left|V(\cdot) \vec{z}\right| \mathbf{1}_{I} \right\|_{L^{\widetilde{p}(\cdot)}(\mathbb{R}^{n-1})}\\
&\sim \left|A_{I,V} \vec{z}\right|.
\end{align*}
This finishes the proof of Lemma \ref{W Wk 1}.
\end{proof}
Now, we give the proof of Theorem \ref{WV ext}.
\begin{proof}[Proof of Theorem \ref{WV ext}]
We first show that ${\mathop \mathrm{\,Ext\,}} \vec{f}$ is well-defined and ${\mathop \mathrm{\,Ext\,}}$
is a bounded linear operator.
For any $\lambda' \in \{0,1\}^{n-1}$,
let $\vec{t}^{(\lambda')} := \{\vec{t}^{(\lambda')}_I\}_{I\in \mathcal{Q}_+(\mathbb{R}^{n-1})}$
with $\vec{t}^{(\lambda')}_I := \langle \vec{f}, \theta^{(\lambda')}_I \rangle$
and, when $\lambda' = \mathbf{0}_{n-1}$,
$ \vec{t}^{(\mathbf{0}_{n-1})}_I:= 0 $ if $I\notin \mathcal{Q}_{0}$.
For any $\lambda' \in \{0,1\}^{n-1}$,
let $\vec{u}^{(\lambda')} := \{\vec{u}^{(\lambda')}_Q\}_{Q\in \mathcal{Q}_+(\mathbb{R}^{n})}$
with
\begin{align*}
\vec{u}^{(\lambda')}_Q :=
\begin{cases}
\displaystyle [l(I)]^\frac12 \vec{t}^{(\lambda')}_I &\text{if}\ Q = Q(I,k_0)\ \text{for some}\ I\in\mathcal{Q}_+(\mathbb{R}^{n-1}),\\
\displaystyle \mathbf{0}_m &\text{otherwise},
\end{cases}
\end{align*}
where $k_0$ is the same as in Remark \ref{wave k0}.
Thus, by this and \eqref{def ext 1},
we obtain, for any $\lambda' \in \{0,1\}^{n-1}$, $I\in\mathcal{Q}_+(\mathbb{R}^{n-1})$, and $x\in\mathbb{R}^n$,
\begin{align*}
\left\langle \vec{f}, \theta^{(\lambda')}_I \right\rangle {\mathop \mathrm{\,Ext\,}} \theta^{(\lambda')}_I(x)
 = \frac{1}{\varphi(-k_0)} \vec{u}^{(\lambda')}_{Q(I,k_0)}
\left[ \theta^{(\lambda')} \otimes \varphi \right]_{Q(I,k_0)}(x).
\end{align*}
Hence, using this
and \eqref{def ext}, we find that
\begin{align*}
{\mathop \mathrm{\,Ext\,}} \vec{f} &= \frac{1}{\varphi(-k_0)}\sum_{I\in\mathcal{Q}_0(\mathbb{R}^{n-1})} \vec{u}^{(\mathbf{0}_{n-1})}_{Q(I,k_0)}
\left[ \theta^{(\mathbf{0}_{n-1})} \otimes \varphi \right]_{Q(I,k_0)}\nonumber \\
&\quad + \frac{1}{\varphi(-k_0)}\sum_{\lambda' \in \Lambda_{n-1}} \sum_{I\in\mathcal{Q}_+(\mathbb{R}^{n-1})} \vec{u}^{(\lambda')}_{Q(I,k_0)}
\left[ \theta^{(\lambda')} \otimes \varphi \right]_{Q(I,k_0)}.
\end{align*}
This, together with Theorem \ref{dep wave 1} and the fact that
$\{[\theta^{\lambda'}\otimes \varphi]_{Q(I,k_0)}\}_{I\in\mathcal{Q}_+(\mathbb{R}^{n-1})}$
is a subset of
$\{ \theta^{(\mathbf{0})}_Q:\ Q\in \mathcal{Q}_0 \}\cup
\{ \theta_Q^{(\lambda)}:\ Q\in\mathcal{Q}_+, \lambda \in \Lambda_n\}$,
further implies that, to prove ${\mathop \mathrm{\,Ext\,}} \vec{f}$ converges in $[\mathcal{S}'(\mathbb{R}^{n-1})]^{m}$,
we only need to show that, for any $\lambda' \in \{0,1\}^{n-1}$,
\begin{align}\label{ext eq 3}
\left\| \vec{u}^{(\lambda')} \right\|_{b^{s(\cdot)}_{p(\cdot),q(\cdot)}(W,\mathbb{R}^n)}
\lesssim \left\| \vec{t}^{(\lambda')} \right\|_{b^{\widetilde{s}(\cdot)-\frac{1}{\widetilde{p}(\cdot)}}_{\widetilde{p}(\cdot),\widetilde{q}(\cdot)}(V,\mathbb{R}^{n-1})}.
\end{align}
Now, fixing $\lambda' \in \{0,1\}^{n-1}$,
let $ \widetilde{t} := \{\widetilde{t}_I\}_{I\in\mathcal{Q}_+(\mathbb{R}^{n-1})}$
with $ \widetilde{t}_I := |A_{I,V}\vec{t}^{(\lambda')}_I| $
and let $\widetilde{u} := \{\widetilde{u}_{Q}\}_{Q\in \mathcal{Q}_+(\mathbb{R}^{n})}$
with
\begin{align*}
\widetilde{u}_{Q} :=
\begin{cases}
\displaystyle |A_{Q}\vec{u}^{(\lambda)}_{Q}| &\ \text{if}\ Q = Q(I,k_0), \\
\displaystyle 0 &\ \text{otherwise}.
\end{cases}
\end{align*}
Observe that $ \widetilde{t}^{(k_0)} = \widetilde{u} $,
where $\widetilde{t}^{(k_0)}$ is as in \eqref{def tk}.
Using this and Lemmas \ref{wp p 1} and \ref{W Wk 1},
we find that
\begin{align}\label{ext eq 4}
\left\| \vec{t}^{(\lambda')} \right\|_{b^{\widetilde{s}(\cdot)-\frac{1}{\widetilde{p}(\cdot)}}_{\widetilde{p}(\cdot),\widetilde{q}(\cdot)}(V,\mathbb{R}^{n-1})}
\lesssim \left\| \widetilde{t} \right\|_{b^{\widetilde{s}(\cdot)-\frac{1}{\widetilde{p}(\cdot)}}_{\widetilde{p}(\cdot),\widetilde{q}(\cdot)}(\mathbb{R}^{n-1})}
\lesssim \left\| \widetilde{t}^{(k)} \right\|_{b^{s(\cdot)}_{p(\cdot),q(\cdot)}(\mathbb{R}^{n})}
 = \left\| \widetilde{u} \right\|_{b^{s(\cdot)}_{p(\cdot),q(\cdot)}(\mathbb{R}^{n})},
\end{align}
which completes the proof of \eqref{ext eq 3}.
Thus, from this, we infer that ${\mathop \mathrm{\,Ext\,}} \vec{f}$ converges in $(\mathcal{S}')^m$.
Furthermore, using \eqref{ext eq 4}, \eqref{ext eq 3}, and Theorem \ref{dep wave 1},
we conclude that
$$ \left\| {\mathop \mathrm{\,Ext\,}} \vec{f} \right\|_{B^{s(\cdot)}_{p(\cdot),q(\cdot)}(W,\mathbb{R}^n)}
\lesssim \left\| \vec{f} \right\|_{B^{s(\cdot)-\frac{1}{p(\cdot)}}_{p(\cdot),q(\cdot)}(V,\mathbb{R}^{n-1})}, $$
that is,  ${\mathop \mathrm{\,Ext\,}}$ is bounded.

Finally, by the definitions of ${\mathop \mathrm{\,Tr\,}}$ and ${\mathop \mathrm{\,Ext\,}}$,
we conclude that, for any $\vec{f} \in B^{s(\cdot)-\frac{1}{p(\cdot)}}_{p(\cdot),q(\cdot)}(V,\mathbb{R}^{n-1})$,
\begin{align*}
\left( {\mathop \mathrm{\,Tr\,}} \circ {\mathop \mathrm{\,Ext\,}} \right) \vec{f} = &\sum_{I \in\mathcal{Q}_0(\mathbb{R}^{n-1})} \left\langle \vec{f}, \theta^{(\mathbf{0}_{n-1})}_I \right\rangle \left( {\mathop \mathrm{\,Tr\,}} \circ {\mathop \mathrm{\,Ext\,}} \right) \theta^{(\mathbf{0}_{n-1})}_I\\
&\quad +  \sum_{\lambda' \in \Lambda_{n-1} } \sum_{I\in \mathcal{Q}_+(\mathbb{R}^{n-1})} \left\langle \vec{f}, \theta^{(\lambda')}_I \right\rangle \left( {\mathop \mathrm{\,Tr\,}} \circ {\mathop \mathrm{\,Ext\,}} \right) \theta^{(\lambda')}_I\\
& = \sum_{I \in\mathcal{Q}_0(\mathbb{R}^{n-1})} \left\langle \vec{f}, \theta^{(\mathbf{0}_{n-1})}_I \right\rangle  \theta^{(\mathbf{0}_{n-1})}_I + \sum_{\lambda' \in \Lambda_{n-1} } \sum_{I\in \mathcal{Q}_+(\mathbb{R}^{n-1})} \left\langle \vec{f}, \theta^{(\lambda')}_I \right\rangle \theta^{(\lambda')}_I
 = \vec{f}
\end{align*}
in $[\mathcal{S}'(\mathbb{R}^{n-1})]^m$.
This finishes the proof of Theorem \ref{WV ext}.
\end{proof}
\begin{remark}
When $p$, $q$, and $s$ are all constant exponents,
Theorems \ref{trace 1} and \ref{WV ext} in this case reduce to,
respectively,  \cite[Theorems 5.6 and 5.10]{bhyy23 3}
with $\tau = 0$ therein.
For the unweighted variable Besov space,
Theorem \ref{trace 1} in this case coincides with
\cite[{\rm (a)} of Theorem 1{\rm (i)}]{n16},
and Theorems \ref{trace 1} and \ref{WV ext} are new even
in the case where $W$ is a scalar variable weight.
\end{remark}

\subsection{Calder\'on--Zygmund Operators}\label{sec CZ}

In this subsection, under some mild assumptions
we establish the boundedness of Calder\'on--Zygmund operators
on $B^{s(\cdot)}_{p(\cdot),q(\cdot)}(W)$
(see, for example, \cite{t91,bhyy23 3}).

The \emph{Calder\'on--Zygmund operator $T$} is formally given by the expression of the form
$$ Tf(x) := \int_{\mathbb{R}^n} \mathcal{K}(x,y) f(y)\,dy, $$
where the \emph{kernel}
$$\mathcal{K}:\ \Delta := \{(x,y) \in \mathbb{R}^n \times \mathbb{R}^n:\ x\neq y\} \to \mathbb{C}$$
is a measurable mapping satisfying size, cancellation, and smoothness conditions
like those of the Hilbert kernel $\mathcal{K}(x,y) := \frac{1}{x-y}$ when $n = 1$
or the Riesz kernels $\mathcal{K}_i(x,y) := \frac{x_i-y_i}{|x-y|^{n+1}}$ with $i\in\{1,\dots,n\}$
when $n\geq 2$.
We refer to \cite{bhyy23 3} for the construction of Calder\'on--Zygmund operators
in Bseov spaces.
Let $\mathcal{D}$ be the set of all infinitely differentiable functions
on $\mathbb{R}^n$ with compact support,
equipped with the classical inductive limit topology,
and $\mathcal{D}'$ be the space of all continuous linear functionals on $\mathcal{D}$,
equipped with the weak-$\ast$ topology
(see, for instance, \cite[Chapters 2.2 and 2.3]{g14} for more details).
We note that, if the Calder\'on--Zygmund operator $T\in \mathcal{L}(\mathcal{S}, \mathcal{S}')$, \
then, by the well-known Schwartz kernel theorem,
we find that there exists $\mathcal{K}\in \mathcal{S}'(\mathbb{R}^n\times \mathbb{R}^n)$
such that,
for any $\varphi,\phi\in \mathcal{S}$,
$$ \left\langle T\varphi ,\phi \right\rangle = \left\langle \mathcal{K},\varphi \otimes \phi \right\rangle, $$
where $\mathcal{K}$ is called the \emph{Schwartz kernel} of $T$.

The following definition is about some basic assumptions of $\mathcal{K}$.
\begin{definition}
Let $T\in \mathcal{L}(\mathcal{S},\mathcal{S}')$ and $\mathcal{K}\in \mathcal{S}'(\mathbb{R}^n\times \mathbb{R}^n) $ be the Schwartz kernel of $T$.
\begin{itemize}
\item[{\rm (i)}] The operator $T$ is said to satisfy the
\emph{weak boundedness property}, denoted by $T\in \rm{WBP}$,
if, for any bounded subset $\mathcal{B}$ of $\mathcal{D}$,
there exists a positive constant $C$, depending on $\mathcal{B}$,
such that, for any $\varphi,\eta\in\mathcal{B}$, $h\in\mathbb{R}^n$, and $r\in (0,\infty)$,
$$ \left| \left\langle T\left(\varphi\left( \frac{\cdot-h}{r} \right) \right),\eta\left( \frac{\cdot-h}{r} \right)  \right\rangle \right|
\leq Cr^n. $$
\item[{\rm (ii)}] For any $l\in (0,\infty)$, we say $T$ has a \emph{Calder\'on--Zygmund kernel}
of order $l$, denoted by $T\in {\rm CZO}(l)$,
if the restriction of $\mathcal{K}$ on the set $\{(x,y)\in \mathbb{R}^n\times\mathbb{R}^n:\ x\neq y\}$
is a continuous function with continuous partial derivatives in the $x$ variable up to order
$ \lfloor\!\lfloor l \rfloor\!\rfloor $,
where $\lfloor\!\lfloor l \rfloor\!\rfloor$ is as in \eqref{def int},
satisfying that there exists a positive constant $C$ such that
\begin{itemize}
\item[{\rm (a)}] for any $\gamma \in\mathbb{Z}_+^n$ with
$ |\gamma| \leq \lfloor\!\lfloor l \rfloor\!\rfloor $ and
for any $x,y\in\mathbb{R}^n$ with $x\neq y$,
$ | \partial_x^\gamma \mathcal{K}(x,y) | \leq C |x-y|^{-n-|\gamma|} $,
\item[{\rm (b)}]for any $\gamma\in\mathbb{Z}_+^n$ with $|\gamma| = \lfloor\!\lfloor l \rfloor\!\rfloor$
and for any $x,y,h\in\mathbb{R}^n$ with $|h| < \frac12 |x-y|$,
$$ \left| \partial_x^\gamma \mathcal{K}(x,y) - \partial_x^\gamma \mathcal{K}(x+h,y) \right| \leq C \left|x-y\right|^{-n-l}|h|^{l^{\ast\ast}}, $$
where $l^{\ast \ast}$ is the same as in \eqref{def int}.
\end{itemize}
For any $l\in(-\infty,0]$, we interpret $T\in{\rm CZO}(l)$ as a void condition.
\end{itemize}
\end{definition}
\begin{remark}
By the definition of ${\rm CZO}(l)$,
it is obvious that, for any $l_1,l_2\in\mathbb{R}$ with $l_1<l_2$,
$ {\rm CZO}(l_1) \subset {\rm CZO}(l_2) $.
\end{remark}
To define the action of Calder\'on--Zygmund operators on polynomials,
we first recall the following result,
which is a special case of \cite[Lemma 2.2.12]{t91}.
\begin{lemma}\label{def pol}
Let $l\in (0,\infty)$ and $T\in {\rm CZO}(l)$,
and let $\{\phi_j\}_{j\in\mathbb{N}} \subset \mathcal{D}$ be a sequence of functions
such that $\sup_{j\in\mathbb{N}} \|\phi_j\|_{L^\infty} < \infty $
and, for any compact set $K$ in $\mathbb{R}^n$,
there exists $j_K\in\mathbb{N}$ such that,
for any $j\geq j_k$ and $x\in K$,
$\phi_j(x) = 1$.
Then the limit
\begin{align}\label{def Tfg}
\left\langle T(f),g \right\rangle := \lim\limits_{j\rightarrow \infty} \left\langle T\left( \phi_jf \right),g \right\rangle
\end{align}
exists for any polynomials $f(y) := y^\gamma$ with $|\gamma| \leq \lfloor\!\lfloor l \rfloor\!\rfloor $
and for any $g\in \mathcal{D}_{\lfloor\!\lfloor l \rfloor\!\rfloor}$, where
$$ \mathcal{D}_{\lfloor\!\lfloor l \rfloor\!\rfloor} := \left\{ g\in\mathcal{D}:\ \int_{\mathbb{R}^n} x^\gamma g(x) = 0
\ \text{for any}\ \gamma \in\mathbb{Z}_+^n \ \text{with}\ |\gamma| \leq \lfloor\!\lfloor l \rfloor\!\rfloor \right\}, $$
and \eqref{def Tfg} is independent of the choice of $\{\phi_j\}_{j\in\mathbb{N}}$.
\end{lemma}
Based on Lemma \ref{def pol}, we give the following definition.
\begin{definition}
Let $l\in (0,\infty)$.
For any $T\in {\rm CZO}(l)$ and $f(y) := y^\gamma$ with $y\in\mathbb{R}^n$
and $|\gamma| \leq \lfloor\!\lfloor l \rfloor\!\rfloor$,
we define $T(y^\gamma) = Tf : \mathcal{D}_{\lfloor\!\lfloor l \rfloor\!\rfloor} \rightarrow \mathbb{C}$
by \eqref{def Tfg}.
\end{definition}

\begin{definition}
Let $E,F\in\mathbb{R}$, $T\in \mathcal{L}(\mathcal{S},\mathcal{S}')$, and $\mathcal{K}\in \mathcal{S}'(\mathbb{R}^n\times \mathbb{R}^n)$ be its Schwartz kernel.
We say that $T\in {\rm CZK}^0(E;F)$ if the restriction of $\mathcal{K}$ to $\{(x,y)\in \mathbb{R}^n\times\mathbb{R}^n:\ x\neq y\}$
is a continuous function such that, for any $\alpha \in \mathbb{Z}_+^n$ with $|\alpha| \leq \lfloor\!\lfloor E \rfloor\!\rfloor$,
$\partial^\alpha_x \mathcal{K}$ exists as a continuous function
and there exists a positive constant $C$
such that $ | \partial^\alpha_x \mathcal{K}(x,y) | \leq C |x-y|^{-n-|\alpha|} $
for any $x,y\in\mathbb{R}^n$ with $x\neq y$,
$$ \left| \partial^\alpha_x \mathcal{K}(x+u,y) - \partial^\alpha_x \mathcal{K}(x,y) \right|
\leq C |u|^{E^{\ast\ast}}|x-y|^{-n-E} $$
for any $\alpha\in\mathbb{Z}_+^n$ with $|\alpha| = \lfloor\!\lfloor E \rfloor\!\rfloor$
and $x,y,u\in \mathbb{R}^n$ with $|u| < \frac12 |x-y|$,
and $$ \left| \partial^\alpha_x \partial^\beta_y \mathcal{K}(x,y) - \partial^\alpha_x\partial^\beta_y \mathcal{K}(x,y+v) \right|
\leq C |v|^{(F-|\alpha|)^{\ast\ast}}|x-y|^{-n-|\alpha|-(F-|\alpha|)} $$
for any $\alpha,\beta\in\mathbb{Z}_+^n$ with $|\alpha| \leq \lfloor\!\lfloor E \rfloor\!\rfloor$
and $|\beta| = \lfloor\!\lfloor F-|\alpha| \rfloor\!\rfloor$ and for any $x,y,v\in\mathbb{R}^n$ with $|v|<\frac12 |x-y|$.

We say that $T\in {\rm CZK}^1(E;F)$ if $T\in {\rm CZK}^0(E;F)$ and,
in addition, for any $\alpha,\beta\in\mathbb{Z}_+^n$ with $|\alpha| =\lfloor\!\lfloor E \rfloor\!\rfloor$
and $|\beta| = \lfloor\!\lfloor F-E \rfloor\!\rfloor$ and for any $x,y,u, v\in\mathbb{R}^n$ with $|u|+|v|<\frac12 |x-y|$,
\begin{align}\label{CZK1}
&\left| \partial^\alpha_x \partial^\beta_y \mathcal{K}(x,y) - \partial^\alpha_x\partial^\beta_y \mathcal{K}(x+u,y)
- \partial^\alpha_x\partial^\beta_y \mathcal{K}(x,y+v) + \partial^\alpha_x\partial^\beta_y \mathcal{K}(x+u,y+v)\right|\nonumber\\
&\quad \leq C |u|^{E^{\ast\ast}}|v|^{(F-E)^{\ast\ast}}|x-y|^{-n-E-(F-E)}.
\end{align}
We write just $T\in {\rm CZK}(E;F)$ if the parameter values are such that \eqref{CZK1}
is void and hence $T\in {\rm CZK}^0(E;F)$ and $T\in {\rm CZK}^1(E;F)$ coincide.
\end{definition}
Indeed, it is obvious that \eqref{CZK1} is void unless $F > E >0$.
\begin{definition}\label{def CZK2}
Let $\sigma\in \{0,1\}$ and $E,F,G,H\in\mathbb{R}$.
We say $T\in {\rm lnCZO}^\sigma(E,F,G,H)$
if $T\in \mathcal{L}(\mathcal{S},\mathcal{S}')$ and its Schwartz kernel $\mathcal{K}\in \mathcal{S}'(\mathbb{R}^n\times \mathbb{R}^n)$ satisfies
\begin{itemize}
\item[{\rm (i)}] $T\in {\rm WBP}$,
\item[{\rm (ii)}] $\mathcal{K}\in {\rm CZK}^\sigma(E;F)$,
\item[{\rm (iii)}] $T(y^\gamma) = 0$ for any $\gamma\in\mathbb{Z}_+^n$ with $|\gamma| \leq G$,
\item[{\rm (iv)}]  $T^\ast(x^\theta) = 0$ for any $\theta\in\mathbb{Z}_+^n$ with $|\theta| \leq H$,
where $T^\ast$ denotes the adjoint operator of $T$ on $L^2$,
\item[{\rm (v)}] there exists a positive constant $C$ such that, for any $\alpha\in\mathbb{Z}_+^n$
with $|\alpha| \leq \lfloor\!\lfloor E\rfloor\!\rfloor+1$ and any $x,y\in\mathbb{R}^n$ with $|x-y| >1$,
$ | \partial^\alpha_x\mathcal{K}(x,y) |\leq C|x-y|^{-(n+F)}. $
\end{itemize}
\end{definition}
\begin{remark}
In Definition \ref{def CZK2},
if we remove the condition (v) of ${\rm lnCZO}^\sigma(E,F,G,H)$,
then ${\rm lnCZO}^\sigma(E,F,G,H)$ reduces to ${\rm CZO}^\sigma(E,F,G,H)$,
which was defined in \cite[Definition 6.17]{bhyy23 3}.
\end{remark}

Now, we recall the definition of smooth atoms.
Differently from $(r,L,M)$-atoms as in Definition \ref{def atom}
with only finite times differentiability,
the following smooth atoms are infinitely differentiable.
\begin{definition}\label{def sm atom}
Let $L,N\in\mathbb{R}$.
A function $a_Q\in C^{\infty}_{\rm c}$ is called an \emph{$(L,M)$-atom} on cube $Q$
if
\begin{itemize}
\item[{\rm (i)}] ${\mathop\mathrm{\,supp\,}} a_Q \subset 3Q$,
\item[{\rm (ii)}] $\int_{\mathbb{R}^n} x^\gamma a_Q(x)\,dx = 0$ if $l(Q) < 1$ and $\gamma\in \mathbb{Z}_+^n$ with $|\gamma| \leq L$,
\item[{\rm (iii)}] $|D^\gamma a_Q(x)| \leq |Q|^{-\frac12-\frac{|\gamma|}{n}}$
for any $x\in\mathbb{R}^n$ and $\gamma\in\mathbb{Z}_+^n$ with $|\gamma| \leq N$.
\end{itemize}
\end{definition}

Applying an argument similar to that used in the proof
of \cite[Proposition 6.5]{bhyy23 3}
with \cite[Corollary 3.15]{bhyy23 3} replaced by Lemma \ref{mbpsi abo},
we obtain the following
conclusion; we omit the details here.

\begin{lemma}\label{CZ ext}
Let $p(\cdot),q(\cdot)\in \mathcal{P}_0\cap LH$, $s(\cdot)\in LH$,
$W\in \mathscr{A}_{p(\cdot),\infty}$, and $L,N\in (0,\infty)$.
If $T\in \mathcal{L}(\mathcal{S},\mathcal{S}')$ maps $(L,N)$-atoms
to $b^{s(\cdot)}_{p(\cdot),q(\cdot)}(W)$-synthesis molecules,
then there exists an operator
$\widetilde{T} \in \mathcal{L}(B^{s(\cdot)}_{p(\cdot),q(\cdot)}(W), B^{s(\cdot)}_{p(\cdot),q(\cdot)}(W))$
that agrees with $T$ on $(\mathcal{S})^m$.
\end{lemma}

Observing that Definition \ref{def sm atom} coincides with \cite[Definition 6,14]{bhyy24}
and Definition \ref{def mole} coincides with \cite[Definition 5.1]{bhyy24},
we can apply the following two lemmas, which are exactly
\cite[Proposition 6.19]{bhyy23 3} and \cite[Proposition 6.24]{bhyy24}, respectively,
and show that the Calder\'on--Zygmund operator maps atoms into molecules.

\begin{lemma}\label{CZK lem 1}
Let $\sigma \in \{0,1\}$, $E,F,G,H,K,L,M,N\in\mathbb{R}$,
and $Q\in \mathcal{Q}_+$.
Suppose that $T\in {\rm CZO}^\sigma(E,F,G,H)$.
Then $T$ maps sufficiently regular atoms on $Q$
to $(K,L,M,N)$-molecules on $Q$ provided that
\begin{align*}
\sigma \geq \mathbf{1}_{(0,\infty)}(N),\quad
\begin{cases}
\displaystyle E\geq N, \\
\displaystyle E > \lfloor N \rfloor^{(+)},
\end{cases}\
\begin{cases}
\displaystyle F\geq (K\wedge M)-n, \\
\displaystyle F > \lfloor L \rfloor,
\end{cases}\
G\geq \lfloor N \rfloor^{(+)},
\end{align*}
and $H\geq \lfloor L \rfloor^{(+)}.$
\end{lemma}

\begin{lemma}\label{CZK lem 2}
Let $\sigma \in \{0,1\}$, $E,F,G,H, K,M,N\in\mathbb{R}$,
and $Q\in \mathcal{Q}_0$.
If $T\in {\rm lnCZO}^\sigma(E,F,G,H)$,
then $T$ maps sufficiently regular non-cancellative atoms on $Q$
to $(K,-1,M,N)$-molecules on $Q$ provided that
\begin{align*}
\sigma \geq \mathbf{1}_{(0,\infty)}(N),\quad
\begin{cases}
\displaystyle E\geq N, \\
\displaystyle E > \lfloor N \rfloor^{(+)},
\end{cases}
F\geq (K\wedge M)-n, \ \text{and}\ G\geq \lfloor N \rfloor^{(+)}.
\end{align*}
\end{lemma}

Combining Lemmas \ref{CZK lem 1}, \ref{CZK lem 2}, and \ref{CZ ext}
with Theorem \ref{mole abo},
we obtain the following result immediately;
we omit details here.

\begin{theorem}\label{CZ theo}
Let $p(\cdot), q(\cdot)\in \mathcal{P}_0\cap LH$ and $s(\cdot)\in LH$
and let $W\in \mathscr{A}_{p(\cdot)}$ and $\mathbb{A} := \{A_Q\}_{Q\in\mathcal{Q}_+}$
be reducing operators of order $p(\cdot)$ for $W$.
Assume that $T\in {\rm lnCZO}^\sigma(E,F,G,H)$,
where $\sigma \in \{0,1\}$ and $E,F,G,H\in\mathbb{R}$ satisfy
\begin{align*}
&\sigma \geq \mathbf{1}_{(0,\infty)}(s_+),\quad E\geq (s_+)^{(+)},
\quad F > \frac{n}{\alpha_W} - n +\left[-s_-\vee C(s,q)\right],
\quad G \geq \lfloor s_+ \rfloor^{(+)}, \\
&\quad\text{and}\quad H\geq \left\lfloor \frac{n}{\alpha_W} - n -s_- \right\rfloor^{(+)},
\end{align*}
where $C(s,q)$ is the same as in \eqref{abo bound con}.
Then $T$ is bounded on $B^{s(\cdot)}_{p(\cdot),q(\cdot)}(W)$.
\end{theorem}
\begin{remark}
When $p$, $q$, and $s$ are all constant exponents and $W\in \mathscr{A}_p$,
the ranges of $E$, $F$, $G$, and $H$ in Theorem \ref{CZ theo} in this case
coincide with the corresponding ones in \cite[Proposition 6.19]{bhyy23 3} in the case $\tau = 0$ therein.
\end{remark}

\noindent\textbf{Author Contributions}\quad 
All authors developed and discussed the results and contributed to the final
manuscript.

\medskip

\noindent\textbf{Data Availability}\quad Data sharing is not applicable 
to this article as no data sets were generated or analyzed.

\section*{Declarations}

\noindent\textbf{Conflict of interest}\quad All authors state no conflict of interest.

\medskip

\noindent\textbf{Informed Consent}\quad Informed consent has been obtained 
from all individuals included in this research work.

\bigskip

\noindent   Dachun Yang (Corresponding author), Wen Yuan and Zongze Zeng

\medskip

\noindent  Laboratory of Mathematics and Complex Systems
(Ministry of Education of China),
School of Mathematical Sciences, Beijing Normal University,
Beijing 100875, The People's Republic of China

\smallskip

\noindent {\it E-mails}:
\texttt{dcyang@bnu.edu.cn} (D. Yang)

\noindent\phantom{{\it E-mails:}}
\texttt{wenyuan@bnu.edu.cn} (W. Yuan)

\noindent\phantom{{\it E-mails:}}
\texttt{zzzeng@mail.bnu.edu.cn} (Z. Zeng)


\begin{thebibliography}{99}
\bibitem{ac21}
M. Abidin and J. Chen,
Global well-posedness for fractional Navier--Stokes equations in variable exponent
Fourier--Besov--Morrey spaces,
Acta Math. Sci. Ser. B (Engl. Ed.) 41 (2021), 164--176.

\vspace{-0.3cm}

\bibitem{ac21 2}
M. Abidin and J. Chen,
Global well-posedness of the generalized rotating
magnetohydrodynamics equations in variable exponent Fourier--Besov spaces,
J. Appl. Anal. Comput. 11 (2021), 1177--1190.

\vspace{-0.3cm}

\bibitem{adh18}
A. Almeida, L. Diening and P. H\"ast\"o,
Homogeneous variable exponent Besov and Triebel--Lizorkin spaces,
Math. Nachr. 291 (2018), 1177--1190.

\vspace{-0.3cm}

\bibitem{ah10}
A. Almeida and P. H\"ast\"o,
Besov spaces with variable smoothness and integrability,
J. Funct. Anal. 258 (2010), 1628--1655.

\vspace{-0.3cm}

\bibitem{ah14}
A. Almeida and P. H\"ast\"o,
Interpolation in variable exponent spaces,
Rev. Mat. Complut. 27 (2014), 657--676.

\vspace{-0.3cm}

\bibitem{b99}
O. V. Besov,
On spaces of functions of variable smoothness defined by pseudodifferential operators,
Tr. Mat. Inst. Steklova 227 (1999),
Issled. po Teor. Differ. Funkts. Mnogikh Perem. i ee Prilozh. 18, 56--74 (in Russian);
translation in Proc. Steklov Inst. Math. 227(4) (1999), 50--69.

\vspace{-0.3cm}

\bibitem{b03}
O.V. Besov,
Equivalent normings of spaces of functions of variable smoothness,
Tr. Mat. Inst. Steklova 243 (2003), Funkt. Prostran. Priblizh. Differ. Uravn.,
87--95 (in Russian);
translation in Proc. Steklov Inst. Math. 243(4) (2003), 80--88.

\vspace{-0.3cm}

\bibitem{b05}
O.V. Besov,
Interpolation, embedding, and extension of spaces of functions of variable smooth-ness,
Tr. Mat. Inst. Steklova 248 (2005), Issled. Teor. Funkts. Differ. Uravn., 52--63, (in Russian);
translation in Proc. Steklov Inst. Math. 248(1) (2005), 47--58.


\vspace{-0.3cm}

\bibitem{bpw16}
K. Bickel, S. Petermichl and B. D. Wick,
Bounds for the Hilbert transform with matrix $A_2$ weights,
J. Funct. Anal. 270 (2016), 1719--1743.

\vspace{-0.3cm}

\bibitem{b01}
M. Bownik,
Inverse volume inequalities for matrix weights,
Indiana Univ. Math. J. 50 (2001), 383--410.

\vspace{-0.3cm}

\bibitem{bc22}
M. Bownik and D. Cruz-Uribe,
Extrapolation and factorization of matrix weights,
Math. Ann. (to appaer) or arXiv:2210.09443.

\vspace{-0.3cm}

\bibitem{bcyy24}
F. Bu, Y. Chen, D. Yang and W. Yuan,
Maximal function and atomic characterizations of matrix-weighted Hardy spaces
with their applications to boundedness of Calder\'on--Zygmund operators,
arXiv:2501.18800.

\vspace{-0.3cm}

\bibitem{bhyy23}
F. Bu, T. Hyt\"onen, D. Yang, and W. Yuan,
Matrix-weighted Besov-type and Triebel--Lizorkin-type spaces I:
$A_p$-dimensions of matrix weights and $\psi$-transform characterizations,
Math. Ann. 391 (2025), 6105--6185.
\vspace{-0.3cm}

\bibitem{bhyy23 2}
F. Bu, T. Hyt\"onen, D. Yang, and W. Yuan,
Matrix-weighted Besov-type and Triebel--Lizorkin-type spaces II:
Sharp boundedness of almost diagonal operators,
J. Lond. Math. Soc. (2) 111 (2025), Paper No. e70094, 59 pp.

\vspace{-0.3cm}

\bibitem{bhyy23 3}
F. Bu, T. Hyt\"onen, D. Yang and W. Yuan,
Matrix-weighted Besov-type and Triebel--Lizorkin-type spaces III:
characterizations of molecules and wavelets, trace theorems,
and boundedness of pseudo-differential operators and Calder\'on--Zygmund operators,
Math. Z. 308 (2024), Paper No. 32, 67 pp.

\vspace{-0.3cm}

\bibitem{bhyy24}
F. Bu, T. Hyt\"onen, D. Yang and W. Yuan,
Besov--Triebel--Lizorkin-type spaces with matrix $A_\infty$ weights,
Sci. China Math. 69 (2026), 383--460.

\vspace{-0.3cm}

\bibitem{bhyy23 1}
F. Bu, D. Yang, W. Yuan and T. Hyt\"onen,
New characterizations and properties of matrix $A_\infty$ weights,
Acta Math. Sin. (Engl. Ser.) (2025),
https://doi.org/10.1007/s10114-025-5143-9.

\vspace{-0.3cm}

\bibitem{byyz25}
F. Bu, D. Yang, W. Yuan and M. Zhang,
Matrix-weighted Besov--Triebel--Lizorkin spaces of optimal scale:
Real-variable characterizations,
invariance on integrable index, and Sobolev-type embedding,
J. Differential Equations (2026), 463 (2026), Paper No. 114140, 101 pp.

\vspace{-0.3cm}

\bibitem{byyz26} F. Bu, D. Yang, W. Yuan and Y. Zhao,
Matrix weights,  maximal operators,
Calder\'on--Zygmund operators, and Besov--Triebel--Lizorkin-type
spaces---a survey, Anal. Theory Appl. 41 (2025), 371--468.

\vspace{-0.3cm}

\bibitem{bbd20}
H.-Q. Bui, T. A. Bui and X. T. Duong,
Weighted Besov and Triebel--Lizorkin spaces associated with operators and applications,
Forum Math. Sigma 8 (2020), Paper No. e11, 95 pp.

\vspace{-0.3cm}

\bibitem{b20}
T. A. Bui,
Besov and Triebel--Lizorkin spaces for Schr\"odinger operators with inverse-square
potentials and applications,
J. Differential Equations 269 (2020), 641--688.

\vspace{-0.3cm}

\bibitem{b20 2}
T. A. Bui,
Hermite pseudo-multipliers on new Besov and Triebel--Lizorkin spaces,
J. Approx. Theory 252 (2020), 105348, 16 pp.

\vspace{-0.3cm}

\bibitem{bd17}
T. A. Bui and X. T. Duong,
Laguerre operator and its associated weighted Besov and Triebel--Lizorkin spaces,
Trans. Amer. Math. Soc. 369 (2017), 2109--2150.

\vspace{-0.3cm}

\bibitem{bd21}
T. A. Bui and X. T. Duong,
Spectral multipliers of self-adjoint operators on Besov and Triebel--Lizorkin spaces
associated to operators,
Int. Math. Res. Not. IMRN 2021 (2021), 18181--18224.

\vspace{-0.3cm}

\bibitem{bd21 2}
T. A. Bui and X. T. Duong,
Higher-order Riesz transforms of Hermite operators
on new Besov and Triebel--Lizorkin spaces,
Constr. Approx. 53 (2021), 85--120.

\vspace{-0.3cm}

\bibitem{cyy25} Y. Chen, D. Yang and W. Yuan, Matrix-weighted
Campanato spaces: Duality and Calder\'on--Zygmund operators,
Acta Math. Sci. Ser. B (Engl. Ed.) (to appera) or arXiv:2508.15195.

\vspace{-0.3cm}

\bibitem{cms25}
R. Chichoune, Z. Mokhtari and K. Saibi, Khedoudj,
Weighted variable Besov space associated with operators,
Rend. Circ. Mat. Palermo (2) 74 (2025), Paper No. 26, 26 pp.

\vspace{-0.3cm}

\bibitem{c01}
M. Christ and M. Goldberg,
Vector $A_2$ weights and a Hardy--Littlewood maximal function,
Trans. Amer. Math. Soc. 353 (2001), 1995--2002.

\vspace{-0.3cm}

\bibitem{cgn17}
G. Cleanthous, A. G. Georgiadis and M. Nielsen,
Discrete decomposition of homogeneous mixed-norm Besov spaces,
in: Functional Analysis, Harmonic Analysis, and Image Processing:
A Collection of Papers in Honor of Bj\"orn Jawerth,
pp. 167--184, Contemp. Math. 693, Amer. Math. Soc., Providence, RI, 2017.

\vspace{-0.3cm}

\bibitem{cgn19}
G. Cleanthous, A. G. Georgiadis and M. Nielsen,
Molecular decomposition and Fourier multipliers for holomorphic Besov and Triebel--Lizorkin spaces,
Monatsh. Math. 188 (2019), 467--493.

\vspace{-0.3cm}

\bibitem{cdh11}
D. Cruz-Uribe, L. Diening and P. H\"ast\"o,
The maximal operator on weighted variable Lebesgue spaces,
Fract. Calc. Appl. Anal. 14 (2011), 361--374.

\vspace{-0.3cm}

\bibitem{cf13}
D. Cruz-Uribe and A. Fiorenza,
Variable Lebesgue Space. Foundations and Harmonic Analysis,
Appl. Number. Harmon. Anal., Birkh\"auser/Springer, Heidelberg, 2013.

\vspace{-0.3cm}

\bibitem{cfn12}
D. Cruz-Uribe, A. Fiorenza and C. J. Neugebauer,
Weighted norm inequalities for the maximal operator on variable Lebesgue spaces,
J. Math. Anal. Appl. 394 (2012), 744--760.


\vspace{-0.3cm}

\bibitem{cp23}
D. Cruz-Uribe and M. Penrod,
Convolution operators in matrix weighted, variable Lebesgue spaces,
Anal. Appl. (Singap.) 22 (2024), 1133--1157.

\vspace{-0.3cm}

\bibitem{cp24}
D. Cruz-Uribe and M. Penrod,
The reverse H\"older inequality for $\mathscr{A}_{p(\cdot)}$
weights with applications to matrix weights,
arXiv: 2411.12849

\vspace{-0.3cm}

\bibitem{cs23}
D. Cruz-Uribe and B. Sweeting,
Weighted weak-type inequalities for maximal operators and singular integrals,
Rev. Mat. Complut. 38 (2025), 183--205.

\vspace{-0.3cm}

\bibitem{cw17}
D. Cruz-Uribe and L. D. Wang,
Extrapolation and weighted norm inequalities in the variable Lebesgue spaces,
Trans. Amer. Math. Soc. 369 (2017), 1205--1235.

\vspace{-0.3cm}

\bibitem{d88}
I. Daubechies,
Orthonormal bases of compactly supported wavelets,
Comm. Pure Appl. Math. 41 (1988), 909--996.

\vspace{-0.3cm}

\bibitem{dhl20}
F. Di Plinio, T. Hyt\"onen and K. Li,
Sparse bounds for maximal rough singular integrals via the Fourier transform,
Ann. Inst. Fourier (Grenoble) 70 (2020), 1871--1902.

\vspace{-0.3cm}

\bibitem{dhr17}
L. Diening, P. Harjulehto, P. H\"ast\"o and M.  R\r{u}{\v z}i{\v c}ka,
Lebesgue and Sobolev Spaces with Variable Exponents,
Lecture Notes in Mathematics 2017, Springer, Heidelberg, 2011.

\vspace{-0.3cm}

\bibitem{dhr09}
L. Diening, P. H\"ast\"o and S. Roudenko,
Function spaces of variable smoothness and integrability,
J. Funct. Anal. 256 (2009), 1731--1768.

\vspace{-0.3cm}

\bibitem{dx14}
B. Dong and J. Xu,
Local characterizations of Besov and Triebel--Lizorkin spaces with variable exponent,
J. Funct. Spaces 2014 (2014), Art. ID 417341, 8 pp.


\vspace{-0.3cm}

\bibitem{d12}
D. Drihem,
Atomic decomposition of Besov spaces with variable smoothness and integrability,
J. Math. Anal. Appl. 389 (2012), 15--31.

\vspace{-0.3cm}

\bibitem{d15}
D. Drihem,
Some properties of variable Besov-type spaces,
Funct. Approx. Comment. Math. 52 (2015), 193--221.

\vspace{-0.3cm}

\bibitem{d15 1}
D. Drihem,
Some characterizations of variable Besov-type spaces,
Ann. Funct. Anal. 6 (2015), 255--288.

\vspace{-0.3cm}

\bibitem{dh17}
D. Drihem and W. Hebbache,
Boundedness of non regular pseudodifferential operators on variable Besov spaces,
J. Pseudo-Differ. Oper. Appl. 8 (2017), 167--189.

\vspace{-0.3cm}

\bibitem{fj90}
M. Frazier and B, Jawerth,
A discrete transform and decompositions of distribution spaces,
J. Funct. Anal. 93 (1990), 34--170.

\vspace{-0.3cm}

\bibitem{fr04}
M. Frazier and S. Roudenko,
Matrix-weighted Besov spaces and conditions of $A_p$ type for
$0<p\leq 1$,
Indiana Univ. Math. J. 53 (2004), 1225--1254.

\vspace{-0.3cm}

\bibitem{fr21}
M. Frazier and S. Roudenko,
Littlewood--Paley theory for matrix-weighted function spaces,
Math. Ann. 380 (2021), 487--537.

\vspace{-0.3cm}

\bibitem{fx11}
J. Fu and J. Xu,
Characterizations of Morrey type Besov and Triebel--Lizorkin spaces with variable exponents,
J. Math. Anal. Appl. 381 (2011), 280--298.

\vspace{-0.3cm}

\bibitem{gkkp17}
A. G. Georgiadis, G. Kerkyacharian, G. Kyriazis and P. Petrushev,
Homogeneous Besov and Triebel--Lizorkin spaces associated to non-negative self-adjoint operators,
J. Math. Anal. Appl. 449 (2017), 1382--1412.

\vspace{-0.3cm}

\bibitem{gkkp19}
A. G. Georgiadis, G. Kerkyacharian, G. Kyriazis and P. Petrushev,
Atomic and molecular decomposition of homogeneous spaces of distributions
associated to non-negative selfadjoint operators,
J. Fourier Anal. Appl. 25 (2019), 3259--3309.

\vspace{-0.3cm}

\bibitem{gn16}
A. G. Georgiadis and M. Nielsen,
Pseudo differential operators on mixed-norm Besov and Triebel--Lizorkin spaces,
Math. Nachr. 289 (2016), 2019--2036.

\vspace{-0.3cm}

\bibitem{g03}
M. Goldberg, Matrix $A_p$ weights via maximal functions,
Pacific J. Math. 211 (2003), 201--220.

\vspace{-0.3cm}

\bibitem{g14}
L. Grafakos, Classical Fourier Analysis, third edition, Grad. Texts in Math. 249, Springer,
New York, 2014.

\vspace{-0.3cm}

\bibitem{gwx23}
P. Guo, S. Wang and J. Xu,
Continuous characterizations of weighted Besov spaces of variable smoothness and integrability,
Filomat 37 (2023), 9913--9930.

\vspace{-0.3cm}

\bibitem{hsz24}
Y. He, Q. Sun and C. Zhuo,
Pointwise characterizations of variable Besov and Triebel--Lizorkin spaces via Haj{\l}asz gradients,
Fract. Calc. Appl. Anal. 27 (2024), 944--969.

\vspace{-0.3cm}

\bibitem{kv12}
H. Kempka and J. Vyb\'iral,
Spaces of variable smoothness and integrability:
characterizations by local means and ball means of differences,
J. Fourier Anal. Appl. 18 (2012), 852--891.

\vspace{-0.3cm}

\bibitem{l89}
H. G. Leopold,
On Besov spaces of variable order of differentiation,
Arch. Math. 53(2) (1989) 178--187.

\vspace{-0.3cm}

\bibitem{l89 2}
H. G. Leopold,
Interpolation of Besov spaces of variable order of differentiation,
Arch. Math. (Basel) 53(2) (1989), 178--187.

\vspace{-0.3cm}

\bibitem{l91}
H. G. Leopold,
On function spaces of variable order of differentiation,
Forum Math. 3(3) (1991), 1--21.

\vspace{-0.3cm}

\bibitem{ls96}
H. G. Leopold and E. Schrohe,
Trace theorems for Sobolev spaces of variable order of differentiation,
Math. Nachr. 179 (1996), 223--245.

\vspace{-0.3cm}

\bibitem{llor24}
A. Lerner, K. Li, S. Ombrosi and I. Rivera-R\'ios,
On the sharpness of some quantitative Muckenhoupt--Wheeden inequalities,
C. R. Math. Acad. Sci. Paris 362 (2024), 1253--1260.

\vspace{-0.3cm}

\bibitem{llor24 2}
A. Lerner, K. Li, S. Ombrosi and I. Rivera-R\'ios,
On some improved weighted weak type inequalities,
Ann. Sc. Norm. Super. Pisa Cl. Sci. (5) (2024),
https://doi.org/10.2422/2036-2145.202407012.

\vspace{-0.3cm}

\bibitem{m93}
Y. Meyer,
Wavelets and Operators, in: Different Perspectives on Wavelets, pp. 35--58,
Proc. Sympos. Appl. Math. 47, Amer. Math. Soc., Providence, RI, 1993.

\vspace{-0.3cm}

\bibitem{mns13}
S. Moura, J. Neves and C. Schneider,
On trace spaces of 2-microlocal Besov spaces with variable integrability,
Math. Nachr. 286 (2013), 1240--1254

\vspace{-0.3cm}

\bibitem{nptv17}
F. Nazarov, S. Petermichl, S. Treil and A. Volberg,
Convex body domination and weighted estimates with matrix weights,
Adv. Math. 318 (2017), 279--306.

\vspace{-0.3cm}

\bibitem{nt96}
F. Nazarov and S. R. Treil,
The hunt for a Bellman function: applications to estimates for singular integral operators and to other classical problems of harmonic analysis, (Russian),
translated from Algebra i Analiz 8 (1996), 32--162, St. Petersburg Math. J. 8 (1997), 721--824.

\vspace{-0.3cm}

\bibitem{n13}
M. Nielsen,
Summation of multiple Fourier series in matrix weighted $L_p$-spaces,
J. Math. 2013, Art. ID 135245, 7 pp.

\vspace{-0.3cm}

\bibitem{nr18}
M. Nielsen and M. G. Rasmussen,
Projection operators on matrix weighted $L^p$ and a simple sufficient Muckenhoupt condition,
Math. Scand. 123 (2018), 72--84.

\vspace{-0.3cm}

\bibitem{ns21}
M. Nielsen and H. {\v S}iki\'c,
Muckenhoupt matrix weights,
J. Geom. Anal. 31 (2021), 8850--8865.

\vspace{-0.3cm}

\bibitem{ns25}
M. Nielsen and H. {\v S}iki\'c,
Muckenhoupt matrix weights for general bases, in: The Mathematical Heritage of
Guido Weiss, pp. 361--386, Appl. Numer. Harmon. Anal., Birkh\"auser/ Springer, Cham, 2025.

\vspace{-0.3cm}

\bibitem{ns26}
M. Nielsen and H. {\v S}iki\'c,
Matrix weights on compact and non-compact domains,
J. Math. Anal. Appl. 555 (2026), Paper No. 130069, 23 pp.

\vspace{-0.3cm}

\bibitem{np25}
Z. Nieraeth and M. Penrod,
Matrix-weighted bounds in variable Lebesgue spaces,
Ann. Fenn. Math. 50 (2025), 519--548.

\vspace{-0.3cm}

\bibitem{n16}
T. Noi,
Trace and extension operators for Besov spaces and Triebel--Lizorkin spaces with variable exponents,
Rev. Mat. Complut. 29 (2016), 341--404.


\vspace{-0.3cm}

\bibitem{os24}
T. Ogawa and S. Shimizu,
Free boundary problems of the incompressible Navier--Stokes equations with non-flat initial surface in the critical Besov space,
Math. Ann. 390 (2024), 3155--3219.

\vspace{-0.3cm}

\bibitem{oao24}
F. Ouidirne, C. Allalou and M. Oukessou,
Analyticity for the fractional Navier--Stokes equations in critical
Fourier--Besov--Morrey spaces with variable exponents,
Bol. Soc. Parana. Mat. (3) 42 (2024), 11 pp.

\vspace{-0.3cm}

\bibitem{r03}
S. Roudenko,
Matrix-weighted Besov spaces,
Trans. Amer. Math. Soc. 355 (2003), 273--314.

\vspace{-0.3cm}

\bibitem{tx05}
L. Tang and J. Xu,
Some properties of Morrey type Besov-Triebel spaces,
Math. Nachr. 278 (2005), 904--917.

\vspace{-0.3cm}

\bibitem{tv97}
S. Treil and A. Volberg,
Wavelets and the angle between past and future, J. Funct. Anal.
143 (1997), 269--308.

\vspace{-0.3cm}

\bibitem{t91}
R. H. Torres,
Boundedness Results for Operators with Singular Kernels on Distribution Spaces,
Mem. Amer. Math. Soc. 90 (1991), no. 442, viii+172 pp.

\vspace{-0.3cm}

\bibitem{v97}
A. Volberg,
Matrix $A_p$ weights via $\mathcal{S}$-functions,
J. Amer. Math. Soc. 10 (1997), 445--466.

\vspace{-0.3cm}

\bibitem{whhy21}
F. Wang, Y. Han, Z. He and D. Yang,
Besov and Triebel--Lizorkin spaces on spaces of homogeneous
type with applications to boundedness of Calder\'on--Zygmund operators,
Dissertationes Math. 565 (2021), 1--113.

\vspace{-0.3cm}

\bibitem{wgx24}
S. Wang, P. Guo and J. Xu,
Characterizations of weighted Besov spaces with variable exponents,
Acta Math. Sin. (Engl. Ser.) 40 (2024), 2855--2878.

\vspace{-0.3cm}

\bibitem{wgx25}
S. Wang, P. Guo and J. Xu,
Precompact sets in matrix weighted Lebesgue spaces with variable exponent,
Georgian Math. J. 32 (2025), 1071--1083.

\vspace{-0.3cm}

\bibitem{wx22}
S. Wang and J. Xu,
Weighted Besov spaces with variable exponents,
J. Math. Anal. Appl. 505 (2022), Paper No. 125478, 27 pp.

\vspace{-0.3cm}

\bibitem{wm58}
N. Wiener and P. Masani,
The prediction theory of multivariate stochastic processes. II. The
linear predictor, Acta Math. 99 (1958), 93--137.

\vspace{-0.3cm}

\bibitem{x05}
J. Xu,
A characterization of Morrey type Besov and Triebel--Lizorkin spaces,
Vietnam J. Math. 33 (2005), 369--379.


\vspace{-0.3cm}

\bibitem{x08}
J.  Xu, Variable Besov and Triebel--Lizorkin spaces, Ann. Acad. Sci. Fenn.
Math. 33 (2008), 511--522.

\vspace{-0.3cm}

\bibitem{x08 2}
J.  Xu,
The relation between variable Bessel potential spaces and Triebel--Lizorkin spaces,
Integral Transforms Spec. Funct. 19 (2008), 599--605.

\vspace{-0.3cm}

\bibitem{xf12}
J. Xu and J. Fu,
Well-posedness for the 2D dissipative quasi-geostrophic equations in the Morrey type Besov space,
Math. Appl. (Wuhan) 25 (2012), 624--630.

\vspace{-0.3cm}

\bibitem{yy08}
D. Yang and W. Yuan,
A new class of function spaces connecting Triebel--Lizorkin spaces and Q spaces,
J. Funct. Anal. 255 (2008), 2760--2809.


\vspace{-0.3cm}

\bibitem{yyz25}
D. Yang, W. Yuan and Z. Zeng,
Variable Muckenhoupt $A_\infty$ weights, arXiv:2509.07786.

\vspace{-0.3cm}

\bibitem{yymz25}	
D. Yang, W. Yuan and M. Zhang,
Matrix-weighted Besov--Triebel--Lizorkin spaces of
optimal scale: Boundedness of pseudo-differential operators,
the trace operator, and Calder\'{o}n--Zygmund operators,
Proc. Steklov Inst. Math. (2025), DOI: 10.4213/tm4504.

\vspace{-0.3cm}

\bibitem{yzy15}
D. Yang, C. Zhuo and W. Yuan,
Besov-type spaces with variable smoothness and integrability,
J. Funct. Anal. 269 (2015), 1840--1898.

\vspace{-0.3cm}

\bibitem{ysy10}
W. Yuan, W. Sickel and D. Yang,
Morrey and Campanato Meet Besov, Lizorkin and Triebel,
Lecture Notes in Mathematics 2005, Springer-Verlag,
Berlin, 2010.

\vspace{-0.3cm}

\bibitem{zd23}
Z. Zeghad and D. Drihem,
Variable Besov-type spaces,
Acta Math. Sin. (Engl. Ser.) 39 (2023), 553--583.

\vspace{-0.3cm}

\bibitem{zy20}
C. Zhuo and D. Yang,
Variable Besov spaces associated with heat kernels,
Constr. Approx. 52 (2020), 479--523.

\end{thebibliography}
\end{document}